\documentclass{amsart}
\usepackage{amssymb}
\usepackage{mathrsfs}
\usepackage{stmaryrd}
\usepackage{imakeidx}
\usepackage[colorlinks = true]{hyperref}
\usepackage{enumitem}
\usepackage{comment}
\input xy \xyoption {all}

\usepackage{tikz}
\usetikzlibrary{cd}

\usepackage{dsfont}

\usepackage{rotating}
\usepackage{lscape}

\newcommand{\Z}{\mathbb{Z}}

\newcommand{\bk}{\Bbbk}

\newcommand{\Gm}{\mathbb{G}_{\mathrm{m}}}

\newcommand{\Mod}{\mathsf{Mod}}

\newcommand{\bX}{\mathbf{X}}

\newcommand{\hatstar}{\mathbin{\widehat{\star}}}
\newcommand{\cB}{\mathcal{B}}

\newcommand{\cN}{\mathcal{N}}

\newcommand{\cP}{\mathcal{P}}
\newcommand{\cS}{\mathcal{S}}
\newcommand{\cZ}{\mathcal{Z}}

\newcommand{\id}{\mathrm{id}}

\newcommand{\simto}{\xrightarrow{\sim}}

\DeclareMathOperator{\Hom}{Hom}
\DeclareMathOperator{\Ext}{Ext}
\DeclareMathOperator{\End}{End}

\newcommand{\Waff}{W_{\mathrm{aff}}}
\newcommand{\Saff}{S_{\mathrm{aff}}}
\newcommand{\Wext}{W_{\mathrm{ext}}}

\newcommand{\bG}{\mathbf{G}}
\newcommand{\bB}{\mathbf{B}}
\newcommand{\bP}{\mathbf{P}}
\newcommand{\bL}{\mathbf{L}}
\newcommand{\bU}{\mathbf{U}}
\newcommand{\bT}{\mathbf{T}}
\newcommand{\bg}{\mathbf{g}}
\newcommand{\bb}{\mathbf{b}}
\newcommand{\bt}{\mathbf{t}}
\newcommand{\bn}{\mathbf{n}}
\newcommand{\bbX}{\mathbb{X}}

\renewcommand{\hat}{\widehat}
\renewcommand{\tilde}{\widetilde}
\makeatletter
\def\lotimes{\@ifnextchar_{\@lotimessub}{\@lotimesnosub}}
\def\@lotimessub_#1{\mathchoice{\mathbin{\mathop{\otimes}^L}_{#1}}%
  {\otimes^L_{#1}}{\otimes^L_{#1}}{\otimes^L_{#1}}}
\def\@lotimesnosub{\mathbin{\mathop{\otimes}^L}}
\makeatother

\makeatletter
\def\lboxtimes{\@ifnextchar_{\@lboxtimessub}{\@lboxtimesnosub}}
\def\@lboxtimessub_#1{\mathchoice{\mathbin{\mathop{\boxtimes}^L}_{#1}}%
  {\boxtimes^L_{#1}}{\boxtimes^L_{#1}}{\boxtimes^L_{#1}}}
\def\@lboxtimesnosub{\mathbin{\mathop{\boxtimes}^L}}
\makeatother

\makeatletter
\def\hatotimes{\@ifnextchar_{\@hatotimessub}{\@hatotimesnosub}}
\def\@hatotimessub_#1{\mathchoice{\mathbin{\mathop{\widehat{\otimes}}}_{#1}}%
  {\widehat{\otimes}_{#1}}{\widehat{\otimes}_{#1}}{\widehat{\otimes}_{#1}}}
\def\@lotimesnosub{\mathbin{\mathop{\widehat{\otimes}}}}
\makeatother

\newcommand{\scO}{\mathscr{O}}

\newcommand{\cI}{\mathcal{I}}

\newcommand{\Rep}{\mathsf{Rep}}

\newcommand{\Coh}{\mathsf{Coh}}
\newcommand{\QCoh}{\mathsf{QCoh}}

\newcommand{\ft}{\mathfrak{t}}
\newcommand{\fn}{\mathfrak{n}}

\newcommand{\fb}{\mathfrak{b}}
\newcommand{\fg}{\mathfrak{g}}

\newcommand{\Ind}{\mathrm{Ind}}

\newcommand{\Ug}{\mathcal{U} \mathfrak{g}}

\newcommand{\fR}{\mathfrak{R}}
\newcommand{\fRs}{\mathfrak{R}^{\mathrm{s}}}

\newcommand{\scF}{\mathscr{F}}
\newcommand{\scG}{\mathscr{G}}
\newcommand{\scU}{\mathscr{U}}
\newcommand{\cU}{\mathcal{U}}
\newcommand{\ZHC}{Z_{\mathrm{HC}}}
\newcommand{\ZFr}{Z_{\mathrm{Fr}}}
\newcommand{\fC}{\mathfrak{C}}
\newcommand{\fD}{\mathfrak{D}}
\newcommand{\tfD}{\widetilde{\mathfrak{D}}}
\newcommand{\fE}{\mathfrak{E}}
\newcommand{\reg}{\mathrm{reg}}

\newcommand{\rs}{\mathrm{rs}}
\newcommand{\HCBim}{\mathsf{HC}}
\newcommand{\DBS}{\mathsf{D}_{\mathrm{BS}}}
\newcommand{\uw}{\underline{w}}

\newcommand{\Bim}{\mathrm{Bim}}
\newcommand{\bgreg}{\bg_{\mathrm{reg}}}
\newcommand{\bgrs}{\bg_{\mathrm{rs}}}

\newcommand{\bS}{\mathbf{S}}
\newcommand{\tbg}{\widetilde{\bg}}

\newcommand{\ext}{\mathrm{ext}}

\newcommand{\Spec}{\mathrm{Spec}}
\newcommand{\op}{\mathrm{op}}

\newcommand{\bbI}{\mathbb{I}}
\newcommand{\bbJ}{\mathbb{J}}

\newcommand{\Sim}{\mathsf{L}}
\newcommand{\Ver}{\mathsf{Z}}
\newcommand{\fin}{\mathrm{fg}}

\newcommand{\Lie}{\mathrm{Lie}}
\newcommand{\omu}{{\overline{\mu}}}
\newcommand{\ola}{{\overline{\lambda}}}
\newcommand{\tla}{{\widetilde{\lambda}}}
\newcommand{\tmu}{{\widetilde{\mu}}}
\newcommand{\AS}{\mathrm{AS}}
\newcommand{\tD}{\widetilde{\mathscr{D}}}
\newcommand{\Fr}{\mathrm{Fr}}
\newcommand{\diag}{\mathrm{diag}}

\newcommand{\WW}{\mathbf{W}}

\newcommand{\SSaff}{\mathbf{S}_{\mathrm{aff}}}

\newcommand{\sfP}{\mathsf{P}}
\newcommand{\sfM}{\mathsf{M}}
\newcommand{\sfC}{\mathsf{C}}
\newcommand{\sfQ}{\mathsf{Q}}

\numberwithin{equation}{section}
\numberwithin{figure}{section}
\newtheorem{thm}{Theorem}[section]
\newtheorem{lem}[thm]{Lemma}
\newtheorem{prop}[thm]{Proposition}
\newtheorem{cor}[thm]{Corollary}
\newtheorem{conj}[thm]{Conjecture}

\theoremstyle{definition}

\theoremstyle{remark}
\newtheorem{rmk}[thm]{Remark}

\title{Hecke action on the principal block}

\author[R. Bezrukavnikov]{Roman Bezrukavnikov}
\address{Department of Mathematics \\ Massachusetts Institute of Technology \\ Cambridge, MA \\ 02139 \\ USA.}
\email{bezrukav@math.mit.edu}

 \author[S.~Riche]{Simon Riche}
 \address{Universit\'e Clermont Auvergne, CNRS, LMBP, F-63000 Clermont-Ferrand, France.}
 \email{simon.riche@uca.fr}
 
\begin{document}

\begin{abstract}
In this paper we construct an action of the affine Hecke category (in its ``Soergel bimodules'' incarnation) on the principal block of representations of a simply-connected semisimple algebraic group over an algebraically closed field of characteristic bigger than the Coxeter number. This confirms a conjecture of G. Williamson and the second author, and provides a new proof of the tilting character formula in terms of antispherical $p$-Kazhdan--Lusztig polynomials.
\end{abstract}

\maketitle

\setcounter{tocdepth}{1}
\tableofcontents

%%%%%%%%%%%%%%%%%%%%%%%%%%%%%%%%%%%%%%%%%%%%%%%%
\section{Introduction}
%%%%%%%%%%%%%%%%%%%%%%%%%%%%%%%%%%%%%%%%%%%%%%%%

%---------------------------------------------
\subsection{Representation theory of reductive algebraic groups and the Hecke category}
%---------------------------------------------

Let $G$ be a connected reductive algebraic group over an algebraically closed field $\bk$ of characteristic $p$ (assumed to be strictly bigger than the Coxeter number of $G$), and let $\Waff$ be the associated affine Weyl group. A choice of a Borel subgroup $B$ in $G$ determines a subset $\Saff \subset \Waff$ of ``simple reflections" such that $(\Waff, \Saff)$ is a Coxeter system. It has been expected for long (following ideas of Verma~\cite{verma} later expanded by Lusztig~\cite{lusztig} in particular) that the combinatorics of the category $\Rep(G)$ of finite-dimensional algebraic $G$-modules should be expressible in terms of the Kazhdan--Lusztig combinatorics of $(\Waff, \Saff)$. This expectation is now known to be exact if $p$ is large (thanks to work of Kazhdan--Lusztig, Kashiwara--Tanisaki and Andersen--Jantzen--Soergel, or a later version of Fiebig), but not for some smaller $p$'s (as shown by Williamson); see~\cite{rw} for details and references.

In~\cite{rw}, G. Williamson and the second author of the present paper started advocating the idea that the combinatorics of $\Rep(G)$ should rather be expressed in terms of the \emph{$p$-Kazhdan--Lusztig} combinatorics, introduced a few years before by G. Williamson (partly in collaboration, see~\cite{jmw,jw}) as some ``combinatorial shadow" of the Hecke category $\DBS$ over $\bk$ attached to $(\Waff, \Saff)$. (Here by ``Hecke category" we mean the diagrammatic category introduced by Elias--Williamson in~\cite{ew}; this category is closely related with those of Soergel bimodules and parity complexes on flag varieties.)
%This principle can be seen as a further extension of ideas of Verma~\cite{verma} and a way to correct Lusztig's conjecture on characters of simple modules~\cite{lusztig}, which is now known to hold in large characteristics but not in full generality (see~\cite{rw} for details and references).
In that paper it was in particular observed that a concrete incarnation of this idea (a character formula for indecomposable tilting modules in the principal block, in terms of antispherical $p$-Kazhdan--Lusztig polynomials) was a consequence of the following conjecture, where for $s \in \Saff$ we denote by $B_s$ the object of $\DBS$ naturally associated with $s$.

\begin{conj}[\cite{rw}]
\label{conj:intro}
There exists a ($\bk$-linear) right action of the monoidal category $\DBS$ on the principal block $\Rep_0(G)$ of $\Rep(G)$ such that for any $s \in \Saff$ the object $B_s$ acts via a functor isomorphic to the wall-crossing functor associated with $s$.
\end{conj}

%a simple ``categorical'' conjecture stating that the wall-crossing functors define an action of the Hecke category of the associated affine Weyl group (in its ``Soergel bimodules'' incarnation) on the principal block of $G$. This 
The formulation of Conjecture~\ref{conj:intro} was motivated in particular by the philosophy of categorical actions of Lie algebras; it was proved in~\cite[Part~II]{rw} (and independently by Elias--Losev~\cite{el}) in the special case when $G=\mathrm{GL}_n(\bk)$ for some $n$, using the machinery of $2$-Kac--Moody algebras~\cite{rouquier}.

Later the ``combinatorial'' consequence of Conjecture~\ref{conj:intro} (the tilting character formula) was proved for general $G$ (in fact, in two very different ways, see~\cite{amrw2} and~\cite{rw-smith}), but these proofs use other tools, and none of them implies the original categorical conjecture. The main result of the present paper is a proof of Conjecture~\ref{conj:intro} (which, as explained above, provides in particular a third general proof of the tilting character formula). Contrarily to the other approaches to such questions, our proof does not involve constructible sheaves in any way; it uses coherent sheaves, but mostly over affine schemes, and hence can be considered essentially algebraic.

\subsection{Localization for Harish-Chandra bimodules}
\label{ss:intro-localization}
%---------------------------------------------

The action from Conjecture~\ref{conj:intro} will be constructed from the natural action of Harish-Chandra $\Ug$-bimo\-dules (where $\Ug$ is the enveloping algebra of the Lie algebra $\fg$ of $G$), i.e.~$G$-equivariant finitely generated $\Ug$-bimodules on which the diagonal action of $\Ug$ is the differential of the $G$-action (see~\S\ref{ss:HCBim} for details). Namely, the category $\Rep(G)$ can be naturally seen as a full subcategory of the category of $G$-equivariant $\Ug$-modules via differentiation. Thus it admits an action of the monoidal category of Harish-Chandra bimodules; moreover, wall-crossing functors (and, more generally, translation functors) can be described as the action of certain specific completed Harish-Chandra bimodules. 

More specifically, recall that (under suitable assumptions on $p$) the center of $\Ug$ identifies with functions on the fiber product
\[
 \fg^{*(1)} \times_{\ft^{*(1)}/W} \ft^*/(W,\bullet)
\]
where the superscript $(1)$ denotes Frobenius twist, $\ft$ is the Lie algebra of a maximal torus $T$ contained in $B$, and $W$ is the Weyl group of $(G,T)$, acting on $\ft^*$ via the ``dot action'' $\bullet$. The subalgebra $\scO(\fg^{*(1)})$ is realized as the ``Frobenius center'' and the subalgebra $\scO(\ft^*/(W,\bullet))$ as the ``Harish-Chandra center."
For $\lambda, \mu \in X^*(T)$, in~\S\S\ref{ss:HCBim-completed}--\ref{ss:monoidal-bimodules} we will construct a certain monoidal category
\[
 \HCBim^{\hat{\lambda},\hat{\mu}}
\]
of ``completed Harish-Chandra bimodules,'' which are certain $G$-equivariant finitely generated modules over
\[
 (\Ug \otimes_{\scO(\fg^{*(1)})} \Ug^\op) \otimes_{\scO(\ft^*/(W,\bullet) \times_{\ft^{*(1)}/W} \ft^*/(W,\bullet))} \scO(\ft^*/(W,\bullet) \times_{\ft^{*(1)}/W} \ft^*/(W,\bullet))^{\hat{\lambda},\hat{\mu}},
\]
where $\scO(\ft^*/(W,\bullet) \times_{\ft^{*(1)}/W} \ft^*/(W,\bullet))^{\hat{\lambda},\hat{\mu}}$ is the completion of $\scO(\ft^*/(W,\bullet) \times_{\ft^{*(1)}/W} \ft^*/(W,\bullet))$ with respect to the maximal ideal determined by $(\lambda,\mu)$. Once this definition is in place,
to construct the desired action it therefore suffices to construct a monoidal functor from the Hecke category $\DBS$ to the category $\HCBim^{\hat{0},\hat{0}}$ sending each $B_s$ to an object isomorphic to the completed bimodule realizing the corresponding wall-crossing functor. This will be realized in Theorem~\ref{thm:action}.

The main tool we will use for this construction is a localization theory for Harish-Chandra bimodules. 
%Recall that $\Ug$ possesses a large central subalgebra $\ZFr$ called the \emph{Frobenius center} and isomorphic to $\scO(\fg^{*(1)})$ (where the superscript $(1)$ denotes Frobenius twist). 
Even though in the end we are interested in $G$-modules, which when seen as $\Ug$-modules have a \emph{trivial} Frobenius central character, we will localize our bimodules on the \emph{regular part} of $\scO(\fg^{*(1)})$, and more precisely on a Kostant section $\bS^{*} \subset \fg^{*(1)}$ to the (co-)adjoint quotient. 
%This is made possible by the fact that the appropriate restriction to such a Kostant section is fully faithful on the ``diagonally induced'' Harish-Chandra bimodules that describe translation functors, see Proposition~\ref{prop:rest-S-fully-faithful}. Hence the bimodules we are interested in can be studied using 
%We will therefore consider bimodules over 
We will therefore set
$\cU_\bS \fg:=\Ug \otimes_{\scO(\fg^{*(1)})} \scO(\bS^{*})$, and consider a certain group scheme over $\ft^*/(W,\bullet) \times_{\ft^{*(1)}/W} \ft^*/(W,\bullet)$ constructed out of the universal centralizer group scheme $\mathbb{J}^*_\bS$ over $\bS^* \cong \ft^{*(1)}/W$.
%(Here $\ft$ is the Lie algebra of a maximal torus $T$ in $G$, $W$ is the Weyl group of $(G,T)$, $\scO(\ft^*/(W,\bullet))$ is realized as the ``Harish-Chandra center" in $\Ug$, and the map $\ft^*/(W,\bullet) \to \ft^{*(1)}/W$ is induced by the Artin--Schreier map $\ft^* \to \ft^{*(1)}$.) More specifically, 
For $\lambda, \mu \in X^*(T)$, we will define a certain monoidal category $\HCBim_\bS^{\hat{\lambda},\hat{\mu}}$ of finitely generated equivariant modules over
\begin{multline*}
\cU_\bS^{\hat{\lambda},\hat{\mu}} :=
 (\cU_\bS\fg \otimes_{\scO(\bS^*)} \cU_\bS\fg^\op) \otimes_{\scO(\ft^*/(W,\bullet) \times_{\ft^{*(1)}/W} \ft^*/(W,\bullet))} \\
 \scO(\ft^*/(W,\bullet) \times_{\ft^{*(1)}/W} \ft^*/(W,\bullet))^{\hat{\lambda},\hat{\mu}};
\end{multline*}
%for the completion $\cU_\bS^{\hat{\lambda},\hat{\mu}}$ of $\cU_\bS \fg \otimes_{\scO(\bS^*)} \cU_\bS \fg$ at the central character determined by the image of $(\lambda,\mu)$ in $\ft^*/(W,\bullet) \times_{\ft^{*(1)}/W} \ft^*/(W,\bullet)$; 
see~\S\ref{ss:HCBim-S} for the precise definition. By construction we have a natural monoidal functor
\[
 \HCBim^{\hat{\lambda},\hat{\mu}} \to \HCBim_\bS^{\hat{\lambda},\hat{\lambda}},
\]
and the main result of Section~\ref{sec:HCBim} (Proposition~\ref{prop:rest-S-fully-faithful}) states that this functor is fully faithful on a certain subcategory $\HCBim^{\hat{\lambda},\hat{\mu}}_\diag$ of ``diagonally induced'' bimodules which contains the objects realizing translation and wall-crossing functors.

It is a classical observation that $\cU_\bS \fg$ is an Azumaya algebra over its center, see Proposition~\ref{prop:Ureg-Azumaya} for details and references; as a consequence, the category of (finitely generated) bimodules over this algebra such that the left and right actions of its center coincide is equivalent to the category of (finitely generated) modules over this center (see~\S\ref{ss:Azumaya-alg} for details and references). This property is not directly applicable to our problem, since the two actions of this center on Harish-Chandra bimodules do \emph{not} coincide in general; however by using bimodules realizing translation to and from the ``most singular'' Harish-Chandra character (namely, the opposite of the half-sum of the positive roots), we construct in Section~\ref{sec:localization} an equivariant splitting bundle for $\cU_\bS^{\hat{\lambda},\hat{\mu}}$ in case $\lambda$ and $\mu$ belong to the lower closure of the fundamental alcove. As a consequence, for such $\lambda,\mu$ we obtain an equivalence of categories between $\HCBim_\bS^{\hat{\lambda},\hat{\mu}}$ and the category of coherent representations of the pullback of $\mathbb{J}^*_\bS$ to the spectrum of $\scO(\ft^*/(W,\bullet) \times_{\ft^{*(1)}/W} \ft^*/(W,\bullet))^{\hat{\lambda},\hat{\mu}}$;
%with respect to the maximal ideal determined by $(\lambda,\mu)$;
see Corollary~\ref{cor:equiv-ModUg-ModZ}.

A general theory of localization for modules over $\Ug$ has been developed by the first author with Mirkovi{\'c} and Rumynin, see~\cite{bmr,bmr2,bm}. The localization that we require here is however slightly different, and the present paper does not formally rely on any substantial result from~\cite{bmr,bmr2,bm}. One difference is that we are interested not in modules but in \emph{bimodules}, which are equivariant for the diagonal $G$-action. Some of the constructions in~\cite{bmr,bmr2,bm} (in particular, the non-canonicity of the choice of splitting bundle) make this theory difficult to use in an equivariant setting, and our construction is slightly different. Finally, as explained above we only need to consider the \emph{regular} part of the Frobenius center, which simplifies the situation a lot, and in particular allows us to work completely at the level of \emph{abelian} categories, without having to consider the more involved \emph{derived} categories.

%-------------------------------------------
\subsection{The Hecke category and representations of the universal centralizer}
\label{ss:intro-Hecke-cat}
%-------------------------------------------

The other crucial ingredient of our proof is a new incarnation of the Hecke category (for any Coxeter system $(\mathcal{W},\mathcal{S})$) recently found by Abe~\cite{abe}. 

The Hecke category is a categorification of the Hecke algebra of $(\mathcal{W},\mathcal{S})$, depending on a choice of extra data (comprising a representation $V$ of $\mathcal{W}$), which admits several different incarnations. An early definition of this category in terms of Soergel bimodules~\cite{soergel} applies to ``reflection faithful'' representations of Coxeter systems, which include natural examples of representations over fields of characteristic $0$ (e.g.~geometric representations of finite Coxeter systems and representations appearing in the theory of Kac--Moody Lie algebras for crystallographic Coxeter systems), but does \emph{not} include important examples over fields of positive characteristic (e.g.~some natural representations of affine Weyl groups of reductive groups).
%is carried out over a characteristic zero field for a finite Coxeter group. 
Under this assumption Soergel bimodules can be
defined as a full subcategory of the category of graded bimodules over the polynomial algebra
$\scO(V)$. More recently Elias and Williamson~\cite{ew} have
proposed a definition of the Hecke category in terms of generators and relations which applies (and behaves as one might expect) in a much greater generality, encompassing the representation of the affine Weyl group that we require. It is in terms of this construction that Conjecture~\ref{conj:intro} was stated. (For more on the Hecke category, see also~\cite{williamson, jw}.)

The main drawback of the construction in~\cite{ew}, however, is that it is much less concrete than Soergel's original definition, and does not involve $\scO(V)$-bimodules.
This drawback is exactly compensated by Abe's work; under minor technical assumptions he proves in~\cite{abe} that the category of Elias and Williamson identifies with a category of ``enhanced Soergel bimodules,'' which are certain graded bimodules over $\scO(V)$ endowed with a decomposition of its tensor product with $\mathrm{Frac}(\scO(V))$ (on the right) parametrized by $\mathcal{W}$.
%has subsequently been drastically simplified by Abe. 

Based on Abe's work, in the case of the affine Weyl group acting on $X^*(T) \otimes_\Z \bk$ through the natural action of $W$, we realize the Hecke category as a full subcategory in ($\Gm$-equivariant) coherent representations of the pullback of $\bbJ^*_\bS$ to
%coherent sheaves on the preimage of the Kostant slice in the (Frobenius twist of the) Steinberg variety of triples,\footnote{Here, by Steinberg variety of triples we mean the fiber product of two copies of the Grothendieck resolution over the dual of the Lie algebra, and not the version involving the Springer resolution.} identified with 
$\ft^{*(1)} \times_{\ft^{*(1)}/W} \ft^{*(1)}$; see
%, equivariant with respect to the pullback of the universal centralizer; see 
Theorem~\ref{thm:Hecke-centralizer}.
%The other crucial ingredient in our approach is essentially a restatement (in our particular case) of the main result of~\cite{abe}. Namely, we observe that Abe's realization of the Hecke category can be described in terms of representations of the pullback of the universal centralizer to $\ft^{*(1)} \times_{\ft^{*(1)}/W} \ft^{*(1)}$. 
This construction allows us to define a monoidal functor from the Hecke category to the category of representations considered in~\S\ref{ss:intro-localization}, and then (using our localization theorem) a monoidal functor 
\begin{equation}
\label{eqn:functor-intro}
\DBS \to \HCBim_\bS^{\hat{0},\hat{0}}.
\end{equation}
(This construction applies more generally for the category $\HCBim_\bS^{\hat{\lambda},\hat{\lambda}}$ when $\lambda$ belongs to the fundamental alcove; in this case natural \'etale maps allow us to identify the completions of the schemes $\ft^*/(W,\bullet) \times_{\ft^{*(1)}/W} \ft^*/(W,\bullet)$ and $\ft^{*(1)} \times_{\ft^{*(1)}/W} \ft^{*(1)}$ at the images of $(\lambda,\lambda)$.)

\begin{rmk}
 Although the concrete incarnation of this idea that is relevant in the present paper is new, the fact that affine Soergel bimodules are closely related with representations of the universal centralizer was already known: it dates back (at least) to~\cite{dodd}; see also~\cite{mr} for an adaptation of these ideas to positive characteristic coefficients.
\end{rmk}

At this point, to conclude our proof it only remains to show that our functor~\eqref{eqn:functor-intro}
sends the objects of the Hecke category labelled by simple reflections to the images in $\HCBim_\bS^{\hat{0},\hat{0}}$ of the bimodules realizing the wall-crossing functors. (In fact, this property will also ensure that the functor takes values in the essential image of our fully faithful functor $\HCBim_\diag^{\hat{0},\hat{0}} \to \HCBim_\bS^{\hat{0},\hat{0}}$, which will imply that it ``lifts'' to a functor from $\DBS$ to $\HCBim^{\hat{0},\hat{0}}$.) In the case when the simple reflection belongs to the finite Weyl group $W$, this can be checked explicitly, using localization at a character involving a weight on the corresponding wall of the fundamental alcove; see Proposition~\ref{prop:image-Bs-finite}. The general case is reduced to this one using a standard trick (used e.g.~in~\cite{riche,bm}), based on the observation that in the \emph{extended} affine Weyl group each simple reflection is conjugate to a simple reflection which belongs to $W$. The concrete proof involves the study of an analogue of the affine braid group action from~\cite{br} in our present context; in this case the situation simplifies however (once again because we work over the regular part of the Frobenius center) and this action in fact factors through an action of the extended affine Weyl group.

\begin{rmk}
One of the motivations for Abe's work~\cite{abe} was an attempt to prove Conjecture~\ref{conj:intro}. What Abe was actually able to construct is rather an action on the principal block of the category of $G_1 T$-modules (where $G_1$ is the first Frobenius kernel of $G$), which is interesting but less applicable to character computations; see~\cite{abe2}.
\end{rmk}

%-------------------------------------------
\subsection{Towards a coherent realization of the Hecke category}
%-------------------------------------------

Thanks to work of Kazhdan--Lusztig~\cite{kl2} and Ginzburg~\cite{cg}, it is known that the Hecke algebra of $(\Waff,\Saff)$ identifies with the Grothendieck group of the category of equivariant coherent sheaves on the Steinberg variety of triples $\mathrm{St}$.\footnote{Here, by \emph{Steinberg variety of triples} we mean the fiber product of two copies of the Grothendieck resolution over the dual of the Lie algebra, and not the version involving the Springer resolution. This distinction is not important for the results of~\cite{kl2,cg}, but it is for our considerations here.} The fiber product $\ft^{*(1)} \times_{\ft^{*(1)}/W} \ft^{*(1)}$ considered  in~\S\ref{ss:intro-Hecke-cat} identifies with the preimage of the Kostant section $\bS^*$ in the Frobenius twist of $\mathrm{St}$, and the construction of~\S\ref{ss:intro-Hecke-cat} can be seen to provide a fully faithful monoidal functor from the Hecke category to the category of equivariant coherent sheaves on the \emph{regular part} of $\mathrm{St}$. In future work we will upgrade this construction to a fully faithful monoidal functor to the category of equivariant coherent sheaves on the whole Steinberg variety.\footnote{We understand that Ivan Losev has found a different proof of this statement, also based on the results of the present paper.} This construction will be part of our project (partly joint with L. Rider) of constructing a ``modular" version of the equivalence constructed by the first author in~\cite{bez}; see~\cite{brr} for a first step towards this goal.

%\begin{rmk}
%We understand that Ivan Losev has found a different proof of this statement, also based on the results of the present paper.
%\end{rmk}

%-------------------------------------------
\subsection{Contents}
%-------------------------------------------

In Section~\ref{sec:Hecke-univ-centralizer} we recall Abe's results, and use them to construct our monoidal functor from the Hecke category to the appropriate category of representations of the universal centralizer.
In Section~\ref{sec:HCBim} we introduce the categories of completed Harish-Chandra bimodules we will work with, and prove that restriction to a Kostant section is fully faithful on diagonally-induced bimodules.
In Section~\ref{sec:localization} we develop our localization theory for Harish-Chandra bimodules.
In Section~\ref{sec:Ug-D} we prove (for later use) some technical results using the relation between $\Ug$ and differential operators on the flag variety of $G$. In
Section~\ref{sec:Hecke-action} we prove the main result of the paper, i.e.~we construct the Hecke action on the principal block and prove that objects associated with simple reflections act via wall-crossing functors.
Finally, Appendix~\ref{sec:index} contains an index of the main notation used in the paper.

%---------------------------------------------
\subsection{Acknowledgements}
%---------------------------------------------

We thank Ivan Losev for his active interest in our work and useful discussions on the subject of this paper, and the referee for his/her careful reading and useful suggestions, which both helped improve the exposition of the paper.

While we were in the final stages of our work, we were informed that J.~Ciappara had obtained a radically different proof of Conjecture~\ref{conj:intro}. His construction of the action relies on the geometric Satake equivalence and the Smith--Treumann theory of~\cite{rw-smith}; these results have now appeared in~\cite{ciappara}. We thank G. Williamson for keeping us informed of this work, and for useful comments.

The project realized in this paper was conceived during the workshop \emph{Modular Representation Theory} organized in Oxford by the Clay Mathematical Institute. We thank the organizers of this conference (G.~Williamson, I.~Losev and M.~Emerton) for the fruitful working atmosphere they have created, and the participants for the inspiring talks. A later part of this work was accomplished at the Institut Henri Poincar\'e, as part of the thematic program \emph{Representation Theory} organized by D.~Hernandez, N.~Jacon, E.~Letellier, S.~R., and \'E. Vasserot, and an even later part was done while both authors were members of the Institute for Advanced Study during the \emph{Special Year on Geometric and Modular Representation Theory} organized by G. Williamson.

R.B.~was supported by NSF Grant No.~DMS-1601953, and his work was partly supported by grants
from the Institute for Advanced Study and Carnegie Corporation of New York.
This project has received funding from the European Research Council (ERC) under the European Union's Horizon 2020 research and innovation programme (S.R., grant agreements No. 677147 and 101002592).

%%%%%%%%%%%%%%%%%%%%%%%%%%%%%%%%
\section{The affine Hecke category and representations of the regular centralizer}
\label{sec:Hecke-univ-centralizer}
%%%%%%%%%%%%%%%%%%%%%%%%%%%%%%%%

In this section we explain that the affine Hecke category attached to a connected reductive algebraic group $\bG$ can be described as a category of representations of (a pullback of) the universal centralizer attached to $\bG$. Our main tool will be a description of the Hecke as a category of ``enhanced Soergel bimodules'' recently obtained by Abe~\cite{abe}. In later sections the group $\bG$ will be chosen as the Frobenius twist of the group $G$ appearing in Conjecture~\ref{conj:intro}; however this constructions applies in a slightly more general context, and might be of independent interest.

%--------------------------------------------------------------
\subsection{The affine Weyl group and the associated Hecke category}
\label{ss:Hecke-cat}
%--------------------------------------------------------------

We let $\bk$ be an algebraically closed field of characteristic $p$ (possibly equal to $0$), and 
$\bG$ be a connected reductive algebraic group over $\bk$. We fix a Borel subgroup $\bB \subset \bG$ and a maximal torus $\bT \subset \bB$. The Lie algebras of $\bG$, $\bB$, $\bT$ will be denoted $\bg$, $\bb$ and $\bt$ respectively. 
We set $\bX:=X^*(\bT)$, resp.~$\bX^\vee := X_*(\bT)$, and denote by $\Phi \subset \bX$, resp.~$\Phi^\vee \subset \bX^\vee$, the root system, resp.~coroot system, of $(\bG,\bT)$. The canonical bijection $\Phi \simto \Phi^\vee$ will be denoted $\alpha \mapsto \alpha^\vee$. The choice of $\bB$ determines a subset $\Phi^+ \subset \Phi$ of positive roots, consisting of the $\bT$-weights in $\bg/\bb$; the corresponding basis of $\Phi$ will be denoted $\Phi^{\mathrm{s}}$. 
In this section we will make the following assumptions:
\begin{enumerate}
\item $p$ is good for $\bG$;
%, and $\ell \geq 5$;
\item 
\label{it:ass-torsion}
neither $\bX/\Z\Phi$ nor $\bX^\vee/\Z\Phi^\vee$ has $p$-torsion;
\item there exists a $\bG$-equivariant isomorphism $\bg \simto \bg^*$.
\end{enumerate}
For simplicity we will fix once and for all a $\bG$-equivariant isomorphism $\kappa : \bg \simto \bg^*$. 
By equivariance this also provides an identification of $\bt$ and $\bt^*$ (where $\bt^*$ is identified with the subspace in $\bg^*$ consisting of linear forms that vanish on all root subspaces).
%Note that the composition $\bt \hookrightarrow \bg \xrightarrow{\kappa} \bg^* \to \bt^*$ is automatically is an isomorphism (where the first and third morphisms are the obvious one), which allows to also identify $\bt$ and $\bt^*$.

Let $\WW=N_{\bG}(\bT)/\bT$ be the Weyl group of $(\bG,\bT)$.
The associated \emph{affine Weyl group} is the semi-direct product
\[
 \Waff := \WW \ltimes \Z \Phi
\]
where $\Z \Phi \subset \bX$ is the lattice generated by the roots. For $\lambda \in \Z\Phi$ we will denote by $t_\lambda$ the image of $\lambda$ in $\Waff$. It is well known that $\Waff$ is generated by the subset $\Saff$ consisting of the reflections $s_\alpha$ with $\alpha \in \Phi^{\mathrm{s}}$, together with the products $t_\beta s_\beta$ where $\beta \in \Phi$ is such that $\beta^\vee$ is a maximal coroot. Moreover, the pair $(\Waff,\Saff)$ is a Coxeter system, see~\cite[\S II.6.3]{jantzen}.
%There is a classical way of 
We will sometimes need to ``enlarge'' this group by considering translations by all elements of $\bX$. Namely, the \emph{extended affine Weyl group} is the semi-direct product
\[
 \Wext := \WW \ltimes \bX.
\]
Then $\Waff$ is a normal subgroup in $\Wext$. 

%The length function for the Coxeter structure on $\WWaff$ considered above is well known to be given by
%\[
%\ell(wt_\lambda) = \sum_{\substack{\alpha \in \Phi^+ \\ w(\alpha) \in \Phi^+}} |\langle \lambda, \alpha^\vee \rangle| + \sum_{\substack{\alpha \in \Phi^+ \\ w(\alpha) \in -\Phi^+}} |1+\langle \lambda, \alpha^\vee \rangle|
%\]
%for $w \in \WW$ and $\lambda \in \Z\mathfrak{R}$. This formula makes sense more generally for $\lambda \in \bX$ (where $t_\lambda$ will denote the image of $\lambda$ in $\WWext$), which allows to extend $\ell$ to the group $\WWext$. If we set
%\[
% \Omega := \{ w \in \WWext \mid \ell(w)=0 \},
%\]
%then it is well known that the action of $\Omega$ on $\WWaff$ by conjugation is via Coxeter group automorphisms, and that multiplication induces a group isomorphism
%\[
% \Omega \ltimes \WWaff \simto \WWext.
%\]

We will consider the balanced
%action of $\Waff$ on $R:=\mathrm{S}(\ft)$ via the quotient $\Waff \to W$. There is an associated 
``realization'' of $\Waff$ over $\bk$ (in the sense of~\cite{ew}) defined as follows:
\begin{itemize}
\item
the underlying $\bk$-vector space is $\bt^*$;
\item
if $\alpha \in \Phi^{\mathrm{s}}$ and $s=s_\alpha$, then
the ``root'' $\alpha_s \in \bt$ (resp.~``coroot'' $\alpha_s^\vee \in \bt^*$) associated with $s$ is the differential of $\alpha^\vee$ (resp.~of $\alpha$);
\item
if $\beta \in \Phi^+$ is such that $\beta^\vee$ is a maximal coroot and $s=t_\beta s_\beta$ then
the ``root'' $\alpha_s$ (resp.~``coroot'' $\alpha_s^\vee$) associated with $s$ is the differential of $-\beta^\vee$ (resp.~of $-\beta$).
\end{itemize}
This realization is an example of a Cartan realization in the sense of~\cite[\S 10.1]{amrw}. Our assumption~\eqref{it:ass-torsion} implies that this realization satisfies the ``Demazure surjectivity'' condition of~\cite{ew}. 
%One can also check by explicit computation that a Cartan realization always satisfies the condition~\cite[Assumption~1.1]{abe3}; therefore the results of~\cite{abe} are applicable in this setting, see~\cite[Theorem~1.3]{abe3}. 
There is an associated action of $\Waff$ on $\bt^*$, which simply is the natural action of $\WW$, seen as an action of $\Waff$ via the projection $\Waff \to \WW$.

We will denote by $\DBS$ the ``diagrammatic Hecke category'' defined by Elias--Williamson~\cite{ew} for the Coxeter system $(\Waff,\Saff)$ and this choice of realization. (For a discussion of this definition, see also~\cite[Chap.~2]{amrw}.) 
The technical conditions necessary for this category to be defined (and well behaved) are somewhat subtle, and not all of them are made explicit in~\cite{ew}; see~\cite[\S 5]{ew2} for a detailed discussion of this question. As explained in~\cite[\S 5.1]{ew2}, a sufficient condition (in addition to the fact that the data define a realization satisfying Demazure surjectivity) that ensures that all the results of~\cite{ew} are applicable is that for any pair of distinct simple reflections $s,t$ such that $st$ has finite order, the restriction of the action to the subgroup generated by $s$ and $t$ is faithful. It follows from~\cite[Lemma~8.1.1]{amrw} that this condition is satisfied in our context, except possibly if $p=2$ and $st=ts$. (The assumptions of~\cite[Lemma~8.1.1]{amrw} are easily checked by hand for Cartan realizations in good characteristic.) However in this case $s$ and $t$ act nontrivially by Demazure surjectivity, and $st$ acts nontrivially since $\alpha_s^\vee$ and $\alpha_t^\vee$ are not colinear thanks to assumption~\eqref{it:ass-torsion}.

The category $\DBS$ is a $\bk$-linear (non additive) monoidal category.
By definition its objects are pairs $(\uw,n)$ where $\uw$ is a word in $\Saff$ and $n \in \Z$; the product is given by concatenation of words and addition of integers, and for any words $\uw,\uw'$ the direct sum of morphism spaces
\[
\bigoplus_{n \in \Z} \Hom_{\DBS}((\uw,0),(\uw',n))
\]
is a graded bimodule over $R:=\scO(\bt^*) = \mathrm{Sym}(\bt)$ (where the grading is such that elements in $\bt$ have degree $2$). Following usual conventions, the object $(\uw,n)$ will rather be denoted $B_{\uw}(n)$. Then there exists a natural ``grading shift'' autoequivalence of $\DBS$ such that $(B_{\uw}(n))(1)=B_{\uw}(n+1)$ for any $\uw$ and any $n \in \Z$.

\begin{rmk}
The Hecke category $\DBS$ (as, more generally, Hecke categories attached to Cartan realizations of crystallographic Coxeter systems) admits an incarnation in terms of parity complexes on a flag variety; see~\cite[Part~III]{rw}. Although important for some other purposes, this realization of the Hecke category will not play any role in the present paper. (The relation between Soergel bimodules and constructible sheaves on flag varieties was first obtained, in a characteristic-$0$ context, by Soergel~\cite{soergel-philo} in the case of finite crystallographic groups---using the earlier definition of Soergel bimodules in this case in~\cite{soergel-comb}---and by H\"arterich for Kac--Moody groups~\cite{haerterich}.)
%In the case of Coxeter systems associated with Kac--Moody groups (i.e.~the \emph{crystallographic} Coxeter systems), and for representations appearing naturally in this theory over a field of characteristic $0$, it is well known that the category of Soergel bimodules can also be described in terms of constructible sheaves on the corresponding flag variety. (The fact that Kazhdan--Lusztig combinatorics in this case is related to constructible sheaves on the flag variety is a fundamental observation of Kazhdan--Lusztig~\cite{kl}. The equivalence of categories with Soergel bimodules is due to Soergel~\cite{soergel-philo} in the case of finite crystallographic groups---using the earlier definition of Soergel bimodules in this case in~\cite{soergel-comb}---and to H\"arterich for Kac--Moody groups~\cite{haerterich}.) A similar result relating the Hecke category of~\cite{ew} to constructible sheaves on the associated flag variety (which holds also for coefficients in a field of positive characteristic) can be found in~\cite[Part~III]{rw}. Although important for some other purposes, this realization of the Hecke category will not play any role in the present paper.
\end{rmk}

%--------------------------------------------------------------
\subsection{Abe's incarnation of the Hecke category}
\label{ss:abes-incarnation}
%--------------------------------------------------------------

%The proof of Theorem~\ref{thm:action} will use 
For our present purposes we will need
a description of $\DBS$ in terms of $R$-bimodules due to Abe~\cite{abe}, which is close to the definition of Soergel bimodules~\cite{soergel}, and which we now recall. Once again, in order to apply these results one needs some technical assumptions. A sufficient condition (in terms of vanishing of quantum binomial coefficients) for the results of~\cite{abe} to be applicable is given in~\cite{abe3}. One can check by explicit computation that this condition is automatically satisfied for Cartan realizations.
%In order to be able to use this description, we will have to make one more assumption, namely that for any distinct $s,s' \in \Saff$, the representation of the subgroup of $\Waff$ generated by $\{s,s'\}$ on $\bt$ is reflection faithful. (EXPLAIN RESTRICTION ON $p$.)

We will denote by $Q$ the fraction field of $R$. Following~\cite{abe}, we denote by $\mathsf{C}'$ the category defined as follows. Objects are pairs consisting of a graded $R$-bimodule $M$ together with a decomposition
\begin{equation}
\label{eqn:dec-bimodule}
M \otimes_R Q = \bigoplus_{w \in \Waff} M_Q^w
\end{equation}
as $(R,Q)$-bimodules
such that:
\begin{itemize}
\item there exist only finitely many $w$'s such that $M_Q^w \neq 0$;
\item for any $w \in \Waff$, $r \in R$ and $m \in M_Q^w$ we have
\begin{equation}
\label{eqn:action-w}
m \cdot r=w(r) \cdot m.
\end{equation}
\end{itemize}
Morphisms in $\mathsf{C}'$ are defined in the obvious way, as morphisms of graded bimodules compatible with the decompositions~\eqref{eqn:dec-bimodule}. We also denote by $\mathsf{C}$ the full subcategory of $\mathsf{C}'$ whose objects are those whose underlying graded $R$-bimodule $M$ is finitely generated as an $R$-bimodule and flat as a right $R$-module. As explained in~\cite[Lemma~2.6]{abe}, the underlying $R$-bimodule of any object in $\mathsf{C}$ is in fact finitely generated as a left and as a right $R$-module; this property shows that the tensor product over $R$ induces in a natural way a monoidal product on $\mathsf{C}$.
%, which will be denoted $\otimes$. 
We also have a ``grading shift'' autoequivalence of $\mathsf{C}$, which only changes the grading of the underlying graded $R$-bimodule in such a way that $M(1)^i=M^{i+1}$.

For $s \in \Saff$, we consider the $s$-invariants $R^s \subset R$, and the graded $R$-bimodule $B_s^{\Bim} := R \otimes_{R^s} R (1)$. This object has a natural ``lift'' as an object in $\mathsf{C}$, which will also be denoted $B_s^{\Bim}$ (see~\cite[\S 2.4]{abe}).

The following result is a special case of~\cite[Theorem~5.6]{abe} (see also~\cite[Theorem~1.3]{abe3}). 
%(The assumptions of this theorem hold in our present setting by Lemma~\ref{lem:ref-faithful}.)

\begin{thm}
\label{thm:abe}
There exists a canonical fully faithful monoidal functor
\[
\DBS \to \mathsf{C}
\]
sending $B_s$ to $B_s^{\Bim}$ for any $s \in \Saff$ and intertwining the grading shifts $(1)$.
\end{thm}

It will be convenient to consider also a slight extension of the category $\mathsf{C}$, adapted to the group $\Wext$. Namely, the action of $\Waff$ on $\bt^*$ extends in a natural way to $\Wext$ (using now the projection $\Wext \to \WW$). We will denote by $\mathsf{C}'_{\ext}$ the category whose objects are pairs consisting a graded $R$-bimodule $M$ together with a decomposition
\begin{equation*}
%\label{eqn:dec-bimodule-ext}
M \otimes_R Q = \bigoplus_{w \in \Wext} M_Q^w
\end{equation*}
as $(R,Q)$-bimodules
such that:
\begin{itemize}
\item there exist only finitely many $w$'s such that $M_Q^w \neq 0$;
\item for any $w \in \Wext$, $r \in R$ and $m \in M_Q^w$ we have $m \cdot r=w(r) \cdot m$,
\end{itemize}
and where morphisms are defined in the obvious way. We will also denote by $\mathsf{C}_\ext$ the full subcategory of $\mathsf{C}'_\ext$ whose objects are those whose underlying graded $R$-bimodule $M$ is finitely generated as an $R$-bimodule and flat as a right $R$-module. It is clear that $\mathsf{C}'$ is a full subcategory in $\mathsf{C}'_\ext$, that $\mathsf{C}$ is a full subcategory in $\mathsf{C}_\ext$, and that the tensor product $\otimes_R$ defines a monoidal structure on $\mathsf{C}_\ext$.

\begin{rmk}
Some adaptations of the ``Soergel calculus" of~\cite{ew} to \emph{extended} affine Weyl groups have been discussed in work of Mackaay--Thiel~\cite{mackaay-thiel} and Elias~\cite[\S 3]{elias}. The analogue of Theorem~\ref{thm:abe} in this context is most likely true, but since it is not needed in this paper we will not investigate this question further.
\end{rmk}

In addition to the objects $B_s^{\Bim}$ considered above, the category $\mathsf{C}_\ext$ also possesses ``standard'' objects $(\Delta_x : x \in \Wext)$ defined as follows. For any $x \in \Wext$, $\Delta_x$ is isomorphic to $R$ as a graded vector space, and the structure of $R$-bimodule is given by
\[
r \cdot m \cdot r' = rmx(r')
\]
for $r,r' \in R$ and $m \in \Delta_x$. The decomposition of $\Delta_x \otimes_R Q$ is defined so that this object is concentrated in degree $x$. For any $x,y \in \Wext$ we have a canonical isomorphism
\begin{equation}
\label{eqn:convolution-Delta}
\Delta_x \otimes_R \Delta_y \simto \Delta_{xy}
\end{equation}
in $\mathsf{C}_\ext$,
defined by $m \otimes m' \mapsto mx(m')$.

\begin{lem}
\label{lem:conjugation-abe}
Let $s,t \in \Saff$ and $x \in \Wext$ be such that $s=xtx^{-1}$. Then there exists a canonical isomorphism
\[
B_s^{\Bim} \cong \Delta_x \otimes_R B_t^{\Bim} \otimes_R \Delta_{x^{-1}}.
\]
\end{lem}

\begin{proof}
The isomorphism of $R$-bimodules
\[
\Delta_x \otimes_R B_t^{\Bim} \otimes_R \Delta_{x^{-1}} \simto B_s^{\Bim}
\]
is defined by
\[
r_1 \otimes (r_2 \otimes r_3) \otimes r_4 \mapsto (r_1x(r_2)) \otimes (x(r_3)x(r_4)).
\]
We leave it to the reader to check that this morphism is well defined, and indeed defines an isomorphism in $\mathsf{C}_\ext$.
\end{proof}

In Section~\ref{sec:Hecke-action} we will also need the following standard claim, for which we refer to~\cite[\S 3.4]{ew}.

\begin{lem}
\label{lem:exact-seq-Bs}
% Assume that $s \in \SSaff \cap \WW$. Then 
For any $s \in \Saff$, there exist exact sequences of $R$-bimodules
 \[
  \Delta_s \hookrightarrow R \otimes_{R^s} R \twoheadrightarrow \Delta_e, \qquad \Delta_e \hookrightarrow R \otimes_{R^s} R \twoheadrightarrow \Delta_s.
 \]
\end{lem}

\subsection{Universal centralizer and Kostant section}
\label{ss:Kostant-section-Abe}
%----------------------------------------------------------------

We will denote by $\bgreg \subset \bg$ the open subset consisting of regular elements, i.e.~elements whose centralizer has dimension $\dim(T)$. The ``regular universal centralizer'' is the affine group scheme
\[
\bbJ_\reg := \bgreg \times_{\bgreg \times \bgreg} (\bG \times \bgreg)
\]
over $\bgreg$, where the morphism $\bgreg \to \bgreg \times \bgreg$ is the diagonal embedding, and the map $\bG \times \bgreg \to \bgreg$ sends $(g,x)$ to $(g \cdot x, x)$. For any $x \in \bgreg$, the fiber of $\bbJ_\reg$ over $x$ is the scheme-theoretic centralizer of $x$ for the adjoint $\bG$-action. By construction $\bbJ_\reg$ is a closed subgroup scheme in $\bG \times \bgreg$, and as explained in~\cite[Corollary~3.3.6]{riche-kostant} it is smooth over $\bgreg$. We will also denote by $\bgreg^*$ the image of $\bgreg$ under $\kappa$, and by $\bbJ_\reg^*$ the smooth affine group scheme over $\bgreg^*$ obtained by pushforward from $\bbJ_\reg$. (It is easily seen that these objects do not depend on the choice of $\kappa$.)

There exists a canonical morphism
\begin{equation}
 \label{eqn:morph-J-T}
 \bt^* \times_{\bt^*/\WW} \bbJ^*_\reg \to (\bt^* \times_{\bt^*/\WW} \bgreg^*) \times \bT
\end{equation}
of group schemes over $\bt^* \times_{\bt^*/\WW} \bgreg^*$, whose construction we now explain. Let $\bn$ be the Lie algebra of the unipotent radical $\bU$ of $\bB$. Recall that the \emph{Grothendieck resolution} is the $\bG$-equivariant vector bundle over $\bG/\bB$ given by
\[
 \tbg := \{(\xi, g\bB) \in \bg^* \times \bG/\bB \mid \xi_{|g \cdot \bn}=0\}.
\]
We have natural maps
\[
 \pi : \tbg \to \bg^*, \qquad \vartheta : \tbg \to \bt^*.
\]
(The morphism $\pi$ is induced by the first projection. The morphism $\vartheta$ sends a pair $(\xi,g\bB)$ to $\xi_{|g \cdot \bb}$, seen as an element in $(g \cdot \bb / g \cdot \bn)^* \cong (\bb/\bn)^* \cong \bt^*$, where the first isomorphism is induced by conjugation by the inverse of any representative for the coset $g\bB$.) If we denote by $\tbg_{\reg}$ the preimage of $\bgreg^*$ in $\tbg$, then these maps induce an isomorphism of schemes
\begin{equation}
\label{eqn:tfgreg}
 \tbg_{\reg} \simto \bgreg^* \times_{\bt^*/\WW} \bt^*,
\end{equation}
see~\cite[Lemma~3.5.3]{riche-kostant}. Moreover, under this identification, by~\cite[Proposition~3.5.6]{riche-kostant} the group scheme $\bt^* \times_{\bt^*/\WW} \bbJ^*_\reg$ identifies with the universal stabilizer associated with the action of $\bG$ on $\tbg_{\reg}$ (defined by the same procedure as for $\bbJ_\reg$ above), which is such that the fiber over $(\xi,g\bB)$ is the scheme-theoretic stabilizer of $\xi$ for the action of $g\bB g^{-1}$. Now as above in the definition of $\vartheta$, there exists for any $g \in \bG$ a canonical isomorphism $g\bB g^{-1} / g\bU g^{-1} \cong \bT$, which allows us to define the wished-for morphism~\eqref{eqn:morph-J-T}.

%regular nilpotent element $e \in \fn$ and an affine subspace $\cS \subset \fg$ as in~\cite[\S 3.1]{riche-kostant} (in other words, a ``Kostant section'' for the adjoint quotient). Such a subspace is automatically contained in $\fgreg$, see~\cite[Equation~(3.1.1)]{riche-kostant}. Moreover, the composition $\cS \hookrightarrow \fg \to \fg/G$ (where the second map is the adjoint quotient morphism) is an isomorphism, see~\cite[Theorem~3.2.2]{riche-kostant}. 

Let $\bgrs \subset \bg$ denote the open subset of semisimple regular elements, and set $\bgrs^*:=\kappa(\bgrs)$. We will denote by $\bbJ_\rs$, resp.~$\bbJ^*_\rs$, the restriction of $\bbJ_\reg$, resp.~$\bbJ^*_\reg$, to $\bgrs$, resp.~$\bgrs^*$. 
%Let also $\bA$ be the ``universal maximal torus,'' which is canonically isomorphic to the quotient of any Borel subgroup in $\bG$ by its unipotent radical. 
Recall that the adjoint quotient $\bg/\bG$ identifies canonically with $\bt/\WW$, see~\cite[\S 4.1]{bc}; as a consequence, under our assumptions the coadjoint quotient $\bg^*/\bG$ identifies canonically with $\bt^*/\WW$.
%e.g.~\cite[Proposition~2.3.2]{riche-kostant}.
%, and set
%\[
%\fD_\rs := \bgrs \times_{\bt/W} \bt
%\]

\begin{lem}
\label{lem:Irs}
The morphism~\eqref{eqn:morph-J-T} restricts to an isomorphism
%There exists a canonical isomorphism
\[
\bbJ^*_\rs \times_{\bt^*/\WW} \bt^* \simto \bT \times (\bgrs^* \times_{\bt^*/\WW} \bt^*)
\]
of group schemes
over $\bgrs^* \times_{\bt^*/\WW} \bt^*$.
\end{lem}

\begin{proof}
It is sufficient to prove the analogous statement for $\bg$ in place of $\bg^*$. If we denote by $\tbg_\rs$ the inverse image of $\bgrs$ under $\pi$, then by~\cite[Lemma~13.4]{jantzen-nilp}, there exists a canonical isomorphism
\[
\bG \times^\bT \bt_\rs \simto \tbg_\rs,
\]
where $\bt_\rs := \bt \cap \bg_\rs$.
From the comments above and the compatibility of universal stabilizers with open embeddings, $\bt \times_{\bt/\WW} \bbJ_\rs$ identifies with the universal stabilizer associated with the action of $\bG$ on $\tbg_\rs$. Since the latter scheme identifies with $\bG \times^\bT \bt_\rs = \bG/\bT \times \bt_\rs$, the universal stabilizer identifies with $\bT \times \tbg_\rs$, as desired.
\end{proof}

%%----------------------------------------------------------
%\subsection{Restriction to a Kostant section}
%\label{ss:restriction-Kostant-Hecke}
%%----------------------------------------------------------

%As in~\S\ref{ss:centralizer-Kostant}
%we choose an affine subspace $\bS \subset \bg$ as in~\cite[\S 3.1]{riche-kostant} (in other words, a ``Kostant section'' for the adjoint quotient), and set $\bS^*:=\kappa(\bS)$. We then have $\bS^* \subset \bgreg^*$, see~\cite[Equation~(3.1.1)]{riche-kostant}, and the composition $\bS^* \hookrightarrow \bg^* \to \bg^*/\bG$ (where the second map is the adjoint quotient) is an isomorphism, see~\cite[Theorem~3.2.2]{riche-kostant}. 
Let us choose a Kostant section to the adjoint quotient, i.e.~a closed subscheme $\bS \subset \bg$ contained in $\bgreg$ and such that the composition $\bS \hookrightarrow \bg \to \bg/\bG$ (where the second map is the adjoint quotient morphism) is an isomorphism. (For a construction of such a section in the present generality, see~\cite[\S 3]{riche-kostant}.) We will denote by $\bS^*$ the image of $\bS$ under $\kappa$, so that the composition $\bS^* \hookrightarrow \bg^* \to \bg^*/\bG$ is an isomorphism, and by $\bbJ^*_\bS$ the restriction of $\bbJ_\reg^*$ to $\bS^*$ (a closed subgroup scheme of $\bG \times \bS^*$, smooth over $\bS^*$).

As explained e.g.~in~\cite[\S 4.4]{mr}, there exists a natural action of the multiplicative group $\Gm$ on $\bS^*$ such that the isomorphism $\bS^* \simto \bg^*/\bG$ is $\Gm$-equivariant, where $t \in \bk^\times$ acts on $\bg^*$ by multiplication by $t^{-2}$, and on $\bg^*/\bG$ by the induced action. The isomorphism $\bt^*/\WW \simto \bg^*/\bG$ is also $\Gm$-equivariant, where the action on $\bt^*/\WW$ is induced by the action on $\bt^*$ where $t \in \bk^\times$ acts by multiplication by $t^{-2}$.
%We also set
%\[
%\bJ^*_{\bS} := \bS^* \times_{\bgreg^*} \bJ^*_\reg.
%\]
%Then $\bJ^*_{\bS}$ is a smooth affine group scheme over $\bS^*$, and
As explained in~\cite[\S 4.5, p.~2302]{mr} there exists a natural $\Gm$-action on $\bbJ^*_{\bS}$ such that the structure morphism $\bbJ^*_{\bS} \to \bS^*$, the multiplication map $\bbJ^*_{\bS} \times_{\bS^*} \bbJ^*_{\bS} \to \bbJ^*_{\bS}$ and the inversion morphism $\bbJ^*_{\bS} \to \bbJ^*_{\bS}$ are $\Gm$-equivariant.

%Consider once again the isomorphism given by the composition
%\[
%\bS^* \hookrightarrow \bg^* \to \bg^*/\bG \cong \bt^*/\WW.
%\]
%Transporting the group scheme $\bbJ^*_\bS$ along the isomorphism $\bS^* \simto \bt^*/\WW$ we obtain a group scheme over $\bt^*/\WW$, which we denote by $\bbJ^*_{\adj}$. The considerations above show that there exists an action of $\Gm$ on $\bbJ^*_{\adj}$ such that the projection $\bbJ^*_{\adj} \to \bt^*/\WW$ is $\Gm$-equivariant, where the action on $\bt^*/\WW$ is induced by the action on $\bt^*$ where $t \in \bk^\times$ acts by multiplication by $t^{-2}$.

%---------------------------------------------------------------
\subsection{Representations of the universal centralizer and Abe's category}
\label{ss:Rep-Abe}
%---------------------------------------------------------------

The actions of $\Gm$ on $\bt^*$ and $\bt^*/\WW$ considered in~\S\ref{ss:Kostant-section-Abe} provide an action on the fiber product $\bt^* \times_{\bt^*/\WW} \bt^*$.
Let us now
consider the category
\[
\Rep^{\Gm}(\bt^* \times_{\bt^*/\WW} \bbJ^*_{\bS} \times_{\bt^*/\WW} \bt^*)
\]
of $\Gm$-equivariant coherent representations of the smooth affine group scheme
\[
\bt^* \times_{\bt^*/\WW} \bbJ^*_{\bS} \times_{\bt^*/\WW} \bt^*
\]
over
$\bt^* \times_{\bt^*/\WW} \bt^*$ (where the morphism $\bbJ^*_\bS \to \bt^*/\WW$ is obtained via the identification $\bS^* \simto \bt^*/\WW$), i.e.~$\bt^* \times_{\bt^*/\WW} \bbJ^*_{\bS} \times_{\bt^*/\WW} \bt^*$-modules equipped with a structure of $\Gm$-equivariant coherent sheaf on $\bt^* \times_{\bt^*/\WW} \bt^*$, such that the action map is $\Gm$-equivariant. Since $R$ is finite over $R^\WW$, this category admits a natural convolution product $\star$, such that the $\scO(\bt^* \times_{\bt^*/\WW} \bt^*)$-module underlying the product $M \star N$ is the tensor product $M \otimes_{R} N$ (where $R=\scO(\bt^*)$ acts on $M$ via the second projection $\bt^* \times_{\bt^*/\WW} \bt^* \to \bt^*$ and on $N$ via the first projection $\bt^* \times_{\bt^*/\WW} \bt^* \to \bt^*$). In this way, $(\Rep^{\Gm}(\bt^* \times_{\bt^*/\WW} \bbJ^*_{\bS} \times_{\bt^*/\WW} \bt^*), \star)$ is a monoidal category. We will denote by
\[
\Rep_{\mathrm{fl}}^{\Gm}(\bt^* \times_{\bt^*/\WW} \bbJ^*_{\bS} \times_{\bt^*/\WW} \bt^*)
\]
the full subcategory of $\Rep^{\Gm}(\bt^* \times_{\bt^*/\WW} \bbJ^*_{\bS} \times_{\bt^*/\WW} \bt^*)$ whose objects are the representations whose underlying coherent sheaves are flat with respect to the second projection $\bt^* \times_{\bt^*/\WW} \bt^* \to \bt^*$. It is not difficult to check that this subcategory is stable under $\star$, hence also admits a canonical structure of monoidal category.

\begin{prop}
\label{prop:univ-centralizer-Hecke}
There exists a canonical fully faithful monoidal functor
\[
\bigl( \Rep_{\mathrm{fl}}^{\Gm}(\bt^* \times_{\bt^*/\WW} \bbJ^*_{\bS} \times_{\bt^*/\WW} \bt^*),\star \bigr) \to (\mathsf{C}_\ext, \otimes_R).
\]
\end{prop}

\begin{proof}
We start by constructing a functor
\begin{equation}
\label{eqn:functor-J-C-2}
\Rep^{\Gm}(\bt^* \times_{\bt^*/\WW} \bbJ^*_{\bS} \times_{\bt^*/\WW} \bt^*) \to \mathsf{C}'_{\ext}.
\end{equation}
By definition, any object in $\Rep^{\Gm}(\bt^* \times_{\bt^*/\WW} \bbJ^*_{\bS} \times_{\bt^*/\WW} \bt^*)$ is in particular a $\Gm$-equivariant coherent sheaf on $\bt^* \times_{\bt^*/\WW} \bt^*$, hence can be seen as a graded $R$-bimodule.
%, which is moreover finitely generated (as a bimodule) and flat as a left $R$-module. 
To equip this graded bimodule with the structure of an object in $\mathsf{C}'_\ext$, we must provide a decomposition of its tensor product with $Q$ parametrized by $\Wext$. In fact, we will provide such a decomposition for its tensor product with $\scO(\bt^*_\rs)$, where $\bt^*_\rs := \bt^* \cap \bg^*_\rs$ (which is sufficient since $Q$ is a further localization of $\scO(\bt^*_\rs)$).

First, the open subset $\bt^*_{\rs} \subset \bt^*$ is the complement of the kernels of the differentials of the coroots. This open subset is stable under the action of $\WW$, and the restriction of this action is free, see~\cite[Lemma~2.3.3]{riche-kostant}. In particular we have an open subset $\bt^*_{\rs} / \WW \subset \bt^*/\WW$, the morphism $\bt^*_\rs \to \bt^*_\rs/\WW$ is \'etale, and the map $(w,x) \mapsto (x,w(x))$ induces an isomorphism of schemes
\[
\WW \times \bt^*_\rs \simto \bt^*_\rs \times_{\bt^*_\rs/\WW} \bt^*_\rs,
\]
see~\cite[Exp.~V, \S 2]{sga1}.
As a consequence, for any coherent sheaf $\scF$ on $\bt^* \times_{\bt^*/\WW} \bt^*$, the tensor product
\[
 \Gamma(\bt^* \times_{\bt^*/\WW} \bt^*, \scF) \otimes_R \scO(\bt^*_\rs)
\]
admits a canonical decomposition (as an $\scO(\bt^*_\rs)$-bimodule) parametrized by $\WW$, such that the action on the factor corresponding to $w \in \WW$ factors through the quotient
\[
\scO(\bt_\rs^* \times \bt^*_\rs) \twoheadrightarrow \scO(\mathrm{Gr}(w,\bt^*_\rs))
\]
(where in the right-hand side $\mathrm{Gr}(w,\bt^*_\rs)$ denotes the graph of $w$ acting on $\bt^*_\rs$), i.e.~satisfies the condition in~\eqref{eqn:action-w}.

Next, let us explain how this decomposition can be refined if $\scF$ belongs to $\Rep(\bt^* \times_{\bt^*/\WW} \bbJ^*_{\bS} \times_{\bt^*/\WW} \bt^*)$. For this, we consider the restriction $\mathbf{M}_w$ of $\bt^* \times_{\bt^*/\WW} \bbJ^*_{\bS} \times_{\bt^*/\WW} \bt^*$ to $\mathrm{Gr}(w,\bt^*_\rs)$. Identifying the latter subscheme with $\bt^*_\rs$ via the first projection and using Lemma~\ref{lem:Irs}, we obtain a canonical isomorphism of group schemes
\[
\mathbf{M}_w \simto \bt^*_\rs \times \bT.
\]
This means that the category of representations of $\mathbf{M}_w$ on coherent sheaves on $\mathrm{Gr}(w,\bt^*_\rs)$ is canonically equivalent to the category of $\bX$-graded coherent shea\-ves on $\bt^*_\rs$. Starting with an object $\scF$ in $\Rep(\bt^* \times_{\bt^*/\WW} \bbJ^*_{\bS} \times_{\bt^*/\WW} \bt^*)$, we therefore obtain a decomposition of $\Gamma(\bt^* \times_{\bt^*/\WW} \bt^*, \scF) \otimes_R \scO(\bt^*_\rs)$ parametrized by $\Wext$ by defining, for $\lambda \in \bX$ and $w \in \WW$, the summand associated with $t_\lambda w$ as the $\lambda$-graded part in the summand associated with $w$ (which is a representation of $\mathbf{M}_w$).
%, where we identify $\bA$ with $\bT$ via the composition
%\[
%\bT \hookrightarrow \bB \twoheadrightarrow \bA.
%\]
This finishes the description of the functor~\eqref{eqn:functor-J-C-2}.

It is clear from construction that this functor sends objects in $\Rep^{\Gm}_{\mathrm{fl}}(\bt^* \times_{\bt^*/\WW} \bbJ^*_{\bS} \times_{\bt^*/\WW} \bt^*)$ to objects in $\mathsf{C}_{\ext}$, which therefore provides the functor of the statement. This functor is also easily seen to be monoidal. Let us now explain why it is fully faithful. Consider $\scF,\scG$ in $\Rep^{\Gm}_{\mathrm{fl}}(\bt^* \times_{\bt^*/\WW} \bbJ^*_{\bS} \times_{\bt^*/\WW} \bt^*)$, and denote their images by $M,N$ (so that the underlying graded bimodule of $M$, resp.~$N$, is $\Gamma(\bt^* \times_{\bt^*/\WW} \bt^*,\scF)$, resp.~$\Gamma(\bt^* \times_{\bt^*/\WW} \bt^*,\scG)$). By construction, morphisms in $\mathsf{C}_{\ext}$ from $M$ to $N$ are morphisms of graded bimodules from $\Gamma(\bt^* \times_{\bt^*/\WW} \bt^*,\scF)$ to $\Gamma(\bt^* \times_{\bt^*/\WW} \bt^*,\scG)$ whose restriction to $\bt^* \times \{\eta\}$ (where $\eta$ is the generic point of $\bt^*$) commutes with the action of the restriction of $\bt^* \times_{\bt^*/\WW} \bbJ^*_{\bS} \times_{\bt^*/\WW} \bt^*$. Now, since by assumption $\Gamma(\bt^* \times_{\bt^*/\WW} \bt^*,\scG)$ is flat as a right $R$-module and $\scO(\bt^* \times_{\bt^*/\WW} \bbJ^*_{\bS} \times_{\bt^*/\WW} \bt^*)$ is flat over $\scO(\bt^* \times_{\bt^*/\WW} \bt^*)$, such a morphism is automatically a morphism of $\scO(\bt^* \times_{\bt^*/\WW} \bbJ^*_{\bS} \times_{\bt^*/\WW} \bt^*)$-comodules. This proves the desired fully faithfulness.
\end{proof}

\begin{lem}
\label{lem:Delta-J}
 For any $w \in \Wext$, the object $\Delta_w$ belongs to the essential image of the functor of Proposition~\ref{prop:univ-centralizer-Hecke}.
\end{lem}

\begin{proof}
 The isomorphism~\eqref{eqn:convolution-Delta} reduces the proof to the case $w$ belongs either to $\WW$ or to $\bX$. The case $w \in \WW$ is obvious: in this case $\Delta_w$ is the image of its underlying graded $R$-bimodule, endowed with the trivial structure as a representation. For the case $w \in \bX$, in view of the construction of the functor in Proposition~\ref{prop:univ-centralizer-Hecke}, 
the claim follows from the fact that the isomorphism of Lemma~\ref{lem:Irs} is the restriction of the morphism~\eqref{eqn:morph-J-T}.
%a morphism of group schemes $\bJ_\reg^* \times_{\bt^*/\WW} \bt^* \to \bT \times (\bg^*_\reg \times_{\bt^*/\WW} \bt^*)$.
%(which will \emph{not} be an isomorphism).
%Recall the Grothendieck resolution  $\tbg$ considered in the proof of Lemma~\ref{lem:Irs}. We will denote by $\tbg_\reg$ the inverse image of $\bg^*_\reg$ under the canonical morphism $\tbg \to \bg^*$. Then by~\cite[Lemma~3.5.3]{riche-kostant} the isomorphism~\eqref{eqn:tbgrs-fiber-product} extends to an isomorphism
%\[
%\tbg_\reg \simto \bg^*_\reg \times_{\bt^*/\WW} \bt^*.
%\]
%Moreover, by~\cite[Proposition~3.5.6]{riche-kostant}, under this isomorphism to pullback $\bJ_\reg^* \times_{\bt^*/\WW} \bt^*$ identifies with the universal centralizer $\widetilde{\bJ}_\reg$ associated with the $\bG$-action on $\tbg_\reg$. Now there exists a canonical morphism of group schemes
%\[
%\widetilde{\bJ}_\reg \to \bA \times \tbg_\reg,
%\]
%from which we obtain the desired morphism~\eqref{eqn:Jreg-morphism}.
\end{proof}

%\begin{rmk}
%\label{rmk:Delta-J}
 For $w \in \Wext$, we will denote by $\Delta_w^{\bbJ}$ the unique object in $\Rep_{\mathrm{fl}}^{\Gm}(\bt^* \times_{\bt^*/\WW} \bbJ^*_{\bS} \times_{\bt^*/\WW} \bt^*)$ which is sent to $\Delta_w$.
%\end{rmk}

%---------------------------------------------------------------
\subsection{Representations of the universal centralizer and the Hecke category}
\label{ss:Rep-Hecke}
%---------------------------------------------------------------

By Proposition~\ref{prop:univ-centralizer-Hecke} the category $\Rep_{\mathrm{fl}}^{\Gm}(\bt^* \times_{\bt^*/\WW} \bbJ^*_{\bS} \times_{\bt^*/\WW} \bt^*)$ can be seen as a full monoidal subcategory in Abe's category $\mathsf{C}_\ext$, and by Theorem~\ref{thm:abe} the same is true for the Hecke category $\DBS$. We now investigate the relation between these two subcategories.

\begin{lem}
\label{lem:essential-images}
The essential image of the functor of Theorem~\ref{thm:abe} is contained in the essential image of the functor of Proposition~\ref{prop:univ-centralizer-Hecke}.
\end{lem}

\begin{proof}
By definition, the category $\DBS$ is generated under convolution and grading shift by the objects $(B_s : s \in \Saff)$. Hence to prove the lemma it suffices to prove each $B_s^\Bim$ belongs to the essential image of the functor of Proposition~\ref{prop:univ-centralizer-Hecke}. 

If $s=s_\alpha$ for some $\alpha \in \Phi^{\mathrm{s}}$, then $B_s^\Bim$ is the image of the appropriate shift of $\scO(\bt^* \times_{\bt^*/\{e,s\}} \bt^*)$, endowed with the trivial structure as a representation of $\bt^* \times_{\bt^*/\WW} \bbJ^*_{\bS} \times_{\bt^*/\WW} \bt^*$. If $s \in \SSaff$ is not of this form, then there exist $x \in \Wext$ and $t \in \Saff$ such that $t=s_\alpha$ for some $\alpha \in \Phi^{\mathrm{s}}$ and $s=xtx^{-1}$. (In fact, such a statement is even true in the braid group associated with $\Wext$: see~\cite[Lemma~6.1.2]{riche} or~\cite[Lemma~2.1.1]{bm} for the proof in the setting where $\bG$ is semisimple and simply connected, from which one can deduce the general case using restriction to the derived subgroup.) By Lemma~\ref{lem:conjugation-abe} we then have $B_s^\Bim \cong \Delta_x \otimes_R B_t^\Bim \otimes_R \Delta_{x^{-1}}$; since $B_t^\Bim$ is now known to belong to the essential image of our functor, and since $\Delta_x$ also satisfies this property by Lemma~\ref{lem:Delta-J}, this finishes the proof.
\end{proof}

From this lemma we deduce the following claim, which will be crucial for our constructions in Section~\ref{sec:Hecke-action}.

\begin{thm}
\label{thm:Hecke-centralizer}
There exists a canonical fully faithful monoidal functor
\[
\DBS \to \Rep_{\mathrm{fl}}^{\Gm}(\bt^* \times_{\bt^*/\WW} \bbJ^*_{\bS} \times_{\bt^*/\WW} \bt^*).
\]
\end{thm}

\section{Some categories of equivariant \texorpdfstring{$\Ug$}{Ug}-bimodules}
\label{sec:HCBim}
%%%%%%%%%%%%%%%%%%%%%%%%%%%%%%%%%%%%%%%%%%%%%%%%

%---------------------------------------------------------
\subsection{Weights}
\label{ss:weights}
%---------------------------------------------------------

From now on we assume that $p>0$. We consider a simply connected semisimple algebraic group $G$ over $\bk$, and its category
\[
\Rep(G)
\]
of finite-dimensional algebraic representations.

For a $\bk$-scheme $X$ we will denote by $X^{(1)}$ the associated Frobenius twist, defined as the fiber product $X^{(1)}:=X \times_{\mathrm{Spec}(\bk)} \mathrm{Spec}(\bk)$, where the morphism $\mathrm{Spec}(\bk) \to \mathrm{Spec}(\bk)$ is associated with the map $x \mapsto x^p$. (The projection $X^{(1)} \to X$ is an isomorphism of $\mathbb{F}_p$-schemes, but not of $\bk$-schemes.) 
 We will assume that
 \[
 \text{$p$ is very good for $G$.}
 \]
 Then the group $\bG := G^{(1)}$ satisfies the assumptions of Section~\ref{sec:Hecke-univ-centralizer}. We will denote by $\Fr : G \to \bG$ the Frobenius morphism of $G$, and will use the same notation for its restriction to the various subgroups considered below.
 
 The subgroups $\bB$, $\bT$, $\bU$ of $\bG$, when seen as subschemes in $G$, determine subgroups $B$, $T$, $U$ whose Frobenius twists are $\bB$, $\bT$, $\bU$ respectively.
%We fix an algebraically closed field $\bk$ of characteristic $p>0$, and a simply connected semisimple algebraic group $G$ over $\bk$. We will denote by $\fg$ the Lie algebra of $G$. The main object of study of the paper will be the category
%\[
%\Rep(G)
%\]
%of finite-dimensional algebraic $G$-modules.
%We choose a Borel subgroup $B \subset G$ and a maximal torus $T \subset B$, and 
We will denote by $\fg$, $\fb$, $\ft$, $\fn$ the respective Lie algebras of $G$, $B$, $T$, $U$ (so that $\bg=\fg^{(1)}$ and similarly for $B$, $U$, $T$), and by
%Let $U$ be the unipotent radical of $B$, $\fn$ be its Lie algebra, and 
$W$ the Weyl group of $(G,T)$. We set $\bbX:=X^*(T)$, and denote by $\fR \subset \bbX$ the root system of $(G,T)$. The choice of $B$ determines a system of positive roots $\fR^+ \subset \fR$, chosen as the $T$-weights in $\fg/\fb$. We will denote by $\fRs \subset \fR$ the corresponding subset of simple roots, and by $\rho \in \bbX$ the halfsum of the positive roots. We also set $\bbX^\vee := X_*(T)$, and denote by $\fR^\vee \subset \bbX^\vee$ the coroot system. The canonical bijection $\fR \simto \fR^\vee$ will be denoted as usual $\alpha \mapsto \alpha^\vee$.

The Frobenius morphism $\Fr$ induces an isomorphism
\[
N_G(T)/T \simto N_{\bG}(\bT)/\bT,
\]
which allows us to identify the Weyl group $\WW$ of $\bG$ with $W$. 
It is a standard fact that the morphism from $\bX=X^*(T^{(1)})$ to $\bbX$ induced by $\Fr : T \to \bT$ is injective, and that its image is $p \cdot \bbX$, which allows us to identify $\bX$ with $p \cdot \bbX$.
The identification $\bX = p \cdot \bbX$ is $W$-equivariant, and the root system $\Phi$ of $(\bG,\bT)$ is
%If we denote by $\Phi$ and $\Psi$ the root systems of $(G^{(1)},T^{(1)})$ and $(G,T)$ respectively, then we also have 
$\Phi = \{p \cdot \alpha : \alpha \in \fR\}$; similarly we have $\Phi^+=p\fR^+$ and $\Phi^{\mathrm{s}}=p\fRs$. In particular, the affine Weyl group $\Waff$ of~\S\ref{ss:Hecke-cat} identifies with $W \ltimes (p \Z\fR)$, and the extended affine Weyl group $\Wext$ identifies with $W \ltimes (p \cdot \bbX)$. Recall also our subset of Coxeter generators $\Saff \subset \Waff$. The subgroup $W \subset \Waff$ is a parabolic subgroup; its longest element will be denoted (as usual) by $w_0$.
%The \emph{affine Weyl group} associated with $G$ is the semi-direct product
%\[
% \Waff := W \ltimes p\Z \fR
%\]
%where $\Z \fR \subset \bbX$ is the lattice generated by $\fR$, and the $W$-action on $p\Z\fR$ is induced by the natural action on $\bbX$. The group $\Waff$ is a normal subgroup in the \emph{extended affine Weyl group}
%\[
% \Wext := W \ltimes p\bbX.
%\]
%Given $\mu \in p\bbX$, we will denote by $t_\mu$ the associated element of $\Wext$. It is well known that the group $\Waff$ is generated by the subset $\Saff$ consisting of the reflections $s_\alpha$ with $\alpha \in \fR^{\mathrm{s}}$, together with the products $t_{p\beta} s_\beta$ where $\beta \in \fR$ is such that $\beta^\vee$ is a maximal coroot. Moreover, the pair $(\Waff,\Saff)$ is a Coxeter system, see~\cite[\S II.6.3]{jantzen}. 
We will consider the ``dot'' action of $\Wext$ (or its subgroup $\Waff$) on $\bbX$ defined by
\[
 (t_\mu w) \bullet \lambda = w(\lambda+\rho)-\rho+\mu
\]
for $\mu \in p\bbX$, $w \in W$ and $\lambda \in \bbX$.
%, where $\rho \in \bbX$ is the halfsum of positive roots. 
%Note that the map $\bbX \to \ft^*$ sending a character to its differential is $\Wext$-equivariant, where $\Wext$ acts on $\bbX$ via the dot-action, and on $\ft^*$ via the projection $\Wext \to W$ and the action of $W$ on $\ft^*$ also denoted $\bullet$ in~\S\ref{ss:weights}.

%The reason why it is legitimate to use the same notation as for the $W$-action on $\ft^*$ is that the map $X^*(T) \to \ft^*$ is $\Waff$-equivariant, where $\Waff$ acts on $X^*(T)$ as above and on $\ft^*$ via the natural quotient morphism $\Waff \to W$ and the $(W,\bullet)$-action on $\ft^*$.

Given a character $\lambda \in \bbX$, we will denote by $\ola \in \ft^*$ the differential of $\lambda$. We set
\[
 \ft^*_{\mathbb{F}_p} := \{\ola : \lambda \in \bbX\} \subset \ft^*.
\]
In this way, the map $\lambda \mapsto \ola$ induces an isomorphism of abelian groups
\[
 \bbX/p\bbX \simto \ft^*_{\mathbb{F}_p}.
\]
(In particular, $\ft^*_{\mathbb{F}_p}$ is finite.)

The group $W$ naturally acts on $\ft^*$. We also have a ``dot'' action of $W$ on $\ft^*$, defined by
\[
 w \bullet \xi := w(\xi+\overline{\rho})-\overline{\rho}.
\]
With this definition the map $\bbX \to \ft^*$ sending $\lambda$ to $\ola$ is $\Wext$-equivariant, where $\Wext$ acts on $\bbX$ via the dot-action and on $\ft^*$ via the projection $\Wext \to W$ and the dot-action of $W$ on $\ft^*$. This observation legitimates the use of the same notation for these actions. It also shows that the subset $\ft^*_{\mathbb{F}_p} \subset \ft^*$ is stable under the dot-action of $W$. Below we will consider the quotient $\ft^*/(W,\bullet)$ of the dot-action of $W$ on $\ft^*$. For $\lambda \in \bbX$, we will denote by $\tla$ the image of $\ola$ in $\ft^*/(W,\bullet)$.

%We will assume throughout the paper that $p$ is very good for $G$. This 
As mentioned above our assumption that $p$ is very good for $G$ implies in particular that the quotient $\bbX/\Z\fR$ has no $p$-torsion, or in other words that
\begin{equation}
\label{eqn:no-torsion}
 \Z\fR \cap p\bbX = p\Z\fR.
\end{equation}
This equality has the following consequences.

\begin{lem}
\label{lem:weights}
 Let $\lambda \in \bbX$.
 \begin{enumerate}
  \item 
  \label{it:lem-no-torsion-1}
  We have
  \[
   \Waff \bullet \lambda = (\Wext \bullet \lambda) \cap (\lambda+\Z\fR).
  \]
\item 
\label{it:lem-no-torsion-2}
The stabilizer of $\ola$ for the dot-action of $W$ on $\ft^*$ is the image under the natural surjection $\Waff \to W$ of the stabilizer of $\lambda$ for the dot-action of $\Waff$ on $\bbX$.
 \end{enumerate}
\end{lem}

\begin{proof}
\eqref{it:lem-no-torsion-1}
 Since $W \bullet \lambda \subset \lambda+\Z\fR$, we have
 \[
  (\Wext \bullet \lambda) \cap (\lambda+\Z\fR) = (W \bullet \lambda + p\bbX) \cap (\lambda+\Z\fR) = W \bullet \lambda + (\Z\fR \cap p\bbX).
 \]
In view of~\eqref{eqn:no-torsion}, the right-hand side equals $W \bullet \lambda + p\Z\fR = \Waff \bullet \lambda$, as desired.

\eqref{it:lem-no-torsion-2} For $w \in W$ we have
\[
 w \bullet \ola = \overline{w \bullet \lambda},
\]
so that $w \bullet \ola=\ola$ iff $w \bullet \lambda \in \lambda+p\bbX$. Since $w \bullet \lambda \in \lambda+\Z\fR$, as above this condition is equivalent to $w \bullet \lambda \in \lambda+p\Z\fR$, i.e.~to the existence of $\mu \in p\Z\fR$ such that $t_\mu w \in \Waff$ stabilizes $\lambda$.
\end{proof}

For any subset $I \subset \fR^{\mathrm{s}}$, we will denote by $W_I \subset W$ the subgroup generated by the reflections $\{s_\alpha : \alpha \in I\}$. Recall that an element of $\bbX$ is called \emph{regular} if its stabilizer in $\Waff$ (for the dot-action) is trivial.
As a consequence of Lemma~\ref{lem:weights}, we obtain in particular the following claim.

\begin{lem}
\label{lem:quotient-etale}
 Let $\lambda \in \bbX$, and assume that the stabilizer of $\lambda$ for the dot-action of $\Waff$ is $W_I$. Then the morphism
 \[
  \ft^* / (W_I,\bullet) \to \ft^*/(W,\bullet)
 \]
 induced by the quotient morphism $\ft^* \to \ft^*/(W,\bullet)$
is \'etale at the image of $\ola$. In particular, if $\lambda$ is regular then the quotient morphism $\ft^* \to \ft^*/(W,\bullet)$
is \'etale at $\ola$.
\end{lem}

\begin{proof}
 By Lemma~\ref{lem:weights}\eqref{it:lem-no-torsion-2}, the stabilizer of $\ola$ for the dot-action of $W$ on $\ft^*$ is $W_I$. Hence the claim follows from the general criterion~\cite[Exp.~V, Proposition~2.2]{sga1}.
\end{proof}

%---------------------------------------------------------
\subsection{The center of the enveloping algebra}
\label{ss:center}
%---------------------------------------------------------

% From now on we fix an algebraically closed field $\bk$ of characteristic $p>0$, and a connected reductive group $G$ over $\bk$. We will denote by $\fg$ the Lie algebra of $G$. We will assume that $p$ is odd, and that these data satisfy Jantzen's standard assumptions, i.e.~that
% \begin{enumerate}
% \item
% the derived subgroup $\mathscr{D}(G)$ of $G$ is simply connected;
% \item
% $p$ is good for $G$
% \item
% $\fg$ admits a nondegenerate $G$-invariant bilinear form.
% \end{enumerate}
% 
% Under these assumptions, 
Consider the universal enveloping algebra $\Ug$ of $\fg$. Its center $Z(\cU\fg)$
%The center $Z(\cU\fg)$ of the enveloping algebra $\cU\fg$ of $\fg$ 
can be described as follows. First we set
\[
\ZHC:=(\cU\fg)^G.
\]
(Here, the subscript ``$\mathrm{HC}$'' stands for Harish-Chandra.) Next, as the Lie algebra of an algebraic group over a field of characteristic $p$, $\fg$ admits a ``restricted $p$-th power'' operation $x \mapsto x^{[p]}$, which stabilizes the Lie algebra of any algebraic subgroup of $G$. We will denote by
\[
\ZFr
\]
the $\bk$-subalgebra of $\cU\fg$ generated by the elements of the form $x^p-x^{[p]}$ for $x \in \fg$.
Then by~\cite[Theorem~2]{mirkovic-rumynin}
%by~\cite[Theorem~3.5]{brown-gordon} (which extends earlier work of Veldkamp and Mirkovi{\'c}--Rumynin), 
multiplication induces an isomorphism
\[
\ZFr \otimes_{\ZFr \cap \ZHC} \ZHC \simto Z(\cU\fg).
\]
It is well known that $\Ug$ is finite as a $\ZFr$-algebra (hence a fortiori as a $Z(\Ug)$-algebra).

These central subalgebras can be described geometrically as follows. 
%Namely, for a $\bk$-scheme $X$ we will denote by $X^{(1)}$ the associated Frobenius twist, defined as the fiber product $X^{(1)}:=X \times_{\mathrm{Spec}(\bk)} \mathrm{Spec}(\bk)$, where the morphism $\mathrm{Spec}(\bk) \to \mathrm{Spec}(\bk)$ is associated with the map $x \mapsto x^p$. (The projection $X^{(1)} \to X$ is an isomorphism of $\Z$-schemes, but not of $\bk$-schemes.) With this notation, 
First, it is well known that the map $x \mapsto x^p-x^{[p]}$ induces a $\bk$-algebra isomorphism
\begin{equation}
\label{eqn:ZFr}
\scO(\fg^{*(1)}) \simto \ZFr.
\end{equation}
We also have $\ZFr \cap \ZHC = (\ZFr)^G$, and the $G$-action on $\fg^{*(1)}$ factors through the Frobenius morphism $\Fr$, so that we obtain an isomorphism
\[
\scO(\fg^*{}^{(1)} / G^{(1)}) \simto \ZFr \cap \ZHC.
\]
On the other hand, 
% let us choose a Borel subgroup $B \subset G$ and a maximal torus $T \subset B$, and let us denote by $\ft$ and $\fb$ their Lie algebras. Let also $U$ be the unipotent radical of $B$, $\fu$ be its Lie algebra, and $W$ be the Weyl group of $(G,T)$. The map sending an element in $X^*(T)$ to its differential induces an embedding $X^*(T)/pX^*(T) \to \ft^*$. Given $\lambda \in X^*(T)$, we will denote by $\overline{\lambda}$ its class in $X^*(T)/pX^*(T)$, identified with its differential (hence considered as an element in $\ft^*$). We denote by $\rho \in X^*(T)$ the half sum of the differentials of the positive roots (chosen as the $T$-weights in $\fu^*$, so that $B$ is ``negative''), and define the ``dot-action'' of $W$ on $\ft^*$ by
% \[
% w \bullet \lambda = w(\lambda+\overline{\rho})-\overline{\rho},
% \]
% where on the right-hand side we consider the obvious action of $W$. Then 
the ``Harish-Chandra isomorphism'' provides a $\bk$-algebra isomorphism
\begin{equation}
\label{eqn:ZHC}
\scO(\ft^*/(W,\bullet)) \simto \ZHC,
\end{equation}
see~\cite[Theorem~1(2)]{mirkovic-rumynin}.
%This isomorphism shows that the datum of a character of $\ZHC$ is equivalent to that of an element in $\ft^*/(W,\bullet)$. 
%The characters we will be mostly interested in are the \emph{integral} ones, i.e.~those coming from characters of $T$. Given $\lambda \in X^*(T)$, we will denote by $\widetilde{\lambda}$ the image of $\overline{\lambda}$ in $\ft^*/(W,\bullet)$.

The \emph{Artin--Schreier morphism}
\[
 \AS : \ft^* \to \ft^{*(1)}
\]
is the morphism associated with the algebra map $\scO(\ft^{*(1)}) \to \scO(\ft^*)$ defined by $h \mapsto h^p-h^{[p]}$ for $h \in \ft$. It is well known that $\AS$ is a Galois covering with Galois group $\ft^*_{\mathbb{F}_p}$ (acting on $\ft^*$ via addition). The morphism $\AS$ is $W$-equivariant, where $W$ acts on $\ft^*$ via the dot-action and on $\ft^{*(1)}$ via the natural action. It therefore induces a morphism
\begin{equation}
\label{eqn:AS-quotients}
 \ft^*/(W,\bullet) \to \ft^*{}^{(1)} / W.
\end{equation}
Recall the Chevalley isomorphism
\[
\ft^*{}^{(1)} / W \simto
\fg^*{}^{(1)} / G^{(1)}
\]
%see~\cite[Theorem~1(3)]{mirkovic-rumynin}.
already encountered in~\S\ref{ss:Kostant-section-Abe}.
Under this identification, the embedding $\ZFr \cap \ZHC \hookrightarrow \ZHC$ is induced by~\eqref{eqn:AS-quotients}.
%the morphism $\ft^*/(W,\bullet) \to \ft^*{}^{(1)} / W$ considered above. 

Combining all these descriptions, and setting
\[
\fC:= \fg^{*(1)} \times_{\ft^{*(1)} / W} \ft^*/(W,\bullet),
\]
 we therefore obtain a $\bk$-algebra isomorphism
 \begin{equation*}
% \label{eqn:ZUg}
 \scO(\fC) \simto Z(\Ug),
 \end{equation*}
 see~\cite[Corollary~3]{mirkovic-rumynin}.
 
 Using this identification one can consider $\Ug$ as an $\scO(\fC)$-algebra. The $G$-action on $\fC$ induced by the adjoint $G$-action on $\Ug$ is the action obtained by pullback via the Frobenius morphism $\Fr$ of the $G^{(1)}$-action on $\fC$ induced by the coadjoint $G^{(1)}$-action on $\fg^{*(1)}$. Using this action, one can therefore see $\Ug$ as a $G$-equivariant $\scO(\fC)$-algebra.
 
 %------------------------------------------
 \subsection{Central reductions}
 \label{ss:central-reductions}
 %------------------------------------------
 
 In view of~\eqref{eqn:ZFr}, the maximal ideals in $\ZFr$ are in a canonical bijection with elements in $\fg^{*(1)}$. Given $\eta \in \fg^{*(1)}$, we will denote by $\mathfrak{m}_\eta \subset \ZFr$ the corresponding maximal ideal, and set
 \[
 \cU_{\eta}\fg := \Ug / \mathfrak{m}_\eta \cdot \Ug.
 \]
 Similarly, in view of~\eqref{eqn:ZHC} the maximal ideals in $\ZHC$ are in a canonical bijection with closed points in $\ft^*/(W,\bullet)$, i.e.~with $(W,\bullet)$-orbits in $\ft^*$. Given a closed point $\xi \in \ft^*/(W,\bullet)$, we will denote by $\mathfrak{m}^\xi \subset \ZHC$ the corresponding maximal ideal, and set
 \[
 \cU^\xi \fg := \Ug / \mathfrak{m}^\xi \cdot \Ug.
 \]
 If $\eta$ and $\xi$ have the same image in $\ft^{*(1)}/W$, then $\mathfrak{m}_\eta \cdot Z(\Ug) + \mathfrak{m}^\xi \cdot Z(\Ug)$ is a maximal ideal in $Z(\Ug)$, and we can also set
 \[
 \cU_{\eta}^\xi\fg := \Ug / (\mathfrak{m}_\eta \cdot \Ug + \mathfrak{m}^\xi \cdot \Ug).
 \]
 
In the cases we will encounter more specifically below, the point $\xi$ will often be the image $\tilde{\lambda}$ of the differential of a character $\lambda \in \bbX$. In this setting we will write $\mathfrak{m}^\lambda$, $\cU^\lambda \fg$ and $\cU_{\eta}^\lambda\fg$ instead of $\mathfrak{m}^{\tilde{\lambda}}$, $\cU^{\tilde{\lambda}} \fg$ and $\cU_{\eta}^{\tilde{\lambda}}\fg$. The image of any element of $\ft^*_{\mathbb{F}_p}$ under the Artin--Schreier map is $0$; therefore, if we denote
by
 \[
 \cN^* \subset \fg^*
 \]
 the preimage under the coadjoint morphism $\fg^* \to \ft^*/W$ of the image of $0$, then given any $\lambda \in \bbX$ the elements $\eta \in \fg^{*(1)}$ whose image in $\ft^{*(1)}/W$ coincides with that of $\tilde{\lambda}$ are exactly those in $\cN^{*(1)}$. 

 \subsection{Harish-Chandra bimodules}
 \label{ss:HCBim}
 %-----------------------------------------------------------

We will denote by $\HCBim$ the category whose objects are the $\Ug$-bimodules $V$ endowed with an (algebraic) action of $G$ which satisfy the following conditions:
\begin{enumerate}
\item the action morphisms $\Ug \otimes V \to V$ and $V \otimes \Ug \to V$ are $G$-equivariant (for the diagonal actions on $\Ug \otimes V$ and $V \otimes \Ug$);
\item 
\label{it:def-HC-2}
the $\fg$-action on $V$ obtained by differentiating the $G$-action is given by $(x,v) \mapsto x \cdot v - v \cdot x$;
\item 
\label{it:def-HC-3}
$V$ is finitely generated both as a left and as a right $\Ug$-module.
\end{enumerate}
Morphisms in the category $\HCBim$ are morphisms of bimodules which also commute with the $G$-actions. Objects in this category are called Harish-Chandra bimodules. 
%Of course, the right $\Ug$-action can be recovered from the left action and the $G$-action, and the left $\Ug$-action can be recovered from the right action and the $G$-action. In other words, there exist natural fully faithful functors
%\[
%\HCBim \to \mathsf{Mod}^G_{\mathrm{fg}}(\Ug) \quad \text{and} \quad
%\HCBim \to \mathsf{Mod}^G_{\mathrm{fg}}(\Ug^{\mathrm{op}}),
%\]
%where $\mathsf{Mod}^G_{\mathrm{fg}}(\Ug)$ is the category of $G$-equivariant finitely generated $\Ug$-modules, and $\mathsf{Mod}^G_{\mathrm{fg}}(\Ug^{\mathrm{op}})$ is the similar category for the opposite algebra.
It is easily seen that the tensor product $\otimes_{\Ug}$ of bimodules endows $\HCBim$ with the structure of a monoidal category, where the $G$-action on the tensor product is the diagonal action. 

If $M$ belongs to $\HCBim$, the $\Ug$-action obtained by differentiating the $G$-action must vanish on $\ZFr \cap (\fg \cdot \Ug)$. In view of condition~\eqref{it:def-HC-2} above, this implies that the two actions of $\ZFr$ on $M$ obtained by restriction of the left and right $\Ug$-actions coincide; in other words, the action of $\Ug \otimes_\bk \Ug^{\mathrm{op}}$ on $M$ must factor through an action of $\Ug \otimes_{\ZFr} \Ug^{\mathrm{op}}$. However, the two actions of $\ZHC$ on a Harish-Chandra bimodule might differ.
%; for this reason, we will not be able to use directly Proposition~\ref{prop:Ureg-Azumaya} to study these modules. 
Note that $\Ug \otimes_{\ZFr} \Ug^{\mathrm{op}}$ is in a natural way a finite algebra over the commutative ring
\[
\cZ := Z(\Ug) \otimes_{\ZFr} Z(\Ug) =  \ZHC \otimes_{\ZFr \cap \ZHC} \ZFr \otimes_{\ZFr \cap \ZHC} \ZHC \cong \scO(\fC \times_{\fg^{*(1)}} \fC).
\]
Note also that since $\Ug \otimes_{\ZFr} \Ug^{\mathrm{op}}$ is finitely generated both as a left and as a right $\Ug$-module, condition~\eqref{it:def-HC-3} above can be equivalently replaced by the condition that the object is finitely generated as a $\Ug$-bimodule (or as a left $\Ug$-module, or as a right $\Ug$-module). It is easily seen that forgetting the right, resp.~left, action of $\Ug$ defines an equivalence of categories
\begin{equation}
\label{eqn:HCBim-Mod}
\HCBim \simto \Mod^G_{\mathrm{fg}}(\Ug), \quad \mathrm{resp.} \quad \HCBim \simto \Mod^G_{\mathrm{fg}}(\Ug^{\mathrm{op}}),
\end{equation}
where $\Mod^G_{\mathrm{fg}}(\Ug)$ is the category of $G$-equivariant finitely generated $\Ug$-modules, and similarly for $\Mod^G_{\mathrm{fg}}(\Ug^{\mathrm{op}})$. For instance, for the first functor, one can reconstruct the right action of $\Ug$ on a $G$-equivariant $\Ug$-module $M$ by setting $m \cdot x := x \cdot m -\varrho(x)(m)$ for $m \in M$ and $x \in \fg$, where $\varrho$ denotes the differential of the $G$-action.

%In particular, this algebra identifies in a canonical way with the global sections of a sheaf of algebras on $\fC \times_{\fg^{*(1)}} \fC$ (namely, the restriction of the sheaf of algebras $\scU \boxtimes \scU^{\mathrm{op}}$ on $\fC \times \fC$).

We will also denote by $\Mod^G_{\mathrm{fg}}(\Ug \otimes_{\ZFr} \Ug^{\mathrm{op}})$ the category of $G$-equivariant finitely generated (left) modules over $\Ug \otimes_{\ZFr} \Ug^{\mathrm{op}}$. As above, since $\Ug \otimes_{\ZFr} \Ug^{\mathrm{op}}$ is finitely generated both as a left and as a right $\Ug$-module, the tensor product $\otimes_{\Ug}$ endows this category with a monoidal structure, and as explained above we have a fully faithful monoidal functor
\begin{equation}
\label{eqn:HCBim-Bim}
\HCBim \to \Mod^G_{\mathrm{fg}}(\Ug \otimes_{\ZFr} \Ug^{\mathrm{op}}).
\end{equation}
If $M$ belongs to $\Mod^G_{\mathrm{fg}}(\Ug \otimes_{\ZFr} \Ug^{\mathrm{op}})$, then one obtains an extra $\cU_0 \fg$-action on $M$ by setting $x \cdot m = (x \otimes 1 - 1 \otimes x)m$ for $x \in \fg$ and $m \in M$. Since $\cU_0 \fg$ identifies canonically with the distribution algebra $\mathrm{Dist}(G_1)$ of the kernel $G_1$ of $\Fr$, $M$ becomes in this way a $G \ltimes G_1$-equivariant $\Ug \otimes_{\ZFr} \Ug^{\mathrm{op}}$-module, where the action of $G \ltimes G_1$ on $\Ug \otimes_{\ZFr} \Ug^{\mathrm{op}}$ is obtained from the $G$-action by composition with the product morphism $G \ltimes G_1 \to G$. As for~\eqref{eqn:HCBim-Mod}, forgetting the right, resp.~left, action of $\Ug$ defines an equivalence of categories
\begin{multline*}
\Mod^G_{\mathrm{fg}}(\Ug \otimes_{\ZFr} \Ug^{\mathrm{op}}) \simto \Mod^{G \ltimes G_1}_{\mathrm{fg}}(\Ug), \\ \mathrm{resp.} \quad \Mod^G_{\mathrm{fg}}(\Ug \otimes_{\ZFr} \Ug^{\mathrm{op}}) \simto \Mod^{G \ltimes G_1}_{\mathrm{fg}}(\Ug^{\mathrm{op}}).
\end{multline*}
From the point of view of these equivalences and those in~\eqref{eqn:HCBim-Mod}, the essential image of~\eqref{eqn:HCBim-Bim} consists of equivariant modules on which the action of $G \ltimes G_1$ factors through the product morphism $G \ltimes G_1 \to G$.

One can construct interesting objects in $\HCBim$ from $G$-modules as follows. 
%Let us denote by $\Rep(G)$ the category of finite-dimensional algebraic $G$-modules. Then 
Given $V$ in $\Rep(G)$ we consider the Harish-Chandra bimodule
\[
V \otimes \Ug,
\]
where the left $\Ug$-action is diagonal (with respect to the action on $V$ obtained by differentiation, and the action on $\Ug$ by left multiplication), the right $\Ug$-action is induced by right multiplication on $\Ug$, and the $G$-action is diagonal (with respect to the given action on $V$ and the adjoint action on $\Ug$). In particular, for $x,y,z \in \Ug$ and $v \in V$ we have
\[
x \cdot (v \otimes z) \cdot y = (x_{(1)} \cdot v) \otimes (x_{(2)} zy),
\]
where we use Sweedler's notation for the comultiplication in the Hopf algebra $\Ug$. It is easily seen that the map $(x \otimes y) \otimes v \mapsto (x_{(1)} \cdot v) \otimes (x_{(2)} y)$ induces an isomorphism of $G$-equivariant $\Ug$-bimodules
\[
(\Ug \otimes \Ug^\op) \otimes_{\Ug} V \simto V \otimes \Ug,
\]
where the tensor product over $\Ug$ in the left-hand side is taken with respect to the morphism $\Ug \to \Ug \otimes \Ug^\op$ defined by $x \mapsto x_{(1)} \otimes S(x_{(2)})$, where $S$ is the antipode. In particular, the modules $V \otimes \Ug$ are ``induced from the diagonal.'' For $V,V'$ in $\Rep(G)$, we have a canonical isomorphism of Harish-Chandra bimodules
\begin{equation}
\label{eqn:isom-tensor-product-diag}
(V \otimes \Ug) \otimes_{\Ug} (V' \otimes \Ug) \simto (V \otimes V') \otimes \Ug.
\end{equation}

One can similarly consider, again for $V$ in $\Rep(G)$, the Harish-Chandra bimodule
\[
\Ug \otimes V
\]
where now the actions of $\Ug$ are defined by
\[
x \cdot (z \otimes v) \cdot y = (xzy_{(1)}) \otimes (S(y_{(2)}) \cdot v)
\]
for $x,y,z \in \Ug$ and $v \in V$
(and the $G$-action is still diagonal).
As above we have an isomorphism of $G$-equivariant $\Ug$-bimodules
\[
(\Ug \otimes \Ug^\op) \otimes_{\Ug} V \simto \Ug \otimes V,
\]
now given by $(x \otimes y) \otimes v \mapsto (xy_{(1)}) \otimes (S(y_{(2)}) \cdot v)$ for $x,y \in \Ug$ and $v \in V$. In particular, the objects $V \otimes \Ug$ and $\Ug \otimes V$ are isomorphic; explicitly the isomorphism is given by
\[
 v \otimes x \mapsto x \otimes (S(x) \cdot v)
\]
for $x \in \Ug$ and $v \in V$.

 %-----------------------------------------------------------
 \subsection{Completed Harish-Chandra bimodules}
 \label{ss:HCBim-completed}
 %-----------------------------------------------------------

Now, we need to adapt the considerations of~\S\ref{ss:HCBim} to the setting of completed Harish-Chandra characters.

Let us set
\[
\mathfrak{D} := \mathrm{Spec}(\ZHC \otimes_{\ZHC \cap \ZFr} \ZHC) \cong \ft^*/(W,\bullet) \times_{\ft^{*(1)}/W} \ft^*/(W,\bullet),
\]
so that $\mathcal{Z} = \scO(\fg^{*(1)} \times_{\ft^{*(1)}/W} \mathfrak{D})$. For $\lambda,\mu \in \bbX$, we will also set
\[
\cI^{\lambda,\mu} := \mathfrak{m}^\lambda \otimes_{\ZHC \cap \ZFr} \ZHC + \ZHC \otimes_{\ZHC \cap \ZFr} \mathfrak{m}^\mu \cdot \ZHC,
\]
and will denote by $\mathfrak{D}^{\hat{\lambda},\hat{\mu}}$ the spectrum of the completion of $\scO(\mathfrak{D})$ with respect to the maximal ideal $\cI^{\lambda,\mu}$. Finally, we set
\begin{multline*}
\cU^{\hat{\lambda},\hat{\mu}} := \bigl( \Ug \otimes_{\ZFr} \Ug^{\mathrm{op}} \bigr) \otimes_{\mathcal{Z}} \scO(\fg^{*(1)} \times_{\ft^{*(1)}/W} \mathfrak{D}^{\hat{\lambda},\hat{\mu}}) \\
\cong \bigl( \Ug \otimes_{\ZFr} \Ug^{\mathrm{op}} \bigr) \otimes_{\scO(\mathfrak{D})} \scO(\mathfrak{D}^{\hat{\lambda},\hat{\mu}}).
\end{multline*}
(Note that $\cU^{\hat{\lambda},\hat{\mu}}$ is \emph{not} the completion of $\Ug \otimes_{\ZFr} \Ug^{\mathrm{op}}$ with respect to the ideal generated by $\cI^{\lambda,\mu}$, since $\Ug \otimes_{\ZFr} \Ug^{\mathrm{op}}$ is not of finite type as an $\scO(\mathfrak{D})$-module.) 
%Below we will consider various categories of $\cU^{\hat{\lambda},\hat{\mu}}$-module; from a different point of view, if we denote (for $\nu \in \bbX$) by $\ZHC^{\hat{\nu}}$ the completion such a module is the same as a bimodule over $\Ug \otimes_{\ZHC}$

The algebra $\scO(\mathfrak{D}^{\hat{\lambda},\hat{\mu}})$ is Noetherian, see~\cite[\href{https://stacks.math.columbia.edu/tag/05GH}{Tag 05GH}]{stacks-project}, hence $\scO(\fg^{*(1)} \times_{\ft^{*(1)}/W} \mathfrak{D}^{\hat{\lambda},\hat{\mu}})$ is Noetherian too, being finitely generated over $\scO(\mathfrak{D}^{\hat{\lambda},\hat{\mu}})$, see~\cite[\href{https://stacks.math.columbia.edu/tag/00FN}{Tag 00FN}]{stacks-project}. Finally, since $\cU^{\hat{\lambda},\hat{\mu}}$ is finitely generated as an $\scO(\fg^{*(1)} \times_{\ft^{*(1)}/W} \mathfrak{D}^{\hat{\lambda},\hat{\mu}})$-module it is left and right Noetherian (as a noncommutative ring), and a $\cU^{\hat{\lambda},\hat{\mu}}$-module is finitely generated if and only if it is finitely generated as an $\scO(\fg^{*(1)} \times_{\ft^{*(1)}/W} \mathfrak{D}^{\hat{\lambda},\hat{\mu}})$-module.

The $G$-action on $\Ug \otimes_{\ZFr} \Ug^{\mathrm{op}}$ induces an algebraic $G$-module structure on $\cU^{\hat{\lambda},\hat{\mu}}$, and we can consider the category of $G$-equivariant finitely generated modules over this algebra; this (abelian) category will be denoted $\Mod^G_{\mathrm{fg}}(\cU^{\hat{\lambda},\hat{\mu}})$. The full subcategory whose objects are the modules $M$ such that the differential of the $G$-action coincides with the action given by $x \cdot m=xm-mx$ for $x \in \fg$ and $m \in M$ will be denoted $\HCBim^{\hat{\lambda},\hat{\mu}}$; its objects will be called \emph{completed Harish-Chandra bimodules}.

Given $\lambda,\mu \in \bbX$,
we have a natural exact functor
\[
\mathsf{C}^{\lambda,\mu} : \Mod^G_{\mathrm{fg}}(\Ug \otimes_{\ZFr} \Ug^{\mathrm{op}}) \to \Mod^G_{\mathrm{fg}}(\cU^{\hat{\lambda},\hat{\mu}}),
\]
defined by
\[
\mathsf{C}^{\lambda,\mu}(M) = \scO(\mathfrak{D}^{\hat{\lambda},\hat{\mu}}) \otimes_{\scO(\mathfrak{D})} M,
\]
which restricts to a functor $\HCBim \to \HCBim^{\hat{\lambda},\hat{\mu}}$.
(Exactness of this functor follows from the fact that $\scO(\mathfrak{D}^{\hat{\lambda},\hat{\mu}})$ is flat over $\scO(\mathfrak{D})$, see~\cite[\href{https://stacks.math.columbia.edu/tag/00MB}{Tag 00MB}]{stacks-project}.)
%
%which sends a $\Ug \otimes_{\ZFr} \Ug^{\mathrm{op}}$-module $M$ to its completion with respect to $\cI^{\lambda,\mu}$. This functor restricts to a functor
%from $\HCBim$ to $\HCBim^{\hat{\lambda},\hat{\mu}}$.
We will denote by
\[
 \HCBim^{\hat{\lambda},\hat{\mu}}_\diag
\]
the full additive subcategory of $\HCBim^{\hat{\lambda},\hat{\mu}}$ whose objects are direct summands of objects of the form $\mathsf{C}^{\lambda,\mu} (V \otimes \Ug)$ with $V$ in $\Rep(G)$.
%images of such Harish-Chandra bimodules under the functor~\eqref{eqn:HC-completion-functor}.
(In view of the comments at the end of~\S\ref{ss:HCBim}, objects in this category will sometimes be referred to as completed diagonally induced Harish-Chandra bimodules.)
In case $\lambda=\mu$, we will set
\[
\cU^{\hat{\lambda}} = \mathsf{C}^{\lambda,\lambda}(\bk \otimes \Ug),
\]
where here $\bk$ is the trivial $G$-module.

For later use, we also introduce some completed bimodules which are closely related to the translation functors for $G$-modules (see~\S\ref{ss:translation} below for details). 
Recall that a weight $\lambda \in \bbX$ is said to belong to the \emph{fundamental alcove}, resp.~to the \emph{closure of the fundamental alcove}, if it satisfies
\[
0 < \langle \lambda+\rho, \alpha^\vee \rangle < p, \quad \text{resp.} \quad 0 \leq \langle \lambda+\rho, \alpha^\vee \rangle \leq p,
\]
for any positive root $\alpha$. With this notation, the set of weights which belong to the closure of the fundamental alcove is a fundamental domain for the $(\Waff,\bullet)$-action on $\bbX$. Moreover, if $\lambda \in \bbX$ belongs to the closure of the fundamental alcove, then its stabilizer in $\Waff$ is the parabolic subgroup generated by the elements $s \in \Saff$ such that $s \bullet \lambda=\lambda$; see~\cite[\S II.6.3]{jantzen}.

\begin{rmk}
\label{rmk:Cox-number}
The weight lattice $\bbX$ contains weights which belong to the fundamental alcove iff $p \geq h$, where $h$ is the Coxeter number of $G$, see~\cite[\S 6.2]{jantzen}. Even though this condition will be imposed only in Section~\ref{sec:Hecke-action}, some of our statements in Section~\ref{sec:Ug-D} involve weights in the fundamental alcove; these statements will simply be empty in case $p<h$.
\end{rmk}

Let $\bbX^+ \subset \bbX$ be the subset of dominant weights determined by $\fR^+$.
For any $\nu \in \bbX^+$, we will denote by $\Sim(\nu)$ the simple $G$-module with highest weight $\nu$, i.e.~the unique simple submodule in $\Ind_{B}^G(\nu)$.
%$\Til(\nu)$ the indecomposable tilting $G$-module with highest weight $\nu$.
Given two weights $\lambda,\mu \in \bbX$ which belong to the closure of the fundamental alcove, we set
\[
\sfP^{\lambda,\mu} := \mathsf{C}^{\lambda,\mu}(\Sim(\nu) \otimes \Ug) \quad \in \HCBim_\diag^{\hat{\lambda},\hat{\mu}},
\]
where $\nu$ is the unique dominant $W$-translate of $\lambda-\mu$.

%-------------------------------------------------------
\subsection{Comparison of completions}
\label{ss:comparison}
%-------------------------------------------------------

% Of course, for $\lambda \in \bbX$, the character $\tilde{\lambda}$ of $\ZHC$ only depends on the image of $\lambda$ in the quotient 
% \[
%  (\bbX/p \bbX) / (W,\bullet)
% \]
% of $\bbX/p \bbX$ by the action of $W$ induced by the $\bullet$-action on $\ft^*$. Therefore we can fix a (finite) subset $\Lambda \subset \bbX$ such that the images of these elements in $\bbX/p \bbX$ are representatives of the orbits for this action, and restrict to the case $\lambda$ belongs to $\Lambda$ whenever convenient.

For notational simplicity, let us now fix a subset $\Lambda \subset \bbX$ such that the map $\lambda \mapsto \tla$ restricts to a bijection $\Lambda \simto \ft^*_{\mathbb{F}_p}/(W,\bullet)$. (In other words, $\Lambda$ is a set of representatives for the $\bullet$-action of $\Wext$ on $\bbX$.)

We will denote by $\cI \subset \scO(\ft^{*(1)}/W)=\ZHC \cap \ZFr$ the maximal ideal corresponding to the image of $0 \in \ft^{*(1)}$. Then in the notation of~\S\ref{ss:center} $\cI \cdot \ZFr$ is the ideal of definition of $\cN^{*(1)} \subset \fg^{*(1)}$, and each ideal $\cI^{\lambda,\mu}$ contains $\cI \cdot \scO(\mathfrak{D})$. We will denote by $\fD^\wedge$ the spectrum of the completion of $\scO(\fD)$ with respect to the ideal $\cI \cdot \scO(\fD)$. Note that since $\scO(\fD)$ is finite as an $\scO(\ft^{*(1)}/W)$-module (because the morphisms $\ft^* \to \ft^{*(1)}$ and $\ft^{*(1)} \to \ft^{*(1)}/W$ are finite), if we denote by $\scO(\ft^{*(1)}/W)^\wedge$ the completion of $\scO(\ft^{*(1)}/W)$ with respect to $\cI$ we have a canonical isomorphism
\begin{equation}
\label{eqn:Dwedge-tensor-prod}
\scO(\ft^{*(1)}/W)^\wedge \otimes_{\scO(\ft^{*(1)}/W)} \scO(\fD) \simto \scO(\fD^\wedge),
\end{equation}
see~\cite[\href{https://stacks.math.columbia.edu/tag/00MA}{Tag 00MA}]{stacks-project}.

\begin{lem}
\label{lem:isom-completions}
The natural morphism
\[
\scO(\fD^\wedge) \to \prod_{\lambda,\mu \in \Lambda} \scO(\fD^{\hat{\lambda},\hat{\mu}})
\]
is a ring isomorphism.
\end{lem}

\begin{proof}
The lemma will basically follow from the observation that the closed points in the fiber of the morphism~\eqref{eqn:AS-quotients} over the (closed) point corresponding to $\cI$ are those corresponding to the ideals $\mathfrak{m}^\lambda$ with $\lambda \in \bbX$, which itself follows from the fact that the fiber of $\AS : \ft^* \to \ft^{*(1)}$ over $0$ is $\ft^*_{\mathbb{F}_p}$.

More precisely, the morphism considered in this statement is the product of the morphisms $\scO(\fD^\wedge) \to \scO(\fD^{\hat{\lambda},\hat{\mu}})$ induced by the natural morphisms $\scO(\fD)/(\cI^n \cdot \scO(\fD)) \twoheadrightarrow \scO(\fD)/(\cI^{\lambda,\mu})^n$. This morphism is clearly a ring morphism; to prove that it is invertible we will construct its inverse.

Let us fix some $n \geq 1$, and consider the quotient $\scO(\fD)/(\cI^n \cdot \scO(\fD))$.
Here as explained above $\scO(\fD)=\ZHC \otimes_{\ZHC \cap \ZFr} \ZHC$ is a finite $\scO(\ft^{*(1)}/W)$-module;
%(since $\ZHC$ is finite, as a submodule of the finite module $\scO(\ft^*)$); 
therefore this algebra is finite-dimensional. Its maximal ideals are in bijection with the maximal ideals of $\scO(\fD)$ containing $\cI \cdot \scO(\fD)$, hence with $\Lambda \times \Lambda$ through
$(\lambda,\mu) \mapsto \cI^{\lambda,\mu}  / \cI^n \cdot \scO(\fD)$.
% with $\lambda \in X^*(T)$. Let us fix a subset $\Lambda \subset X^*(T)$ of representatives of the orbits of the action of $W$ on $X^*(T)/p \cdot X^*(T)$ induced by the dot-action. 
 In view of the general theory of Artin rings (see e.g.~\cite[Chap.~8]{amcd}), for any $\lambda,\mu \in \Lambda$ the quotient
%\begin{multline*}
% (\ZHC \otimes_{\ZHC \cap \ZFr} \ZHC) / \\
% \bigl( \cI^n \cdot (\ZHC \otimes_{\ZHC \cap \ZFr} \ZHC) + (\mathfrak{m}^\lambda \otimes_{\ZHC \cap \ZFr} \ZHC + \ZHC \otimes_{\ZHC \cap \ZFr} \mathfrak{m}^\mu)^m \bigr)
%\end{multline*}
\[
\scO(\fD)/(\cI^n \cdot \scO(\fD) + (\cI^{\lambda,\mu})^m)
\]
does not depend on $m$ for $m \gg 0$, and the natural morphism from $\scO(\fD)/(\cI^n \cdot \scO(\fD))$ to the product of these rings (over $\Lambda \times \Lambda$)
 %if we denote this quotient by $(\ZHC / I^n \cdot \ZHC)^{\hat{\lambda}}$ then the natural morphism
%\[
%\ZHC / I^n \cdot \ZHC \to \prod_{\lambda \in \Lambda} (\ZHC / I^n \cdot \ZHC)^{\hat{\lambda}}
%\]
is an isomorphism. 
%Now we have
%\[
%\cZ = (\ZHC \otimes_{\ZHC \cap \ZFr} \ZHC) \otimes_{\ZHC \cap \ZFr} \ZFr.
%\]
%From the preceding considerations we deduce that for $m \gg 0$ the natural morphism
%\[
%\cZ/(\cI^n \cdot \cZ) \to \prod_{\lambda,\mu \in \Lambda} \cZ/ (\cI^n \cdot \cZ + (\cI^{\lambda,\mu})^m)
%\]
%is a ring isomorphism.

Now we are ready to define the wished-for inverse morphism
\[
\prod_{\lambda,\mu \in \Lambda} \scO(\fD^{\hat{\lambda},\hat{\mu}}) \to \scO(\fD^\wedge).
%\prod_{\lambda,\mu \in \Lambda} \cZ^{\hat{\lambda},\hat{\mu}} \to \cZ^\wedge.
\]
For this it suffices to define, for any $n \geq 1$, a ring morphism
\begin{equation}
\label{eqn:morphism-completions}
\prod_{\lambda,\mu \in \Lambda} \scO(\fD^{\hat{\lambda},\hat{\mu}}) \to \scO(\fD)/(\cI^n \cdot \scO(\fD)).
\end{equation}
We fix $m$ such that the natural morphism
\begin{equation}
\label{eqn:morphism-completions-2}
\scO(\fD)/(\cI^n \cdot \scO(\fD)) \to \prod_{\lambda,\mu \in \Lambda} \scO(\fD)/(\cI^n \cdot \scO(\fD) + (\cI^{\lambda,\mu})^m)
\end{equation}
is an isomorphism. Then we have natural ring morphisms
\[
\prod_{\lambda,\mu \in \Lambda} \scO(\fD^{\hat{\lambda},\hat{\mu}}) \to \prod_{\lambda,\mu \in \Lambda} \scO(\fD) / (\cI^{\lambda,\mu})^m \to \prod_{\lambda,\mu \in \Lambda} \scO(\fD)/(\cI^n \cdot \scO(\fD) + (\cI^{\lambda,\mu})^m).
\]
Composing with the inverse of~\eqref{eqn:morphism-completions-2} we deduce the desired map~\eqref{eqn:morphism-completions}.

It is easy (and left to the reader) to check that the two morphisms considered above are inverse to each other.
\end{proof}

\begin{rmk}
\label{rmk:completion-tensor-product-ZHC}
If we denote by $\ZHC^\wedge$ the completion of $\ZHC$ with respect to the ideal $\cI \cdot \ZHC$, then as in~\eqref{eqn:Dwedge-tensor-prod} we have a canonical isomorphism $\scO(\ft^{*(1)}/W)^\wedge \otimes_{\scO(\ft^{*(1)}/W)} \ZHC \simto \ZHC^\wedge$, and therefore a canonical isomorphism
\[
\scO(\fD^\wedge) \simto \ZHC^\wedge \otimes_{\scO(\ft^{*(1)}/W)^\wedge} \ZHC^\wedge.
\]
Denoting by $\ZHC^{\hat{\lambda}}$ the completion of $\ZHC$ with respect to the ideal $\mathfrak{m}^\lambda$, for $\lambda \in \bbX$, the same considerations as for Lemma~\ref{lem:isom-completions} show that we have a canonical isomorphism
\[
\ZHC^\wedge \simto \prod_{\lambda \in \Lambda} \ZHC^{\hat{\lambda}},
\]
from which we obtain a decomposition
\[
\scO(\fD^\wedge) \simto \prod_{\lambda,\mu \in \Lambda} \ZHC^{\hat{\lambda}} \otimes_{\scO(\ft^{*(1)}/W)^\wedge} \ZHC^{\hat{\mu}}.
\]
This decomposition is in fact ``the same'' as the decomposition of Lemma~\ref{lem:isom-completions}, in the sense that for any $\lambda,\mu \in \Lambda$ we have an identification
\begin{equation}
\label{eqn:D-la-mu-tens}
\scO(\fD^{\hat{\lambda},\hat{\mu}}) \simto \ZHC^{\hat{\lambda}} \otimes_{\scO(\ft^{*(1)}/W)^\wedge} \ZHC^{\hat{\mu}}.
\end{equation}
\end{rmk}

%\begin{rmk}
% Another interpretation of the considerations in the proof above (in case $n=1$) is that given $\lambda,\mu \in \bbX$ there exists $N \in \Z_{\geq 0}$ such that for any $m \geq N$ we have $(\cI^{\lambda,\mu})^N \subset \cI \cdot \cZ + (\cI^{\lambda,\mu})^m$. As a consequence we have $(\cI^{\lambda,\mu})^N \cdot \cZ^{\hat{\lambda},\hat{\mu}} \subset \cI \cdot \cZ^{\hat{\lambda},\hat{\mu}}$. We also have $\cI \cdot \cZ^{\hat{\lambda},\hat{\mu}} \subset \cI^{\lambda,\mu} \cdot \cZ^{\hat{\lambda},\hat{\mu}}$; we deduce that for any finitely generated $\cZ^{\hat{\lambda},\hat{\mu}}$-module $M$ the natural morphism
% \[
%  M \to \varprojlim_{n \geq 1} M / \cI^n \cdot M
% \]
%is an isomorphism.
%\end{rmk}

We will also set 
\[
\cU^\wedge := \bigl( \cU\fg \otimes_{\ZFr} \cU\fg^\op \bigr) \otimes_{\scO(\fD)} \scO(\fD^\wedge) \cong \bigl( \cU\fg \otimes_{\ZFr} \cU\fg^\op \bigr) \otimes_{\scO(\ft^{*(1)}/W)} \scO(\ft^{*(1)}/W)^\wedge,
\]
where the isomorphism uses~\eqref{eqn:Dwedge-tensor-prod}.
% the completion of the $\cZ$-algebra $\cU\fg \otimes_{\ZFr} \cU\fg^\op$ with respect to the ideal $\cI \cdot \cZ$. By the same considerations as in~\S\ref{ss:HCBim-completed}, since $\cU\fg \otimes_{\ZFr} \cU\fg^\op$ is finite as a $\cZ$-module, we have a canonical isomorphism
%\[
%\cZ^\wedge \otimes_{\cZ} (\cU\fg \otimes_{\ZFr} \cU\fg) \simto \cU^\wedge.
%\]
%Comparing with~\eqref{eqn:completion-tensor} and using 
Then, as in~\S\ref{ss:HCBim-completed}, $\cU^\wedge$ is a left and right Noetherian ring, endowed with a structure of algebraic $G$-module, and Lemma~\ref{lem:isom-completions} implies that the natural morphism
\begin{equation}
\label{eqn:isom-completions-U}
\cU^\wedge \to \prod_{\lambda,\mu \in \Lambda} \cU^{\hat{\lambda},\hat{\mu}}
\end{equation}
is an algebra isomorphism. Below we will consider various categories of $\cU^\wedge$-modules; in fact we have
\begin{multline*}
\cU^\wedge \cong \bigl( \Ug \otimes_{\scO(\ft^{*(1)}/W)} \scO(\ft^{*(1)}/W)^\wedge \bigr) \otimes_{\ZFr \otimes_{\scO(\ft^{*(1)}/W)} \scO(\ft^{*(1)}/W)^\wedge}\\
 \bigl( \Ug \otimes_{\scO(\ft^{*(1)}/W)} \scO(\ft^{*(1)}/W)^\wedge \bigr)^\op;
\end{multline*}
a $\cU^\wedge$-module is therefore the same as a $(\Ug \otimes_{\scO(\ft^{*(1)}/W)} \scO(\ft^{*(1)}/W)^\wedge)$-bimodule on which the left and right actions of $\ZFr \otimes_{\scO(\ft^{*(1)}/W)} \scO(\ft^{*(1)}/W)^\wedge$ coincide.

%Following the same procedure as for the categories $\Mod^G_{\fin}(\cU^{\hat{\lambda},\hat{\mu}})$ and $\HCBim^{\hat{\lambda},\hat{\mu}}$, we can define the categories 
We will denote by $\Mod^G_{\fin}(\cU^\wedge)$ the abelian category of $G$-equivariant finitely generated $\cU^\wedge$-modules.
% and $\HCBim^\wedge$ of $G$-equivariant finitely generated $\cU^\wedge$-modules and Harish-Chandra $\cU^\wedge$-modules.
In view of~\eqref{eqn:isom-completions-U}, we have a canonical equivalence of categories
\begin{equation}
\label{eqn:Mod-Uwedge-sum}
\Mod^G_{\fin}(\cU^\wedge) \cong \bigoplus_{\lambda,\mu \in \Lambda} \Mod^G_{\fin}(\cU^{\hat{\lambda},\hat{\mu}}).
%\qquad \HCBim^\wedge \cong \bigoplus_{\lambda,\mu \in \Lambda} \HCBim^{\hat{\lambda},\hat{\mu}}.
\end{equation}
We also have a canonical exact functor
\begin{equation}
\label{eqn:Cwedge}
\mathsf{C}^\wedge : \Mod^G_{\fin}(\Ug \otimes_{\ZFr} \Ug^{\mathrm{op}}) \to \Mod^G_{\fin}(\cU^\wedge)
%\HCBim \to \HCBim^\wedge
\end{equation}
defined by $\mathsf{C}^\wedge(M)=\scO(\fD^\wedge) \otimes_{\scO(\fD)} M$.
%sending a module to its completion with respect to $\cI \cdot \cZ$, or in other words to its tensor product with $\cZ^\wedge$ over $\cZ$, which restricts to a functor from $\HCBim$ to $\HCBim^\wedge$. 
For the same reasons as above, for any $M$ in $\Mod^G_{\fin}(\Ug \otimes_{\ZFr} \Ug^{\mathrm{op}})$ we have a canonical isomorphism
\[
\mathsf{C}^\wedge(M) \cong \bigoplus_{\lambda,\mu \in \Lambda} \mathsf{C}^{\lambda,\mu}(M).
\]

An object $M$ in $\Mod^G_{\fin}(\cU^\wedge)$ will be called a \emph{completed Harish-Chandra bimodule} if the differential of the $G$-action coincides with the action given by $x \cdot m = xm-mx$ for $x \in \fg$ and $m \in M$, and we will denote by $\HCBim^\wedge$ the full subcategory of $\Mod^G_{\fin}(\cU^\wedge)$ consisting of such objects; this terminology is compatible with that of~\S\ref{ss:HCBim-completed} in the sense that~\eqref{eqn:Mod-Uwedge-sum} restricts to an equivalence of categories
\[
\HCBim^\wedge \cong \bigoplus_{\lambda,\mu \in \Lambda} \HCBim^{\hat{\lambda},\hat{\mu}}.
\]
We will denote by $\HCBim^\wedge_\diag$ the full additive subcategory of $\Mod^G_{\fin}(\cU^\wedge)$ whose objects are the direct summands of objects of the form $\mathsf{C}^\wedge(V \otimes \Ug)$ with $V$ in $\Rep(G)$. With this definition,~\eqref{eqn:Mod-Uwedge-sum} restricts further to an equivalence of categories
\[
\HCBim^\wedge_\diag \cong \bigoplus_{\lambda,\mu \in \Lambda} \HCBim^{\hat{\lambda},\hat{\mu}}_\diag.
\]

%-------------------------------------------------------
\subsection{Monoidal structure}
\label{ss:monoidal-bimodules}
%-------------------------------------------------------

We now want to define some analogue of the monoidal structure on $\Mod^G_{\mathrm{fg}}(\Ug \otimes_{\ZFr} \Ug^{\mathrm{op}})$ for the categories $\Mod^G_{\mathrm{fg}}(\cU^{\hat{\lambda},\hat{\mu}})$. More specifically,
given $\lambda,\mu,\nu$ in $\bbX$ we want to define a canonical bifunctor
\begin{equation}
\label{eqn:hatotimes}
(-) \hatotimes_\Ug (-) : \Mod^G_{\mathrm{fg}}(\cU^{\hat{\lambda},\hat{\mu}}) \times \Mod^G_{\mathrm{fg}}(\cU^{\hat{\mu},\hat{\nu}}) \to \Mod^G_{\mathrm{fg}}(\cU^{\hat{\lambda},\hat{\nu}})
\end{equation}
right exact on both sides, which restricts to a bifunctor
\[
\HCBim^{\hat{\lambda},\hat{\mu}} \times \HCBim^{\hat{\mu},\hat{\nu}} \to \HCBim^{\hat{\lambda},\hat{\nu}},
\]
these bifunctors satisfying natural unit and associativity axioms. Explicitly, we require that:
\begin{itemize}
\item
 in case $\mu=\lambda$ we have a canonical isomorphism of functors
\[
\cU^{\hat{\lambda}} \hatotimes_\Ug (-) \cong \id,
\]
and in case $\nu=\mu$ we have a canonical isomorphism
\[
(-) \hatotimes_{\Ug} \cU^{\hat{\mu}} \cong \id;
\]
\item
for four weights $\lambda,\mu,\nu,\eta \in \bbX$ we have an isomorphism
\[
\bigl( (-) \hatotimes_\Ug (-) \bigr) \hatotimes_\Ug (-) \simto (-) \hatotimes_\Ug \bigl( (-) \hatotimes_\Ug (-) \bigr)
\]
of functors from
\[
\Mod^G_{\mathrm{fg}}(\cU^{\hat{\lambda},\hat{\mu}}) \times \Mod^G_{\mathrm{fg}}(\cU^{\hat{\mu},\hat{\nu}}) \times \Mod^G_{\mathrm{fg}}(\cU^{\hat{\nu},\hat{\eta}})
\]
to $\Mod^G_{\mathrm{fg}}(\cU^{\hat{\lambda},\hat{\eta}})$.
\end{itemize}
 In particular, in case $\lambda=\mu=\nu$, this construction will equip $\Mod^G_{\mathrm{fg}}(\cU^{\hat{\lambda},\hat{\lambda}})$ with the structure of a monoidal category.
 
 For this we can assume that all the weights involved belong to the subset $\Lambda$ chosen in~\S\ref{ss:comparison}. It therefore suffices to construct a monoidal structure of the category $\Mod^G_{\fin}(\cU^\wedge)$, with monoidal unit $\mathsf{C}^\wedge(\bk \otimes \Ug)$; the bifunctor~\eqref{eqn:hatotimes} will then be deduced by restriction to the factor $\Mod^G_{\mathrm{fg}}(\cU^{\hat{\lambda},\hat{\mu}}) \times \Mod^G_{\mathrm{fg}}(\cU^{\hat{\mu},\hat{\nu}})$ in the decomposition~\eqref{eqn:Mod-Uwedge-sum}.
 
% Each object of $\Mod^G_{\fin}(\cU^\wedge)$ can be regarded as a $G$-equivariant coherent sheaf on the $G$-scheme
% \[
% \fg^{*(1)} \times_{\ft^{*(1)}/W} \fD^\wedge.
% \]
% Now, 
Recall that we have
\[
\cU^\wedge = \bigl( \cU\fg \otimes_{\ZFr} \cU\fg^\op \bigr) \otimes_{\scO(\ft^{*(1)}/W)} \scO(\ft^{*(1)}/W)^\wedge.
\]
Given $M,N$ in $\Mod^G_{\fin}(\cU^\wedge)$, we set
\[
M \hatotimes_\Ug N := M \otimes_{\Ug \otimes_{\scO(\ft^{*(1)}/W)} \scO(\ft^{*(1)}/W)^\wedge} N,
\]
where the right action of $\Ug \otimes_{\scO(\ft^{*(1)}/W)} \scO(\ft^{*(1)}/W)^\wedge$ on $M$ is induced by the action of the second copy of $\Ug$, and the left action on $N$ is induced by the action of the first copy of $\Ug$. This tensor product
admits compatible actions of $\cU\fg \otimes_{\ZFr} \cU\fg^\op$ (induced by the action of the first copy of $\Ug$ on $M$, and of the second copy of $\Ug$ on $N$) and of $\scO(\ft^{*(1)}/W)^\wedge$, hence of $\cU^\wedge$. This modules is moreover finitely generated, since it is finitely generated over $\scO(\fg^{*(1)}) \otimes_{\scO(\ft^{*(1)}/W)} \scO(\ft^{*(1)}/W)^\wedge$, and it admits a canonical (diagonal) $G$-module structure. 
%Setting
%\[
% M \hatotimes_\Ug N := M \otimes_{\Ug \otimes_{\scO(\ft^{*(1)}/W)} \scO(\ft^{*(1)}/W)^\wedge} N,
%\]
%we therefore 
In this way we obtain a bifunctor
\[
\Mod^G_{\mathrm{fg}}(\cU^\wedge) \times \Mod^G_{\mathrm{fg}}(\cU^\wedge) \to \Mod^G_{\mathrm{fg}}(\cU^\wedge),
\]
which is easily seen to provide a monoidal structure with monoidal unit $\mathsf{C}^\wedge(\bk \otimes \Ug)$. It is easily seen also that the functor~\eqref{eqn:Cwedge} has a canonical monoidal structure; using~\eqref{eqn:isom-tensor-product-diag} we deduce that the full subcategory $\HCBim^\wedge_\diag$ is a monoidal subcategory.

If $\lambda,\mu,\nu \in \bbX$ and if $M$ belongs to $\Mod^G_{\fin}(\cU^{\hat{\lambda},\hat{\mu}})$ and $N$ belongs to $\Mod^G_{\fin}(\cU^{\hat{\mu},\hat{\nu}})$, then seeing $M$ and $N$ as objects in $\Mod^G_{\fin}(\cU^\wedge)$ via~\eqref{eqn:Mod-Uwedge-sum} the product $M \hatotimes_{\Ug} N$ belongs to the factor $\Mod^G_{\fin}(\cU^{\hat{\lambda},\hat{\nu}})$, which provides the desired bifunctor~\eqref{eqn:hatotimes}. (In fact, the action of the left copy of $\ZHC$ factors through an action of $\ZHC^{\hat{\lambda}}$, and that of the right copy factors through an action of $\ZHC^{\hat{\nu}}$, which justifies the claim in view of~\eqref{eqn:D-la-mu-tens}.) From the corresponding property for $\Mod^G_{\fin}(\cU^\wedge)$ we deduce that the subcategories $\HCBim^{\hat{\lambda},\hat{\mu}}$ and $\HCBim^{\hat{\lambda},\hat{\mu}}_\diag$ are stable (in the obvious sense) under this bifunctor. In this setting the functor $\mathsf{C}^{\hat{\lambda},\hat{\mu}}$ satisfies
 \begin{equation*}
%  \label{eqn:completion-tensor-product}
 \mathsf{C}^{\hat{\lambda},\hat{\mu}}(M \otimes_{\Ug} N) \cong \bigoplus_{\nu \in \Lambda} \mathsf{C}^{\hat{\lambda},\hat{\nu}}(M) \hatotimes_{\Ug} \mathsf{C}^{\hat{\nu},\hat{\mu}}(N)
 \end{equation*}
 for any $M,N$ in $\Mod^G_{\fin}(\Ug \otimes_{\ZFr} \Ug^{\mathrm{op}})$.

 \subsection{Restriction to the Kostant section}
 \label{ss:centralizer-Kostant}
 %---------------------------------------------------------
 
 Recall the constructions of~\S\ref{ss:Kostant-section-Abe} applied to the group $\bG=G^{(1)}$. In particular, we have a Kostant section $\bS^* \subset \bg^*=\fg^{*(1)}$, and group schemes $\bbJ^*_\reg$ over $\fg_\reg^{*(1)}$ and $\bbJ^*_\bS$ over $\bS^*$.

We also set
\[
\bbI^*_\bS := (G \times \bS^{*}) \times_{G^{(1)} \times \bS^{*}} \bbJ^{*}_\bS.
%, \quad \text{resp.} \quad \bbI^*_\adj := (G \times \ft^{*(1)}/W) \times_{G^{(1)} \times \ft^{*(1)}/W} \bbJ^{*}_\adj
\]
where the map $G \times \bS^{*} \to G^{(1)} \times \bS^{*}$
%, resp.~$G \times \ft^{*(1)}/W \to G^{(1)} \times \ft^{*(1)}/W$,
 is the product of the Frobenius morphism of $G$ and the identity of $\bS^{*}$.
% , resp.~$\ft^{*(1)}/W$. 
 Since $G$ is smooth its Frobenius morphism is flat (see e.g.~\cite[\S 1.1]{bk}), and therefore $\bbI^*_\bS$
% , resp.~$\bbI^*_\adj$, 
 is a flat affine group scheme over $\bS^{*}$.
% , resp.~$\ft^{*(1)}/W$. 
 By construction $\bbI^*_\bS$ contains $G_1 \times \bS^*$ as a normal subgroup, and the quotient identifies with $\bbJ^{*}_\bS$.
 %; the same holds for $\bbI^*_\adj$.
Note also that the morphism~\eqref{eqn:morph-J-T} induces (after restriction to $\bS$ and  composition with the projection $\bbI^*_\bS \to \bbJ^*_\bS$) a group-scheme morphism
\begin{equation}
 \label{eqn:morph-I-T}
 \ft^{*(1)} \times_{\ft^{*(1)}/W} \bbI^*_\bS \to (\ft^{*(1)} \times_{\ft^{*(1)}/W} \bS^*) \times T^{(1)}.
 %, \\ \ft^{*(1)} \times_{\ft^{*(1)}/W} \bbI^*_\adj \to \ft^{*(1)} \times T^{(1)}.
\end{equation}

Finally, we set
\[
\cU_{\bS}\fg := \Ug \otimes_{\ZFr} \scO(\bS^{*}),
\]
where $\scO(\bS^{*})$ is seen as a $\ZFr$-algebra via the identification~\eqref{eqn:ZFr}. If we set
\[
\fC_\bS := \bS^{*} \times_{\ft^{*(1)} / W} \ft^*/(W,\bullet),
\]
then the projection $\fC_\bS \to \ft^*/(W,\bullet)$ is an isomorphism, and
$\cU_{\bS}\fg$ is an $\scO(\fC_\bS)$-algebra. Recall that the algebra $\cU\fg$ can be seen as a $G$-equivariant $\scO(\fC)$-algebra (see~\S\ref{ss:center}). Using the general construction recalled in~\cite[\S 2.2]{mr}, from this we deduce on $\cU_{\bS} \fg$ a natural structure of module for the flat affine group scheme $\fC_\bS \times_{\bS^{*}} \bbI^*_\bS$ over $\fC_\bS$, such that the multiplication morphism is equivariant.
%-equivariant $\scO(\fC_\cS)$-algebra.

%---------------------------------------------------------------
\subsection{(Completed) Harish-Chandra bimodules for \texorpdfstring{$\cU_{\bS}\fg$}{USg}}
\label{ss:HCBim-S}
%---------------------------------------------------------------

We now want to define, given $\lambda,\mu \in \bbX$, categories analogous to $\Mod^G_{\fin}(\cU^{\hat{\lambda},\hat{\mu}})$ and $\HCBim^{\hat{\lambda},\hat{\mu}}$ but for the algebra $\cU_{\bS}\fg$ in place of $\Ug$.
We start with the non-completed version.

%The $G$-action on $\Ug$ restricts to a structure of $\bbI^*_\cS$-module on $\cU_\cS \fg$. Below we will also 
Let us consider the category $\Mod^{\bbI}_{\mathrm{fg}}(\cU_\bS \fg \otimes_{\scO(\bS^{*})} \cU_\bS \fg^{\mathrm{op}})$ of finitely generated $\cU_\bS \fg \otimes_{\scO(\bS^{*})} \cU_\bS \fg^{\mathrm{op}}$-modules endowed with a compatible structure of $\bbI^*_\bS$-module. Since $\bbI^*_\bS$ is flat over $\bS^{*}$, this category is abelian.
Here $\cU_\bS \fg \otimes_{\scO(\bS^{*})} (\cU_\bS \fg)^{\mathrm{op}}$ is an algebra over
\[
\cZ_\bS := \scO(\ft^*/(W,\bullet)) \otimes_{\scO(\ft^{*(1)}/W)} \scO(\bS^{*}) \otimes_{\scO(\ft^{*(1)}/W)} \scO(\ft^*/(W,\bullet)),
\]
%\cong \scO(\ft^*/(W,\bullet)) \otimes_{\scO(\ft^{*(1)}/W)} \scO(\ft^*/(W,\bullet)) = \scO(\fD).
%\end{multline*}
which identifies with $\scO(\fD)$ via the composition of natural algebra morphisms
\[
\scO(\fD) \to \scO(\fg^{*(1)} \times_{\ft^{*(1)}/W} \fD) = \cZ \to \cZ_\bS. 
\]
As in~\S\ref{ss:HCBim} the tensor product $\otimes_{\cU_\bS \fg}$ defines a monoidal structure on this category, and using the construction of~\cite[\S 2.2]{mr} considered above the functor $\scO(\bS^{*}) \otimes_{\ZFr} (-)$ defines a monoidal functor
\begin{equation}
\label{eqn:rest-bimodules}
\Mod^G_{\mathrm{fg}}(\Ug \otimes_{\ZFr} \Ug^{\mathrm{op}}) \to \Mod^{\bbI}_{\mathrm{fg}}(\cU_\bS \fg \otimes_{\scO(\bS^{*})} \cU_\bS \fg^{\mathrm{op}}).
\end{equation}
An object $M$ in $\Mod^{\bbI}_{\mathrm{fg}}(\cU_\bS \fg \otimes_{\scO(\bS^{*})} \cU_\bS \fg^{\mathrm{op}})$ will be called a \emph{Harish-Chandra $\cU_\bS \fg$-bimodule} if the restriction of the action of $\bbI^*_\bS$ to $G_1 \times \bS^*$, seen as an action of the algebra $\mathrm{Dist}(G_1) = \cU_0 \fg$, coincides with the action determined by the rule $x \cdot m = xm-mx$ for $x \in \fg$ and $m \in M$. We will denote by $\HCBim_\bS$ the full subcategory of $\Mod^{\bbI}_{\mathrm{fg}}(\cU_\bS \fg \otimes_{\scO(\bS^{*})} \cU_\bS \fg^{\mathrm{op}})$ consisting of such objects; then the functor~\eqref{eqn:rest-bimodules} restricts to a functor
\[
\HCBim \to \HCBim_{\bS}.
\]
As in~\S\ref{ss:HCBim}, for any $M$ in $\Mod^{\bbI}_{\mathrm{fg}}(\cU_\bS \fg \otimes_{\scO(\bS^{*})} \cU_\bS \fg^{\mathrm{op}})$ the action of $\bbI^*_\bS$ on $M$ extends to an action of the semi-direct product
\[
\bbI^*_\bS \ltimes G_1 = \bbI^*_\bS \ltimes_\bS (G_1 \times \bS);
\]
the Harish-Chandra $\cU_\bS \fg$-bimodules are those objects on which this action factors through the multiplication morphism $\bbI^*_\bS \ltimes G_1 \to \bbI^*_\bS$.

Now we add completions to the picture. Given $\lambda,\mu \in \bbX$, we will denote by $\cZ_\bS^{\hat{\lambda},\hat{\mu}}$ the completion of $\cZ_\bS$ with respect to the maximal ideal $\cI^{\lambda,\mu}_\bS := \cI^{\lambda,\mu} \cdot \cZ_\bS$, so that we have a canonical isomorphism $\scO(\fD^{\hat{\lambda},\hat{\mu}}) \simto \cZ_\bS^{\hat{\lambda},\hat{\mu}}$.
%one can define as in~\S\ref{ss:HCBim-completed} 
We also set
\[
\cU^{\hat{\lambda},\hat{\mu}}_\bS := \cZ_\bS^{\hat{\lambda},\hat{\mu}} \otimes_{\cZ_\bS} \bigl( \cU_\bS \fg \otimes_{\scO(\bS^{*})} \cU_\bS \fg^\op \bigr).
\]
In this setting $\cU_\bS \fg \otimes_{\scO(\bS^{*})} \cU_\bS \fg^\op$ is finitely generated as a $\cZ_\bS$-module, so that $\cU^{\hat{\lambda},\hat{\mu}}_\bS$ identifies with
the completion of $\cU_\bS \fg \otimes_{\scO(\bS^{*})} \cU_\bS \fg^\op$ with respect to the ideal
$\cI^{\lambda,\mu} \cdot (\cU_\bS \fg \otimes_{\scO(\bS^{*})} \cU_\bS \fg^\op)$.
% = \mathfrak{m}^\lambda \otimes_{\scO(\ft^{*(1)}/W)} \scO(\ft^*/(W,\bullet)) + \scO(\ft^*/(W,\bullet)) \otimes_{\scO(\ft^{*(1)}/W)} \mathfrak{m}^\mu
%(Here $I^{\lambda,\mu}_\cS$ is a maximal ideal in $\scO \bigl( \ft^*/(W,\bullet) \times_{\ft^{*(1)}/W} \ft^*/(W,\bullet) \bigr)$.) 
%Note that here the algebra $\cU_\cS \fg \otimes_{\scO(\cS^{*(1)})} (\cU_\cS \fg)^\op$ is finitely generated as a $\cZ_\cS$-module, so that the natural morphism
%\[
%\scO(\fD^{\hat{\lambda},\hat{\mu}}) \otimes_{\scO(\fD)} \bigl( \cU_\cS \fg \otimes_{\scO(\cS^{*(1)})} (\cU_\cS \fg)^\op \bigr) \to \cU^{\hat{\lambda},\hat{\mu}}_\cS
%\]
%is an isomorphism.
If we denote by $\bbI^{\hat{\lambda},\hat{\mu}}_\bS$ the pullback of $\bbI^*_\bS$ under the natural morphism $\mathrm{Spec}(\cZ_\bS^{\hat{\lambda},\hat{\mu}}) \to \bS^{*}$, then $\bbI^{\hat{\lambda},\hat{\mu}}_\bS$ is a flat group scheme over $\mathrm{Spec}(\cZ_\bS^{\hat{\lambda},\hat{\mu}})$, and
the $\bbI^*_\bS$-module structure on $\cU_\bS \fg$ induces a natural $\bbI^{\hat{\lambda},\hat{\mu}}_\bS$-module structure on $\cU^{\hat{\lambda},\hat{\mu}}_\bS$.

The algebra $\cU^{\hat{\lambda},\hat{\mu}}_\bS$ is left and right Noetherian, and
we will denote by $\Mod^{\bbI}_{\mathrm{fg}}(\cU_\bS^{\hat{\lambda},\hat{\mu}})$ the abelian category of $\bbI^{\hat{\lambda},\hat{\mu}}_\bS$-equivariant finitely generated $\cU_\bS^{\hat{\lambda},\hat{\mu}}$-modules.
% can also be defined by the same procedure as for $\Mod^G_{\mathrm{fg}}(\cU^{\hat{\lambda},\hat{\mu}})$, namely as the category of finitely generated $\cU_\cS^{\hat{\lambda},\hat{\mu}}$-modules $M$ together with a $(\fC_\cS \times_{\cS^{*(1)}} \fC_\cS) \times_{\cS^{*(1)}} \bbI^*_\cS$-module structure on each quotient $M/ (\cI^{\lambda,\mu}_\cS)^n \cdot M$, such that the quotient morphism
%\[
%M/ (\cI^{\lambda,\mu}_\cS)^{n+1} \cdot M \to M/ (\cI^{\lambda,\mu}_\cS)^n \cdot M
%\]
%is equivariant for any $n \geq 1$, and similarly for the action morphism.
Note that we have
\[
\cU^{\hat{\lambda},\hat{\mu}}_\bS = \cU^{\hat{\lambda},\hat{\mu}} \otimes_{\ZFr} \scO(\bS^{*});
\]
in particular, the functor $\scO(\bS^{*}) \otimes_{\ZFr} (-)$ defines a natural functor
\begin{equation}
\label{eqn:restr-S-bimodules}
\Mod^G_{\mathrm{fg}}(\cU^{\hat{\lambda},\hat{\mu}}) \to \Mod^{\bbI}_{\mathrm{fg}}(\cU_\bS^{\hat{\lambda},\hat{\mu}}).
\end{equation}
(Standard arguments show that this functor is exact, but this will not play any role below.)
%Similarly, if we set
%\[
%\cZ_\cS^{\hat{\lambda},\hat{\mu}} := \scO(\cS^{*(1)}) \otimes_{\scO(\fg^{*(1)})} \cZ^{\hat{\lambda},\hat{\mu}},
%\]
%then
%%\[
%%\fC_\cS^{\hat{\lambda},\hat{\mu}} := \cS^{*(1)} \times_{\fg^{*(1)}} \fC^{\hat{\lambda},\hat{\mu}},
%%\]
%$\cU_\cS^{\hat{\lambda},\hat{\mu}}$ is a $\cZ_\cS^{\hat{\lambda},\hat{\mu}}$-algebra,
%%an $\scO(\fC_\cS^{\hat{\lambda},\hat{\mu}})$-algebra, 
%and $\cZ_\cS^{\hat{\lambda},\hat{\mu}}$ identifies with the completion of the algebra $\cZ_\cS$ with respect to $\cI_\cS^{\lambda,\mu}$.
%$\fC_\cS^{\hat{\lambda},\hat{\mu}}$ identifies with the completion of the algebra $\scO(\ft^*/(W,\bullet) \times_{\ft^{*(1)}/W} \ft^*/(W,\bullet))$ with respect to $I^{\lambda,\mu}_\cS$.
If $\lambda,\mu \in \bbX$ belong to the closure of the fundamental alcove, we will also set
\[
 \sfP^{\lambda,\mu}_\bS := \scO(\bS^{*}) \otimes_{\scO(\fg^{*(1)})} \sfP^{\lambda,\mu}.
\]

One defines the notion of \emph{completed Harish-Chandra $\cU_\bS \fg$-bimodules} by imposing the same condition as for $\HCBim_\bS$. The full subcategory of $\Mod^{\bbI}_{\mathrm{fg}}(\cU_\bS^{\hat{\lambda},\hat{\mu}})$ consisting of such objects will be denoted $\HCBim_\bS^{\hat{\lambda},\hat{\mu}}$; it is clear that the functor~\eqref{eqn:restr-S-bimodules} restricts to a functor
\[
 \HCBim^{\hat{\lambda},\hat{\mu}} \to \HCBim_\bS^{\hat{\lambda},\hat{\mu}}.
\]

%=============================

Using considerations similar to those of~\S\ref{ss:monoidal-bimodules} one constructs, again for $\lambda,\mu,\nu \in \bbX$, a canonical bifunctor
\begin{equation}
\label{eqn:hatotimes-S}
(-) \hatotimes_{\cU_\bS \fg} (-) : \Mod^{\bbI}_{\mathrm{fg}}(\cU_\bS^{\hat{\lambda},\hat{\mu}}) \times \Mod^{\bbI}_{\mathrm{fg}}(\cU_\bS^{\hat{\mu},\hat{\nu}}) \to \Mod^{\bbI}_{\mathrm{fg}}(\cU_\bS^{\hat{\lambda},\hat{\nu}})
\end{equation}
which factors through a bifunctor
\[
\HCBim_\bS^{\hat{\lambda},\hat{\mu}} \times \HCBim_\bS^{\hat{\mu},\hat{\nu}} \to \HCBim_\bS^{\hat{\lambda},\hat{\nu}},
\]
this construction being unital, associative, and compatible in the natural way with the bifunctors $(-) \hatotimes_{\cU \fg} (-)$ via the functors~\eqref{eqn:restr-S-bimodules}. More explicitly, one remarks that if $\cZ_\bS^\wedge$ is the completion of $\cZ_\bS$ with respect to the ideal $\cI \cdot \cZ_\bS$, then if we set
\[
\cU_\bS^\wedge := \cZ_\bS^\wedge \otimes_{\cZ_\bS} \bigl( \cU_\bS \fg \otimes_{\scO(\bS^{*})} \cU_\bS \fg^\op \bigr) = \scO(\bS^{*(1)}) \otimes_{\scO(\fg^{*(1)})} \cU^\wedge,
\]
 as in~\eqref{eqn:isom-completions-U} we have a canonical algebra isomorphism
\[
\cU_\bS^\wedge \simto \prod_{\lambda,\mu \in \Lambda} \cU_\bS^{\hat{\lambda},\hat{\mu}}.
\]
If we denote by $\bbI_\bS^\wedge$ the pullback of $\bbI^*_\bS$ to $\mathrm{Spec}(\cZ_\bS^\wedge)$, then one can consider the abelian 
category $\Mod_{\fin}^{\mathbb{I}}(\cU_\bS^\wedge)$ of finitely generated $\bbI_\bS^\wedge$-equivariant $\cU_\bS^\wedge$-modules, and the bifunctor
\begin{equation}
\label{eqn:hatotimes-S-wedge}
(-) \hatotimes_{\cU_\bS \fg} (-) : \Mod_{\fin}^{\bbI}(\cU_\bS^\wedge) \times \Mod_{\fin}^{\bbI}(\cU_\bS^\wedge) \to \Mod_{\fin}^{\bbI}(\cU_\bS^\wedge)
\end{equation}
defined by
\[
M \hatotimes_{\cU_\bS \fg} N = M \otimes_{\cU_\bS \fg \otimes_{\scO(\ft^{*(1)}/W)} \scO(\ft^{*(1)}/W)^\wedge} N
%\varprojlim_{n \geq 1} \, (M/\cI^n \cdot M) \otimes_{\cU_\cS \fg} (N / \cI^n \cdot N).
\]
defines a monoidal structure on this category. Note that any finitely generated $\cU_\bS^\wedge$-module $M$ is also finitely generated as a $\cZ_\bS^\wedge$-module, so that the natural morphism $M \to \varprojlim_n M/\cI^n \cdot M$ is an isomorphism (see~\cite[\href{https://stacks.math.columbia.edu/tag/00MA}{Tag 00MA}]{stacks-project}); the monoidal product considered above therefore satisfies
\[
M \hatotimes_{\cU_\bS \fg} N \cong \varprojlim_{n \geq 1} \, (M/\cI^n \cdot M) \otimes_{\cU_\bS \fg} (N / \cI^n \cdot N)
\]
as $\cU_\bS^\wedge$-modules,
for any $M,N$ in $\Mod_{\fin}^{\bbI}(\cU_\bS^\wedge)$. The functor $\scO(\bS^{*}) \otimes_{\ZFr} (-)$ also induces a functor
\begin{equation}
\label{eqn:rest-wedge-S}
\Mod^G_{\fin}(\cU^\wedge) \to \Mod_{\fin}^{\bbI}(\cU_\bS^\wedge)
\end{equation}
which admits a canonical monoidal structure. The composition of this functor with the functor $\mathsf{C}^\wedge$ of~\eqref{eqn:Cwedge} will be denoted $\mathsf{C}^\wedge_\bS$.

As in~\eqref{eqn:Mod-Uwedge-sum}
we have
\begin{equation}
\label{eqn:Mod-UwedgeS-sum}
\Mod^{\bbI}_{\fin}(\cU_\bS^\wedge) \cong \bigoplus_{\lambda,\mu \in \Lambda} \Mod^{\bbI}_{\fin}(\cU_\bS^{\hat{\lambda},\hat{\mu}}),
\end{equation}
and the bifunctor~\eqref{eqn:hatotimes-S} is then obtained by restriction of~\eqref{eqn:hatotimes-S-wedge} to the appropriate summands. In case $\lambda=\mu=\nu$, this bifunctor equips $\Mod^{\bbI}_{\fin}(\cU_\bS^{\hat{\lambda},\hat{\lambda}})$ with a structure of monoidal category, with unit object
\[
\cU_\bS^{\hat{\lambda}} := \scO(\bS^{*}) \otimes_{\scO(\fg^{*(1)})} \cU^{\hat{\lambda}}.
\]
With this definition, it is clear that the functor~\eqref{eqn:restr-S-bimodules} is compatible with the bifunctors $(-) \hatotimes_{\cU \fg} (-)$ and $(-) \hatotimes_{\cU_\bS \fg} (-)$ in the obvious way.

\begin{lem}
\label{lem:convolution-adjoint}
For any $\lambda,\mu \in \bbX$ which belong to the closure of the fundamental alcove and any $\nu \in \bbX$, the functor
\[
\sfP^{\lambda,\mu}_\bS \hatotimes_{\cU_\bS \fg} (-) : \Mod^{\bbI}_{\mathrm{fg}}(\cU_\bS^{\hat{\mu},\hat{\nu}}) \to \Mod^{\bbI}_{\mathrm{fg}}(\cU_\bS^{\hat{\lambda},\hat{\nu}})
%, \quad \text{resp.} \quad (-) \hatotimes_{\cU_\cS \fg} \mathbb{P}^{\lambda,\mu}_\cS
\]
is both left and right adjoint to the functor
\[
\sfP^{\mu,\lambda}_\bS \hatotimes_{\cU_\bS \fg} (-) : \Mod^{\bbI}_{\mathrm{fg}}(\cU_\bS^{\hat{\lambda},\hat{\nu}}) \to \Mod^{\bbI}_{\mathrm{fg}}(\cU_\bS^{\hat{\mu},\hat{\nu}}).
%, \quad \text{resp.} \quad (-) \hatotimes_{\cU_\cS \fg} \mathbb{P}^{\lambda,\mu}_\cS
\]
A similar property holds for the functors $(-) \hatotimes_{\cU_\bS \fg} \sfP^{\lambda,\mu}_\bS$ and $(-) \hatotimes_{\cU_\bS \fg} \sfP^{\mu,\lambda}_\bS$.
\end{lem}

\begin{proof}
We prove the case of convolution on the left; convolution on the right can be treated similarly. We remark that for any $V \in \Rep(G)$, the functor
\[
\mathsf{C}_\bS^\wedge(V \otimes \Ug) \hatotimes_{\cU_\bS \fg} (-) : \Mod_{\fin}^{\bbI}(\cU_\bS^\wedge) \to \Mod_{\fin}^{\bbI}(\cU_\bS^\wedge)
\]
is both left and right adjoint to the functor
\[
\mathsf{C}_\bS^\wedge(V^* \otimes \Ug) \hatotimes_{\cU_\bS \fg} (-) : \Mod_{\fin}^{\bbI}(\cU_\bS^\wedge) \to \Mod_{\fin}^{\bbI}(\cU_\bS^\wedge).
\]
(In fact, these functors can be realized more concretely as tensor product with $V$ and $V^*$ respectively.) On the other hand, the inclusion functor
\[
\Mod^{\bbI}_{\mathrm{fg}}(\cU_\bS^{\hat{\lambda},\hat{\nu}}) \to \Mod_{\fin}^{\bbI}(\cU_\bS^\wedge)
\]
(see~\eqref{eqn:Mod-UwedgeS-sum})
is both left and right adjoint to the corresponding projection functor
\[
\Mod_{\fin}^{\bbI}(\cU_\bS^\wedge) \to \Mod^{\bbI}_{\mathrm{fg}}(\cU_\bS^{\hat{\lambda},\hat{\nu}}),
\]
and similarly for $\mu$ in place of $\lambda$. The desired claim follows, since the functors $\sfP^{\lambda,\mu}_\bS \hatotimes_{\cU_\bS \fg} (-)$ and $\sfP^{\mu,\lambda}_\bS \hatotimes_{\cU_\bS \fg} (-)$ are isomorphic to compositions of functors of this form. (More specifically, if $\nu \in \bbX$ is the only dominant $W$-translate of $\mu-\lambda$, the functor $\sfP^{\lambda,\mu}_\bS \hatotimes_{\cU_\bS \fg} (-)$ involves the module $\Sim(\nu)$, and the functor $\sfP^{\mu,\lambda}_\bS \hatotimes_{\cU_\bS \fg} (-)$ involves the module $\Sim(-w_0(\nu))$; here we fix an isomorphism $\Sim(\nu)^* \cong \Sim(-w_0(\nu))$.)
\end{proof}

%-----------------------------------------------------------------------------
\subsection{Restriction to the Kostant section for diagonally induced bimodules}
%-----------------------------------------------------------------------------

In this subsection we aim at proving the following claim.

\begin{prop}
\label{prop:rest-S-fully-faithful}
For any $\lambda,\mu \in \bbX$, the functor~\eqref{eqn:restr-S-bimodules} is fully faithful on the subcategory $\HCBim^{\hat{\lambda},\hat{\mu}}_\diag$.
\end{prop}

The proof of this proposition will use some standard properties stated in the following lemma. Here, $k$ is a commutative ring, $A$ is a left Noetherian $k$-algebra, and $H$ is an affine $k$-group scheme. For any commutative $k$-algebra $k'$ we set $H_{k'}:=\mathrm{Spec}(k') \times_{\mathrm{Spec}(k)} H$. If $A$ admits an action of $H$ by algebra automorphisms, we denote by $\Mod^H(A)$ the category of $H$-equivariant $A$-modules. (This category is abelian if $H$ is flat over $k$.)

\begin{lem}
\phantomsection
\label{lem:modules}
\begin{enumerate}
\item 
\label{it:lem-mod-1}
If $M,N$ are $A$-modules with $M$ finitely generated, for any flat $k$-module $V$ we have a canonical isomorphism
\[
\Hom_A(M,N \otimes_k V) \simto \Hom_A(M,N) \otimes_k V
\]
where $N \otimes_k V$ is regarded as an $A$-module for the action on the first factor.
\item 
\label{it:lem-mod-2}
If $H$ is flat over $k$, and if $M,N$ are $H$-equivariant $A$-modules, with $M$ finitely generated as an $A$-module, then the $k$-module $\Hom_A(M,N)$ admits a canonical structure of $H$-module, and we have a canonical isomorphism
\[
\Hom_{\Mod^H(A)}(M,N) \simto \bigl( \Hom_A(M,N) \bigr)^H.
\]
\item 
\label{it:lem-mod-3}
If $k'$ is a flat commutative $k$-algebra, for any $H$-module we have a canonical isomorphism
\[
(k' \otimes_k M)^{H_{k'}} \simto k' \otimes_k M^H.
\]
\end{enumerate}
\end{lem}

\begin{proof}
\eqref{it:lem-mod-1} We consider a presentation $A^{\oplus n} \to A^{\oplus m} \to M \to 0$; we then have exact sequences
\[
0 \to \Hom_A(M,N \otimes_k V) \to \Hom_A(A^{\oplus m},N \otimes_k V) \to \Hom_A(A^{\oplus n},N \otimes_k V)
\]
and
\[
0 \to \Hom_A(M,N) \otimes_k V \to \Hom_A(A^{\oplus m},N) \otimes_k V \to \Hom_A(A^{\oplus n},N) \otimes_k V,
\]
where we use the flatness of $V$. It is clear that the second and third terms in these exact sequences identify, and we deduce an identification of the first terms.

\eqref{it:lem-mod-2}
We first consider the morphism
\[
\Hom_A(M,N) \to \Hom_A(M,N \otimes_k \scO(H))
\]
which sends a morphism $\varphi : M \to N$ to the composition
\[
M \xrightarrow{(\id \otimes S) \otimes \Delta_M} M \otimes_k \scO(H) \xrightarrow{\varphi \otimes \id} N \otimes_k \scO(H) \to N \otimes_k \scO(H)
\]
where the third morphism sends $n \otimes g$ to $n_{(1)} \otimes f_{(2)} g$ in Sweedler's notation. By~\eqref{it:lem-mod-1} we have $\Hom_A(M,N \otimes_k \scO(H)) \cong \Hom_A(M,N) \otimes_k \scO(H)$, so that this morphism can be seen as a morphism $\Hom_A(M,N) \to \Hom_A(M,N) \otimes_k \scO(H)$, which can be checked to provide an $\scO(H)$-comodule structure (i.e.~an $H$-module structure) on $\Hom_A(M,N)$. The isomorphism $\Hom_{\Mod^H(A)}(M,N) \simto ( \Hom_A(M,N) )^H$ is then clear from definitions.

\eqref{it:lem-mod-3}
See~\cite[\S I.2.10, Equation (3)]{jantzen}.
\end{proof}

With these tools we can give the proof of the proposition.

\begin{proof}[Proof of Proposition~\ref{prop:rest-S-fully-faithful}]
To prove the proposition, it suffices to prove that the functor~\eqref{eqn:rest-wedge-S} is fully faithful on the subcategory $\HCBim^\wedge_\diag$, which will follow if we prove that it induces an isomorphism
\[
\Hom_{\Mod^G_{\fin}(\cU^\wedge)}(\mathsf{C}^\wedge(M),\mathsf{C}^\wedge(V \otimes \Ug)) \simto \Hom_{\Mod^{\bbI}_{\fin}(\cU_\bS^\wedge)}(\mathsf{C}_\bS^\wedge(M),\mathsf{C}_\bS^\wedge(V \otimes \Ug))
\]
for any $M$ in $\Mod^G_\fin(\Ug \otimes_{\ZFr} \Ug^\op)$ and $V \in \Rep(G)$.

For $N$ in $\Mod^G_\fin(\Ug \otimes_{\ZFr} \Ug^\op)$, we have
\begin{multline*}
\Hom_{\Mod^G_{\fin}(\cU^\wedge)}(\mathsf{C}^\wedge(M),\mathsf{C}^\wedge(N)) \cong \bigl( \Hom_{\cU^\wedge}(\mathsf{C}^\wedge(M),\mathsf{C}^\wedge(N)) \bigr)^G \\
\cong \bigl( \Hom_{\Ug \otimes_{\ZFr} \Ug^\op}(M,\mathsf{C}^\wedge(N)) \bigr)^G \cong \bigl( \Hom_{\Ug \otimes_{\ZFr} \Ug^\op}(M,N) \otimes_{\scO(\fD)} \scO(\fD^\wedge) \bigr)^G,
\end{multline*}
where the first isomorphism uses Lemma~\ref{lem:modules}\eqref{it:lem-mod-2}, and the third one uses Lem\-ma~\ref{lem:modules}\eqref{it:lem-mod-1}. Using Lemma~\ref{lem:modules}\eqref{it:lem-mod-3}, we deduce isomorphisms
\begin{multline*}
\Hom_{\Mod^G_{\fin}(\cU^\wedge)}(\mathsf{C}^\wedge(M),\mathsf{C}^\wedge(N)) \cong \bigl( \Hom_{\Ug \otimes_{\ZFr} \Ug^\op}(M,N) \bigr)^G \otimes_{\scO(\fD)} \scO(\fD^\wedge) \\
\cong \Hom_{\Mod^G(\Ug \otimes_{\ZFr} \Ug^\op)}(M,N) \otimes_{\scO(\fD)} \scO(\fD^\wedge).
\end{multline*}

Assuming now that $N=V \otimes \Ug$, we claim that the functor $\scO(\bS^{*}) \otimes_{\ZFr} (-)$ induces an isomorphism
\begin{multline}
\label{eqn:res-isom}
\Hom_{\Mod^G_\fin(\Ug \otimes_{\ZFr} \Ug^\op)}(M,V \otimes \Ug) \simto \\
 \Hom_{\Mod^{\bbI}_\fin(\cU_\bS \fg \otimes_{\scO(\bS^{*})} \cU_\bS \fg^\op)}(\scO(\bS^{*}) \otimes_{\ZFr} M, V \otimes \cU_{\bS} \fg).
\end{multline}
%where $\Mod^{\bbI}(\cU_\cS \fg \otimes_{\scO(\cS^{*(1)})} \cU_\cS \fg^\op)$ is the category of $\bbI^*_\cS$-equivariant $\cU_\cS \fg \otimes_{\scO(\cS^{*(1)})} \cU_\cS \fg^\op$-modules. 
In fact, the algebra $\Ug \otimes_{\ZFr} \Ug^\op$ is a $G$-equivariant finite $\scO(\fg^{*(1)})$-algebra. Therefore, it identifies with the global sections of a $G$-equivariant coherent sheaf of $\scO_{\fg^{*(1)}}$-algebras $\scU$ on $\fg^{*(1)}$. Moreover, the restriction $\scU_\bS$ of $\scU$ to $\bS^{*}$ is an $\bbI^*_\bS$-equivariant sheaf of $\scO_{\bS^{*}}$-algebras on $\bS^{*}$, whose global sections are $\cU_\bS \fg \otimes_{\scO(\bS^{*})} (\cU_\bS \fg)^\op$. Consider the open embedding $j : \fg^{*(1)}_\reg \hookrightarrow \fg^{*(1)}$, and set $\scU_\reg := j^*(\scU)$. Let us denote by $\QCoh^G(\fg^{*(1)}_\reg, \scU_\reg)$ the category of $G$-equivariant quasi-coherent sheaves on $\fg^{*(1)}_\reg$ equipped with a structure of $\scU_\reg$-module, compatible with the $G$-equivariant structure in the natural way, and by $\Coh^G(\fg^{*(1)}_\reg, \scU_\reg)$ the subcategory of coherent modules. Then we have a restriction functor
\[
 j^* : \Mod^G(\Ug \otimes_{\ZFr} \Ug^\op) \to \QCoh^G(\fg^{*(1)}_\reg, \scU_\reg)
\]
%where $\Mod^G(\Ug \otimes_{\ZFr} \Ug^\op)$ is the category of all (non necessarily finitely generated) $G$-equivariant $\Ug \otimes_{\ZFr} \Ug^\op$-modules, 
which admits a right adjoint
\[
 j_* : \QCoh^G(\fg^{*(1)}_\reg, \scU_\reg) \to \Mod^G(\Ug \otimes_{\ZFr} \Ug^\op)
\]
coinciding with the usual pushforward functor at the level of quasi-coherent sheaves on $\fg^{*(1)}_\reg$ and $\fg^{*(1)}$. Since the complement of $\fg^{*(1)}_\reg$ has codimension $\geq 2$, the natural morphism $\scO_{\fg^{*(1)}} \to j_* j^* \scO_{\fg^{*(1)}}$ is an isomorphism. Since $\Ug$ is free over $\ZFr$ it follows that the morphism $V \otimes \Ug \to j_* j^* (V \otimes \Ug)$ is also an isomorphism, and then that the functor $j^*$ induces an isomorphism
\[
\Hom_{\Mod^G_{\fin}(\Ug \otimes_{\ZFr} \Ug^\op)}(M,V \otimes \Ug) \simto \Hom_{\Coh^G(\fg^{*(1)}_\reg, \scU_\reg)}(j^* M, j^*(V \otimes \Ug)).
\]
It is a standard observation (see~\cite[Proposition~3.3.11]{riche-kostant}) that restriction to $\bS^{*}$ induces an equivalence of abelian categories
\[
 \Coh^{G^{(1)}}(\fg_\reg^{*(1)}) \simto \Rep(\bbJ^*_\bS),
\]
where the right-hand side denotes the category of representations of the affine group scheme $\bbJ^*_{\bS}$ (see~\S\ref{ss:centralizer-Kostant}) on coherent $\scO_{\bS^{*}}$-modules. The same considerations provide an equivalence of categories
\[
 \Coh^{G}(\fg_\reg^{*(1)}) \simto \Rep(\bbI^*_\bS).
\]
(Here we use the fact that the Frobenius morphism of $G$ is flat and surjective, hence faithfully flat.)
This equivalence is monoidal with respect to the natural tensor product on each side, and the image of the algebra $\scU_\reg$ is $\scU_\bS$; therefore it induces an equivalence of abelian categories
\[
 \Coh^G(\fg^{*(1)}_\reg, \scU_\reg) \simto \Mod^{\bbI}_{\fin}(\cU_\bS \fg \otimes_{\scO(\bS^{*})} \cU_\bS \fg^\op).
\]
From this we finally obtain that~\eqref{eqn:res-isom} is an isomorphism.

Combining~\eqref{eqn:res-isom} with the preceding isomorphisms we obtain a canonical isomorphism
\begin{multline*}
\Hom_{\Mod^G_{\fin}(\cU^\wedge)}(\mathsf{C}^\wedge(M),\mathsf{C}^\wedge(V \otimes \Ug)) \cong \\
 \Hom_{\Mod^{\bbI}_\fin(\cU_\bS \fg \otimes_{\scO(\bS^{*})} \cU_\bS \fg^\op)}(\scO(\bS^{*}) \otimes_{\ZFr} M, V \otimes \cU_{\bS} \fg) \otimes_{\scO(\fD)} \scO(\fD^\wedge).
\end{multline*}
Considerations similar to those of the beginning of the proof allow to identify the right-hand side with
\[
\Hom_{\Mod^{\bbI}_{\fin}(\cU_\bS^\wedge)}(\mathsf{C}_\bS^\wedge(M),\mathsf{C}_\bS^\wedge(V \otimes \Ug)),
\]
which finishes the proof.
\end{proof}

\section{Localization for Harish-Chandra bimodules}
\label{sec:localization}
%%%%%%%%%%%%%%%%%%%%%%%%%%%%%%

%-----------------------------------------------------------
\subsection{Azumaya algebras}
\label{ss:Azumaya-alg}
%-----------------------------------------------------------

We start by recalling the basic theory of Azumaya algebras.

Let $R$ be a commutative ring. Recall that an $R$-module $P$ is called \emph{faithfully projective} if it is projective of finite type and if moreover the only $R$-module $M$ such that $P \otimes_R M = 0$ is $M=0$. By~\cite[Chap.~I, Lemme~6.2]{ko} this condition is equivalent to requiring that $P$ is projective of finite type and faithful (i.e.~its annihilator in $R$ is trivial). An $R$-module $P$ is finitely generated and projective iff it is finitely presented and moreover the localization $P_{\mathfrak{p}}$ is free over $R_{\mathfrak{p}}$ for any $\mathfrak{p} \in \Spec(R)$, see~\cite[Chap.~I, Lemme~5.2]{ko} or~\cite[\href{https://stacks.math.columbia.edu/tag/00NX}{Tag 00NX}]{stacks-project}. In this setting, $P$ is faithful iff the rank of $P_{\mathfrak{p}}$ is positive for any $\mathfrak{p}$, see~\cite[Chap.~I, Lemme~6.1]{ko}. This notion is important in Morita theory since if $P$ is a faithfully projective $R$-module, then we obtain quasi-inverse equivalences of categories
\[
\xymatrix{
\Mod(R) \ar@<0.5ex>[r] & \Mod(\End_R(P)) \ar@<0.5ex>[l]
}
\]
given by $M \mapsto P \otimes_R M$ and $N \mapsto \Hom_R(P,R) \otimes_{\End_R(P)} N$ where $\Mod(A)$ is the category of left $A$-modules for any ring $A$; see~\cite[Chap.~I, Lemme~7.2]{ko}. In case $R$ is Noetherian, the ring $\End_R(P)$ is left Noetherian (as a noncommutative ring), and these equivalences restrict to equivalences
\begin{equation}
\label{eqn:equiv-faith-proj}
\xymatrix{
\Mod_{\mathrm{fg}}(R) \ar@<0.5ex>[r] & \Mod_{\mathrm{fg}}(\End_R(P)) \ar@<0.5ex>[l]
}
\end{equation}
between subcategories of finitely generated modules. (Here, a left $\End_R(P)$-module is finitely generated iff it is finitely generated as an $R$-module.)

Let $A$ be an $R$-algebra.\footnote{Let us insist that by an $R$-algebra we mean a (not necessarily commutative) ring $A$ endowed with a ring morphism from $R$ to the \emph{center} of $A$.}
Recall (see~\cite[\S III.5]{ko}) that $A$ is called an \emph{Azumaya $R$-algebra} if it satisfies one of the following equivalent conditions:
\begin{itemize}
\item $A$ is faithfully projective as an $R$-module, and the morphism sending $a \otimes b$ to the map $x \mapsto axb$ is an isomorphism of $R$-algebras
\begin{equation*}
%\label{eqn:Azumaya-matrices}
A \otimes_R A^{\mathrm{op}} \simto \mathrm{End}_R(A) ;
\end{equation*}
\item $A$ is finite as an $R$-module, the ring morphism $R \to A$ is injective, and moreover for any maximal ideal $\mathfrak{m} \subset R$ the finite-dimensional $R/\mathfrak{m}$-algebra $A/\mathfrak{m} A$ is a central simple algebra.
\end{itemize}
In particular, the first characterization and the facts recalled above show that in this case we have canonical equivalences of categories
\[
\xymatrix{
\Mod(R) \ar@<0.5ex>[r] & \Mod(A \otimes_R A^{\mathrm{op}}). \ar@<0.5ex>[l]
}
\]
 \subsection{Azumaya property of \texorpdfstring{$\cU_\bS\fg$}{Ug}}
 \label{ss:Azumaya-Ug}
 %---------------------------------------------------------
 
% In view of the identification~\eqref{eqn:ZUg}, the $Z(\cU\fg)$-algebra $\cU\fg$ defines a (coherent) sheaf of $\scO_{\fC}$-algebras $\scU$ on the affine scheme $\fC$. 
%% By assumption, there exists an isomorphism of $G$-varieties $\fg \simto \fg^*$, and we will denote by $\fg^*_\reg$ the image under this isomorphism of the open subvariety $\fg_\reg \subset \fg$ consisting of regular elements. (It is clear that $\fg^*_\reg$ does not depend on the choice of $G$-equivariant isomorphism $\fg \to \fg^*$.) 
% We set
% \[
% \fC_\reg := \fg_\reg^{*(1)} \times_{\ft^*{}^{(1)} / W} \ft^*/(W,\bullet),
% \]
% so that $\fC_\reg$ is an open subscheme in $\fC$. We will denote by $\scU_\reg$ the restriction of $\scU$ to $\fC_\reg$ (a coherent sheaf of $\scO_{\fC_\reg}$-algebras).
 
 The following property is standard (see~\cite{brown-goodearl, brown-gordon}); we recall its proof for the reader's convenience.
 
 \begin{prop}
 \label{prop:Ureg-Azumaya}
The $\scO(\fC_\bS)$-algebra $\cU_\bS \fg$ is Azumaya.
% The sheaf of algebras $\scU_\reg$ is an Azumaya algebra over $\fC_\reg$.
 \end{prop}
 
 \begin{proof}
 We will use the second characterization of Azumaya algebras recalled in~\S\ref{ss:Azumaya-alg}.
 Since $\Ug$ is finite over $Z(\Ug)=\scO(\fC)$, $\cU_\bS \fg$ is finite over $\scO(\fC_\bS)$. To prove that the morphism $\scO(\fC_\bS) \to \cU_\bS \fg$ is injective, we consider the composition
 \[
 G \times \bS^{*} \to G^{(1)} \times \bS^{*} \to \fg^{*(1)}_\reg \to \fg^{*(1)}.
 \]
 Here the first morphism is flat since $G$ is smooth (see~\S\ref{ss:centralizer-Kostant}), the second morphism is smooth by~\cite[Lemma~3.3.1]{riche-kostant}, and the third one is an open embedding; this composition is therefore flat. The algebras $Z(\Ug)$ and $\Ug$ can be considered as $G$-equivariant coherent sheaves on $\fg^{*(1)}$, and the induced morphism $\scO(\fC_\bS) \to \cU_\bS \fg$ is obtained from the embedding $Z(\Ug) \to \Ug$ by the pullback functor
 \[
 \Coh^G(\fg^{*(1)}) \to \Coh^G(G \times \bS^{*})
 \]
 associated with the flat morphism considered above,
 followed by the obvious equivalences $\Coh^G(G \times \bS^{*}) \cong \Coh(\bS^{*}) \cong \Mod_{\fin}(\scO(\bS^{*}))$; it is therefore injective.
 
 What is left to prove is that if $\mathfrak{m} \subset Z(\Ug)$ is a maximal ideal which belongs to $\fC_\bS=\bS^{*(1)} \times_{\ft^{*(1)} / W} \ft^*/(W,\bullet)$, then  $\Ug/\mathfrak{m} \Ug$ is a central simple algebra. In fact, this property holds more generally if $\mathfrak{m}$ belongs to $\fC_\reg := \fg_\reg^{*(1)} \times_{\ft^*{}^{(1)} / W} \ft^*/(W,\bullet)$. Indeed,
let $N$ be the maximal dimension of a simple $\Ug$-module. 
% As explained in~\cite[\S 1.10]{brown-goodearl}, the algebra $\Ug$ is a PI ring. 
 By~\cite[Proposition~3.1]{brown-goodearl}, 
% the maximal dimension $M(\fg)$ of a $\Ug$-module is its PI-degree, and 
 if $\mathfrak{m} \subset Z(\Ug)$ is a maximal ideal such that $\Ug/\mathfrak{m} \Ug$ admits a simple module $V$ of dimension $N$, then $\Ug/\mathfrak{m} \Ug$ is a central simple algebra; more specifically, the algebra morphism $\Ug/\mathfrak{m} \Ug \to \End_\bk(V)$ is an isomorphism. Now by~\cite[Theorem~4.4]{ps} we have $N=p^{\# \fR^+}$. And by~\cite[Theorem~5.6]{ps}, if $\mathfrak{m}$ belongs to $\fC_\reg$ then any simple $\Ug/\mathfrak{m} \Ug$-module has dimension divisible by $p^{\# \fR^+}$ hence equal to $N$.
 \end{proof}
 
% As a consequence of Proposition~\ref{prop:Ureg-Azumaya} we deduce that the action morphisms induce an isomorphism of coherent $\scO_{\fC_\reg}$-algebras
% \begin{equation}
% \label{eqn:isom-Ureg-End}
% \scU_\reg \otimes_{\scO_{\fC_\reg}} (\scU_\reg)^{\mathrm{op}} \simto \mathscr{E}nd_{\scO_{\fC_\reg}}( \scU_\reg),
% \end{equation}
% see~\eqref{eqn:isom-Azumaya}. 

%As a consequence of 
It follows in particular from Proposition~\ref{prop:Ureg-Azumaya}
%, $\cU_{\cS} \fg$ is an Azumaya $\scO(\fC_\cS)$-algebra. (Here, since the morphism $\cS^{*(1)} \to \ft^{*(1)} / W$ is an isomorphism, the projection $\fC_\cS \to \ft^*/(W,\bullet)$ is an isomorphism.) In particular, 
that $\cU_{\bS} \fg$ is faithfully projective as an $\scO(\fC_\bS)$-module.
% , and we have a canonical isomorphism
% \begin{equation}
% \label{eqn:USg-splitting}
% \cU_{\cS} \fg \otimes_{\scO(\fC_\cS)} \cU_\cS \fg \simto \End_{\scO(\fC_\cS)} (\cU_\cS \fg).
% \end{equation}
 
% One can make the Azumaya property from Proposition~\ref{prop:Ureg-Azumaya} a bit more concrete as follows. 
Below we will use a slightly more concrete version of Proposition~\ref{prop:Ureg-Azumaya}, as follows. 
 First we need to recall the definition of baby Verma modules. 
% Given a Borel subgroup $B' \subset G$, we will denote by $\radu(B')$ its unipotent radical. 
 Consider some element $\eta \in \fg^{*(1)}$, and some Borel subgroup $B' \subset G$ with unipotent radical $U'$ such that $\eta$ vanishes on $\Lie(U')^{(1)}$. (Such a Borel subgroup exists for any $\eta$, see~\cite[Lemma~6.6]{jantzen-Lie}.) Then $\eta$ defines an element $\eta'$ in $(\Lie(B')/\Lie(U'))^{*(1)}$.
 %, which identifies canonically with $\Lie(B/U)^* \cong \ft^*$. 
 Let $\xi \in \ft^*$ be an element whose image under the map
 \[
  \ft^* \cong (\Lie(B)/\Lie(U))^* \simto (\Lie(B')/\Lie(U'))^* \to (\Lie(B')/\Lie(U'))^{*(1)}
 \]
is $\eta'$, where the second map is induced by conjugation by an element $g \in G$ such that $gBg^{-1}=B'$ (it is well known that the isomorphism does not depend on the choice of $g$), and the second one is the Artin--Schreier map associated with the torus $B'/U'$.
%  under the Artin--Schreier map is $\chi$, i.e.~such that
%  \[
%   \xi(h)^p-\xi(h^{[p]}) = \eta(h)^p
%  \]
% for any $h \in T'$. 
Then we can consider the associated \emph{baby Verma module}
\[
 \Ver_{\eta,B'}(\xi) := \cU_\eta\fg \otimes_{\cU_\eta \Lie(B')} \bk_\xi,
\]
where $\cU_\eta \Lie(B')$ is the central reduction (with respect to the Frobenius center) of the enveloping algebra of $\Lie(B')$ at the image of $\eta$ in $\Lie(B')^{*(1)}$, and $\bk_\xi$ is its $1$-dimensional module defined by the image of $\xi$ in $(\Lie(B')/\Lie(U'))^*$.
%
%  assume that $\chi$ vanishes on $\fb^{(1)} \subset \fg^{(1)}$. In particular, this implies that $\chi$ belongs to $\cN^{*(1)}$. Hence for any $\lambda \in X^*(T)$ we can consider the maximal ideal $\mathfrak{m}_\chi \cdot Z(\Ug) + \mathfrak{m}^\lambda \cdot Z(\Ug) \subset Z(\Ug)$, and the associated central reduction $\cU^\lambda_\chi \fg$. The character $\lambda$ defines a $1$-dimensional $B$-module, hence a $1$-dimensional $\cU_0 \fb$-module, where $\cU_0 \fb = \cU \fb / (\mathfrak{m}_\chi \cap \cU \fb) \cdot \cU \fb$, which we will denote $\bk_\lambda$. The \emph{baby Verma module} associated with $(\chi,\lambda)$ is the $\cU_\chi^\lambda \fg$-module
%  \[
%  Z_\chi(\lambda) := \cU_\chi \fg \otimes_{\cU_0 \fb} \bk_\lambda.
%  \]
 This module has dimension $p^{\# \fR^+}$; if we assume furthermore that $\eta \in \fg^{*(1)}_\reg$, then the considerations in the proof of Proposition~\ref{prop:Ureg-Azumaya} therefore imply that this module is simple, and that the algebra morphism
 \begin{equation}
 \label{eqn:U-End-Z}
 \cU^{\xi'}_{\eta} \fg \to \End_\bk(\Ver_{\eta,B'}(\xi))
 \end{equation}
 is an isomorphism, where we denote by $\xi'$ the image of $\xi$ in $\ft^*/(W,\bullet)$.
%  under the morphism
%  \[
%   \Lie(T')^* \cong \Lie(B'/B'_{\mathrm{u}}) \cong \Lie(B/B_{\mathrm{u}}) \cong \ft^* \to \ft^*/(W,\bullet).
%  \]
% (Here, $B_{\mathrm{u}}$ denotes the unipotent radical of $B$, and the isomorphism $B'/B'_{\mathrm{u}} \cong B/B_{\mathrm{u}}$ is induced by conjugation by an element $g \in G$ such that $gB'b^{-1}=B$; it is well known that this isomorphism does not depend on the choice of $g$.)

% \begin{rmk}
% \label{rmk:Verma-extension-ZHC}
%  It is usual to consider the baby Verma modules $Z_{\eta,B'}(\xi)$ as $\cU\fg$-modules. But in fact these objects are naturally modules for a larger algebra, namely $\Ug \otimes_{\ZHC} \scO(\ft^*)$, where the algebra $\scO(\ft^*)$ acts via the character defined by $\xi$. Taking this action into account, the isomorphism~\eqref{eqn:U-End-Z} can be interpreted as an isomorphism
%  \[
%   \bk_\xi \otimes_{\scO(\ft^*)} \bigl( \cU_{\eta} \fg \otimes_{\ZHC} \scO(\ft^*) \bigr) \simto \End_\bk(Z_{\eta,B'}(\xi))
%  \]
% \end{rmk}

%--------------------------------------------------------------
\subsection{Splitting bundles for the algebras \texorpdfstring{$\cU_\bS^{\hat{\lambda},\hat{\mu}}$}{US}}
\label{ss:splitting-bundles}
%--------------------------------------------------------------

%%------------------------------------------------------------------
%\subsection{Some categories of coherent sheaves}
%%------------------------------------------------------------------

%Below we will construct 
Our goal in this section is to construct some tools that will allow us to study the categories $\Mod^G_{\mathrm{fg}}(\cU^{\hat{\lambda},\hat{\mu}})$ and $\HCBim^{\hat{\lambda},\hat{\mu}}$ via geometric methods. First we introduce the categories of modules that will be involved in these constructions. Our model will be the category $\Coh^G(\fC \times_{\fg^{*(1)}} \fC)$ of $G$-equivariant coherent sheaves on $\fC \times_{\fg^{*(1)}} \fC$, or in other words of $G$-equivariant finitely generated $\cZ$-modules, which is a monoidal category for the operation sending a pair $(M,N)$ to
\[
M \otimes_{Z(\Ug)} N,
\]
where in the tensor product $Z(\Ug)$ acts on $M$ via the right action and on $N$ via the left action. The $\cZ$-action on $M \otimes_{Z(\Ug)} N$ comes from the left action of $Z(\Ug)$ on $M$ and the right action on $N$. In practice however, we will have to restrict to $\bS^{*}$, and add generalized characters to this picture.

First, recall the group scheme $\bbI^\wedge_\bS$ over $\mathrm{Spec}(\cZ_\bS^\wedge)$ introduced in~\S\ref{ss:HCBim-S}. We consider the abelian category $\Mod^{\bbI}_{\fin}(\cZ_\bS^\wedge)$ of representations of the flat affine group scheme $\bbI^\wedge_\bS$ on finitely generated $\cZ_\bS^\wedge$-modules. Here a $\cZ_\bS^\wedge$-module is a $\ZHC \otimes_{\scO(\ft^{*(1)}/W)} \scO(\ft^{*(1)}/W)^\wedge$-bimodule on which the left and right actions of $\scO(\ft^{*(1)}/W)^\wedge$ coincide; the category of such modules therefore admits a canonical tensor product, which preserves finitely generated modules, and induces a monoidal product 
\[
 (-) \hatstar_{\bS} (-) : \Mod^{\bbI}_{\fin}(\cZ_\bS^\wedge) \times \Mod^{\bbI}_{\fin}(\cZ_\bS^\wedge) \to \Mod^{\bbI}_{\fin}(\cZ_\bS^\wedge)
\]
with unit object $\ZHC \otimes_{\scO(\ft^{*(1)}/W)} \scO(\ft^{*(1)}/W)^\wedge$, with diagonal $\cZ_\bS^\wedge$-module structure and trivial action of $\bbI^\wedge_\bS$. If we set
\[
\bbJ_\bS^\wedge := \mathrm{Spec}(\cZ_\bS^\wedge) \times_\bS \bbJ^*_\bS,
\]
then one can also consider the abelian category $\Mod^{\bbJ}_{\fin}(\cZ_\bS^\wedge)$ of representations of $\bbJ_\bS^\wedge$ on on finitely generated $\cZ_\bS^\wedge$-modules; the quotient morphism $\bbI_\bS^\wedge \to \bbJ^\wedge_\bS$ induces an exact and fully faithful functor
\[
\Mod^{\bbJ}_{\fin}(\cZ_\bS^\wedge) \to \Mod^{\bbI}_{\fin}(\cZ_\bS^\wedge)
\]
whose image is stabilized by the convolution product $\hatstar_\bS$.

Next, for $\lambda,\mu \in \bbX$ we have the affine group scheme $\bbI_\bS^{\hat{\lambda},\hat{\mu}}$ over $\mathrm{Spec}(\cZ_\bS^{\hat{\lambda},\hat{\mu}})$ also introduced in~\S\ref{ss:HCBim-S}, and we can consider the abelian category $\Mod^{\bbI}_{\fin}(\cZ_\bS^{\hat{\lambda},\hat{\mu}})$ of representations of this affine group scheme on finitely generated $\cZ_\bS^{\hat{\lambda},\hat{\mu}}$-modules. As in~\eqref{eqn:Mod-UwedgeS-sum} we have a decomposition
\[
\Mod^{\bbI}_{\fin}(\cZ_\bS^\wedge) \cong \bigoplus_{\lambda,\mu \in \Lambda} \Mod^{\bbI}_{\fin}(\cZ_\bS^{\hat{\lambda},\hat{\mu}}).
\]
Given $M$ in $\Mod^{\bbI}_{\fin}(\cZ_\bS^{\hat{\lambda},\hat{\mu}})$ and $N$ in $\Mod^{\bbI}_{\fin}(\cZ_\bS^{\hat{\mu},\hat{\nu}})$, seen as objects in $\Mod^{\bbI}_{\fin}(\cZ_\bS^\wedge)$, the object $M \hatstar_{\bS} N$ belongs to $\Mod^{\bbI}_{\fin}(\cZ_\bS^{\hat{\lambda},\hat{\nu}})$; in other words, the bifunctor $\hatstar_{\bS}$ restricts to a bifunctor
\begin{equation}
\label{eqn:convolution-ZS}
\Mod^{\bbI}_{\fin}(\cZ_\bS^{\hat{\lambda},\hat{\mu}}) \times \Mod^{\bbI}_{\fin}(\cZ_\bS^{\hat{\mu},\hat{\nu}}) \to \Mod^{\bbI}_{\fin}(\cZ_\bS^{\hat{\lambda},\hat{\nu}})
\end{equation}
for any $\lambda,\mu,\nu \in \bbX$. In particular, each category $\Mod^{\bbI}_{\fin}(\cZ_\bS^{\hat{\lambda},\hat{\lambda}})$ admits a monoidal structure; the unit object $\cZ_\bS^{\hat{\lambda}}$ for this structure is $\ZHC^{\hat{\lambda}}$, endowed with the natural structure of $\cZ_\bS^{\hat{\lambda},\hat{\lambda}}$-module (induced by the product morphism $\ZHC \otimes_{\ZHC \cap \ZFr} \ZHC \to \ZHC$) and the trivial structure as a representation of $\bbI^{\hat{\lambda},\hat{\lambda}}_\bS$.

One can also play the same game starting with $\bbJ^*_\bS$ in place of $\bbI^*_\bS$; one obtains in this way affine group schemes $\bbJ_\bS^{\hat{\lambda},\hat{\mu}}$, categories $\Mod^{\bbJ}_{\fin}(\cZ_\bS^{\hat{\lambda},\hat{\mu}})$ and exact fully faithful functors
\[
\Mod^{\bbJ}_{\fin}(\cZ_\bS^{\hat{\lambda},\hat{\mu}}) \to \Mod^{\bbI}_{\fin}(\cZ_\bS^{\hat{\lambda},\hat{\mu}})
\]
whose essential images are stabilized by the bifunctor~\eqref{eqn:convolution-ZS} in the obvious sense.

\begin{rmk}
In view of~\eqref{eqn:D-la-mu-tens}, for any $\lambda,\mu$ we have an algebra isomorphism
\[
\cZ_\bS^{\hat{\lambda},\hat{\mu}} \simto \ZHC^{\hat{\lambda}} \otimes_{\scO(\ft^{*(1)}/W)^\wedge} \ZHC^{\hat{\mu}}.
\]
From this point of view, the bifunctor~\eqref{eqn:convolution-ZS} is induced by the tensor product $(-) \otimes_{\ZHC^{\hat{\mu}}} (-)$.
\end{rmk}

Recall that a weight $\lambda \in \bbX$ is said to belong to the \emph{lower closure of the fundamental alcove} if it satisfies
\[
0 \leq \langle \lambda+\rho, \alpha^\vee \rangle < p
\]
for any positive root $\alpha$.
Recall also the completed bimodules introduced in~\S\ref{ss:HCBim-completed}. In particular, given $\lambda,\mu \in \bbX$ which belong to the lower closure of the fundamental alcove, we have the objects
%If $\lambda \in X^*(T)$ belongs to the lower closure of the fundamental alcove, we set
\begin{align*}
\sfP^{\lambda,-\rho} &= \mathsf{C}^{\lambda,-\rho} \bigl( \Sim(\lambda+\rho) \otimes \Ug \bigr) \quad \in \HCBim^{\hat{\lambda},\hat{-\rho}}_\diag, \\
%the object of $\Mod^G_{\mathrm{fg}}(\cU^{\hat{\lambda},\hat{-\rho}})$ obtained by completion with respect to the ideal $I^{\lambda,-\rho} \cdot (\Ug \otimes_{\ZFr} \Ug)$.
%Similarly, if $\mu \in X^*(T)$ belongs to the lower closure of the fundamental alcove 
%we can consider the indecomposable tilting $G$-module $\Til(\mu+\rho)$, and the Harish-Chandra bimodule
%\[
%\Ug \otimes_\bk \Til(\mu+\rho)
%\]
%where now the \emph{right} $\Ug$-action is diagonal. The object of $\Mod^G_{\mathrm{fg}}(\cU^{\hat{-\rho},\hat{\mu}})$ obtained by completion will be denoted
\sfP^{-\rho,\mu} &= \mathsf{C}^{-\rho,\mu} \bigl(\Sim(-w_0 \mu+\rho) \otimes \Ug \bigr) \quad \in \HCBim^{\hat{-\rho},\hat{\mu}}_\diag.
\end{align*}
%where $w_0$ is the longest element in $W$.
%where we use the notation introduced in~\S\ref{ss:HCBim}.
%Finally, for 
%$\lambda,\mu \in X^*(T)$ in the lower closure of the fundamental alcove 
%such $\lambda,\mu$
We set
\[
\sfM^{\lambda,\mu} := \sfP^{\lambda,-\rho} \hatotimes_{\Ug} \sfP^{-\rho,\mu} \quad \in \HCBim^{\hat{\lambda},\hat{\mu}}_\diag.
\]
%so that $M^{\hat{\lambda},\hat{\mu}}$ is an object in $\HCBim^{\hat{\lambda},\hat{\mu}}$.
We also set $\sfM^{\lambda,\mu}_\bS := \scO(\bS^{*}) \otimes_{\scO(\fg^{*(1)})} \sfM^{\lambda,\mu}$, so that
%under the functor~\eqref{eqn:restr-S-bimodules}.
% ; we then have
\[
\sfM^{\lambda,\mu}_\bS = \sfP^{\lambda,-\rho}_\bS \hatotimes_{\cU_{\bS} \fg} \sfP^{-\rho,\mu}_\bS
\]
where we use the notation of~\S\ref{ss:HCBim-S}.

The main technical result of this section is the following theorem. Its proof will be given in~\S\ref{ss:proof-splitting}, after some preliminaries treated in~\S\ref{ss:study-fibers}.
%we will prove the following result.

\begin{thm}
\label{thm:splitting}
For any $\lambda,\mu \in \bbX$ in the lower closure of the fundamental alcove, the 
%$\scO(\fC_\cS^{\hat{\lambda},\hat{\mu}})$-module 
$\cZ_\bS^{\hat{\lambda},\hat{\mu}}$-module
$\sfM^{\lambda,\mu}_\bS$ is faithfully projective, and the natural algebra morphism
\[
\cU_\bS^{\hat{\lambda},\hat{\mu}} \to \End_{\cZ_\bS^{\hat{\lambda},\hat{\mu}}}(\sfM^{\lambda,\mu}_\bS)
\]
is an isomorphism.
\end{thm}

\subsection{Study of some stalks}
\label{ss:study-fibers}
%--------------------------------------------------------------

%Let us set
%\[
% \fZ_{\bS} := \Spec(\cZ_\bS),
%\]
%which we identify with
Recall that $\Spec(\cZ_\bS)$ identifies naturally with
\[
 \fD=\ft^*/(W,\bullet) \times_{\ft^{*(1)}/W} \ft^*/(W,\bullet),
\]
see~\S\ref{ss:HCBim-S}.
We also set
\[
 \tfD := \ft^* \times_{\ft^{*(1)}} \ft^*.
\]

Since the Artin--Schreier map $\ft^* \to \ft^{*(1)}$ is a Galois covering with Galois group $\ft^*_{\mathbb{F}_p}$, we have a canonical isomorphism
\[
 \ft^*_{\mathbb{F}_p} \times \ft^* \simto \tfD
\]
defined by $(\eta,\xi) \mapsto (\eta+\xi,\xi)$. For $\lambda \in \bbX$ we will denote by $\tfD(\lambda)$ the image of $\{\overline{\lambda + \rho}\} \times \ft^*$ in $\tfD$; if $\widetilde{\Lambda} \subset \bbX$ is a subset of representatives for the quotient $\ft^*_{\mathbb{F}_p}=\bbX/p \bbX$, we then have
\[
 \tfD = \bigsqcup_{\lambda \in \widetilde{\Lambda}} \tfD(\lambda).
\]

Recall that if $A$ is a finitely generated $\bk$-algebra, by~\cite[\href{https://stacks.math.columbia.edu/tag/02J6}{Tag 02J6}]{stacks-project} $\Spec(A)$ is a Jacobson space in the sense of~\cite[\href{https://stacks.math.columbia.edu/tag/005T}{Tag 005T}]{stacks-project}; in other words, closed points are dense in any closed subset of $\Spec(A)$. Here
we have a natural finite morphism
\begin{equation}
\label{eqn:morphism-tfD-fD}
 \tfD \to \fD,
\end{equation}
% which factors through a (finite) morphism
% \[
%  \tfZ_\cS \to (\fZ_\cS)_{\mathrm{red}},
% \]
% where the right-hand side is the reduced subscheme associated with $\fZ_\cS$. 
and it is easily seen that the image of this morphism contains all the closed points of $\fD$; this morphism is therefore surjective.
% since closed points are dense in $\fZ_\cS$ (see~\cite[\href{https://stacks.math.columbia.edu/tag/02J6}{Tag 02J6}]{stacks-project}). 
 For any $\lambda \in \bbX$ we denote by $\fD(\lambda)$ the scheme-theoretic image of $\tfD(\lambda)$ in $\fD$; in other words $\scO(\fD(\lambda))$ is the image of the composition
 \[
 \scO(\fD) \to \scO(\tfD) \twoheadrightarrow \scO(\tfD(\lambda)),
\] 
see~\cite[\href{https://stacks.math.columbia.edu/tag/056A}{Tag 056A}]{stacks-project}. The morphism~\eqref{eqn:morphism-tfD-fD} then factors through a finite morphism $\tfD(\lambda) \to \fD(\lambda)$, which is surjective since its image is closed and dense, see~\cite[\href{https://stacks.math.columbia.edu/tag/01R8}{Tag 01R8}]{stacks-project}.
 Since $\tfD(\lambda)$ is integral, so is $\fD(\lambda)$. Moreover, one can check that
\[
 \fD(\lambda) = \fD(\mu) \quad \text{iff} \quad \tla=\tmu,
 %$\lambda$ and $\mu$ have the same image in $\ft^*_\Z/(W,\bullet)$.
\]
where the operation $\lambda \mapsto \tla$ is as in~\S\ref{ss:weights}.
%(Here, the action of $W$ on $X^*(T)/pX^*(T)$ is induced by the nonshifted action on $X^*(T)$.) 
If $\Lambda \subset \bbX$ is (as in~\S\ref{ss:comparison}) a subset of representatives for $\ft^*_{\mathbb{F}_p}/(W,\bullet)$, we therefore have
\[
 \fD = \bigcup_{\lambda \in \Lambda} \fD(\lambda),
\]
and this constitues the decomposition of $\fD$ into its irreducible components.

Let us consider the open subset
\[
 \ft^*_\circ := \{ \xi \in \ft^* \mid \forall w \in W, \, w \bullet \xi - \xi \notin \ft^*_{\mathbb{F}_p} \smallsetminus \{0\}\} \subset \ft^*.
\]
Then $\ft^*_\circ$ is stable under the $(W,\bullet)$-action, and is in fact the pullback of an open subset of $\ft^*/(W,\bullet)$, which therefore identifies with the quotient $\ft^*_\circ / (W,\bullet)$.

Recall the Grothendieck resolution $\tbg$\footnote{Of course, $\tbg$ is the Frobenius twist of the Grothendieck resolution attached to the group $G$.} (for the reductive group $\bG=G^{(1)}$) and the morphism $\vartheta : \tbg \to \ft^{*(1)}$, see~\S\ref{ss:Kostant-section-Abe} and~\S\ref{ss:centralizer-Kostant}.
If we denote by $\widetilde{\bS}^*$ the (scheme-theoretic) preimage of $\bS^*$ in $\tbg$, then by~\cite[Proposition~3.5.5]{riche-kostant} the morphism $\vartheta$ restricts to an isomorphism $\widetilde{\bS}^* \simto \ft^{*(1)}$. In concrete terms, this means that given $\zeta \in \ft^{*(1)}/W$ identified with an element in $\bS^{*}$, the datum of a preimage of $\zeta$ in $\ft^{*(1)}$ is equivalent to the datum of a Borel subgroup $B' \subset G$ such that $\zeta_{|\Lie(U')^{(1)}}=0$, where $U'$ is the unipotent radical of $B'$.

\begin{prop}
\label{prop:fiber-Verma}
Let $\lambda \in \bbX$ be a weight which belongs to the lower closure of the fundamental alcove.
%, and denote by $\nu$ the unique dominant $W$-translate of $\lambda+\rho$. 
Consider some element $\xi \in \ft^*_\circ$, and denote by $(\zeta_1,\zeta_2) \in \fD(\lambda)$ the image of $(\xi + \overline{\lambda+\rho},\xi) \in \tfD(\lambda)$ in $\fD$
%some closed point $(\zeta_1,\zeta_2) \in \fZ^0_\cS(\lambda)$, and some element $\xi \in \ft^*_\circ$ such that $(\xi + \overline{\lambda+\rho},\xi) \in \tfZ_\cS(\lambda)$ has image $(\zeta_1,\zeta_2)$ in $\fZ_\cS$.
%which is the image of some element $(\xi + \overline{\lambda+\rho},\xi) \in \tfZ_\cS(\lambda)$.
%and let $\xi \in \ft^*$ be a preimage of $\zeta_2$. 
and by $\eta \in \bS^{*}$ the element corresponding to the images of $\zeta_1$ and $\zeta_2$ in $\ft^{*(1)}/W$.
As explained above the image of $\xi$ in $\ft^{*(1)}$ determines a Borel subgroup $B' \subset G$ with unipotent radical $U'$ such that $\eta_{|\Lie(U')^{(1)}}=0$.

If we denote by $i : \Spec(\bk) \to \fD$ the morphism defined by $(\zeta_1,\zeta_2)$, there exists an
%canonical pair $(\eta,B')$ where $\eta \in \cS^{*(1)}$ and $B' \subset G$ is a Borel subgroup, with unipotent radical $U'$, such that $\eta$ vanishes on $\Lie(U')$, and a 
isomorphism of $\cU_\eta^{\zeta_1} \fg \otimes (\cU_\eta^{\zeta_2} \fg)^{\op}$-modules
\[
 i^* \bigl( \Sim(\lambda+\rho) \otimes \cU_\bS \fg \bigr) \cong \Ver_{\eta,B'}(\xi+\overline{\lambda+\rho}) \otimes \Ver_{\eta,B'}(\xi)^*.
\]
% The sheaf $j_\lambda^*(\Sim(\nu) \otimes \cU_\cS \fg)$ is a locally free $\scO_{\fZ_\cS^0(\lambda)}$-module of rank $p^{\dim(G)-\dim(T)}$.
\end{prop}

\begin{proof}
% In view of Lemma~\ref{lem:projective-rk}, to prove the proposition it suffices to prove that for any morphism $i : \Spec(\bk) \to \fZ_\cS^0(\lambda)$ the pullback $i^*j_\lambda^*(\Sim(\nu) \otimes \cU_\cS \fg)$ has dimension $p^{\dim(G)-\dim(T)}$. Here the morphism $i$ defines a closed point $(\zeta_1,\zeta_2) \in \fZ_\cS^0(\lambda)$, and 
 By definition we have
 \[
 i^* \bigl( \Sim(\lambda+\rho) \otimes \cU_\bS \fg \bigr) \cong \bk_{\zeta_1} \otimes_{\ZHC} \bigl( \Sim(\lambda+\rho) \otimes \cU_\eta^{\zeta_2} \fg \bigr).
\]
%  Let us denote by $\eta$ the image of $\zeta_1$ (or equivalently $\zeta_2$) under the map
%  \[
%   \ft^*/(W,\bullet) \to \ft^{*(1)}/W \cong \cS^{*(1)},
%  \]
% and choose some Borel subgroup $B' \subset G$, with unipotent radical $U'$, such that $\eta$ vanishes on $\Lie(U')$ (see~\S\ref{ss:Azumaya-Ug}). This choice defines an element
% \[
%  \eta^0 \in \Lie(B'/U')^{*(1)} \cong \ft^{*(1)},
% \]
% whose image in $\ft^{*(1)}/W$ corresponds to $\eta$ under the identification $\ft^{*(1)}/W \cong \cS^{*(1)}$. Let now $\xi \in \ft^*$ be an element whose image in $\ft^*/(W,\bullet)$ is $\zeta_2$, and whose image in $\ft^{*(1)}$ is $\eta^0$. (It is easy to see that such an element indeed exists.) 
%Let also $\xi'$ be the image of $\xi$ under the identification $\ft^* \cong \Lie(gTg^{-1})^*$ induced by conjugation by $g$. 
By construction,
the image of $\xi$ in $\ft^{*(1)}$ corresponds to the element in the space $(\Lie(B')/\Lie(U'))^{*(1)}$ defined by $\eta$;
by~\eqref{eqn:U-End-Z}, we therefore have a canonical isomorphism
\[
 \cU_\eta^{\zeta_2} \fg \simto \End_{\bk}(\Ver_{\eta, B'}(\xi)) \cong \Ver_{\eta, B'}(\xi) \otimes \Ver_{\eta, B'}(\xi)^*,
\]
under which the action of $\cU\fg$ induced by left multiplication on the left-hand side corresponds to the natural action on $\Ver_{\eta, B'}(\xi)$. We deduce an isomorphism
\[
 i^* \bigl( \Sim(\lambda+\rho) \otimes \cU_\bS \fg \bigr) \cong \bk_{\zeta_1} \otimes_{\ZHC} \bigl( \Sim(\lambda+\rho) \otimes \Ver_{\eta, B'}(\xi) \bigr) \otimes \Ver_{\eta, B'}(\xi)^*,
 \]
 which shows that to conclude the proof it suffices to construct an isomorphism of $\cU_\eta^{\zeta_1}$-modules
 \begin{equation}
\label{eqn:fiber-Vermas}
\bk_{\zeta_1} \otimes_{\ZHC} \bigl( \Sim(\lambda+\rho) \otimes \Ver_{\eta, B'}(\xi) \bigr) \cong \Ver_{\eta,B'}(\xi + \overline{\lambda+\rho}).
 \end{equation}

 As above we have a canonical isomorphism
 \[
  \cU_\eta^{\zeta_1} \fg \simto \End_{\bk}(\Ver_{\eta, B'}(\xi+\overline{\lambda+\rho}));
 \]
therefore, any $\cU_\eta^{\zeta_1} \fg$-module is isomorphic to a direct sum of copies of $\Ver_{\eta, B'}(\xi+\overline{\lambda+\rho})$. To analyze how many copies we have for the specific module in the left-hand side of~\eqref{eqn:fiber-Vermas}, we observe that
\begin{multline*}
 \Hom_{\cU_\eta^{\zeta_1} \fg}(\bk_{\zeta_1} \otimes_{\ZHC} \bigl( \Sim(\lambda+\rho) \otimes \Ver_{\eta, B'}(\xi) \bigr), \Ver_{\eta,B'}(\xi + \overline{\lambda+\rho})) = \\
 \Hom_{\cU_\eta \fg}( \Sim(\lambda+\rho) \otimes \Ver_{\eta, B'}(\xi), \Ver_{\eta,B'}(\xi + \overline{\lambda+\rho})) \cong \\
 \Hom_{\cU_\eta \fg}( \Ver_{\eta, B'}(\xi), \Sim(-w_0\lambda+\rho) \otimes \Ver_{\eta,B'}(\xi + \overline{\lambda+\rho})).
\end{multline*}

We now consider the $\cU_\eta \fg$-module $\Sim(-w_0\lambda+\rho) \otimes \Ver_{\eta,B'}(\xi + \overline{\lambda+\rho})$, and more specifically the direct summand on which $\ZHC$ acts with a generalized character corresponding to $\zeta_2$.
%On the other hand, $(\zeta_1, \zeta_2)$ belongs to $\fZ_\cS^0(\lambda)$. This element admits a preimage in $\ft^* \times_{\ft^{*(1)}} \ft^*$ of the form $(\tilde{\xi},\xi)$ for some $\tilde{\xi} \in \ft^*$. This point must belong to the component $\fZ_\cS(w\lambda)$ for some $w \in W$, and we then have $\tilde{\xi}=\xi+w\lambda$. 
%Now we will analyze the module $\Sim(\nu) \otimes \Ver_{\eta, B'}(\xi)$, following standard considerations.
We have a canonical isomorphism of $\cU_\eta \fg$-modules
\[
 \Sim(-w_0\lambda+\rho) \otimes \Ver_{\eta, B'}(\xi+\overline{\lambda+\rho}) \cong \cU_\eta \fg \otimes_{\cU_\eta \Lie(B')} \bigl( \Sim(-w_0\lambda+\rho)_{|B'} \otimes \bk_{\xi+\overline{\lambda+\rho}} \bigr).
\]
The $B'$-module $\Sim(-w_0\lambda+\rho)_{|B'}$ admits a filtration
\[
 0 \subset M_1 \subset \cdots \subset M_n = \Sim(-w_0\lambda+\rho)_{|B'}
\]
where each $M_i/M_{i-1}$ is $1$-dimensional; moreover these modules are associated with the characters of $B'/U' \cong B/U \cong T$ corresponding to the $T$-weights of $\Sim(-w_0\lambda+\rho)$, counted with multiplicities. This filtration induces a filtration of $\Sim(-w_0\lambda+\rho)_{|B'} \otimes \bk_{\xi+\overline{\lambda+\rho}}$, and then of $\Sim(-w_0\lambda+\rho) \otimes \Ver_{\eta, B'}(\xi+\overline{\lambda+\rho})$, whose subquotients are of the form $\Ver_{\eta,B'}(\xi+\overline{\lambda+\rho+\mu})$, where $\mu$ runs over the $T$-weights of $\Sim(-w_0\lambda+\rho)$, counted with multiplicities.

We claim that there exists exactly one subquotient in this filtration on which $\ZHC$ acts via the character $\zeta_2$, corresponding to the multiplicity-$1$ weight $-\lambda-\rho$ of $\Sim(-w_0\lambda+\rho)$. Indeed, assume that $\ZHC$ acts with character $\zeta_2$ on $\Ver_{\eta,B'}(\xi+\overline{\lambda+\rho+\mu})$. Then there exists $w \in W$ such that $\xi+\overline{\lambda+\rho+\mu}=w \bullet \xi$. 
%By definition of $\fZ_\cS^0(\lambda)$, the point 
Since $\xi$ belongs to $\ft^*_\circ$, this condition implies that $\xi+\overline{\lambda+\rho+\mu}=\xi$, hence that $\lambda+\rho+\mu \in p \bbX$. On the other hand, $\mu$ is a weight of $\Sim(-w_0\lambda+\rho)$, hence it belongs to $-w_0\lambda+\rho + \Z\fR = -\lambda - \rho+\Z\fR$. 
%Now we have $p\bbX \cap \Z\fR = p\Z\fR$, see~\S\ref{ss:weights}; hence 
In view of~\eqref{eqn:no-torsion}
these conditions imply that $\lambda+\rho+\mu \in p\Z\fR$, i.e.~that $\lambda+\mu \in -\rho+p\Z\fR=\Waff \bullet (-\rho)$. By~\cite[Lemma~II.7.7]{jantzen} (applied to the pair of elements $(\lambda,-\rho)$), there must then exist $w \in \Waff$ such that $w \bullet \lambda =\lambda$ and $\lambda+\mu= w \bullet(-\rho)$. Here, since $\lambda$ belongs to the lower closure of the fundamental alcove, the first condition implies that $w \in W$ (see~\S\ref{ss:HCBim-completed}); it follows that $w \bullet (-\rho)=-\rho$, hence that $\lambda+\mu=-\rho$, which finishes the proof of our claim.

% the point $(\xi+\overline{\lambda+\rho},\xi+\overline{\lambda+\rho+\mu}) \in \tfZ_\cS$ has image $(\zeta_1,\zeta_2)$ in $\fZ_\cS$. Since $(\zeta_1, \zeta_2) \in \fZ_\cS^0(\lambda)$, this means that $\overline{-\mu-\rho} \in W \bullet \overline{\lambda}$, or in other words that there exists $y \in W$ such that
% \[
% -\rho+\mu \in y \bullet (-w_0\lambda) + pX^*(T).
% \]
% (Here, as in~\cite{jantzen} we also denote by $\bullet$ the action of $W$ on $X^*(T)$ defined by $w \bullet \nu=w(\nu+\rho)-\rho$.)
% Now $\mu$ is a weight of $\Til(-w_0\lambda+\rho)$, hence it belongs to $-w_0\lambda+\rho + \Z\fR = y \bullet(-w_0\lambda) + \rho+\Z\fR$. Our assumptions ensure that $pX^*(T) \cap \Z\fR = p\Z\fR$; hence we must have 
% \[
% -\rho+\mu \in y \bullet (-w_0\lambda) + p\Z\fR.
% \]
% By~\cite[Lemma~II.7.7]{jantzen} (applied to the pair of elements $(-w_0\lambda,-\rho)$ in the closure of the fundamental alcove), in this situation there exists $w \in \Waff$ such that $w \bullet \lambda=\lambda$ and $w \bullet (-\rho)=\lambda+\mu$. The first condition implies that $w \in W$, so that $w \bullet (-\rho)=-\rho$, and finally $\mu=-\lambda-\rho$, which finishes the proof of our claim.

This claim implies that the direct summand of $\Sim(-w_0\lambda+\rho) \otimes \Ver_{\eta,B'}(\xi + \overline{\lambda+\rho})$ corresponding to the generalized character of $\ZHC$ given by $\zeta_2$ is isomorphic to $\Ver_{\eta, B'}(\xi)$; it follows that
\[
 \Hom_{\cU_\eta^{\zeta_1} \fg} \Bigl( \bk_{\zeta_1} \otimes_{\ZHC} \bigl( \Sim(\lambda+\rho) \otimes \Ver_{\eta, B'}(\xi) \bigr), \Ver_{\eta,B'}(\xi + \overline{\lambda+\rho}) \Bigr)
\]
is $1$-dimensional, which finally proves~\eqref{eqn:fiber-Vermas}.
\end{proof}

% \begin{rmk}
%  Assume that $G$ admits a minuscule weight, i.e.~a weight $\mu$ such that $\langle \mu, \alpha^\vee \rangle \in \{0,1\}$ for any $\alpha \in \fR^+$, with the value $1$ attained at least once. Then the weight $-\rho+p\mu$ also belongs to the closure of the fundamental alcove. The same arguments as for Proposition~\ref{prop:fiber-Verma} show that if $\lambda$ belongs to the closure of the fundamental alcove and satisfies
%  \[
%   \langle \lambda, \alpha^\vee \rangle < p
%  \]
% for all roots $\alpha$ such that $\langle \mu,\alpha \rangle = 0$, and
%  \[
%   \langle \lambda, \alpha^\vee \rangle > 0
%  \]
%  for all roots $\alpha$ such that $\langle \mu,\alpha \rangle > 0$, and if $(\zeta_1,\zeta_2) \in \fZ_\cS(\lambda)^0$ is the image of some element $(\xi + \overline{\lambda+\rho},\xi) \in \tfZ_\cS(\lambda)$, defining $\eta$ and $B'$ as in Proposition~\ref{prop:fiber-Verma} we have
%  \[
%  i^* j_\lambda^* \bigl( \Sim(\nu) \otimes \cU_\cS \fg \bigr) \cong Z_{\eta,B'}(\xi+\overline{\lambda}) \otimes Z_{\eta,B'}(\xi)^*,
% \]
% where $i$ is the embedding of $(\zeta_1,\zeta_2)$ and $\nu$ is the unique dominant $W$-translate of $\lambda-p\mu+\rho$.
% \end{rmk}

The statement of Proposition~\ref{prop:fiber-Verma} is not symmetric, in that the conditions we impose imply that $\zeta_2$ necessarily belongs to $\ft^*_\circ / (W,\bullet)$, whereas $\zeta_1$ might not. Below we will also need the other variant of this statement, in which the \emph{first} component has to belong to $\ft^*_\circ / (W,\bullet)$. Its proof is analogous to that of Proposition~\ref{prop:fiber-Verma}. (More precisely, in this case the counterpart of~\eqref{eqn:fiber-Vermas} can be obtained directly, without recourse to the computation in the paragraph following this equation.)

\begin{prop}
\label{prop:fiber-Verma-prime}
Let $\mu \in \bbX$ be a weight which belongs to the lower closure of the fundamental alcove.
%, and denote by $\nu$ the unique dominant $W$-translate of $\lambda+\rho$. 
Consider some element $\xi \in \ft^*_\circ$, and denote by
%closed point 
$(\zeta_1,\zeta_2) \in \fD(-w_0 \mu)$ the image of $(\xi ,\xi + \overline{\mu+\rho}) \in \tfD(-w_0 \mu)$ in $\fD$
%which is the image of some element $(\xi + \overline{\lambda+\rho},\xi) \in \tfZ_\cS(\lambda)$.
%and let $\xi \in \ft^*$ be a preimage of $\zeta_2$. 
and by $\eta \in \bS^{*}$ the element corresponding to the images of $\zeta_1$ and $\zeta_2$ in $\ft^{*(1)}/W$. As explained above Proposition~\ref{prop:fiber-Verma} the image of $\xi$ in $\ft^{*(1)}$ determines a Borel subgroup $B' \subset G$ with unipotent radical $U'$ such that $\eta_{|\Lie(U')^{(1)}}=0$.

If we denote by $i : \Spec(\bk) \to \fD$ the morphism defined by $(\zeta_1,\zeta_2)$, there exists an
%canonical pair $(\eta,B')$ where $\eta \in \cS^{*(1)}$ and $B' \subset G$ is a Borel subgroup, with unipotent radical $U'$, such that $\eta$ vanishes on $\Lie(U')$, and a 
isomorphism of $\cU_\eta^{\zeta_1} \fg \otimes (\cU_\eta^{\zeta_2} \fg)^{\op}$-modules
\[
 i^* \bigl( \Sim(-w_0\mu+\rho) \otimes \cU_\bS \fg \bigr) \cong \Ver_{\eta,B'}(\xi) \otimes \Ver_{\eta,B'}(\xi+\overline{\mu+\rho})^*.
\]
% The sheaf $j_\lambda^*(\Sim(\nu) \otimes \cU_\cS \fg)$ is a locally free $\scO_{\fZ_\cS^0(\lambda)}$-module of rank $p^{\dim(G)-\dim(T)}$.
\end{prop}

%--------------------------------------------------------------
\subsection{Proof of Theorem~\ref{thm:splitting}}
\label{ss:proof-splitting}
%--------------------------------------------------------------

The proof of Theorem~\ref{thm:splitting} will require two more preliminary lemmas.

\begin{lem}
\label{lem:projective-rk}
 Let $X$ be a reduced scheme locally of finite type over $\bk$, and let $\scF$ be a coherent sheaf on $X$. Assume that there exists $d \geq 0$ such that for any morphism $i : \Spec(\bk) \to X$ the pullback $i^*(\scF) \in \Coh(\Spec(\bk))=\mathsf{Vect}_\bk$ has dimension $d$. Then $\scF$ is a locally free $\scO_X$-module of rank $d$.
\end{lem}

\begin{proof}
 Of course we can assume that $X$ is also affine and of finite type, i.e.~that $X=\Spec(A)$ for some finitely generated reduced $\bk$-algebra $A$. 
 %Note that by~\cite[\href{https://stacks.math.columbia.edu/tag/02J6}{Tag 02J6}]{stacks-project} $\Spec(A)$ is a Jacobson space in the sense of~\cite[\href{https://stacks.math.columbia.edu/tag/005T}{Tag 005T}]{stacks-project}; in other words, 
With this notation, recall that closed points are dense in any closed subset of $\Spec(A)$, see~\S\ref{ss:study-fibers}.
 
 Let us denote by $M$ the $A$-module corresponding to $\scF$. In this setting the datum of a morphism $i : \Spec(\bk) \to X$ is equivalent to the datum of a maximal ideal $\mathfrak{m} \subset A$, and we have $i^*(\scF)=M/\mathfrak{m} \cdot M$. In view of~\cite[\href{https://stacks.math.columbia.edu/tag/0FWG}{Tag 0FWG}]{stacks-project}, to show that $M$ is locally free of rank $d$ it suffices to prove that for any $\mathfrak{p} \in \Spec(A)$ we have
 \[
  \dim_{A_{\mathfrak{p}}/\mathfrak{p} A_{\mathfrak{p}}}(M_{\mathfrak{p}} / \mathfrak{p} M_{\mathfrak{p}})=d.
 \]
Now by~\cite[Theorem~7.33]{peskine}, the function
\[
 \mathfrak{p} \mapsto \dim_{A_{\mathfrak{p}}/\mathfrak{p} A_{\mathfrak{p}}}(M_{\mathfrak{p}} / \mathfrak{p} M_{\mathfrak{p}})
\]
is upper semi-continuous. By assumption, this function is constant (equal to $d$) on the subset of $\Spec(A)$ consisting of maximal ideals, i.e.~of closed points. Hence the open subset
\[
\{\mathfrak{p} \in \Spec(A) \mid \dim_{A_{\mathfrak{p}}/\mathfrak{p} A_{\mathfrak{p}}}(M_{\mathfrak{p}} / \mathfrak{p} M_{\mathfrak{p}}) \leq d \}
\]
contains all closed points, hence is the whole of $\Spec(A)$. On the other hand the open subset
\[
\{\mathfrak{p} \in \Spec(A) \mid \dim_{A_{\mathfrak{p}}/\mathfrak{p} A_{\mathfrak{p}}}(M_{\mathfrak{p}} / \mathfrak{p} M_{\mathfrak{p}}) \leq d-1 \}
\]
does not contain any closed point, hence it is empty.
\end{proof}

%We can now give the proof of Theorem~\ref{thm:splitting}. This proof will require one last preliminary lemma.

\begin{lem}
\label{lem:morph-etale}
 The morphism
 \[
  \ft^*/(W,\bullet) \times_{\ft^{*(1)}/W} \ft^*/(W,\bullet) \times_{\ft^{*(1)}/W} \ft^*/(W,\bullet) \to \fD
 \]
 induced by projection on the first and third factors
is \'etale at any point of the form $(\widetilde{\lambda},\widetilde{-\rho},\widetilde{\mu})$ with $\lambda,\mu \in \bbX$.
\end{lem}

\begin{proof}
 To prove this claim it suffices to prove that the morphism $\ft^*/(W,\bullet) \to \ft^{*(1)}/W$ is \'etale at $\widetilde{-\rho}$. The dot-action of $W$ and the natural action of $\ft^*_{\mathbb{F}_p}$ on $\ft^*$ combine to provide an action of the semi-direct product $\ft^*_{\mathbb{F}_p} \rtimes W$ (where $W$ acts on $\ft^*_{\mathbb{F}_p}$ through the natural, unshifted, action) defined by $(\ola w) \bullet \xi = w(\xi+\overline{\rho})-\overline{\rho}+\ola$ for $\ola \in \ft^*_{\mathbb{F}_p}$ and $w \in W$. Moreover, the composition $\ft^* \to \ft^*/(W,\bullet) \to \ft^{*(1)}/W$ is the quotient morphism for this action. Since $\overline{-\rho}$ is stabilized by $W$, the claim then follows from~\cite[Exp.~V, Proposition~2.2]{sga1}.
%  However, as observed in the course of the proof of Lemma~\ref{lem:U-rho}, the morphism $\ft^*/(W,\bullet) \to \ft^{*(1)}/W$ induces an isomorphism between the completion of $\scO(\ft^{*(1)}/W)$ with respect to $I$
% %the ideal corresponding to the image of $0 \in \ft^*$ 
% and the completion of $\scO(\ft^*/(W,\bullet))$ with respect to $\mathfrak{m}^{-\rho}$. In view of~\cite[\href{https://stacks.math.columbia.edu/tag/039M}{Tag 039M}]{stacks-project} this implies the desired property.
\end{proof}

For $\lambda \in \bbX$, whose image in $\ft^*_{\mathbb{F}_p}/(W,\bullet)$ is that of $\lambda' \in \Lambda$, we set
\[
 \fD_\circ(\lambda) := \Bigl(\ft^* / (W,\bullet) \times_{\ft^{*(1)}/W} \ft^*_\circ / (W,\bullet) \Bigr) \smallsetminus \left( \bigcup_{\mu \in \Lambda  \smallsetminus \{\lambda'\}} \fD(\mu) \right).
\]
Then $\fD_\circ(\lambda)$ is an open subset of $\fD$, contained in $\fD(\lambda)$. We will denote by $j_\lambda : \fD_\circ(\lambda) \to \fD$ the embedding.
Continuing with the same notation, we also set 
\[
 \fD_\circ'(\lambda) := \Bigl(\ft_\circ^* / (W,\bullet) \times_{\ft^{*(1)}/W} \ft^* / (W,\bullet) \Bigr) \smallsetminus \left( \bigcup_{\mu \in \Lambda  \smallsetminus \{\lambda'\}} \fD(\mu) \right),
\]
and we denote by $j_\lambda' : \fD_\circ'(\lambda) \to \fD$ the open embedding.

\begin{proof}[Proof of Theorem~\ref{thm:splitting}]
 Let $\lambda,\mu \in \bbX$ belong to the lower closure of the fundamental alcove. 
% Note that $(\widetilde{\lambda},\widetilde{-\rho}) \in \fZ^0_\cS(\lambda)$; in fact, 
The vector $\overline{-\rho}$ belongs to $\ft^*_\circ$ since this point is stable under the dot-action of $W$. On the other hand, if $\nu \in \bbX$ is such that $(\widetilde{\lambda},\widetilde{-\rho}) \in \fD(\nu)$, then there exists $\xi \in \ft^*$ such that the point $(\xi+\overline{\nu+\rho},\xi) \in \tfD$  has image $(\widetilde{\lambda},\widetilde{-\rho})$ in $\fD$; we then have $\xi \in W \bullet \overline{-\rho} = \{\overline{-\rho}\}$ and $\xi + \overline{\nu+\rho} \in W \bullet \overline{\lambda}$, so that $\tla=\widetilde{\nu}$. We have finally checked that $(\widetilde{\lambda},\widetilde{-\rho}) \in \fD_\circ(\lambda)$; similar considerations show that $(\widetilde{-\rho},\widetilde{\mu}) \in \fD'_\circ(-w_0\mu)$.
 
 By construction, $\fD_\circ(\lambda) \times_{\ft^*/(W,\bullet)} \fD'_\circ(-w_0\mu)$ is an open subscheme in the fiber product $\ft^*/(W,\bullet) \times_{\ft^{*(1)}/W} \ft^*/(W,\bullet) \times_{\ft^{*(1)}/W} \ft^*/(W,\bullet)$.
 Consider the morphism
 \[
  f : \fD_\circ(\lambda) \times_{\ft^*/(W,\bullet)} \fD'_\circ(-w_0\mu) \to \ft^*/(W,\bullet)
 \]
induced by projection on the middle factor. The algebra $\cU_\bS \fg$ is an $\scO(\ft^*/(W,\bullet))$-algebra; it therefore defines a coherent sheaf of $\scO_{\ft^*/(W,\bullet)}$-algebras $\mathscr{A}$ on $\ft^*/(W,\bullet)$. Consider also the projections
\begin{align*}
 p &: \fD_\circ(\lambda) \times_{\ft^*/(W,\bullet)} \fD'_\circ(-w_0\mu) \to \fD_\circ(\lambda), \\
 q &: \fD_\circ(\lambda) \times_{\ft^*/(W,\bullet)} \fD'_\circ(-w_0\mu) \to \fD'_\circ(-w_0 \mu).
\end{align*}
The sheaves $p^* j_\lambda^* ( \Sim(\lambda+\rho) \otimes \cU_\bS \fg )$ and $q^* (j'_{-w_0\mu})^* ( \Sim(-w_0\mu+\rho) \otimes \cU_\bS \fg)$ are naturally sheaves of (right and left, respectively) modules for $f^* \mathscr{A}$, so that we can consider the tensor product
\begin{equation}
\label{eqn:splitting-bundle}
 p^* j_\lambda^* ( \Sim(\lambda+\rho) \otimes \cU_\bS \fg ) \otimes_{f^* \mathscr{A}} q^* (j'_{-w_0\mu})^* ( \Sim(-w_0\mu+\rho) \otimes \cU_\bS \fg ).
\end{equation}

We claim that this sheaf is a locally free $\scO_{\fD_\circ(\lambda) \times_{\ft^*/(W,\bullet)} \fD'_\circ(-w_0\mu)}$-module, of rank $p^{2\# \fR^+}$. In fact, by Lemma~\ref{lem:projective-rk}, to prove this it suffices to prove that for any closed point $(\zeta_1,\zeta_2,\zeta_3) \in \fD_\circ(\lambda) \times_{\ft^*/(W,\bullet)} \fD'_\circ(-w_0\mu)$, denoting by $i : \Spec(\bk) \to \fD_\circ(\lambda) \times_{\ft^*/(W,\bullet)} \fD'_\circ(-w_0\mu)$ the corresponding morphism, the vector space
\begin{equation}
\label{eqn:fiber-splitting-bundle}
 i^* \bigl( p^* j_\lambda^* ( \Sim(\lambda+\rho) \otimes \cU_\cS \fg ) \otimes_{f^* \mathscr{A}} q^* (j'_{-w_0\mu})^* ( \Sim(-w_0\mu+\rho) \otimes \cU_\cS \fg ) \bigr)
\end{equation}
has dimension $p^{2\#\fR^+}$. If we denote by $i_1 : \Spec(\bk) \to \fD$ and $i_2 : \Spec(\bk) \to \fD$ the embeddings of the points $(\zeta_1,\zeta_2)$ and $(\zeta_2,\zeta_3)$ respectively, then this vector space can be written as
\[
 i_1^* \bigl( \Sim(\lambda+\rho) \otimes \cU_\bS \fg \bigr) \otimes_{\cU_\eta^{\zeta_2} \fg} i_2^* \bigl( \Sim(-w_0\mu+\rho) \otimes \cU_\bS \fg \bigr),
\]
where $\eta \in \bS^{*}$ is the image of the $\zeta_i$'s. Let $\xi \in \ft^*$ be such that $(\zeta_1,\zeta_2)$ is the image of $(\xi + \overline{\lambda+\rho},\xi)$, and let $B' \subset G$ be the Borel subgroup with unipotent radical $U'$ such that $\eta_{|\Lie(U')^{(1)}}=0$ determined by the image of $\xi$ in $\ft^{*(1)}$ (see the comments above Proposition~\ref{prop:fiber-Verma}). By Proposition~\ref{prop:fiber-Verma}  we have
\[
 i_1^* \bigl( \Sim(\lambda+\rho) \otimes \cU_\bS \fg \bigr) \cong \Ver_{\eta,B'}(\xi+\overline{\lambda+\rho}) \otimes \Ver_{\eta,B'}(\xi)^*.
\]
Similarly, if $\xi' \in \ft^*$ is such that $(\zeta_2,\zeta_3)$ is the image of $(\xi',\xi' + \overline{\mu+\rho})$, and if $B'' \subset G$ is the Borel subgroup with unipotent radical $U''$ such that $\eta_{|\Lie(U'')^{(1)}}=0$ determined by the image of $\xi'$ in $\ft^{*(1)}$, then by Proposition~\ref{prop:fiber-Verma-prime} we have
\[
 i_2^* \bigl( \Sim(-w_0 \mu+\rho) \otimes \cU_\bS \fg \bigr) \cong \Ver_{\eta,B''}(\xi') \otimes \Ver_{\eta,B''}(\xi'+\overline{\mu+\rho})^*.
\]
Here $\Ver_{\eta,B'}(\xi)$ and $\Ver_{\eta,B''}(\xi')$ are two simple modules over the matrix algebra $\cU_\eta^{\zeta_2} \fg$, see~\S\ref{ss:Azumaya-Ug}; they must therefore be isomorphic. Fixing an isomorphism $\varphi : \Ver_{\eta,B'}(\xi) \simto \Ver_{\eta,B''}(\xi')$, we obtain a pairing
\[
 \Ver_{\eta,B'}(\xi)^* \otimes \Ver_{\eta,B''}(\xi') \to \bk
\]
defined by $f \otimes v \mapsto f(\varphi^{-1}(v))$, which induces an isomorphism
\[
 \Ver_{\eta,B'}(\xi)^* \otimes_{\cU_\eta^{\zeta_2} \fg} \Ver_{\eta,B''}(\xi') \simto \bk.
\]
Combining these observations we obtain that the vector space in~\eqref{eqn:fiber-splitting-bundle} is isomorphic to
\[
 \Ver_{\eta,B'}(\xi+\overline{\lambda+\rho}) \otimes \Ver_{\eta,B''}(\xi'+\overline{\mu+\rho})^*,
\]
hence has dimension $p^{2\#\fR^+}$, as desired.
 
 Now we consider the morphism
 \[
  \fD_\circ(\lambda) \times_{\ft^*/(W,\bullet)} \fD'_\circ(-w_0\mu) \to \fD
 \]
obtained from that of Lemma~\ref{lem:morph-etale} by restriction to the open subset 
\[
\fD_\circ(\lambda) \times_{\ft^*/(W,\bullet)} \fD'_\circ(-w_0\mu) \subset \ft^*/(W,\bullet) \times_{\ft^{*(1)}/W} \ft^*/(W,\bullet) \times_{\ft^{*(1)}/W} \ft^*/(W,\bullet).
\]
This lemma ensures that this morphism is \'etale at $(\widetilde{\lambda},\widetilde{-\rho},\widetilde{\mu})$; it therefore identifies the completion of $\fD_\circ(\lambda) \times_{\ft^*/(W,\bullet)} \fD'_\circ(-w_0\mu)$ at $(\widetilde{\lambda},\widetilde{-\rho},\widetilde{\mu})$ with the completion of $\fD$ at $(\widetilde{\lambda},\widetilde{\mu})$, i.e.~with the spectrum of $\cZ_\bS^{\hat{\lambda},\hat{\mu}}$. By construction the completion of the sheaf~\eqref{eqn:splitting-bundle} at $(\widetilde{\lambda},\widetilde{-\rho},\widetilde{\mu})$ is $\sfM^{\lambda,\mu}_\bS$; since this sheaf is locally free this proves that $\sfM^{\lambda,\mu}_\bS$ is faithfully projective. In fact, since the ring $\cZ_\bS^{\hat{\lambda},\hat{\mu}}$ is local, this module is even free (of rank $p^{2\#\fR^+}$) by~\cite[\href{https://stacks.math.columbia.edu/tag/00NZ}{Tag 00NZ}]{stacks-project}.

Finally we consider the natural morphism
\[
\cU_\bS^{\hat{\lambda},\hat{\mu}} \to \End_{\cZ_\bS^{\hat{\lambda},\hat{\mu}}}(\sfM^{\lambda,\mu}_\bS).
\]
Here, both sides are finite free as modules over $\cZ_\bS^{\hat{\lambda},\hat{\mu}}$. In fact, for the right-hand side this follows from the same property for the module $\sfM^{\lambda,\mu}_\bS$, which we have seen above. For the left-hand side, we observe that $\cU_\bS \fg$ is finite projective over $\scO(\fC_\bS)$ by Proposition~\ref{prop:Ureg-Azumaya}; it follows that $\cU_\bS \fg \otimes_{\scO(\bS^{*})} \cU_\bS \fg^\op$ is finite projective over $\cZ_\bS$, and finally that $\cU_\bS^{\hat{\lambda},\hat{\mu}}$ is finite projective, hence finite free (again by~\cite[\href{https://stacks.math.columbia.edu/tag/00NZ}{Tag 00NZ}]{stacks-project}), over the local ring $\cZ_\bS^{\hat{\lambda},\hat{\mu}}$. Given this property, to prove that our morphism is an isomorphism it suffices to prove that it is invertible after application of the functor $\bk \otimes_{\cZ_\bS^{\hat{\lambda},\hat{\mu}}} (-)$. Now if we denote by $\chi \in \fg^{*(1)}$ the point corresponding to the image of $0$ in $\ft^{*(1)}/W$ under the identification $\bS^* \simto \ft^{*(1)}/W$, we have
\[
 \bk \otimes_{\cZ_\bS^{\hat{\lambda},\hat{\mu}}} \cU_\bS^{\hat{\lambda},\hat{\mu}} = \cU_\chi^\lambda \fg \otimes (\cU_\chi^\mu \fg)^\op.
\]
%where $\chi:=\varkappa(e)$, seen as a point in $\fg^{*(1)}$. (See~\S\ref{ss:centralizer-Kostant} for the definitions of $\varkappa$ and $e$.) 
On the other hand, since $\sfM^{\lambda,\mu}_\bS$ is a free module we have
\[
 \bk \otimes_{\cZ_\bS^{\hat{\lambda},\hat{\mu}}} \End_{\cZ_\bS^{\hat{\lambda},\hat{\mu}}}(\sfM^{\lambda,\mu}_\bS) \cong \End_\bk( \bk \otimes_{\cZ_\bS^{\hat{\lambda},\hat{\mu}}} \sfM^{\lambda,\mu}_\bS),
\]
and applying the considerations above with $\xi=\xi'=\overline{-\rho}$
%considering elements $\xi,\xi'$ as above for the triple $(\lambda,-\rho,\mu)$ 
we have
\begin{equation}
\label{eqn:fiber-M-lambda-mu}
 \bk \otimes_{\cZ_\bS^{\hat{\lambda},\hat{\mu}}} \sfM^{\lambda,\mu}_\bS \cong \Ver_{\chi,B'}(\overline{\lambda}) \otimes \Ver_{\chi,B'}(\overline{\mu})^*,
\end{equation}
%for appropriate Borel subgroups $B',B''$. 
where $B' \subset G$ is the unique Borel subgroup with unipotent radical $U'$ such that $\chi_{|\Lie(U')^{(1)}}=0$.
By~\eqref{eqn:U-End-Z} our morphism is indeed an isomorphism, which finishes the proof.
\end{proof}

%-------------------------------------------------------------------
\subsection{Localization for Harish-Chandra bimodules}
\label{ss:localization-HC}
%-------------------------------------------------------------------

The main consequence of Theorem~\ref{thm:splitting} that will be used below is the following statement. (See~\S\ref{ss:HCBim-S} for the definition of $\cU_\bS^{\hat{\lambda}}$, and~\S\ref{ss:splitting-bundles} for that of $\cZ_\bS^{\hat{\lambda}}$.)

\begin{cor}
\label{cor:equiv-ModUg-ModZ}
For any $\lambda,\mu \in \bbX$ in the lower closure of the fundamental alcove, the functor $\sfM^{\lambda,\mu}_\bS \otimes_{\cZ_\bS^{\hat{\lambda},\hat{\mu}}} (-)$ induces an equivalence of abelian categories
\[
\mathscr{L}_{\lambda,\mu} : \Mod^{\bbI}_{\fin}(\cZ_\bS^{\hat{\lambda},\hat{\mu}}) \simto \Mod^{\bbI}_{\mathrm{fg}}(\cU_\bS^{\hat{\lambda},\hat{\mu}})
\]
which restricts to an equivalence of abelian subcategories
\[
\Mod^{\bbJ}_{\fin}(\cZ_\bS^{\hat{\lambda},\hat{\mu}}) \simto \HCBim_\bS^{\hat{\lambda},\hat{\mu}}.
\]
Moreover, 
% for $\lambda,\mu,\nu \in \bbX$ in the lower closure of the fundamental alcove, these equivalences intertwine the bifunctors
% \[
% \hatotimes_{\cU_\cS \fg} : \Mod^{\mathbb{I}}_{\fin}(\cU_\cS^{\hat{\lambda},\hat{\mu}}) \times \Mod^{\mathbb{I}}_{\fin}(\cU_\cS^{\hat{\nu},\hat{\nu}}) \to \Mod^{\mathbb{I}}_{\fin}(\cU_\cS^{\hat{\lambda},\hat{\nu}})
% \]
% and
% \[
% \hatstar_\cS : \Mod^{\mathbb{I}}_{\fin}(\cZ_\cS^{\hat{\lambda},\hat{\mu}}) \times \Mod^{\mathbb{I}}_{\fin}(\cU_\cS^{\hat{\nu},\hat{\nu}}) \to \Mod^{\mathbb{I}}_{\fin}(\cU_\cS^{\hat{\lambda},\hat{\nu}})
% \]
% in the obvious way and, 
in case $\lambda=\mu$, there exists a canonical isomorphism
\begin{equation}
\label{eqn:L-units}
\mathscr{L}_{\lambda,\lambda}(\cZ_\bS^{\hat{\lambda}}) \cong \cU_\bS^{\hat{\lambda}}.
\end{equation}
%it sends the unit object $\cZ_\cS^{\hat{\lambda}}$ to the unit object $\cU_\cS^{\hat{\lambda}}$, therefore defining an equivalence of monoidal categories.
\end{cor}

\begin{proof}
The properties stated in Theorem~\ref{thm:splitting} ensure that the functor
\[
\sfM^{\lambda,\mu}_\bS \otimes_{\cZ_\bS^{\hat{\lambda},\hat{\mu}}} (-)
\]
induces an equivalence of abelian categories
\[
\Mod_{\mathrm{fg}}(\cZ_\bS^{\hat{\lambda},\hat{\mu}}) \simto \Mod_{\mathrm{fg}}(\cU_\bS^{\hat{\lambda},\hat{\mu}}),
\]
see~\eqref{eqn:equiv-faith-proj}. Adding the $\bbI^{\hat{\lambda},\hat{\mu}}_\bS$-actions in the picture we obtain the desired equivalence
\[
\Mod^{\mathbb{I}}_{\fin}(\cZ_\bS^{\hat{\lambda},\hat{\mu}}) \simto \Mod^{\bbI}_{\mathrm{fg}}(\cU_\bS^{\hat{\lambda},\hat{\mu}}).
\]

Let us now identify the subcategory corresponding to completed Harish-Chandra $\cU_\bS \fg$-bimodules under this equivalence. Recall that any object in $\Mod^{\bbI}_{\mathrm{fg}}(\cU_\bS^{\hat{\lambda},\hat{\mu}})$ has a canonical action of $\bbI^{\hat{\lambda},\hat{\mu}}_\bS \ltimes G_1$, and that such an object is a Harish-Chandra bimodule iff this action factors through the product morphism $\bbI^{\hat{\lambda},\hat{\mu}}_\bS \ltimes G_1 \to \bbI^{\hat{\lambda},\hat{\mu}}_\bS$, i.e.~iff the action of the kernel $\mathbb{K}_\bS^{\hat{\lambda},\hat{\mu}}$ of this map is trivial. Now if $V$ is in $\Mod^{\mathbb{I}}_{\fin}(\cZ_\bS^{\hat{\lambda},\hat{\mu}})$, the action of $\bbI^{\hat{\lambda},\hat{\mu}}_\bS \ltimes G_1$ on $\sfM^{\lambda,\mu}_\bS \otimes_{\cZ_\bS^{\hat{\lambda},\hat{\mu}}} V$ is diagonal, induced by the action on $\sfM^{\lambda,\mu}_\bS$ and the action on $V$ obtained by pullback under the morphism $\bbI^{\hat{\lambda},\hat{\mu}}_\bS \ltimes G_1 \to \bbI^{\hat{\lambda},\hat{\mu}}_\bS$ given by projection on the first factor. By construction the module $\sfM^{\lambda,\mu}_\bS$ is a completed Harish-Chandra $\cU_\bS \fg$-bimodule, so that the action on this factor does factor though the product morphism $\bbI^{\hat{\lambda},\hat{\mu}}_\bS \ltimes G_1 \to \bbI^{\hat{\lambda},\hat{\mu}}_\bS$, and moreover $\sfM^{\lambda,\mu}_\bS$ is free over $\cZ_\bS^{\hat{\lambda},\hat{\mu}}$. Hence $\sfM^{\lambda,\mu}_\bS \otimes_{\cZ_\bS^{\hat{\lambda},\hat{\mu}}} V$ is a completed Harish-Chandra bimodule iff the action of $\mathbb{K}_\bS^{\hat{\lambda},\hat{\mu}}$ on $V$ is trivial, or in other words iff the action of the subgroup scheme $G_1 \times \Spec(\cZ_\bS^{\hat{\lambda},\hat{\mu}})$ on $V$ is trivial, or finally iff the action of $\bbI^{\hat{\lambda},\hat{\mu}}_\bS$ factor through the quotient morphism $\bbI^{\hat{\lambda},\hat{\mu}}_\bS \to \bbJ^{\hat{\lambda},\hat{\mu}}_\bS$. This proves that our equivalence restricts to an equivalence $\Mod^{\bbJ}_{\fin}(\cZ_\bS^{\hat{\lambda},\hat{\mu}}) \simto \HCBim_\bS^{\hat{\lambda},\hat{\mu}}$.

Finally, we consider the special case $\lambda=\mu$, and construct a canonical isomorphism
$\mathscr{L}_{\lambda,\lambda}(\cZ_\bS^{\hat{\lambda}}) \cong \cU_\bS^{\hat{\lambda}}$.
Adjunction (see Lemma~\ref{lem:convolution-adjoint}) provides a canonical morphism
\[
\sfP_\bS^{\lambda,-\rho} \hatotimes_{\cU_\bS \fg} \sfP_\bS^{-\rho,\lambda} \to \cU_\bS^{\hat{\lambda}},
\]
which factors through a morphism
\[
\mathscr{L}_{\lambda,\lambda}(\cZ_\bS^{\hat{\lambda}}) = \sfM_\bS^{\lambda,\lambda} \otimes_{\cZ_\bS^{\hat{\lambda},\hat{\lambda}}} \cZ_\bS^{\hat{\lambda}} \to \cU_\bS^{\hat{\lambda}}.
\]
Here, by the same considerations as in the proof of Theorem~\ref{thm:splitting}, both sides are finite free modules over the local ring $\cZ_\bS^{\hat{\lambda}}$; to prove that this morphism is an isomorphism it therefore suffices to check that the induced morphism
\[
\left( \sfM_\bS^{\lambda,\lambda} \otimes_{\cZ_\bS^{\hat{\lambda},\hat{\lambda}}} \cZ_\bS^{\hat{\lambda}} \right) \otimes_{\cZ_\bS^{\hat{\lambda}}} \bk \to \cU_\bS^{\hat{\lambda}} \otimes_{\cZ_\bS^{\hat{\lambda}}} \bk
\]
is invertible. The right-hand side identifies with $\cU_\chi^{\lambda} \fg$, where $\chi$ is as at the end of the proof of Theorem~\ref{thm:splitting}, and by~\eqref{eqn:fiber-M-lambda-mu} the left-hand side identifies with
$\Ver_{\chi,B'}(\overline{\lambda}) \otimes \Ver_{\chi,B'}(\overline{\lambda})^*$,
where $B' \subset G$ is the unique Borel subgroup with unipotent radical $U'$ such that $\chi_{|\Lie(U')^{(1)}}=0$; the desired claim is therefore clear from the isomorphism~\eqref{eqn:U-End-Z}.
\end{proof}

\begin{rmk}
\label{rmk:L-monoidal}
 We will prove later (at least in the special case when $\mu$ belongs to the fundamental alcove, see~\S\ref{ss:monoidality-L}) that for any $\lambda,\mu,\nu \in \bbX$ in the lower closure of the fundamental alcove the equivalences of Corollary~\ref{cor:equiv-ModUg-ModZ} intertwine the bifunctors
 \[
 \hatotimes_{\cU_\bS \fg} : \Mod^{\bbI}_{\fin}(\cU_\bS^{\hat{\lambda},\hat{\mu}}) \times \Mod^{\bbI}_{\fin}(\cU_\bS^{\hat{\mu},\hat{\nu}}) \to \Mod^{\bbI}_{\fin}(\cU_\bS^{\hat{\lambda},\hat{\nu}})
 \]
 (see~\S\ref{ss:HCBim-S}) and
 \[
 \hatstar_\bS : \Mod^{\bbI}_{\fin}(\cZ_\bS^{\hat{\lambda},\hat{\mu}}) \times \Mod^{\bbI}_{\fin}(\cZ_\bS^{\hat{\mu},\hat{\nu}}) \to \Mod^{\bbI}_{\fin}(\cZ_\bS^{\hat{\lambda},\hat{\nu}})
 \]
(see~\S\ref{ss:splitting-bundles}).
\end{rmk}

%%%%%%%%%%%%%%%%%%%%%%%%%%%%%%%%%%%%%%%%%%%%%%%%%
\section{\texorpdfstring{$\Ug$}{Ug} and differential operators on the flag variety}
\label{sec:Ug-D}
%%%%%%%%%%%%%%%%%%%%%%%%%%%%%%%%%%%%%%%%%%%%%%%%%

In this section we study the equivalences $\mathscr{L}_{\lambda,\mu}$ of Corollary~\ref{cor:equiv-ModUg-ModZ} further, using the relation between the algebra $\Ug$ and differential operators on the flag variety of $G$.

%----------------------------------------
\subsection{Universal twisted differential operators}
\label{ss:Ug-D-reg}
%----------------------------------------

Set $\cB:=G/B$, and consider the natural projection morphism
\[
 \omega : G/U \to \cB.
\]
Here $G/U$ admits a natural action of $T$ induced by multiplication on the right on $G$, and $\omega$ is a (Zariski locally trivial) $T$-torsor.
The sheaf of universal twisted differential operators on $\cB$ is the quasi-coherent sheaf of algebras
\[
 \tD := \omega_*(\mathscr{D}_{G/U})^T,
\]
where the exponent means $T$-invariants. The actions of $G$ and $T$ on $G/U$ induce a canonical algebra morphism
\begin{equation}
\label{eqn:morph-Ug-D}
 \Ug \otimes_{\ZHC} \scO(\ft^*) \to \Gamma(\cB,\tD),
\end{equation}
see~\cite[Lemma~3.1.5]{bmr}. 
%(This morphism is in fact an isomorphism, see~\cite[Proposition~3.4.1]{bmr}, but this fact will not be used below.)

Recall the Grothendieck resolution $\tbg$ for the group $\bG=G^{(1)}$ and the morphism $\vartheta : \tbg \to \ft^{*(1)}$, see~\S\ref{ss:Kostant-section-Abe} and~\S\ref{ss:centralizer-Kostant}.
Consider the Frobenius morphism $\Fr_\cB : \cB \to \cB^{(1)}$ and the natural morphism $f : \tbg \times_{\ft^{*(1)}} \ft^* \to \cB^{(1)}$. As explained in~\cite[\S 2.3]{bmr}, there exists a canonical algebra morphism
\[
 f_* \scO_{\tbg \times_{\ft^{*(1)}} \ft^*} \to (\Fr_\cB)_* \tD
\]
which takes values in the center of $(\Fr_\cB)_* \tD$, and which makes $(\Fr_\cB)_* \tD$ a locally finitely generated $f_* \scO_{\tbg \times_{\ft^{*(1)}} \ft^*}$-module. Since all the morphisms involved in this construction are affine, using this morphism one can consider $\tD$ as a coherent sheaf of $\scO_{\tbg \times_{\ft^{*(1)}} \ft^*}$-algebras on $\tbg \times_{\ft^{*(1)}} \ft^*$. (We will not introduce a different notation for this sheaf of algebras.)

Recall also (see~\S\ref{ss:study-fibers}) that we denote by $\widetilde{\bS}^*$ the preimage of $\bS^*$ under the natural morphism $\pi : \tbg \to \fg^{*(1)}$, and that the morphism $\vartheta$ restricts to an isomorphism $\widetilde{\bS}^* \simto \ft^{*(1)}$; in particular, $\widetilde{\bS}^*$ is an affine scheme.
We set
\[
 \tD_\bS := \tD_{|\widetilde{\bS}^{*} \times_{\ft^{*(1)}} \ft^*}.
 %, \quad \bbD_\cS := \Gamma(\widetilde{\cS}^{*(1)} \times_{\ft^{*(1)}} \ft^*, \tD_\cS).
\]
%As for $\tD$, $\tD_\cS$ can be seen either as a coherent sheaf of $\scO_{\widetilde{\cS}^{*(1)} \times_{\ft^{*(1)}} \ft^*}$-algebras on $\widetilde{\cS}^{*(1)} \times_{\ft^{*(1)}} \ft^*$, or as a quasi-coherent sheaf of algebras on $\cB$. We 
We will
also set
\[
 \widetilde{\cU}_\bS\fg := \cU_\bS\fg \otimes_{\ZHC} \scO(\ft^*).
\]

\begin{lem}
\label{lem:US-D}
 The morphism~\eqref{eqn:morph-Ug-D} induces an algebra isomorphism
 \[
  \widetilde{\cU}_\bS\fg \simto \Gamma \bigl( \widetilde{\bS}^{*} \times_{\ft^{*(1)}} \ft^*, \tD_\bS \bigr).
 \]
\end{lem}

\begin{proof}
 Consider the natural morphism
 \[
  h : \widetilde{\bS}^{*} \times_{\ft^{*(1)}} \ft^* \to \bS^{*} \times_{\ft^{*(1)}/W} \ft^*/(W,\bullet).
 \]
If we still denote by $\cU_\bS\fg$ the sheaf of $\scO_{\bS^{*} \times_{\ft^{*(1)}/W} \ft^*/(W,\bullet)}$-algebras associated with this $\scO(\bS^{*} \times_{\ft^{*(1)}/W} \ft^*/(W,\bullet))$-algebra, then as in~\cite[Proposition~5.2.1]{bmr} the morphism~\eqref{eqn:morph-Ug-D} induces a canonical isomorphism of sheaves of algebras
\[
 h^*(\cU_\bS\fg) \simto \tD_\bS.
\]
Now $h$ induces an isomorphism
\[
 \widetilde{\bS}^{*} \times_{\ft^{*(1)}} \ft^* \to \bigl( \bS^{*} \times_{\ft^{*(1)}/W} \ft^*/(W,\bullet) \bigr) \times_{\ft^*/(W,\bullet)} \ft^*
\]
(in fact, both sides identify canonically with $\ft^*$)
so that the claim follows by taking global sections.
\end{proof}

\begin{rmk}
 One can give a different proof of Lemma~\ref{lem:US-D} as follows. By~\cite[Proposition~3.4.1]{bmr}, the morphism~\eqref{eqn:morph-Ug-D} is an isomorphism; in other words, identifying quasi-coherent sheaves on $\fg^{*(1)}$ and $\scO(\fg^{*(1)})$-modules, we have a canonical isomorphism of sheaves of $\scO_{\fg^{*(1)}}$-algebras
 \[
  g_* \tD \cong \Ug \otimes_{\ZHC} \scO(\ft^*),
 \]
where $g : \tbg \times_{\ft^{*(1)}} \ft^* \to \fg^{*(1)} \times_{\ft^{*(1)}/W} \ft^*$ is the morphism induced by $\pi$. Restricting this isomorphism first to $\fg^{*(1)}_\reg \times_{\ft^{*(1)}/W} \ft^*$ and then to $\bS^{*} \times_{\ft^{*(1)}/W} \ft^*$ we deduce the isomorphism of the lemma, since $g$ restricts to an isomorphism on the preimage of $\fg^{*(1)}_\reg \times_{\ft^{*(1)}/W} \ft^*$ (see~\eqref{eqn:tfgreg}).
\end{rmk}

% Recall that a weight $\lambda \in \bbX$ is called \emph{regular} if its stabilizer for the dot-action of $\Waff$ is trivial. The use of the algebra $\cU_\cS\fg \otimes_{\ZHC} \scO(\ft^*)$ is particularly suited for the study of modules with regular characters due to the following claim.
% 
% \begin{lem}
% \label{lem:quotient-etale}
%  Let $\lambda \in \bbX$ be regular. Then the quotient morphism
%  \[
%   \ft^* \to \ft^*/(W,\bullet)
%  \]
% is \'etale at $\ola$.
% \end{lem}
% 
% \begin{proof}
%  By Lemma~\ref{lem:weights}, the stabilizer of $\ola$ for the dot-action of $W$ on $\ft^*$ is also trivial. Hence the claim follows from the general criterion~\cite[Exp.~V, Proposition~2.2]{sga1}.
% \end{proof}

%----------------------------------------
\subsection{Study of some equivariant \texorpdfstring{$\cU_\bS\fg$}{USg}-bimodules}
\label{ss:Ug-D-bimodules}
%----------------------------------------

Given any $\lambda \in \bbX$, we have a line bundle $\scO_{\cB}(\lambda)$ on $\cB$ attached naturally to $\lambda$. (Our normalization is that of~\cite{jantzen}, so that line bundles attached to dominant weights are ample.) This line bundle identifies with the direct summand of $\omega_* \scO_{G/U}$ consisting of sections which have weight $\lambda$ for the $T$-action induced by right multiplication on $G$; it therefore admits a natural action of the sheaf of algebras $\tD$. Using this action and the natural action on $\tD$, we obtain a left action of $\tD$ on the tensor product
\[
 \scO_\cB(\lambda) \otimes_{\scO_\cB} \tD.
\]
As for $\tD$ itself, this module can be also considered as a sheaf of modules on $\tbg \times_{\ft^{*(1)}} \ft^*$. We 
can therefore consider the $\bk$-vector space
%set
\begin{equation}
\label{eqn:def-Dlambda}
 %\bbD_{\bS,\lambda}:=
 \Gamma \bigl( \widetilde{\bS}^{*} \times_{\ft^{*(1)}} \ft^*, (\scO_\cB(\lambda) \otimes_{\scO_\cB} \tD)_{|\widetilde{\bS}^{*} \times_{\ft^{*(1)}} \ft^*} \bigr),
\end{equation}
which in view of Lemma~\ref{lem:US-D} admits a natural left action of $\widetilde{\cU}_\bS\fg$.
The tensor product $\scO_\cB(\lambda) \otimes_{\scO_\cB} \tD$ also admits a natural right action of $\tD$, induced by right multiplication on the second factor. The action of $\pi_* \scO_{\tbg}$ on $(\Fr_\cB)_* \scO_\cB(\lambda)$ being trivial, the two actions of this subalgebra on $(\Fr_\cB)_* (\scO_\cB(\lambda) \otimes_{\scO_\cB} \tD)$ coincide, and the space~\eqref{eqn:def-Dlambda} therefore also admits a right action of $\widetilde{\cU}_\bS\fg$; moreover these actions combine to provide an action of $\widetilde{\cU}_\bS\fg \otimes_{\scO(\bS^{*})} (\widetilde{\cU}_\bS\fg)^\op$.
By construction the action of the central subalgebra
\[
 \scO(\ft^*) \otimes_{\scO(\bS^{*})} \scO(\ft^*) \cong \scO(\ft^* \times_{\ft^{*(1)}/W} \ft^*)
\]
factors through an action of the image of the closed embedding $\ft^* \to \ft^* \times_{\ft^{*(1)}/W} \ft^*$ given by
\[
 \xi \mapsto (\xi + \ola,\xi).
\]
The object $\scO_\cB(\lambda) \otimes_{\scO_\cB} \tD$ also admits a natural structure of $G$-equivariant quasi-coherent sheaf, compatible with the actions considered above. The module~\eqref{eqn:def-Dlambda} therefore also admits a natural and compatible structure of module for the group scheme
\[
 \ft^* \times_{\ft^{*(1)}/W} \bbI_\bS^* \times_{\ft^{*(1)}/W} \ft^*,
\]
see~\S\S\ref{ss:centralizer-Kostant}--\ref{ss:HCBim-S}.

For $\lambda,\mu \in \bbX$, we will denote by
\[
 \widetilde{\cU}_\bS^{\hat{\lambda},\hat{\mu}}
\]
the completion of the $\scO(\ft^* \times_{\ft^{*(1)}/W} \ft^*)$-algebra 
$\widetilde{\cU}_\bS\fg \otimes_{\scO(\bS^{*})} (\widetilde{\cU}_\bS\fg)^\op$
at the ideal corresponding to the point $(\ola,\omu) \in \ft^* \times_{\ft^{*(1)}/W} \ft^*$. Copying the constructions in~\S\ref{ss:HCBim-S} (replacing $\cU_\bS^{\hat{\lambda},\hat{\mu}}$ by $\widetilde{\cU}_\bS^{\hat{\lambda},\hat{\mu}}$ and $\ft^*/(W,\bullet) \times_{\ft^{*(1)}/W} \bbI_\bS^* \times_{\ft^{*(1)}/W} \ft^* /(W,\bullet)$ by $\ft^* \times_{\ft^{*(1)}/W} \bbI_\bS^* \times_{\ft^{*(1)}/W} \ft^*$) we define the category $\Mod^{\bbI}_{\fin}(\widetilde{\cU}_\bS^{\hat{\lambda},\hat{\mu}})$. Copying the definition of $\hatotimes_{\cU_\bS \fg}$ we obtain, for $\lambda,\mu,\nu \in \bbX$, a bifunctor
\[
 (-) \hatotimes_{\widetilde{\cU}_\bS \fg} (-) : \Mod^{\bbI}_{\fin}(\widetilde{\cU}_\bS^{\hat{\lambda},\hat{\mu}}) \times \Mod^{\bbI}_{\fin}(\widetilde{\cU}_\bS^{\hat{\mu},\hat{\nu}}) \to \Mod^{\bbI}_{\fin}(\widetilde{\cU}_\bS^{\hat{\lambda},\hat{\nu}}).
\]
For any $\lambda,\mu \in \bbX$ we have a natural ``forgetful'' functor
\[
 \Mod^{\bbI}_{\fin}(\widetilde{\cU}_\bS^{\hat{\lambda},\hat{\mu}}) \to \Mod^{\bbI}_{\fin}(\cU_\bS^{\hat{\lambda},\hat{\mu}}),
\]
which we will usually omit from notation. In case $\lambda$ and $\mu$ are regular, this functor is an equivalence by Lemma~\ref{lem:quotient-etale}. In case $\mu$ is regular, for $M \in \Mod^{\bbI}_{\fin}(\widetilde{\cU}_\bS^{\hat{\lambda},\hat{\mu}})$ and $N \in \Mod^{\bbI}_{\fin}(\widetilde{\cU}_\bS^{\hat{\mu},\hat{\nu}})$ we also have a canonical identification
\[
 M \hatotimes_{\cU_\bS\fg} N \simto M \hatotimes_{\widetilde{\cU}_\bS\fg} N.
\]

For $\lambda,\mu \in \bbX$, we will denote by $\sfQ_{\lambda,\mu}$ the completion of the module 
\[
 \Gamma \bigl( \widetilde{\bS}^{*} \times_{\ft^{*(1)}} \ft^*, (\scO_\cB(\lambda-\mu) \otimes_{\scO_\cB} \tD)_{|\widetilde{\bS}^{*} \times_{\ft^{*(1)}} \ft^*} \bigr)
 \]
%$\bbD_{\cS,\lambda-\mu}$ 
at the ideal of $\scO(\ft^* \times_{\ft^{*(1)}/W} \ft^*)$ corresponding to the element $(\ola,\omu)$. In view of the remarks above, this object can equivalently be obtained by completing this module at the ideal of $\scO(\ft^*)$ corresponding to $\ola$ for the left action, or at the ideal of $\scO(\ft^*)$ corresponding to $\omu$ for the right action. This construction provides an object in $\Mod_{\fin}^{\bbI}(\widetilde{\cU}_\bS^{\hat{\lambda},\hat{\mu}})$, hence a fortiori an object in $\Mod_{\fin}^{\bbI}(\cU_\bS^{\hat{\lambda},\hat{\mu}})$.

\begin{lem}
\label{lem:B-transitivity}
For $\lambda,\mu,\nu \in \bbX$, 
%and assume that $\mu$ is regular.
% (i.e.~that its stabilizer for the dot-action of $\Waff$ is trivial). 
 there exists a canonical isomorphism
 \[
  \sfQ_{\lambda,\mu} \hatotimes_{\widetilde{\cU}_\bS \fg} \sfQ_{\mu,\nu} \simto \sfQ_{\lambda,\nu}
 \]
 in $\Mod_{\fin}^{\bbI}(\widetilde{\cU}_\bS^{\hat{\lambda},\hat{\nu}})$. In particular, in case $\mu$ is regular there exists a canonical isomorphism
 \[
  \sfQ_{\lambda,\mu} \hatotimes_{\cU_\bS \fg} \sfQ_{\mu,\nu} \simto \sfQ_{\lambda,\nu}
 \]
 in $\Mod_{\fin}^{\bbI}(\cU_\bS^{\hat{\lambda},\hat{\nu}})$.
\end{lem}

\begin{proof}
 There exist canonical isomorphisms
 \begin{multline*}
  \bigl( \scO_\cB(\lambda-\mu) \otimes_{\scO_\cB} \tD \bigr) \otimes_{\tD} \bigl( \scO_\cB(\mu-\nu) \otimes_{\scO_\cB} \tD \bigr) \simto \\
  \scO_\cB(\lambda-\mu) \otimes_{\scO_\cB} \scO_\cB(\mu-\nu) \otimes_{\scO_\cB} \tD \cong \scO_\cB(\lambda-\nu) \otimes_{\scO_\cB} \tD.
 \end{multline*}
%where the first map is given locally by $f \otimes \partial \otimes f' \otimes \partial' \mapsto f \otimes (\partial \cdot f') \otimes (\partial \partial')$. 
The desired isomorphism follows by restriction to $\widetilde{\bS}^{*} \times_{\ft^{*(1)}} \ft^*$ and then completion at $(\ola,\overline{\nu})$.
%, using Lemma~\ref{lem:US-D} and the fact that the completion of $\cU_\cS\fg$ at the central ideal corresponding to $\tmu \in \ft^*/(W,\bullet)$ identifies with the completion of $\scO(\ft^*) \otimes_{\ZHC} \cU_\cS \fg$ at the central ideal corresponding to $\omu \in \ft^*$ under our assumption, by Lemma~\ref{lem:quotient-etale}.
%as in the proof of Lemma~\ref{lem:translation-bimod-B}.
\end{proof}

If $\lambda \in \bbX$ is regular, Lemma~\ref{lem:quotient-etale} and Lemma~\ref{lem:US-D} imply that we have
\[
\cU_\bS^{\hat{\lambda}} \cong \sfQ_{\lambda,\lambda},
\]
where the left-hand side is as in~\S\ref{ss:HCBim-S}.
Hence the functor of convolution on the left, resp.~right, with $\sfQ_{\lambda,\lambda}$ is isomorphic to the identity functor of $\Mod_{\fin}^{\bbI}(\cU_\bS^{\hat{\lambda},\hat{\mu}})$, resp.~$\Mod_{\fin}^{\bbI}(\cU_\bS^{\hat{\mu},\hat{\lambda}})$, for any $\mu \in \bbX$. Combining this observation with
Lemma~\ref{lem:B-transitivity}, we see that if $\lambda \in \bbX$ belongs to the fundamental alcove, then for any $w \in \Wext$ the object $\sfQ_{\lambda, w \bullet \lambda}$ is invertible in the monoidal category $\Mod_{\fin}^{\bbI}(\cU_\bS^{\hat{\lambda},\hat{\lambda}})$, with inverse $\sfQ_{w \bullet \lambda, \lambda}$.

Recall from~\eqref{eqn:morph-I-T} that we have a canonical morphism of group schemes
\[
 \ft^{*(1)} \times_{\ft^{*(1)}/W} \bbI^*_\bS \to \ft^{*(1)} \times T^{(1)}.
\]
%Restricting to $\cS^{*(1)}$ and then 
Taking the fiber product with the morphism $\ft^* \times_{\ft^{*(1)}/W} \ft^* \to \ft^* \xrightarrow{\mathrm{AS}} \ft^{*(1)}$ (where the first morphism is the first projection) we obtain a morphism of group schemes
\[
 \ft^* \times_{\ft^{*(1)}/W} \bbI_\bS^* \times_{\ft^{*(1)}/W} \ft^* \to (\ft^* \times_{\ft^{*(1)}/W} \ft^*) \times T^{(1)}.
\]
Using this morphism, for any character $\eta$ of $T^{(1)}$ we obtain a structure of representation of $\ft^* \times_{\ft^{*(1)}/W} \bbI_\bS^* \times_{\ft^{*(1)}/W} \ft^*$ on $\scO(\ft^* \times_{\ft^{*(1)}/W} \ft^*)$ defined by this character. Tensoring with this representation we obtain an autoequivalence of $\Mod_{\fin}^{\bbI}(\widetilde{\cU}_\bS^{\hat{\lambda},\hat{\mu}})$, which we denote $M \mapsto M \langle \eta \rangle$.

%Note that the morphism from $X^*(T^{(1)})$ to $\bbX$ induced by the Frobenius morphism $T \to T^{(1)}$ is injective, and that its image is $p \cdot \bbX$. We will therefore identify $X^*(T^{(1)})$ with $p \cdot \bbX$ via this morphism.

Recall from~\S\ref{ss:weights} that we identify the lattice of characters of $T^{(1)}$ with $p \cdot \bbX$.

\begin{lem}
\label{lem:B-translation}
 For any $\lambda,\nu \in \bbX$, there exists a canonical isomorphism
 \[
  \sfQ_{\lambda+p\nu,\lambda} \cong \sfQ_{\lambda,\lambda} \langle p\nu \rangle
 \]
 in $\Mod_{\fin}^{\bbI}(\widetilde{\cU}_\bS^{\hat{\lambda},\hat{\lambda}})$.
\end{lem}

\begin{proof}
By definition, $\sfQ_{\lambda+p\nu,\lambda}$ is the completion at the ideal corresponding to $(\ola,\ola)$ of the $\widetilde{\cU}_\bS\fg \otimes_{\scO(\bS^{*})} (\widetilde{\cU}_\bS\fg)^\op$-module 
\[
 \Gamma \bigl( \widetilde{\bS}^{*} \times_{\ft^{*(1)}} \ft^*, (\scO_\cB(p\nu) \otimes_{\scO_\cB} \tD)_{|\widetilde{\bS}^{*} \times_{\ft^{*(1)}} \ft^*} \bigr).
 \]
%$\bbD_{\cS,p\nu}$. 
If we denote by $U^+$ the unipotent radical of the Borel subgroup opposite to $B$, then $U^+B/B \subset \cB$ is an open subvariety isomorphic to $U^+$, and the projection $\widetilde{\bS}^* \to \cB$ factors through a morphism $\widetilde{\bS}^* \to U^+B/B$, see~\cite[Lemma~4.8]{mr}. As a consequence, the sheaf $(\scO_\cB(p\nu) \otimes_{\scO_\cB} \tD)_{|\widetilde{\bS}^{*} \times_{\ft^{*(1)}} \ft^*}$ can be obtained as a further restriction of $(\scO_\cB(p\nu) \otimes_{\scO_\cB} \tD)_{|U^+B/B}$.

Since $\tD$ acts on $\scO_{\cB}(p\nu)$, we have an action of the algebra $\Ug \otimes_{\ZHC} \scO(\ft^*)$ on $\Gamma(U^+B/B, \scO_\cB(p\nu))$, see~\eqref{eqn:morph-Ug-D}. 
We have
\begin{multline*}
\Gamma \bigl( U^+B/B, \scO_\cB(p\nu) \bigr)=\\
\{f : U^+B \to \bk \mid \forall b \in B, \forall x \in U^+B, \, f(xb^{-1})=(p\nu)(b) \cdot f(x)\}.
\end{multline*}
In this space we have a canonical vector, namely the function $f : U^+B \to \bk$ defined by $f(u_1 t u_2)=(p\nu)^{-1}(t)$ for all $u_1 \in U^+$, $t \in T$ and $u_2 \in U$. This section does not vanish on $U^+B/B$, hence induces an isomorphism of line bundles $\scO_{U^+B/B} \simto \scO_{\cB}(p\nu)_{|U^+B/B}$. We claim that it is furthermore annihilated by the action of $\fg \subset \Ug$ and $\ft \subset \scO(\ft^*)$. In fact, the second case is clear. For the action of $\fg$, in case $\nu \in \bbX^+$ the claim follows from the fact that our vector is the restriction of the unique (up to scalar) vector of weight $p\nu$ in $\Gamma \bigl( \cB, \scO_\cB(p\nu) \bigr)$ (see~\cite[Proof of Proposition~II.2.6]{jantzen}), and that this vector belongs to the $G$-submodule $\Sim(p\nu)$, on which the action of $\fg$ is well known to vanish. From this we deduce the general case by using the Leibniz rule for the action on tensor products of line bundles.

Tensoring this section with the unit in $\tD$ we obtain a section of $(\scO_\cB(p\nu) \otimes_{\scO_\cB} \tD)_{|U^+B/B}$. The right action on this section provides an isomorphism
\[
 \tD_{|U^+B/B} \to (\scO_\cB(p\nu) \otimes_{\scO_\cB} \tD)_{|U^+B/B},
\]
which commutes with the natural left and right actions of $\Ug \otimes_{\ZHC} \scO(\ft^*)$. Restricting further we obtain an isomorphism
\[
 \tD_{|\widetilde{\bS}^{*} \times_{\ft^{*(1)}} \ft^*} \simto (\scO_\cB(p\nu) \otimes_{\scO_\cB} \tD)_{|\widetilde{\bS}^{*} \times_{\ft^{*(1)}} \ft^*},
\]
and then taking global sections and completing an isomorphism of $\widetilde{\cU}_\bS^{\hat{\lambda},\hat{\lambda}}$-modules $\sfQ_{\lambda,\lambda} \simto \sfQ_{\lambda+p\nu,\lambda}$. Taking the action of $\ft^* \times_{\ft^{*(1)}/W} \bbI_\bS^* \times_{\ft^{*(1)}/W} \ft^*$ into account, this provides the desired isomorphism $\sfQ_{\lambda,\lambda} \langle p\nu \rangle \simto \sfQ_{\lambda+p\nu,\lambda}$.
\end{proof}

%-------------------------------------------
\subsection{Relation with translation bimodules}
\label{ss:relation}
%-------------------------------------------

%For $\lambda,\mu \in \bbX$, we will denote by $\bbB_{\lambda,\mu}$ the completion of the module $\bbD_{\cS,\lambda-\mu}$ at the ideal of $\scO(\ft^* \times_{\ft^{*(1)}/W} \ft^*)$ corresponding to the element $(\lambda,\mu)$. In view of the remarks above, this object can equivalently be obtained by completing $\bbD_{\cS,\lambda-\mu}$ at the ideal of $\scO(\ft^*)$ corresponding to $\lambda$ for the left action, or by completing $\bbD_{\cS,\lambda-\mu}$ at the ideal of $\scO(\ft^*)$ corresponding to $\mu$ for the right action. It admits a natural action of the completion of the algebra $(\cU_\cS\fg \otimes_{\ZHC} \scO(\ft^*)) \otimes_{\scO(\cS^{*(1)})} (\cU_\cS\fg^\op \otimes_{\ZHC} \scO(\ft^*))$ at the ideal of $\scO(\ft^* \times_{\ft^{*(1)}/W} \ft^*)$ corresponding to $(\lambda,\mu)$, hence also the structure of an object in $\Mod_{\fin}^{\bbI}(\cU_\cS^{\hat{\lambda},\hat{\mu}})$.

Recall from~\S\ref{ss:HCBim-completed} the notation $\fD=\ft^*/(W,\bullet) \times_{\ft^{*(1)}/W} \ft^*/(W,\bullet)$. 
In our constructions below we will also have to work with variants of this scheme where one of the two copies of $\ft^*/(W,\bullet)$ is replaced by $\ft^*$. We therefore introduce the notation
%Here we will also consider the fiber product
\[
\fE:=\ft^* \times_{\ft^{*(1)}/W} \ft^*/(W,\bullet), \qquad \fE':=\ft^*/(W,\bullet) \times_{\ft^{*(1)}/W} \ft^*.
\]

We now explain the relation between the objects $\sfQ_{\lambda,\mu}$ and the ``translation bimodules'' introduced in~\S\ref{ss:HCBim-completed}.

\begin{lem}
\label{lem:translation-bimod-B}
 Let $\lambda,\mu \in \bbX$ in the closure of the fundamental alcove, with
% $\lambda$ belonging to 
 one of them in
 the fundamental alcove itself.
%  and 
%% $\mu$ belonging to 
% the other one in
% the closure of the fundamental alcove. 
 Then for any $w \in \Wext$ we have
 \[
  \sfP_\bS^{\mu,\lambda} \cong \sfQ_{w \bullet \mu,w \bullet \lambda}
  %, \quad \bbP_\cS^{\lambda,\mu} \cong \sfB_{w \bullet \lambda,w \bullet \mu}
 \]
 in $\Mod_{\fin}^{\bbI}(\cU_\bS^{\hat{\mu},\hat{\lambda}})$.
 % and $\Mod_{\fin}^{\bbI}(\cU_\cS^{\hat{\lambda},\hat{\mu}})$.
\end{lem}

\begin{proof}
To fix notation we assume that $\lambda$ is in the fundamental alcove; the other case 
%We prove the first isomorphism; the second one 
can be obtained similarly. 
It is clear that we can assume that $w \in W$.
Let $\nu \in \bbX$ be the unique dominant weight which belongs to $W(\mu-\lambda)$. Then
 by definition, $\sfP_\bS^{\mu,\lambda}$ is the completion of the module
 \[
  \Sim(\nu) \otimes \cU_\bS\fg
 \]
at the ideal corresponding to the point $(\tmu,\tla) \in \ft^*/(W,\bullet) \times_{\ft^{*(1)}/W} \ft^*/(W,\bullet)$. Now by Lemma~\ref{lem:quotient-etale}
%$\lambda$ has trivial stabilizer for the dot-action of $\Waff$, hence $\ola$ has trivial stabilizer for the dot-action of $W$ (see Lemma~\ref{lem:weights}), and similarly for $w \bullet \ola$ so that 
the quotient morphism $\ft^* \to \ft^*/(W,\bullet)$ is \'etale at $w \bullet \ola$. %(see~\cite[Exp.~V, Proposition~2.2]{sga1}). 
It follows that $\sfP_\bS^{\mu,\lambda}$ can also be obtained as the completion of the $\cU_\bS \fg \otimes_{\scO(\bS^{*})} (\widetilde{\cU}_\bS\fg)^\op$-module
\[
 \Sim(\nu) \otimes \widetilde{\cU}_\bS\fg
\]
with respect to the ideal of $\scO(\fE')$ corresponding to $(\tmu,w \bullet \ola)$. 

By Lemma~\ref{lem:US-D} we have canonical isomorphisms
\[
 \Sim(\nu) \otimes \widetilde{\cU}_\cS\fg \cong \Sim(\nu) \otimes \Gamma \bigl( \widetilde{\bS}^{*} \times_{\ft^{*(1)}} \ft^*, \tD_\bS \bigr) \cong \Gamma(\widetilde{\bS}^{*} \times_{\ft^{*(1)}} \ft^*, \Sim(\nu) \otimes \tD_\bS).
\]
It is a classical fact that the coherent sheaf $\Sim(\nu) \otimes \scO_\cB$ on $\cB$ admits a filtration whose subquotients have the form $\scO_\cB(\eta)$ where $\eta$ runs over the weights of $\Sim(\nu)$ (counted with multiplicities). We deduce a similar filtration for the sheaf $\Sim(\nu) \otimes \tD$, and then for its restriction to $\widetilde{\bS}^{*} \times_{\ft^{*(1)}} \ft^*$. (Here we use the fact that restriction along the closed embedding $\widetilde{\bS}^{*} \hookrightarrow \tbg$ is exact on the category $\QCoh^G(\tbg)$, which follows from the same arguments as those used at the beginning of the proof of Proposition~\ref{prop:Ureg-Azumaya}.)
%since it identifies with the composition of pullback along the flat morphism $G \times \widetilde{\cS}^{*(1)} \to \tfg^{(1)}$ composed with the obvious equivalence $\QCoh^G(G \times \widetilde{\cS}^{*(1)}) \cong \QCoh(\widetilde{\cS}^{*(1)})$. 
In other words, we have obtained a filtration of $\Sim(\nu) \otimes \widetilde{\cU}_\bS\fg$ with subquotients 
%$\bbD_{\cS,\eta}$ 
\begin{equation}
\label{eqn:D-proof}
 \Gamma \bigl( \widetilde{\bS}^{*} \times_{\ft^{*(1)}} \ft^*, (\scO_\cB(\eta) \otimes_{\scO_\cB} \tD)_{|\widetilde{\bS}^{*} \times_{\ft^{*(1)}} \ft^*} \bigr)
\end{equation}
where $\eta$ runs over the weights of $\Sim(\nu)$ (counted with multiplicities). This filtration is clearly compatible with the action of $\cU_\bS \fg \otimes_{\scO(\bS^{*})} (\widetilde{\cU}_\bS\fg)^\op$ and the natural structure of module over the group scheme $\ft^*/(W,\bullet) \times_{\ft^{*(1)}/W} \bbI_\bS^* \times_{\ft^{*(1)}/W} \ft^*$.

Let us denote by $\varpi : \ft^* \to \ft^*/(W,\bullet)$ the quotient morphism.
The irreducible components of $\fE'$ are parametrized by $\ft^*_{\mathbb{F}_p}$, with the component corresponding to $\overline{\gamma}$ being the image of the closed embedding $\ft^* \to \fE'$ given by $\xi \mapsto (\varpi(\xi+\overline{\gamma}),\xi)$. The components containing the point $(\tmu,w \bullet \ola)$ correspond to the elements $\overline{\gamma} \in \ft^*_{\mathbb{F}_p}$ such that $w \bullet \ola+\overline{\gamma} \in W \bullet \omu$, i.e.~$\ola + w^{-1}\overline{\gamma} \in W \bullet \omu$. On the other hand, the module~\eqref{eqn:D-proof} is supported on the component corresponding to $\overline{\eta}$. Hence, after completion at $(\tmu,w \bullet \ola)$, the only subquotients that survive are those corresponding to the weights $\eta$ such that $\ola + w^{-1}\overline{\eta} \in W \bullet \omu$, i.e. $\lambda+w^{-1}\eta \in \Wext \bullet \mu$. Since $w^{-1}\eta$ is a weight of $\Sim(\nu)$, it belongs to $\mu-\lambda+\Z\fR$, so that $\lambda+w^{-1}\eta \in \mu+\Z\fR$. By Lemma~\ref{lem:weights}\eqref{it:lem-no-torsion-1} the condition that $\lambda+w^{-1}\eta \in \Wext \bullet \mu$ is therefore equivalent to $\lambda+w^{-1}\eta \in \Waff \bullet \mu$. Now by~\cite[Lemma~II.7.7]{jantzen} this condition is satisfied only when $\lambda+w^{-1}\eta=\mu$, i.e.~$\eta=w(\mu-\lambda)$. We deduce the desired isomorphism, since $w \bullet \mu-w \bullet \lambda=w(\mu-\lambda)$.
\end{proof}

\begin{rmk}
\label{rmk:translation-bimod-B}
Let $\lambda,\mu \in \bbX$ belonging to the closure of the fundamental alcove, and assume that the stabilizer of $\lambda$ for the dot-action of $\Waff$ is contained in the stabilizer of $\mu$. Then, if we denote by $\scO(\fE)^{\hat{\lambda},\hat{\mu}}$ the completion of $\scO(\fE)$ at the ideal corresponding to $(\ola,\tmu)$, the same considerations as in the proof of Lemma~\ref{lem:translation-bimod-B} show that there exists an isomorphism
\[
\sfQ_{\lambda,\mu} \cong \scO(\fE)^{\hat{\lambda},\hat{\mu}} \otimes_{\cZ_\bS^{\hat{\lambda},\hat{\mu}}} \sfP_\bS^{\lambda,\mu}.
\]
\end{rmk}

Recall that given a simple reflection $s \in \Saff$, a weight $\lambda \in \bbX$ belonging to the closure of the fundamental alcove is said to be on the \emph{wall} corresponding to $s$ if $s \bullet \lambda=\lambda$.

\begin{lem}
\label{lem:transl-bimod-ses}
 Let $\lambda,\mu \in \bbX$, with $\lambda$ belonging to the fundamental alcove and $\mu$ on exactly one wall of the fundamental alcove, attached to the simple reflection $s$. Let also $w \in W$. 
 \begin{enumerate}
 \item
 If $ws \bullet \lambda > w \bullet \lambda$, then there exists an exact sequence
 \[
  \sfQ_{w \bullet \lambda, w \bullet \lambda} \hookrightarrow \sfP_\bS^{\lambda,\mu} \hatotimes_{\cU_\bS\fg} \sfP_\bS^{\mu,\lambda} \twoheadrightarrow \sfQ_{ws \bullet \lambda, w \bullet \lambda}
 \]
in $\Mod_{\fin}^{\bbI}(\cU_\bS^{\hat{\lambda},\hat{\lambda}})$.
\item
If $ws \bullet \lambda < w \bullet \lambda$, then there exists an exact sequence
 \[
  \sfQ_{ws \bullet \lambda, w \bullet \lambda} \hookrightarrow \sfP_\bS^{\lambda,\mu} \hatotimes_{\cU_\bS\fg} \sfP_\bS^{\mu,\lambda} \twoheadrightarrow \sfQ_{w \bullet \lambda, w \bullet \lambda}
 \]
in $\Mod_{\fin}^{\bbI}(\cU_\bS^{\hat{\lambda},\hat{\lambda}})$.
\end{enumerate}
\end{lem}

\begin{proof}
 By Lemma~\ref{lem:translation-bimod-B} we have
 \[
  \sfP_\bS^{\lambda,\mu} \hatotimes_{\cU_\bS\fg} \sfP_\bS^{\mu,\lambda} \cong \sfP_\bS^{\lambda,\mu} \hatotimes_{\cU_\bS\fg} \sfQ_{w \bullet \mu, w \bullet \lambda}.
 \]
Hence, if we denote by $\nu$ the unique dominant weight in $W(\lambda-\mu)$, this object can be obtained by completing the bimodule
\begin{multline*}
 \Sim(\nu) \otimes \Gamma \bigl( \widetilde{\bS}^{*} \times_{\ft^{*(1)}} \ft^*, (\scO_\cB(w \bullet \mu - w \bullet \lambda) \otimes_{\scO_\cB} \tD)_{|\widetilde{\bS}^{*} \times_{\ft^{*(1)}} \ft^*} \bigr) \cong \\
 \Gamma \bigl( \widetilde{\bS}^{*} \times_{\ft^{*(1)}} \ft^*, \Sim(\nu) \otimes (\scO_\cB(w \bullet \mu - w \bullet \lambda) \otimes_{\scO_\cB} \tD)_{|\widetilde{\bS}^{*} \times_{\ft^{*(1)}} \ft^*} \bigr)
\end{multline*}
with respect to the ideal of $\scO(\fE')$ corresponding to $(\tla, w \bullet \ola)$. 
%Now we have
%\[
% \Sim(\nu) \otimes \bbD_{\cS,w \bullet \mu - w \bullet \lambda} \cong \Gamma \bigl( \widetilde{\cS}^{*(1)} \times_{\ft^{*(1)}} \ft^*, \Sim(\nu) \otimes (\scO_\cB(w \bullet \mu - w \bullet \lambda) \otimes_{\scO_\cB} \tD)_{|\widetilde{\cS}^{*(1)} \times_{\ft^{*(1)}} \ft^*} \bigr).
%\]
Hence, as in the proof of Lemma~\ref{lem:translation-bimod-B}, if we choose an enumeration $\eta_1, \cdots, \eta_n$ of the $T$-weights of $\Sim(\nu)$ (counted with multiplicities) such that $\eta_i < \eta_j$ implies $i<j$, then this bimodule admits a filtration
\begin{multline*}
 \{0\}=M_0 \subset M_1 \subset \cdots \\
 \subset M_n = \Gamma \bigl( \widetilde{\bS}^{*} \times_{\ft^{*(1)}} \ft^*, (\scO_\cB(w \bullet \mu - w \bullet \lambda) \otimes_{\scO_\cB} \tD)_{|\widetilde{\bS}^{*} \times_{\ft^{*(1)}} \ft^*} \bigr)
\end{multline*}
such that
\[
M_i/M_{i-1} \cong \Gamma \bigl( \widetilde{\bS}^{*} \times_{\ft^{*(1)}} \ft^*, (\scO_\cB(w \bullet \mu - w \bullet \lambda+ \eta_i) \otimes_{\scO_\cB} \tD)_{|\widetilde{\bS}^{*} \times_{\ft^{*(1)}} \ft^*} \bigr)
\]
for any $i$. The subquotient $M_i/M_{i-1}$ survives after completion at the ideal corresponding to $(\tla, w \bullet \ola)$ iff
\[
 w \bullet \omu - w \bullet \ola + \overline{\eta_i} \in W \bullet \ola - w \bullet \ola,
\]
i.e.~iff
\[
 \mu + w^{-1}\eta_i \in \Wext \bullet \lambda.
\]
Here $w^{-1} \eta_i$ is a weight of $\Sim(\nu)$, hence $\mu + w^{-1} \eta_i$ belongs to $\lambda+\Z\fR$; in view of Lemma~\ref{lem:weights}\eqref{it:lem-no-torsion-1}, this condition is therefore equivalent to $\mu + w^{-1}\eta_i \in \Waff \bullet \lambda$. Since the stabilizer of $\mu$ for the dot-action of $\Waff$ is $\{e,s\}$, by~\cite[Lemma~II.7.7]{jantzen} this condition is satisfied for two values of $\eta_i$, corresponding to
\[
 \mu + w^{-1}\eta_i = \lambda \quad \text{and} \quad \mu + w^{-1}\eta_i=s \bullet \lambda,
\]
i.e.
\[
 w \bullet \mu + \eta_i = w \bullet \lambda \quad \text{and} \quad w \bullet \mu + \eta_i=ws \bullet \lambda.
\]
Hence $\sfP_\bS^{\lambda,\mu} \hatotimes_{\cU_\bS\fg} \sfP_\bS^{\mu,\lambda}$ admits a two-step filtration, with subquotients isomorphic respectively to $\sfQ_{w \bullet \lambda, w \bullet \lambda}$ and $\sfQ_{ws \bullet \lambda, w \bullet \lambda}$. The order in which these subquotients appear depends on wether $ws \bullet \lambda > w \bullet \lambda$ or $ws \bullet \lambda < w \bullet \lambda$, and are as indicated in the statement.
\end{proof}

%-----------------------------------------
\subsection{Convolution with translation bimodules}
\label{ss:convolution-trans-bimod}
%-----------------------------------------

%Recall from~\S\ref{ss:HCBim-completed} the notation $\fD=\ft^*/(W,\bullet) \times_{\ft^{*(1)}/W} \ft^*/(W,\bullet)$. Here we will also consider the fiber product
%\[
%\fE:=\ft^* \times_{\ft^{*(1)}/W} \ft^*/(W,\bullet).
%\]

Let $\lambda,\mu \in \bbX$, and assume that $\lambda$ is regular. Then there exists a canonical algebra morphism
\begin{equation}
\label{eqn:morph-completions}
 \cZ_\bS^{\hat{\mu},\hat{\lambda}} \to \cZ_\bS^{\hat{\lambda},\hat{\lambda}}
\end{equation}
which can be defined as follows. The algebra $\cZ_\bS^{\hat{\mu},\hat{\lambda}}$ is by definition the completion of $\scO(\fD)$ at the ideal corresponding to $(\tmu,\tla)$. Hence it admits a canonical morphism to the completion $\scO(\fE)^{\hat{\mu},\hat{\lambda}}$ of $\scO(\fE)$ at the ideal corresponding to $(\omu,\tla)$. Now the morphism $\ft^* \to \ft^*$ defined by $\xi \mapsto \xi+\omu-\ola$ induces an automorphism of
 $\fE$
sending $(\ola,\tla)$ to $(\omu,\tla)$, which therefore induces an isomorphism
\begin{equation}
\label{eqn:isom-completions-1}
 \scO(\fE)^{\hat{\mu},\hat{\lambda}} \simto \scO(\fE)^{\hat{\lambda},\hat{\lambda}},
\end{equation}
where the right-hand side is the completion of $\scO(\fE)$ at the ideal corresponding to $(\ola,\tla)$. Finally, the natural morphism
\begin{equation}
\label{eqn:isom-completions-2}
\cZ_\bS^{\hat{\lambda},\hat{\lambda}} \to \scO(\fE)^{\hat{\lambda},\hat{\lambda}}
\end{equation}
is an isomorphism by Lemma~\ref{lem:quotient-etale}, since $\lambda$ is regular; combining these constructions we obtain the wished-for morphism~\eqref{eqn:morph-completions}.

%hence such that the induced morphism
%\[
% \scO(\ft^* \times_{\ft^{*(1)}/W} \ft^*/(W,\bullet)) \to \scO(\ft^* \times_{\ft^{*(1)}/W} \ft^*/(W,\bullet))
%\]
%sends the ideal corresponding to $(\omu,\tla)$ to that corresponding to $(\ola,\tla)$, hence induces a morphism between the corresponding completions. Finally, the completion of $\scO(\ft^* \times_{\ft^{*(1)}/W} \ft^*/(W,\bullet))$ at the ideal corresponding to $(\ola,\tla)$ identifies with $\cZ_\cS^{\hat{\lambda},\hat{\lambda}}$ by Lemma~\ref{lem:quotient-etale}; this defines the desired morphism.

Our goal in this subsection is to prove the following claim, which involves the equivalence constructed in Corollary~\ref{cor:equiv-ModUg-ModZ}.

\begin{prop}
\label{prop:localization-wall-crossing}
%Assume that $p \neq 2$.
 Let $\lambda,\mu \in \bbX$, with $\lambda$ belonging to the fundamental alcove and $\mu$ on exactly one wall of the fundamental alcove, attached to a simple reflection $s$ which belongs to $W$. Then there exists an isomorphism
 \[
  \sfP_\bS^{\lambda,\mu} \hatotimes_{\cU_\bS\fg} \sfP_\bS^{\mu,\lambda} \cong \mathscr{L}_{\lambda,\lambda} \left( \cZ_\bS^{\hat{\lambda},\hat{\lambda}} \otimes_{\cZ_\bS^{\hat{\mu},\hat{\lambda}}} \cZ_\bS^{\hat{\lambda}} \right)
 \]
 in $\Mod^{\bbI}_{\fin}(\cU_\bS^{\hat{\lambda},\hat{\lambda}})$, where $\cZ_\bS^{\hat{\lambda}}$ is regarded as a $\cZ_\bS^{\hat{\mu},\hat{\lambda}}$-module via the morphism~\eqref{eqn:morph-completions} and $\cZ_\bS^{\hat{\lambda},\hat{\lambda}} \otimes_{\cZ_\bS^{\hat{\mu},\hat{\lambda}}} \cZ_\bS^{\hat{\lambda}}$ is endowed with the trivial structure as a representation.
\end{prop}

\begin{rmk}
\label{rmk:canonicity-isom}
 From the proof below one can check that the isomorphism in Proposition~\ref{prop:localization-wall-crossing} is ``canonical'' in that it depends only on the choice of an adjunction $(\sfP_\bS^{\lambda,\mu} \hatotimes_{\cU_\bS\fg} (-), \sfP_\bS^{\mu,\lambda} \hatotimes_{\cU_\bS\fg} (-))$, which can be defined by a choice of an isomorphism $\Sim(\nu)^* \cong \Sim(-w_0(\nu))$ where $\nu$ is the only $W$-translate of $\mu-\lambda$; see the proof of Lemma~\ref{lem:convolution-adjoint}. From this proof it is clear also that the morphism
 \[
  \sfP_\bS^{\lambda,\mu} \hatotimes_{\cU_\bS\fg} \sfP_\bS^{\mu,\lambda} \to \cU_\bS^{\hat{\lambda}}
 \]
defined by this adjunction corresponds under $\mathscr{L}_{\lambda,\lambda}$ to the morphism
\[
 \cZ_\bS^{\hat{\lambda},\hat{\lambda}} \otimes_{\cZ_\bS^{\hat{\mu},\hat{\lambda}}} \cZ_\bS^{\hat{\lambda}} \to \cZ_\bS^{\hat{\lambda}}
\]
given by the action of $\cZ_\bS^{\hat{\lambda},\hat{\lambda}}$ on $\cZ_\bS^{\hat{\lambda}}$.
\end{rmk}

The proof of this proposition will use two preliminary lemmas.

\begin{lem}
\label{lem:freeness}
If $s$ is a simple reflection in $W$, then
$\scO(\ft^*)$ is free of rank $2$ as a module over $\scO(\ft^*/(\{e,s\},\bullet))$.
\end{lem}

\begin{proof}
First, translating by $\rho$ we can reduce the $\bullet$-action to the standard action; it therefore suffices to prove that $\scO(\ft^*)$ is free of rank two over the subalgebra $\scO(\ft^*)^s$ of $s$-invariants. Next, recall that we have a $W$-equivariant isomorphism $\ft \simto \ft^*$, induced by a choice of $G$-equivariant isomorphism $\fg \simto \fg^*$. We are therefore reduced to proving that $\scO(\ft)$ is free of rank two over $\scO(\ft)^s$. Now, standard arguments show that $(1,\overline{\rho})$ is a basis of $\scO(\ft)$ over $\scO(\ft)^s$; see e.g.~\cite[Claim~3.11]{ew}.
\end{proof}

\begin{lem}
\label{lem:wall-crossing-diag}
%Assume that $p \neq 2$.
Let $\lambda,\mu \in \bbX$, with $\lambda$ belonging to the fundamental alcove and $\mu$ on exactly one wall of the fundamental alcove, attached to a simple reflection $s$ which belongs to $W$. Then there exist isomorphisms of functors which make the following diagrams commutative, where the upper horizontal arrow in the left-hand side is the restriction-of-scalars functor associated with the morphism~\eqref{eqn:morph-completions}:
 \[
\xymatrix@C=1.8cm{
\Mod_{\fin}^{\bbI}(\cZ_\bS^{\hat{\lambda},\hat{\lambda}}) \ar[d]_-{\mathscr{L}_{\lambda, \lambda}} \ar[r] & \Mod_{\fin}^{\bbI}(\cZ_\bS^{\hat{\mu},\hat{\lambda}}) \ar[d]^-{\mathscr{L}_{\mu,\lambda}} \\
\Mod_{\fin}^{\bbI}(\cU_\bS^{\hat{\lambda},\hat{\lambda}}) \ar[r]^-{\sfP_\bS^{\mu,\lambda} \hatotimes_{\cU_\bS\fg}(-)} & \Mod_{\fin}^{\bbI}(\cU_\bS^{\hat{\mu},\hat{\lambda}}),
}
\quad
%\]
%(where the upper horizontal arrow is the restriction-of-scalars functor associated with the morphism~\eqref{eqn:morph-completions})
%and
% \[
\xymatrix@C=1.8cm{
\Mod_{\fin}^{\bbI}(\cZ_\cS^{\hat{\mu},\hat{\lambda}}) \ar[d]_-{\mathscr{L}_{\mu,\lambda}} \ar[r]^-{\cZ_\bS^{\hat{\lambda},\hat{\lambda}} \otimes_{\cZ_\bS^{\hat{\mu},\hat{\lambda}}} (-)} & \Mod_{\fin}^{\bbI}(\cZ_\bS^{\hat{\lambda},\hat{\lambda}}) \ar[d]^-{\mathscr{L}_{\lambda,\lambda}} \\
\Mod_{\fin}^{\bbI}(\cU_\bS^{\hat{\mu},\hat{\lambda}}) \ar[r]^-{\sfP_\bS^{\lambda,\mu} \hatotimes_{\cU_\bS\fg}(-)} & \Mod_{\fin}^{\bbI}(\cU_\bS^{\hat{\lambda},\hat{\lambda}}).
}
\]
%commutative.
\end{lem}

\begin{proof}
By definition we have
\[
\sfP_\bS^{\mu,\lambda} \hatotimes_{\cU_\bS\fg} \mathscr{L}_{\lambda,\lambda}(-) \cong \left( \sfP_\bS^{\mu,\lambda} \hatotimes_{\cU_\bS\fg} \sfP_\bS^{\lambda,-\rho} \hatotimes_{\cU_\bS\fg} \sfP_\bS^{-\rho,\lambda} \right) \otimes_{\cZ_\bS^{\hat{\lambda},\hat{\lambda}}} (-).
\]
Using Lemma~\ref{lem:B-transitivity} and Lemma~\ref{lem:translation-bimod-B}, we deduce that
\[
\sfP_\bS^{\mu,\lambda} \hatotimes_{\cU_\bS\fg} \mathscr{L}_{\lambda,\lambda}(-) \cong \left( \sfQ_{\mu,-\rho} \hatotimes_{\cU_\bS\fg} \sfP_\bS^{-\rho,\lambda} \right) \otimes_{\cZ_\bS^{\hat{\lambda},\hat{\lambda}}} (-).
\]
Hence to prove the commutativity of the diagram of the left it suffices to construct an isomorphism
\[
\left( \sfP_\bS^{\mu,-\rho} \hatotimes_{\cU_\bS\fg} \sfP_\bS^{-\rho,\lambda} \right) \otimes_{\cZ_\bS^{\hat{\mu},\hat{\lambda}}} \cZ_\bS^{\hat{\lambda},\hat{\lambda}} \simto \sfQ_{\mu,-\rho} \hatotimes_{\cU_\bS\fg} \sfP_\bS^{-\rho,\lambda}
\]
or in other words (in view of~\eqref{eqn:isom-completions-1} and~\eqref{eqn:isom-completions-2}) an isomorphism
\begin{equation}
\label{eqn:morph-commutative-diag-1}
\left( \sfP_\bS^{\mu,-\rho} \hatotimes_{\cU_\bS\fg} \sfP_\bS^{-\rho,\lambda} \right) \otimes_{\cZ_\bS^{\hat{\mu},\hat{\lambda}}} \scO(\fE)^{\hat{\mu},\hat{\lambda}} \simto \sfQ_{\mu,-\rho} \hatotimes_{\cU_\bS\fg} \sfP_\bS^{-\rho,\lambda}.
\end{equation}

To construct a morphism as in~\eqref{eqn:morph-commutative-diag-1} it suffices to construct a morphism
\begin{equation}
\label{eqn:morph-commutative-diag-2}
\sfP_\bS^{\mu,-\rho} \hatotimes_{\cU_\bS\fg} \sfP_\bS^{-\rho,\lambda} \to \sfQ_{\mu,-\rho} \hatotimes_{\cU_\bS\fg} \sfP_\bS^{-\rho,\lambda}
\end{equation}
in $\Mod_{\fin}^{\bbI}(\cU_\bS^{\hat{\mu},\hat{\lambda}})$. By Remark~\ref{rmk:translation-bimod-B} we have
\begin{equation}
\label{eqn:morph-commutative-diag-3}
\sfQ_{\mu,-\rho} \cong \scO(\fE)^{\hat{\mu},\hat{-\rho}} \otimes_{\cZ_\bS^{\hat{\mu},\hat{-\rho}}} \sfP_\bS^{\mu,-\rho};
\end{equation}
in particular there exists a natural morphism $\sfP_\bS^{\mu,-\rho} \to \sfQ_{\mu,-\rho}$, which allows to define the wished-for morphism~\eqref{eqn:morph-commutative-diag-2}, hence the morphism~\eqref{eqn:morph-commutative-diag-1}. 

Now we claim that $\scO(\fE)^{\hat{\mu},\hat{\lambda}}$, resp.~$\scO(\fE)^{\hat{\mu},\hat{-\rho}}$, is free of rank $2$ over the algebra $\cZ_\bS^{\hat{\mu},\hat{\lambda}}$, resp.~$\cZ_\bS^{\hat{\mu},\hat{-\rho}}$, which in view of~\eqref{eqn:morph-commutative-diag-3} will imply that the morphism~\eqref{eqn:morph-commutative-diag-1} is an isomorphism. The two cases are similar, so that we only consider $\scO(\fE)^{\hat{\mu},\hat{\lambda}}$. 
% By Lemma~\ref{lem:weights}, the stabilizer of $\omu$ for the dot-action of $W$ is $\{e,s\}$; by~\cite[Exp.~V, Proposition~2.2]{sga1} this implies that the natural morphism
% \[
% \ft^*/(\{e,s\},\bullet) \to \ft^*/(W,\bullet)
% \]
% is \'etale at the image of $\omu$, so 
It follows from Lemma~\ref{lem:quotient-etale}
that $\cZ_\bS^{\hat{\mu},\hat{\lambda}}$ identifies with the completion
\[
\scO(\ft^*/(\{e,s\},\bullet) \times_{\ft^{*(1)}/W} \ft^*/(W,\bullet))^{\hat{\mu},\hat{\lambda}}
\]
of $\scO(\ft^*/(\{e,s\},\bullet) \times_{\ft^{*(1)}/W} \ft^*/(W,\bullet))$ with respect to the ideal corresponding to the image of $(\omu,\tla)$. Now $\scO(\fE)$ is free of rank $2$ over $\scO(\ft^*/(\{e,s\},\bullet) \times_{\ft^{*(1)}/W} \ft^*/(W,\bullet))$ by Lemma~\ref{lem:freeness}, and its completion with respect to the ideal corresponding to $(\omu,\tla)$ coincides with its completion with respect to the ideal of $\scO(\ft^*/(\{e,s\},\bullet) \times_{\ft^{*(1)}/W} \ft^*/(W,\bullet))$ corresponding to the image of $(\omu,\tla)$ (because $(\omu,\tla)$ is the only closed point in the fiber over its image in $\ft^*/(\{e,s\},\bullet) \times_{\ft^{*(1)}/W} \ft^*/(W,\bullet)$). The desired claim follows.

We have finally proved the commutativity of the left diagram of the lemma. The commutativity of the right diagram follows from that of the left one by adjunction, in view of Lemma~\ref{lem:convolution-adjoint}.
\end{proof}

\begin{proof}[Proof of Proposition~\ref{prop:localization-wall-crossing}]
Lemma~\ref{lem:wall-crossing-diag} provides isomorphisms
\[
\sfP_\bS^{\lambda,\mu} \hatotimes_{\cU_\bS\fg} \sfP_\bS^{\mu,\lambda} \cong \mathscr{L}_{\lambda,\lambda} (\cZ_\bS^{\hat{\lambda},\hat{\lambda}} \otimes_{\cZ_\bS^{\hat{\mu},\hat{\lambda}}} \mathscr{L}_{\mu,\lambda}^{-1} (\sfP_\bS^{\mu,\lambda})) \cong \mathscr{L}_{\lambda,\lambda} (\cZ_\bS^{\hat{\lambda},\hat{\lambda}} \otimes_{\cZ_\bS^{\hat{\mu},\hat{\lambda}}} \mathscr{L}_{\lambda,\lambda}^{-1} (\cU_\bS^{\hat{\lambda}})).
\]
The desired claim follows, using the isomorphism~\eqref{eqn:L-units}.
\end{proof}

%----------------------------------------
\subsection{Monoidality of the functors \texorpdfstring{$\mathscr{L}_{\lambda,\lambda}$}{L}}
\label{ss:monoidality-L}
%----------------------------------------

Our goal in this subsection is to prove the following claim, announced in Remark~\ref{rmk:L-monoidal}.

\begin{prop}
\label{prop:L-monoidal}
 Let $\lambda,\nu \in \bbX$ in the lower closure of the fundamental alcove, and let $\mu \in \bbX$ be in the fundamental alcove. Then for $M \in \Mod_{\fin}^{\bbI}(\cZ_\bS^{\hat{\lambda},\hat{\mu}})$ and $N \in \Mod_{\fin}^{\bbI}(\cZ_\bS^{\hat{\mu},\hat{\nu}})$ there exists a canonical (in particular, bifunctorial) isomorphism
 \[
  \mathscr{L}_{\lambda,\nu}(M \hatstar_\bS N) \cong \mathscr{L}_{\lambda,\mu}(M) \hatotimes_{\cU_\bS\fg} \mathscr{L}_{\mu,\nu}(N).
 \]
In case $\lambda=\mu=\nu$, this isomorphism and~\eqref{eqn:L-units} define on $\mathscr{L}_{\lambda,\lambda}$ the structure of a monoidal functor.
\end{prop}

\begin{proof}
Recall the completion $\scO(\ft^{*(1)}/W)^\wedge$ introduced in~\S\ref{ss:comparison}.
 By definition and Remark~\ref{rmk:completion-tensor-product-ZHC}, if we set $\cU_\bS\fg^\wedge :=\cU_\bS \fg \otimes_{\scO(\ft^{*(1)}/W)} \scO(\ft^{*(1)}/W)^\wedge$ we have
 \[
  \mathscr{L}_{\lambda,\mu}(M) = \sfM_\bS^{\lambda,\mu} \otimes_{\cZ_\bS^{\hat{\lambda},\hat{\mu}}} M = \left( \sfP_\bS^{\lambda,-\rho} \otimes_{\cU_\bS\fg^\wedge} \sfP_\bS^{-\rho,\mu} \right) \otimes_{\ZHC^{\hat{\lambda}} \otimes_{\scO(\ft^{*(1)}/W)^\wedge} \ZHC^{\hat{\mu}}} M
 \]
and
\[
  \mathscr{L}_{\mu,\nu}(N) = \sfM_\bS^{\mu,\nu} \otimes_{\cZ_\bS^{\hat{\mu},\hat{\nu}}} M = \left( \sfP_\bS^{\mu,-\rho} \otimes_{\cU_\bS\fg^\wedge} \sfP_\bS^{-\rho,\nu} \right) \otimes_{\ZHC^{\hat{\mu}} \otimes_{\scO(\ft^{*(1)}/W)^\wedge} \ZHC^{\hat{\nu}}} N.
 \]
 We deduce that
 \begin{multline*}
 \mathscr{L}_{\lambda,\mu}(M) \hatotimes_{\cU_\bS\fg} \mathscr{L}_{\mu,\nu}(N) \cong \\
 \bigl( \sfP_\bS^{\lambda,-\rho} \otimes_{\cU_\bS\fg^\wedge} \sfP_\bS^{-\rho,\mu} \otimes_{\cU_\bS\fg^\wedge} \sfP_\bS^{\mu,-\rho} \otimes_{\cU_\bS\fg^\wedge} \sfP_\bS^{-\rho,\nu}\bigr) \\
 \otimes_{\ZHC^{\hat{\lambda}} \otimes_{\scO(\ft^{*(1)}/W)^\wedge} \ZHC^{\hat{\mu}} \otimes_{\scO(\ft^{*(1)}/W)^\wedge} \ZHC^{\hat{\nu}}} (M \otimes_{\ZHC^{\hat{\mu}}} N).
  \end{multline*}
  (Here, $\ZHC^{\hat{\lambda}} \otimes_{\scO(\ft^{*(1)}/W)^\wedge} \ZHC^{\hat{\mu}} \otimes_{\scO(\ft^{*(1)}/W)^\wedge} \ZHC^{\hat{\nu}}$ identifies with the completion of $\scO(\ft^*/(W,\bullet) \times_{\ft^{*(1)}/W} \ft^*/(W,\bullet) \times_{\ft^{*(1)}/W} \ft^*/(W,\bullet))$ with respect to the ideal corresponding to $(\tla,\tmu,\widetilde{\nu})$.)
%
% By definition we have
% \[
%  \mathscr{L}_{\lambda,\mu}(M) = \bbM_\cS^{\lambda,\mu} \otimes_{\cZ_\cS^{\hat{\lambda},\hat{\mu}}} M = \left( \bbP_\cS^{\lambda,-\rho} \hatotimes_{\cU_\cS\fg} \bbP_\cS^{-\rho,\mu} \right) \otimes_{\cZ_\cS^{\hat{\lambda},\hat{\mu}}} M
% \]
%and
%\[
%  \mathscr{L}_{\mu,\nu}(N) = \bbM_\cS^{\mu,\nu} \otimes_{\cZ_\cS^{\hat{\mu},\hat{\nu}}} N = \left( \bbP_\cS^{\mu,-\rho} \hatotimes_{\cU_\cS\fg} \bbP_\cS^{-\rho,\nu} \right) \otimes_{\cZ_\cS^{\hat{\mu},\hat{\nu}}} N.
% \]
% The object $\mathscr{L}_{\lambda,\mu}(M) \hatotimes_{\cU_\cS\fg} \mathscr{L}_{\mu,\nu}(N)$ is therefore the projective limit (over $n$) of the objects
% \begin{multline*}
%  \Bigl( \bbP_\cS^{\lambda,-\rho}/(\cI^n \cdot \bbP_\cS^{\lambda,-\rho}) \otimes_{\cU_\cS\fg} \bbP_\cS^{-\rho,\mu} / (\cI^n \cdot \bbP_\cS^{-\rho,\mu}) \otimes_{\cU_\cS\fg} \bbP_\cS^{\mu,-\rho} / (\cI^n \cdot \bbP_\cS^{\mu,-\rho}) \\ \otimes_{\cU_\cS\fg} \bbP_\cS^{-\rho,\nu} / (\cI^n \cdot \bbP_\cS^{-\rho,\nu})  \Bigr) \otimes_{\scO(\ft^*/(W,\bullet) \times_{\ft^{*(1)}/W} \ft^*/(W,\bullet) \times_{\ft^{*(1)}/W} \ft^*/(W,\bullet))} \\
%  \bigl( M/(\cI^n \cdot M) \otimes_{\scO(\ft^*/(W,\bullet))} N/(\cI^n \cdot N) \bigr).
% \end{multline*}
By Lemma~\ref{lem:translation-bimod-B} and Lemma~\ref{lem:B-transitivity} we have
\[
 \sfP_\bS^{-\rho,\mu} \hatotimes_{\cU_\bS\fg} \sfP_\bS^{\mu,-\rho} \cong \sfQ_{-\rho,\mu} \hatotimes_{\cU_\bS\fg} \sfQ_{\mu,-\rho} \cong \sfQ_{-\rho,-\rho}.
\]
By Lemma~\ref{lem:US-D}, $\sfQ_{-\rho,-\rho}$ identifies with the completion of $\widetilde{\cU}_\bS\fg$ with respect to the ideal of $\scO(\ft^*)$ corresponding to $\overline{-\rho}$. Via this identification, the action of $\ZHC^{\hat{\mu}}$ is given by the composition
\[
\ZHC^{\hat{\mu}} \to \scO(\ft^*)^{\hat{\mu}} \simto \scO(\ft^*)^{\hat{-\rho}}
\]
where $\scO(\ft^*)^{\hat{\mu}}$ and $\scO(\ft^*)^{\hat{-\rho}}$ are the completions of $\scO(\ft^*)$ with respect to the ideals corresponding to $\overline{\lambda}$ and $\overline{-\rho}$ respectively, the first map is induced by the embedding $\ZHC \cong \scO(\ft^*/(W,\bullet)) \to \scO(\ft^*)$, and the second one is defined in terms of translation, in a way similar to that considered at the beginning of~\S\ref{ss:convolution-trans-bimod}. Here the first map is an isomorphism by Lemma~\ref{lem:quotient-etale}, from which we obtain an isomorphism
 \begin{multline*}
 \mathscr{L}_{\lambda,\mu}(M) \hatotimes_{\cU_\bS\fg} \mathscr{L}_{\mu,\nu}(N) \cong \\
 \bigl( \sfP_\bS^{\lambda,-\rho} \otimes_{\cU_\bS\fg^\wedge} \sfP_\bS^{-\rho,\nu}\bigr)
 \otimes_{\ZHC^{\hat{\lambda}} \otimes_{\scO(\ft^{*(1)}/W)^\wedge} \ZHC^{\hat{-\rho}} \otimes_{\scO(\ft^{*(1)}/W)^\wedge} \ZHC^{\hat{\nu}}} (M \otimes_{\ZHC^{\hat{\mu}}} N).
  \end{multline*}
  Finally we use the fact that the morphism $\scO(\ft^{*(1)}/W)^\wedge \to \ZHC^{\hat{-\rho}}$ is an isomorphism (see the proof of Lemma~\ref{lem:morph-etale}) to deduce the wished-for isomorphism
   \[
 \mathscr{L}_{\lambda,\mu}(M) \hatotimes_{\cU_\bS\fg} \mathscr{L}_{\mu,\nu}(N) \cong   \mathscr{L}_{\lambda,\nu}(M \hatstar_\bS N).
 \]

  In case $\lambda=\mu=\nu$, the fact that the relevant isomorphisms define a monoidal structure on $\mathscr{L}_{\lambda,\lambda}$ is clear from constructions.
\end{proof}

\begin{rmk}
 Proposition~\ref{prop:L-monoidal} also holds in case $\mu$ is singular (in the lower closure of the fundamental alcove). This case can be treated using the methods of~\S\ref{ss:Ug-D-sing} below; since it is not needed in this paper, we omit the details.
\end{rmk}

%----------------------------------------
\subsection{Singular analogues}
\label{ss:Ug-D-sing}
%----------------------------------------

%Below we will also need a ``singular analogue'' of the considerations of~\S\ref{ss:Ug-D-reg}. Namely, l
Let $I \subset \fRs$ be a subset, and let $P_I \subset G$ be the associated standard (i.e., containing $B$) parabolic subgroup of $G$. (In practice, only the case $\#I=1$ will be considered below.) Let $U_I \subset P_I$ be the unipotent radical of $P_I$, and let $L_I$ be the Levi factor containing $T$, so that $P_I \cong L_I \ltimes U_I$. Let $\cP_I:=G/P_I$, and consider the natural projection
\[
 \omega_I : G/U_I \to \cP_I.
\]
The group $L_I$ acts naturally on $G/U_I$ on the right, via the action induced by multiplication on the right on $G$; this action makes $\omega_I$ a (Zariski locally trivial) $L_I$-torsor. We set
\[
 \tD_I := (\omega_I)_* (\mathscr{D}_{G/U_I})^{L_I},
\]
where the exponent means $L_I$-invariants. The actions of $G$ and $L_I$ on $G/U_I$ induce a canonical algebra morphism
\begin{equation}
\label{eqn:morph-Ug-D-I}
 \Ug \otimes_{\ZHC} \scO(\ft^*/(W_I,\bullet)) \to \Gamma(\cP_I,\tD_I),
\end{equation}
see~\cite[Proposition~1.2.3]{bmr2}.

Let also $\bP_I:=P_I^{(1)}$, a parabolic subgroup of $\bG=G^{(1)}$ with unipotent radical $\bU_I:=U_I^{(1)}$.
Let $\tbg_I$ be the parabolic Grothendieck resolution (for the group $\bG$) associated with $I$, defined as
\[
 \tbg_I := \bG \times^{\bP_I} (\bg/\Lie(\bU_I))^*.
\]
Here $\tbg_I$ is a vector bundle over $\bG/\bP_I = \cP_I^{(1)}$, and if $\bL_I:=L_I^{(1)}$ there is a natural morphism
\[
\tbg_I \to \Lie(\bL_I)^*/\bL_I \cong \ft^{*(1)}/W_I,
\]
where $W_I \subset W$ is as in~\S\ref{ss:weights} (seen here as the Weyl group of $(\bL_I,\bT)$). 
%(Here we use a version of the Chevalley isomorphism for a group which is not semisimple; a reference for the most general version of this isomorphism, which covers the case considered here, is~\cite[\S 4.1]{bc}.) 

Consider the induced morphism $f_I : \tbg_I \times_{\ft^{*(1)}/W_I} \ft^*/(W_I,\bullet) \to \cP_I^{(1)}$, and the Frobenius morphism $\Fr_{\cP_I} : \cP_I \to \cP_I^{(1)}$. As explained in~\cite[\S 1.2.1]{bmr2}, there exists a canonical algebra morphism
\[
 (f_I)_* \scO_{\tbg_I \times_{\ft^{*(1)}/W_I} \ft^*/(W_I,\bullet)} \to (\Fr_{\cP_I})_* \tD_I,
\]
where the morphism $\ft^*/(W_I,\bullet) \to \ft^{*(1)}/W_I$ is induced by the Artin--Schreier map. This morphism takes values in the center of $(\Fr_{\cP_I})_* \tD_I$, and makes $(\Fr_{\cP_I})_* \tD_I$ a locally finitely generated $(f_I)_* \scO_{\tbg_I \times_{\ft^{*(1)}/W_I} \ft^*/(W_I,\bullet)}$-module. Since all the morphisms involved in this construction are affine, using this morphism one can consider $\tD_I$ as a coherent sheaf of $\scO_{\tbg_I \times_{\ft^{*(1)}/W_I} \ft^*/(W_I,\bullet)}$-algebras on $\tbg_I \times_{\ft^{*(1)}/W_I} \ft^*/(W_I,\bullet)$. (We will not introduce a different notation for this sheaf of algebras.)

We also have a canonical morphism $\tbg_I \to \fg^{*(1)}$, and we denote by $\widetilde{\bS}^*_I$ the (scheme-theoretic) inverse image of $\bS^*$ under this morphism. As in the case $I=\varnothing$, using~\cite[Remark~3.5.4]{riche-kostant} one can check that the morphism $\tbg_I \to \ft^{*(1)}/W_I$ considered above restricts to an isomorphism $\widetilde{\bS}^*_I \simto \ft^{*(1)}/W_I$; in particular, this scheme is affine. We set
\[
 \tD_{I,\bS} := (\tD_I)_{|\widetilde{\bS}_I^{*} \times_{\ft^{*(1)}/W_I} \ft^*/(W_I,\bullet)}.
 %, \quad \bbD_{I,\cS}=\Gamma(\widetilde{\cS}_I^{*(1)} \times_{\ft^{*(1)}/W_I} \ft^*/(W_I,\bullet), \tD_{I,\cS}).
\]
%As for $\tD_I$, $\tD_{I,\cS}$ can be seen either as a coherent sheaf of $\scO_{\widetilde{\cS}_I^{*(1)} \times_{\ft^{*(1)}/W_I} \ft^*/(W_I,\bullet)}$-algebras on $\widetilde{\cS}_I^{*(1)} \times_{\ft^{*(1)}/W_I} \ft^*/(W_I,\bullet)$, or as a quasi-coherent sheaf of algebras on $\cP_I$.

The following lemma is a parabolic analogue of Lemma~\ref{lem:US-D}, for which the same proof applies.

\begin{lem}
\label{lem:US-D-I}
 The morphism~\eqref{eqn:morph-Ug-D-I} induces an algebra isomorphism
 \[
  \cU_\cS\fg \otimes_{\ZHC} \scO(\ft^*/(W_I,\bullet)) \simto \Gamma(\widetilde{\bS}_I^{*} \times_{\ft^{*(1)}/W_I} \ft^*/(W_I,\bullet), \tD_{I,\bS}).
 \]
\end{lem}

Let
\[
 \bbX_I := \{ \lambda \in \bbX \mid \forall \alpha \in I, \, \langle \lambda,\alpha^\vee \rangle = 0\},
\]
so that $\bbX_I$ identifies with the character lattice of $L_I$.
Then any $\lambda \in \bbX_I$ defines a line bundle $\scO_{\cP_I}(\lambda)$ on $\cP_I$, from which one can define the space
\begin{equation}
\label{eqn:DI-def}
% \bbD_{I,\cS,\lambda} := 
 \Gamma \bigl( \widetilde{\bS}_I^{*} \times_{\ft^{*(1)}/W_I} \ft^*/(W_I,\bullet), (\scO_{\cP_I}(\lambda) \otimes \tD_I)_{|\widetilde{\bS}_I^{*} \times_{\ft^{*(1)}/W_I} \ft^*/(W_I,\bullet)} \bigr).
\end{equation}
This object admits a natural action of the algebra
\[
 (\cU_\bS\fg \otimes_{\ZHC} \scO(\ft^*/(W_I,\bullet))) \otimes_{\scO(\bS^{*})} (\cU_\bS\fg^\op \otimes_{\ZHC} \scO(\ft^*/(W_I,\bullet)))
\]
and of the group scheme 
\[
 \ft^*/(W_I,\bullet) \times_{\ft^{*(1)}/W} \bbI_\bS^* \times_{\ft^{*(1)}/W} \ft^*/(W_I,\bullet).
\]
Since $\lambda$ is $W_I$-invariant, the map $\xi \mapsto \ola + \xi$ factors through an isomorphism
\[
\tau_\lambda^I : \ft^*/(W_I,\bullet) \simto \ft^*/(W_I,\bullet),
\]
and 
the action of the subalgebra $\scO(\ft^*/(W_I,\bullet) \times_{\ft^{*(1)}/W} \ft^*/(W_I,\bullet))$ on~\eqref{eqn:DI-def} factors through 
%an action of the algebra of functions on the image of 
the morphism induced by the closed embedding
\[
\tau^I_\lambda \times \id : \ft^*/(W_I,\bullet) \to \ft^*/(W_I,\bullet) \times_{\ft^{*(1)}/W} \ft^*/(W_I,\bullet).
\]

Given $\lambda,\mu \in \bbX$ such that $\lambda-\mu \in \bbX_I$, one can then define the object
\[
 \sfQ^I_{\lambda,\mu} \in \Mod^{\bbI}_\fin(\cU_\bS^{\hat{\lambda},\hat{\mu}})
\]
as the completion of the module
\[
\Gamma \bigl( \widetilde{\bS}_I^{*} \times_{\ft^{*(1)}/W_I} \ft^*/(W_I,\bullet), (\scO_{\cP_I}(\lambda-\mu) \otimes \tD_I)_{|\widetilde{\bS}_I^{*} \times_{\ft^{*(1)}/W_I} \ft^*/(W_I,\bullet)} \bigr)
\]
%$\bbD_{I,\cS,\lambda-\mu}$ 
at the ideal of $\scO(\ft^*/(W_I,\bullet) \times_{\ft^{*(1)}/W} \ft^*/(W_I,\bullet))$ corresponding to the image of $(\ola,\omu)$. As for $\sfQ_{\lambda,\mu}$, this object can be obtained by completing the module at the ideal of $\scO(\ft^*/(W_I,\bullet))$ corresponding to the image of $\ola$ with respect to the left action, or at the ideal of $\scO(\ft^*/(W_I,\bullet))$ corresponding to the image of $\omu$ with respect to the right action.

 \begin{lem}
 \label{lem:B-transitivity-sing}
  Let $\lambda,\mu,\nu \in \bbX$.
%such that $\lambda-\mu$ and $\mu-\nu$ belong to $\bbX_I$, and assume that the stabilizer of $\mu$ for the dot-action of $\Waff$ is $W_I$. Then there exists a canonical isomorphism
%  \[
%   \sfB^I_{\lambda,\mu} \hatotimes_{\cU_\cS \fg} \bbB^I_{\mu,\nu} \simto \bbB^I_{\lambda,\nu}.
%  \]
\begin{enumerate}
\item
\label{it:transitivity-sing-1}
Assume that the stabilizer of $\mu$ for the dot-action of $\Waff$ is $W_I$, and that $\nu \in -\rho + \bbX_I$. Then there exists a canonical isomorphism
\[
\sfQ_{\lambda,\mu} \hatotimes_{\cU_\bS \fg} \sfQ_{\mu,\nu}^I \simto \sfQ_{\lambda,\nu}
\]
in $\Mod_{\fin}^{\bbI}(\cU_\bS^{\hat{\lambda},\hat{\nu}})$.
Similarly, if the stabilizer of $\mu$ for the dot-action of $\Waff$ is $W_I$, and $\lambda \in -\rho + \bbX_I$, then there exists a canonical isomorphism
\[
\sfQ^I_{\lambda,\mu} \hatotimes_{\cU_\bS \fg} \sfQ_{\mu,\nu} \simto \sfQ_{\lambda,\nu}
\]
in $\Mod_{\fin}^{\bbI}(\cU_\bS^{\hat{\lambda},\hat{\nu}})$.
\item
\label{it:transitivity-sing-2}
Assume that the stabilizer of $\mu$ for the dot-action of $\Waff$ is $W_I$, and that $\lambda,\nu \in -\rho + \bbX_I$. Then there exists a canonical isomorphism
\[
\sfQ^I_{\lambda,\mu} \hatotimes_{\cU_\bS \fg} \sfQ_{\mu,\nu}^I \simto \sfQ^I_{\lambda,\nu}
\]
in $\Mod_{\fin}^{\bbI}(\cU_\bS^{\hat{\lambda},\hat{\nu}})$.
\end{enumerate}
 \end{lem}
 
 \begin{proof}
 \eqref{it:transitivity-sing-1}
 We only prove the first isomorphism; the proof of the second one is similar.
 Our assumptions ensure that $\mu-\nu \in \bbX_I$, so that the object $\sfQ_{\mu,\nu}^I$ is well defined.
 Consider the natural morphism $a_I : \cB \to \cP_I$. By~\cite[Proposition~1.2.3]{bmr2} there exists a canonical morphism of sheaves of algebras
 \begin{equation}
 \label{eqn:morph-tDI-tD}
 \tD_I \to (a_I)_* \tD.
 \end{equation}
 By the projection formula, and since $(a_I)_* \scO_\cB \cong \scO_{\cP_I}$, we also have
 \[
 (a_I)_* \scO_{\cB}(\mu-\nu) \cong \scO_{\cP_I}(\mu-\nu),
% (a_I)_* (\scO_{\cB}(\mu-\nu) \otimes_{\scO_\cB} \tD) &\cong \scO_{\cP_I}(\mu-\nu) \otimes_{\scO_{\cP_I}} (a_I)_* \tD,
 \]
 and via this isomorphism the action of $\tD_I$ on $\scO_{\cP_I}(\mu-\nu)$ is obtained by restriction of scalars along~\eqref{eqn:morph-tDI-tD} from the natural action of $(a_I)_* \tD$ on $(a_I)_* \scO_{\cB}(\mu-\nu)$. We deduce a natural isomorphism
 \begin{multline*}
  (a_I)_* \bigl( \scO_\cB(\lambda-\mu) \otimes_{\scO_\cB} \tD \bigr) \otimes_{\tD_I} \bigl( \scO_{\cP_I}(\mu-\nu) \otimes_{\scO_{\cP_I}} \tD_I \bigr) \simto \\
  (a_I)_* (\scO_\cB(\lambda-\nu) \otimes_{\scO_{\cB}} \tD),
 \end{multline*}
 defined by a formula similar to that considered in the proof of Lemma~\ref{lem:B-transitivity}. The desired isomorphism follows by restricting to $\widetilde{\bS}_I^{*} \times_{\ft^{*(1)}/W_I} \ft^*/(W_I,\bullet)$ and then completing, using Lemma~\ref{lem:US-D-I} and the fact that the natural morphism $\ft^*/(W_I,\bullet) \to \ft^*/(W,\bullet)$ is \'etale at the image of $\omu$, see Lemma~\ref{lem:quotient-etale}.
 %as follows from our assumption and~\cite[Exp.~V, Proposition~2.2]{sga1}.
 
\eqref{it:transitivity-sing-2}
The proof is similar to that of Lemma~\ref{lem:B-transitivity}.
 \end{proof}

%----------------------------------------
\subsection{Conjugation of wall-crossing bimodules}
\label{ss:wall-crossing-bimod}
%----------------------------------------

The following proposition will eventually reduce the question of the description of the bimodules realizing wall-crossing functors for $G$ to the case of wall-crossing functors attached to simple reflections which belong to $W$.

\begin{prop}
\label{prop:conjugation-wall-crossing}
Consider elements $s \in \Saff$, $s' \in \Saff \cap W$ and $w \in \Wext$ such that $s'=wsw^{-1}$.
Let $\lambda, \mu, \mu' \in \bbX$, 
%let $s \in \Saff$, and let $s' \in \Saff \cap W$. 
with $\lambda$ belonging to the fundamental alcove, and $\mu$, resp.~$\mu'$, belonging to the wall of the fundamental alcove attached to $s$, resp.~$s'$, and on no other wall. 
%Let also $w \in \Wext$ be such that $s'=wsw^{-1}$. 
Then there exists an isomorphism
\[
\sfP_\bS^{\lambda,\mu'} \hatotimes_{\cU_\bS \fg} \sfP_\bS^{\mu',\lambda} \cong \sfQ_{\lambda, w \bullet \lambda} \hatotimes_{\cU_\bS \fg} \bigl( \sfP_\bS^{\lambda,\mu} \hatotimes_{\cU_\bS \fg} \sfP_\bS^{\mu,\lambda} \bigr) \hatotimes_{\cU_\bS \fg} \sfQ_{w \bullet \lambda, \lambda}
\]
in $\Mod_{\fin}^{\bbI}(\cU_\bS^{\hat{\lambda},\hat{\lambda}})$.
\end{prop}

\begin{proof}
By Lemma~\ref{lem:translation-bimod-B} we have isomorphisms
\[
\sfP_\bS^{\lambda,\mu} \cong \sfQ_{w \bullet \lambda, w \bullet \mu}, \quad \sfP_\bS^{\mu,\lambda} \cong \sfQ_{w \bullet \mu, w \bullet \lambda}.
\]
Using Lemma~\ref{lem:B-transitivity}, we deduce isomorphisms
\begin{multline*}
\sfQ_{\lambda, w \bullet \lambda} \hatotimes_{\cU_\bS \fg} \bigl( \sfP_\bS^{\lambda,\mu} \hatotimes_{\cU_\bS \fg} \sfP_\bS^{\mu,\lambda} \bigr) \hatotimes_{\cU_\bS \fg} \sfQ_{w \bullet \lambda, \lambda} \\
\cong \sfQ_{\lambda, w \bullet \lambda} \hatotimes_{\cU_\bS \fg} \sfQ_{w \bullet \lambda, w \bullet \mu} \hatotimes_{\cU_\bS \fg} \sfQ_{w \bullet \mu, w \bullet \lambda} \hatotimes_{\cU_\bS \fg} \sfQ_{w \bullet \lambda, \lambda}
\cong \sfQ_{\lambda,w \bullet \mu} \hatotimes_{\cU_\bS \fg} \sfQ_{w \bullet \mu, \lambda}.
\end{multline*}
Now the stabilizers of both $\mu'$ and $w \bullet \mu$ for the dot-action of $\Waff$ is $W_{\{\alpha\}}$, where $\alpha \in \fR^{\mathrm{s}}$ is the simple reflection such that $s'=s_\alpha$. By Lemma~\ref{lem:B-transitivity-sing}\eqref{it:transitivity-sing-1}, it follows that we have isomorphisms
\[
\sfQ_{\lambda,w \bullet \mu} \cong \sfQ_{\lambda,\mu'} \hatotimes_{\cU_\bS \fg} \sfQ^{\{\alpha\}}_{\mu', w \bullet \mu}, \quad \sfQ_{w \bullet \mu, \lambda} \cong \sfQ^{\{\alpha\}}_{w \bullet \mu, \mu'} \hatotimes_{\cU_\bS \fg} \sfQ_{\mu',\lambda},
\]
from which we obtain an isomorphism
\[
\sfQ_{\lambda,w \bullet \mu} \hatotimes_{\cU_\bS \fg} \sfQ_{w \bullet \mu, \lambda} \cong \sfQ_{\lambda,\mu'} \hatotimes_{\cU_\bS \fg} \sfQ^{\{\alpha\}}_{\mu', w \bullet \mu} \hatotimes_{\cU_\bS \fg} \sfQ^{\{\alpha\}}_{w \bullet \mu, \mu'} \hatotimes_{\cU_\bS \fg} \sfQ_{\mu',\lambda}.
\]
Then by Lemma~\ref{lem:B-transitivity-sing}\eqref{it:transitivity-sing-2} we have
\[
\sfQ^{\{\alpha\}}_{\mu', w \bullet \mu} \hatotimes_{\cU_\bS \fg} \sfQ^{\{\alpha\}}_{w \bullet \mu, \mu'} \cong \sfQ^{\{\alpha\}}_{\mu', \mu'},
\]
which implies (using again Lemma~\ref{lem:B-transitivity-sing}\eqref{it:transitivity-sing-1}) that
\[
\sfQ_{\lambda,w \bullet \mu} \hatotimes_{\cU_\bS \fg} \sfQ_{w \bullet \mu, \lambda} \cong \sfQ_{\lambda,\mu'} \hatotimes_{\cU_\bS \fg} \sfQ_{\mu',\lambda}.
\]
The desired claim follows, using again Lemma~\ref{lem:translation-bimod-B}.
\end{proof}

%%%%%%%%%%%%%%%%%%%%%%%%%%%%%%%%%%%%%%%%%%%%%%%%
\section{Hecke action on the principal block}
\label{sec:Hecke-action}
%%%%%%%%%%%%%%%%%%%%%%%%%%%%%%%%%%%%%%%%%%%%%%%%

%We come back to the setting and notation of Sections~\ref{sec:HCBim}%
%--\ref{sec:localization}--\ref{sec:Ug-D}, assuming now that 
In this section we assume that $p>h$ where $h$ is the Coxeter number of $G$ (see Remark~\ref{rmk:Cox-number}). In particular, this ensures that $p$ is very good for $G$, so that the results of the previous sections are applicable.

%--------------------------------------------
\subsection{Categories of \texorpdfstring{$G$}{G}-modules and \texorpdfstring{$G$}{G}-equivariant \texorpdfstring{$\Ug$}{Ug}-modules}
\label{ss:categories}
%--------------------------------------------

We now take a closer look at the category $\Rep(G)$ of finite-dimensional algebraic $G$-modules, and review its
decomposition into ``blocks.'' This will involve the notation introduced in~\S\S\ref{ss:weights}--\ref{ss:central-reductions}.

Recall (see~\S\ref{ss:HCBim-completed}) that for any $\lambda \in \bbX^+$ we have a simple $G$-module $\Sim(\lambda)$ of highest weight $\lambda$, and that all simple $G$-modules are of this form.
The \emph{linkage principle} (see~\cite[Corollary~II.6.17]{jantzen}) states that for $\lambda,\mu \in \bbX^+$ we have
\[
 \Ext^1_{\Rep(G)}(\mathsf{L}(\lambda),\mathsf{L}(\mu)) \neq 0 \quad \Rightarrow \quad \Waff \bullet \lambda = \Waff \bullet \mu.
\]
As a consequence, if for a $\Waff$-orbit $\mathbf{c} \subset \bbX$ we denote by $\Rep_{\mathbf{c}}(G)$ the Serre subcategory of $\Rep(G)$ generated by the simple objects $\mathsf{L}(\lambda)$ with $\lambda \in \mathbf{c} \cap \bbX^+$, then we have a direct sum decomposition
\begin{equation}
\label{eqn:Rep-blocks}
 \Rep(G) = \bigoplus_{\mathbf{c} \in \bbX/(\Waff,\bullet)} \Rep_{\mathbf{c}}(G).
\end{equation}
For $\lambda \in \bbX$, we will write $[\lambda]$ for the $\Waff$-orbit of $\lambda$. 
We will also set
\[
 \Rep_{\langle \lambda \rangle}(G) = \bigoplus_{\substack{\mathbf{c} \in \bbX/(\Waff,\bullet) \\ \mathbf{c} \subset \Wext \bullet \lambda}} \Rep_{\mathbf{c}}(G).
\]
%Below we will 
%assume that $p>h$ for $h$ the Coxeter number of $G$, and we will be 
%be primarily interested in the ``principal block'' $\Rep_{[0]}(G)$ and its ``extended version'' $\Rep_{\langle 0\rangle}(G)$.

Consider the category $\Mod_{\mathrm{fg}}^G(\Ug)$ of $G$-equivariant finitely generated $\Ug$-modules. For $\xi \in \ft^*/(W,\bullet)$, we will denote by
\[
\Mod_{\mathrm{fg}}^{G,\xi}(\Ug)
\]
the full subcategory of $\Mod_{\mathrm{fg}}^G(\Ug)$ whose objects are the modules annihilated by a power of the ideal $\mathfrak{m}^\xi \subset \ZHC$. As for other similar notations, in case $\xi=\tilde{\lambda}$ for some $\lambda \in \bbX$, we will write $\Mod_{\mathrm{fg}}^{G,\lambda}(\Ug)$ for $\Mod_{\mathrm{fg}}^{G,\tilde{\lambda}}(\Ug)$. If we denote by $\Mod_{\mathrm{fg}}^{G,\wedge}(\Ug)$ the category of $G$-equivariant finitely generated $\Ug$-modules annihilated by a power of the ideal $\cI \subset \ZHC \cap \ZFr$ defined in~\S\ref{ss:comparison}, then as e.g.~in~\eqref{eqn:Mod-Uwedge-sum} we have a canonical decomposition
\[
 \Mod_{\mathrm{fg}}^{G,\wedge}(\Ug) = \bigoplus_{\lambda \in \Lambda} \Mod_{\mathrm{fg}}^{G,\lambda}(\Ug),
\]
where $\Lambda \subset \bbX$ is as in~\eqref{eqn:Mod-Uwedge-sum}.

There is a natural fully faithful functor
\begin{equation}
\label{eqn:For-Rep}
\Rep(G) \to \Mod_{\mathrm{fg}}^G(\Ug)
\end{equation}
sending a $G$-module $V$ to itself, with its $G$-module structure, and with the $\Ug$-module structure obtained by differentiating the $G$-action. The essential image of this functor consists of the finite-dimensional $G$-equivariant $\Ug$-modules having the property that their $\Ug$-module structure is obtained from their $G$-module structure by differentiation. Since, for any $\lambda \in \bbX^+$, the action of $\ZHC$ on $\Sim(\lambda)$ factors through the quotient $\ZHC/\mathfrak{m}^\lambda$,                                                                                                                                                                  the functor~\eqref{eqn:For-Rep} restricts to a functor
\[
\Rep_{[\lambda]}(G) \to \Mod_{\mathrm{fg}}^{G,\lambda}(\Ug)
\]
for any $\lambda \in \bbX$. Since $\mathfrak{m}^\lambda$ only depends on the orbit $\Wext \bullet \lambda$, in this way we also obtain a fully faithful functor
\begin{equation}
\label{eqn:Rep-Ug}
\Rep_{\langle \lambda\rangle}(G) \to \Mod_{\mathrm{fg}}^{G,\lambda}(\Ug).
\end{equation}

%-------------------------------------------------------------
\subsection{Action of completed bimodules}
\label{ss:action-bimodules}
%-------------------------------------------------------------

Recall the category $\Mod^G_{\fin}(\cU^{\wedge})$ introduced in~\S\ref{ss:comparison}.
There exists a canonical bifunctor
\[
(-) \hatotimes_{\Ug} (-) : \Mod^G_{\fin}(\cU^{\wedge}) \times \Mod_{\fin}^{G,\wedge}(\Ug) \to \Mod_{\mathrm{fg}}^{G,\wedge}(\Ug)
\]
which can be defined as follows. Consider some $M$ in $\Mod^G_{\fin}(\cU^{\wedge})$ and some $V$ in $\Mod_{\mathrm{fg}}^{G,\wedge}(\Ug)$. By definition, there exists $m \in \Z_{\geq 1}$ such that $\cI^m$ acts trivially on $V$. Then the tensor product
\[
(M / \cI^m \cdot M) \otimes_{\Ug} V
\]
is a finitely generated left $\Ug$-module (where in the tensor product we consider the right $\Ug$-action on $M/\cI^m \cdot M$), which does not depend on the choice of $m$, and which admits a natural (diagonal) structure of algebraic $G$-module. Moreover the action of $\cI$ on this module is nilpotent. We can therefore take this as the definition of $M \hatotimes_{\Ug} V$.

The bifunctor $\hatotimes_{\Ug}$ defines on $\Mod_{\fin}^{G,\wedge}(\Ug)$ a structure of module category for the monoidal category $\Mod^G_{\fin}(\cU^{\wedge})$. It is also easily seen that for $\lambda,\mu \in \bbX$ this bifunctor restricts to a bifunctor
\[
 \Mod^G_{\fin}(\cU^{\hat{\lambda},\hat{\mu}}) \times \Mod_{\fin}^{G,\mu}(\Ug) \to \Mod_{\mathrm{fg}}^{G,\lambda}(\Ug)
\]
(where the category $\Mod^G_{\fin}(\cU^{\hat{\lambda},\hat{\mu}})$ is as in~\S\ref{ss:HCBim-completed}),
which itself restricts to a bifunctor
\[
\HCBim^{\hat{\lambda},\hat{\mu}} \times \Rep_{\langle \mu\rangle}(G)  \to \Rep_{\langle \lambda\rangle}(G)
\]
under the embeddings $\HCBim^{\hat{\lambda},\hat{\mu}} \to \Mod^G_{\fin}(\cU^{\hat{\lambda},\hat{\mu}})$ and~\eqref{eqn:Rep-Ug}.

\subsection{Relation with translation functors}
\label{ss:translation}
%---------------------------------------------------------------

Recall the definition of the translation functors for $G$-modules from~\cite[Chap.~II.7]{jantzen}.
Fix $\lambda,\mu \in \bbX$, and denote by $\nu$ the only dominant $W$-translate of $\lambda-\mu$. Then the translation functor
\[
T^\lambda_\mu : \Rep_{[\mu]}(G) \to \Rep_{[\lambda]}(G)
\]
is the functor sending an object $V$ to the direct summand of $\Sim(\nu) \otimes V$ which belongs to $\Rep_{[\lambda]}(G)$ in the decomposition provided by~\eqref{eqn:Rep-blocks}. We will consider these functors only in case $\lambda$ and $\mu$ both belong to the closure of the fundamental alcove. In this setting, we have defined in~\S\ref{ss:HCBim-completed} an object
$\sfP^{\lambda,\mu} \in \HCBim_{\diag}^{\hat{\lambda},\hat{\mu}}$.

%obtained as the image of the Harish-Chandra bimodule $\Sim(\nu) \otimes \Ug$ (with the diagonal left action of $\Ug$, the right action induced by right multiplication on $\Ug$, and the diagonal $G$-action) under the functor~\eqref{eqn:HC-completion-functor}.

\begin{lem}
\label{lem:translation-bimodules}
%Assume either that $\lambda$ belongs to the closure of the facet of $\mu$, or that $\mu$ is regular and $\lambda$ is on exactly one wall of the fundamental alcove. Then 
Let $\lambda,\mu \in \bbX$ belonging to the closure of the fundamental alcove.
The composition
\[
\Rep_{[\mu]}(G) \to \Rep_{\langle \mu\rangle}(G) \xrightarrow{\sfP^{\lambda,\mu} \hatotimes_{\Ug} (-)} \Rep_{\langle\lambda\rangle}(G)
\]
is canonically isomorphic to the composition
\[
\Rep_{[\mu]}(G) \xrightarrow{T_\mu^\lambda} \Rep_{[\lambda]}(G) \to \Rep_{\langle \lambda\rangle}(G).
\]
\end{lem}

\begin{proof}
By definition, the first functor sends a module $V$ in $\Rep_{[\mu]}(G)$ to the quotient
\[
(\Sim(\nu) \otimes V) / (\mathfrak{m}^\lambda)^n \cdot (\Sim(\nu) \otimes V)
\]
for $n \gg 0$, i.e.~to the direct sum of the factors in $\Sim(\nu) \otimes V$ corresponding to orbits included in $\Wext \bullet \lambda$ in the decomposition provided by~\eqref{eqn:Rep-blocks}. However, 
%the diagonalizable group $Z(G)$ acts on $\Sim(\nu) \otimes V$ via a character; namely, the image of $\lambda$ in $X^*(Z(G)) \cong \bbX / \Z \cdot \Phi$, where $\Phi$ is the root lattice. It follows that 
all the $T$-weights in $\Sim(\nu) \otimes V$ belong to $\lambda + \Z \fR$. 
%Since $\Z \fR \cap (p \cdot \bbX) = p\Z\fR$ (see~\S\ref{ss:weights}), it follows 
In view of Lemma~\ref{lem:weights}\eqref{it:lem-no-torsion-1}, this implies
that $[\lambda]$ is the only $\Waff$-orbit contained in $\Wext \bullet \lambda$ that can contribute to the direct sum above.
%a character; it follows that only one of these direct summands can be nonzero, namely that corresponding to the orbit $[\lambda]$. This proves the desired claim. 
\end{proof}

\begin{rmk}
See~\cite[Lemma~4.3.1]{riche} for a different proof of this claim, under more restrictive assumptions which would be sufficient for our present purposes.
\end{rmk}

\subsection{Main result}
\label{ss:main-result}
%--------------------------------------------------------------

We now consider the category $\DBS$ of~\S\ref{ss:Hecke-cat} associated with the group $\bG=G^{(1)}$.
We also fix a weight $\lambda$ in the fundamental alcove. (Such a weight exists since $p \geq h$.) For any $s \in \Saff$, we choose a weight $\mu_s \in \bbX$ in the closure of the fundamental alcove, which lies on the wall associated with $s$ but on no other wall. (For the existence of such a weight, see~\cite[\S II.6.3]{jantzen}.) Once these choices have been made, the \emph{wall-crossing functor} associated with $s \in \Saff$ is the composition
\[
\Theta_s := T_{\mu_s}^\lambda \circ T_\lambda^{\mu_s} : \Rep_{[\lambda]}(G) \to \Rep_{[\lambda]}(G).
\]

The main result of the present section (and of this paper) is the following.

\begin{thm}
\label{thm:action}
There exists a monoidal functor
\[
\Psi^\lambda : \DBS \to \HCBim^{\hat{\lambda},\hat{\lambda}}
\]
such that
\[
\Psi^\lambda(B_s) \cong \sfP^{\lambda,\mu_s} \hatotimes_{\Ug} \sfP^{\mu_s,\lambda}
\]
for any $s \in \Saff$.
\end{thm}

The proof of this theorem will be explained in~\S\ref{ss:proof-thm-action}. Before, we show that (as explained in the introduction) this theorem implies the main conjecture of~\cite{rw}.

\begin{cor}
\label{cor:conj}
There exists a $\bk$-linear right action of the monoidal category $\DBS$ on $\Rep_{[\lambda]}(G)$ such that for any $s \in \Saff$ the action of the object $B_s$ is isomorphic to $\Theta_s$.
\end{cor}

\begin{proof}
As explained in~\S\ref{ss:action-bimodules}, there exists a canonical (left) action of the category $\HCBim^{\hat{\lambda},\hat{\lambda}}$ on the category $\Rep_{\langle \lambda\rangle}(G)$.
The category $\DBS$ admits a canonical autoequivalence $\imath$ which satisfies $\imath(X \cdot Y)=\imath(Y) \cdot \imath(X)$ for any $X,Y \in \DBS$, see e.g.~\cite[\S 4.2]{rw}. Using this autoequivalence, the functor of Theorem~\ref{thm:action} therefore provides a right action of $\DBS$ on $\Rep_{\langle \lambda\rangle}(G)$ such that $B_s$ acts via the bimodule $\sfP^{\lambda,\mu_s} \hatotimes_{\Ug} \sfP^{\mu_s,\lambda}$, for any $s \in \Saff$. By Lemma~\ref{lem:translation-bimodules}, the action of this bimodule stabilizes the subcategory $\Rep_{[\lambda]}(G)$, and its action on this summand is isomorphic to $\Theta_s$. We have therefore constructed the desired action.
%In view of Lemma~\ref{lem:translation-bimodules}, Theorem~\ref{thm:action} therefore implies that there exists a right action of the monoidal category $\DBS$ on $\Rep_{[\lambda]}(G)$, where $B_s$ acts by the wall-crossing functor $T_{\mu_s}^\lambda T_\lambda^{\mu_s}$ for any $s \in \Saff$. 
\end{proof}

\begin{rmk}
 It is clear from the proof of Corollary~\ref{cor:conj} that the existence of a \emph{right} action of $\DBS$ with the required action of each $B_s$ is equivalent to the existence of a \emph{left} action with the same property. The reason why Conjecture~\ref{conj:intro} mentions a right action is that it makes the comparison with the combinatorics of the category $\Rep_{[ \lambda ]}(G)$ easier.
\end{rmk}

See~\S\ref{ss:image-morphisms} for a discussion of what can be said about the images under $\Psi^\lambda$ of the generating morphisms of $\DBS$.

% EXPLAIN MORPHISMS.
% 
% This
% confirms~\cite[Conjecture~5.1.1]{rw}, or rather the slightly less precise statement considered in~\cite[Remark~5.1.2(3)]{rw}. (As explained in \emph{loc}.~\emph{cit}., this version is however sufficient to deduce all the applications considered in~\cite[Part~I]{rw}, and in particular the character formula for tilting modules in $\Rep_{[\lambda]}(G)$.)

%-------------------------------------------------------------------
\subsection{Proof of Theorem~\ref{thm:action}}
\label{ss:proof-thm-action}
%-------------------------------------------------------------------

Recall that
%that by definition the group scheme $\bbJ^*_\adj$ coincides with $\bbJ^*_\bS$ under the identification $\bS^* \simto \ft^{*(1)}/W$. Hence
%construction we have a canonical morphism $\bbI^*_\cS \to (\bbJ^*_\cS)^{(1)}$ of group schemes over $\cS^{*(1)}$, see~\S\ref{ss:centralizer-Kostant}. Now if the Kostant section $\bS \subset \bg=\fg^{(1)}$ of~\S\ref{ss:Kostant-section-Abe} is chosen as $\cS^{(1)}$, and if the isomorphism $\kappa$ of~\S\ref{ss:Hecke-cat} is chosen as $\varkappa^{(1)}$, then the group scheme $(\bbJ^*_\cS)^{(1)}$ identifies with the group scheme $\bJ_\bS^*$ of~\S\ref{ss:Kostant-section-Abe}. In this way, 
Theorem~\ref{thm:Hecke-centralizer} provides a monoidal functor
\begin{equation}
 \label{eqn:functor-DBS-RepJ}
\DBS \to \Rep^{\Gm}(\ft^{*(1)} \times_{\ft^{*(1)}/W} \bbJ^*_\bS \times_{\ft^{*(1)}/W} \ft^{*(1)}).
\end{equation}
% Pulling back under the morphism $\bbI^*_\cS \to (\bbJ^*_\cS)^{(1)}$ we deduce a monoidal functor
% \begin{equation}
% \label{eqn:functor-DBS-RepI}
% \DBS \to \Rep(\ft^{*(1)} \times_{\ft^{*(1)}/W} \bbI^*_\cS \times_{\ft^{*(1)}/W} \ft^{*(1)}),
% \end{equation}
% where the monoidal product in the right-hand side is defined by the same recipe as for $\Rep(\ft^{*(1)} \times_{\ft^{*(1)}/W} (\bbJ^*_\cS)^{(1)} \times_{\ft^{*(1)}/W} \ft^{*(1)})$. For consistency of notation, in this section we will denote by $(\Delta^{\bbJ}_w : w \in \Wext)$ the objects considered at the end of~\S\ref{ss:Rep-Abe}; we will also denote by $(\Delta^{\bbI}_w : w \in \Wext)$ their images in $\Rep(\ft^{*(1)} \times_{\ft^{*(1)}/W} \bbI^*_\cS \times_{\ft^{*(1)}/W} \ft^{*(1)})$.
%
Since $\lambda$ belongs to the fundamental alcove, its stabilizer for the dot-action on $\bbX$ is trivial, so that the quotient morphism
\[
\ft^* \to \ft^*/(W,\bullet)
\]
is \'etale at $\ola$, see Lemma~\ref{lem:quotient-etale}. Similarly, the Artin--Schreier map
\[
\ft^* \to \ft^{*(1)}
\]
is \'etale (everywhere, hence in particular at $\ola$), and sends $\ola$ to $0$. Using these maps we obtain morphisms
\[
\ft^{*(1)} \times_{\ft^{*(1)}/W} \ft^{*(1)} \leftarrow \ft^* \times_{\ft^{*(1)}/W} \ft^* \to \ft^*/(W,\bullet) \times_{\ft^{*(1)}/W} \ft^*/(W,\bullet)
\]
\'etale at $(\ola,\ola)$,
which identify the algebra $\cZ_{\bS}^{\hat{\lambda},\hat{\lambda}}$ from~\S\ref{ss:HCBim-S}
%= \scO(\fC_\cS^{\hat{\lambda},\hat{\lambda}})$ 
with the completion $\scO(\ft^{*(1)} \times_{\ft^{*(1)}/W} \ft^{*(1)})^{\hat{0},\hat{0}}$ of $\scO(\ft^{*(1)} \times_{\ft^{*(1)}/W} \ft^{*(1)})$ with respect to the maximal ideal corresponding to $(0,0)$. Using this identification, the pullback functor associated with the natural morphism $\mathrm{Spec}(\scO(\ft^{*(1)} \times_{\ft^{*(1)}/W} \ft^{*(1)})^{\hat{0},\hat{0}}) \to \ft^{*(1)} \times_{\ft^{*(1)}/W} \ft^{*(1)}$
%completion with respect to this maximal ideal 
induces a monoidal functor
\begin{equation}
\label{eqn:completion-bbJ}
 \Rep^{\Gm}(\ft^{*(1)} \times_{\ft^{*(1)}/W} \bbJ^*_\bS \times_{\ft^{*(1)}/W} \ft^{*(1)}) \to \Mod^{\bbJ}_{\fin}(\cZ_\bS^{\hat{\lambda},\hat{\lambda}}),
\end{equation}
where the category on the right-hand side is as in~\S\ref{ss:splitting-bundles}.
Precomposing this functor with~\eqref{eqn:functor-DBS-RepJ}, and then composing with the equivalence of Corollary~\ref{cor:equiv-ModUg-ModZ} we obtain a monoidal functor
\begin{equation}
\label{eqn:functor-DBS-RepUS}
 \Psi_\bS^\lambda : \DBS \to \HCBim_\bS^{\hat{\lambda},\hat{\lambda}}.
\end{equation}

\begin{prop}
\label{prop:image-Bs-finite}
 For any $s \in \Saff \cap W$, 
% the functor~\eqref{eqn:functor-DBS-RepUS} sends the object
 there exists an isomorphism
 \[
 \Psi_\bS^\lambda(B_s) \cong \sfP_\bS^{\lambda,\mu_s} \hatotimes_{\cU_\bS\fg} \sfP_\bS^{\mu_s,\lambda}.
 \]
\end{prop}

\begin{proof}
%First, let us assume that $s$ belongs to $W$, i.e.~that $s=s_\alpha$ for some $\alpha \in \fR^{\mathrm{s}}$. 
In the course of the proof of Lemma~\ref{lem:essential-images} we have seen that the image of $B_s$ in $\Rep^\Gm(\ft^{*(1)} \times_{\ft^{*(1)}/W} \bbJ^*_\bS \times_{\ft^{*(1)}/W} \ft^{*(1)})$ is $\scO(\ft^{*(1)} \times_{\ft^{*(1)}/\{e,s\}} \ft^{*(1)})$, endowed with the trivial structure as a representation. 
%It follows that the same is true for its image under~\eqref{eqn:functor-DBS-RepI}. 
On the other hand, by Proposition~\ref{prop:localization-wall-crossing} the wall-crossing bimodule $\sfP_\bS^{\lambda,\mu_s} \hatotimes_{\cU_\bS\fg} \sfP_\bS^{\mu_s,\lambda}$ corresponds to the object $\cZ_\bS^{\hat{\lambda},\hat{\lambda}} \otimes_{\cZ_\bS^{\hat{\mu},\hat{\lambda}}} \cZ_\bS^{\hat{\lambda}}$ (again endowed with the trivial structure as a representation) under the equivalence $\mathscr{L}_{\lambda,\lambda}$. 
Recall that $\cZ_\bS^{\hat{\lambda},\hat{\lambda}}$ identifies with the completion of $\scO(\ft^* \times_{\ft^{*(1)}/W} \ft^*)$ at the ideal corresponding to $(\ola,\ola)$.
The considerations in the proof of Lemma~\ref{lem:wall-crossing-diag}, together with the fact that the quotient morphism $\ft^* \to \ft^*/(W,\bullet)$ is \'etale at $\ola$, imply that the algebra $\cZ_\bS^{\hat{\mu},\hat{\lambda}}$ identifies with the completion of $\scO(\ft^*/(\{e,s\},\bullet) \times_{\ft^{*(1)}/W} \ft^*)$ at the ideal corresponding to the image of $(\omu,\ola)$, and that via this identification the morphism $\cZ_\bS^{\hat{\mu},\hat{\lambda}} \to \cZ_\bS^{\hat{\lambda},\hat{\lambda}}$ is induced by the natural morphism
\[
 \ft^* \times_{\ft^{*(1)}/W} \ft^* \to \ft^*/(\{e,s\},\bullet) \times_{\ft^{*(1)}/W} \ft^*
\]
sending $(\ola,\ola)$ to the image of $(\omu,\ola)$. This morphism fits in a natural commutative diagram
% {\small
% \[
%  \xymatrix@C=0.4cm{
%  \ft^{*(1)} \times_{\ft^{*(1)}/W} \ft^{*(1)} \ar[d] & \ft^* \times_{\ft^{*(1)}/W} \ft^* \ar[l] \ar[r] \ar[d] & \ft^*/(W,\bullet) \times_{\ft^{*(1)}/W} \ft^*/(W,\bullet) \ar@{=}[d] \\
%  \ft^{*(1)}/\{e,s\} \times_{\ft^{*(1)}/W} \ft^{*(1)} & \ft^*/(\{e,s\},\bullet) \times_{\ft^{*(1)}/W} \ft^* \ar[l] \ar[r] & \ft^*/(W,\bullet) \times_{\ft^{*(1)}/W} \ft^*/(W,\bullet)
%  }
% \]
% }
\[
 \xymatrix{
 \ft^* \times_{\ft^{*(1)}/W} \ft^* \ar[r] \ar[d] & \ft^{*(1)} \times_{\ft^{*(1)}/W} \ft^{*(1)} \ar[d] \\
 \ft^*/(\{e,s\},\bullet) \times_{\ft^{*(1)}/W} \ft^* \ar[r] & \ft^{*(1)}/\{e,s\} \times_{\ft^{*(1)}/W} \ft^{*(1)}
 }
\]
where the right vertical arrow is induced by the natural quotient morphism $\ft^{*(1)} \to \ft^{*(1)}/\{e,s\}$ and the horizontal arrows are induced by the Artin--Schreier map.
Here the morphism on the upper row is \'etale at $(\ola,\ola)$, and that on the lower row is \'etale at the image of $(\omu,\ola)$ by the same arguments as for Lemma~\ref{lem:morph-etale}.
%we observe that the morphism $\ft^* \to \ft^{*(1)}/\{e,s\}$ is the quotient morphism for the natural dot-action of $\ft^*_\Z \rtimes \{e,s\}$. Since the stabilizer of $\omu$ for this action is $\{e,s\}$, the claim follows once again from~\cite[Exp.~V, Proposition~2.2]{sga1}. 
This observation shows that $\cZ_\bS^{\hat{\lambda},\hat{\lambda}} \otimes_{\cZ_\bS^{\hat{\mu},\hat{\lambda}}} \cZ_\bS^{\hat{\lambda}}$ identifies with the $\scO(\ft^{*(1)} \times_{\ft^{*(1)}/W} \ft^{*(1)})^{\hat{0},\hat{0}}$-module
\[
 \scO(\ft^{*(1)} \times_{\ft^{*(1)}/W} \ft^{*(1)})^{\hat{0},\hat{0}} \otimes_{\scO(\ft^{*(1)}/\{e,s\} \times_{\ft^{*(1)}/W} \ft^{*(1)})^{\hat{0},\hat{0}}} \scO(\ft^{*(1)})^{\hat{0}},
\]
where $\scO(\ft^{*(1)}/\{e,s\} \times_{\ft^{*(1)}/W} \ft^{*(1)})^{\hat{0},\hat{0}}$ is the completion of $\scO(\ft^{*(1)}/\{e,s\} \times_{\ft^{*(1)}/W} \ft^{*(1)})$ at the ideal corresponding to the image of $(0,0)$, and $\scO(\ft^{*(1)})^{\hat{0}}$ is the completion of $\scO(\ft^{*(1)})$ (seen as an $\scO(\ft^{*(1)}/\{e,s\} \times_{\ft^{*(1)}/W} \ft^{*(1)})$-module in the natural way) at the ideal corresponding to $0$. Using the same considerations as in the proof of Lemma~\ref{lem:wall-crossing-diag}, it is easily seen that this module identifies with the completion of $\scO(\ft^{*(1)} \times_{\ft^{*(1)}/\{e,s\}} \ft^{*(1)})$, which finishes the proof of our claim.
\end{proof}

\begin{rmk}
\label{rmk:canonical-isom-s}
 The isomorphism constructed in the proof of Proposition is essentially canonical, in the sense that it only depends on the isomorphism of Proposition~\ref{prop:localization-wall-crossing} (for the pair $(\lambda,\mu_s)$), which itself only depends on a choice of isomorphism $\Sim(\nu_s)^* \cong \Sim(-w_0(\nu_s))$ where $\nu_s$ is the only dominant $W$-translate of $\mu_s-\lambda$, see Remark~\ref{rmk:canonicity-isom}.
\end{rmk}

Recall the objects $(\Delta_w^{\bbJ} : w \in W_\ext)$ introduced at the end of~\S\ref{ss:Rep-Abe}, and the objects $(\sfQ_{\nu,\eta} : \nu,\eta \in \bbX)$ introduced in~\S\ref{ss:Ug-D-bimodules}.

\begin{lem}
\label{lem:image-braid-group-bimodule-L}
 For any $w \in \Wext$, the object $\mathscr{L}_{\lambda,\lambda}^{-1}(\sfQ_{\lambda,w \bullet \lambda})$ is isomorphic to the image of $\Delta_w^{\bbJ}$ under the functor~\eqref{eqn:completion-bbJ}.
\end{lem}

\begin{proof}
 Write $w = t_\nu x$ with $\nu \in p\bbX$ and $x \in W$. Then by Lemma~\ref{lem:B-transitivity} we have
 \[
  \sfQ_{\lambda,w \bullet \lambda} = \sfQ_{\lambda,x \bullet \lambda+p\nu} \cong \sfQ_{\lambda,x \bullet \lambda} \hatotimes_{\cU_\bS\fg} \sfQ_{x \bullet \lambda,x \bullet \lambda+p\nu}.
 \]
It follows from Lemma~\ref{lem:B-translation} and the proof of Lemma~\ref{lem:Delta-J} that $\mathscr{L}_{\lambda,\lambda}^{-1}(\sfQ_{x \bullet \lambda,x \bullet \lambda+p\nu})$ is the image of $\Delta_{t_\nu}^\bbJ$. In view of~\eqref{eqn:convolution-Delta} and the monoidality of $\mathscr{L}_{\lambda,\lambda}^{-1}$ (see Proposition~\ref{prop:L-monoidal}), to conclude it therefore suffices to prove that $\mathscr{L}_{\lambda,\lambda}^{-1}(\sfQ_{\lambda,x \bullet \lambda})$ is the image of $\Delta_x^\bbJ$. In turn, if $x=s_1 \cdots s_r$ is a reduced expression (with each $s_i$ in $W \cap \Saff$) then again by by Lemma~\ref{lem:B-transitivity} we have
\[
 \sfQ_{\lambda,x\bullet\lambda} \cong \sfQ_{\lambda,s_1\bullet\lambda} \hatotimes_{\cU_\bS\fg} \sfQ_{s_1 \bullet \lambda,s_1s_2\bullet\lambda} \hatotimes_{\cU_\bS\fg} \cdots \hatotimes_{\cU_\bS\fg} \sfQ_{(s_1 \cdots s_{r-1}) \bullet \lambda,x\bullet\lambda}.
\]
If we write $y_j = s_1 \cdots s_{j}$ for $j \in \{0, \cdots, r\}$ then, by monoidality of $\mathscr{L}_{\lambda,\lambda}$,
to conclude it suffices to prove that $\mathscr{L}_{\lambda,\lambda}^{-1}(\sfQ_{y_{i-1} \bullet \lambda,y_i \bullet \lambda})$ is the image of $\Delta_{s_i}^\bbJ$ for any $i \in \{1, \cdots, r\}$.

Fix $i \in \{1, \cdots, r\}$. We have $y_i \bullet \lambda < y_{i-1} \bullet \lambda$. By Lemma~\ref{lem:transl-bimod-ses} we therefore have an exact sequence
\[
 \sfQ_{y_i \bullet \lambda,y_{i-1} \bullet \lambda} \hookrightarrow \sfP_\bS^{\lambda,\mu_{s_i}} \hatotimes_{\cU_\bS\fg} \sfP_\bS^{\mu_{s_i},\lambda} \twoheadrightarrow \sfQ_{y_{i-1} \bullet \lambda, y_{i-1} \bullet \lambda}.
\]
As seen in the course of the proof of Proposition~\ref{prop:image-Bs-finite}, $\mathscr{L}_{\lambda,\lambda}^{-1}(\sfP_\bS^{\lambda,\mu_{s_i}} \hatotimes_{\cU_\bS\fg} \sfP_\bS^{\mu_{s_i},\lambda})$ corresponds to the completion of $\scO(\ft^{*(1)} \times_{\ft^{*(1)}/\{e,s_i\}} \ft^{*(1)})$ (with the trivial structure as a representation of the appropriate group scheme), and it is clear that $\mathscr{L}_{\lambda,\lambda}^{-1}(\sfQ_{y_{i-1} \bullet \lambda, y_{i-1} \bullet \lambda})$ is the completion of $\scO(\ft^{*(1)})$. The object $\mathscr{L}_{\lambda,\lambda}^{-1}(\sfQ_{y_i \bullet \lambda,y_{i-1} \bullet \lambda})$ is therefore the kernel of a surjection from the former completion to the latter completion. However, up to an automorphism of the completion of $\scO(\ft^{*(1)})$ there exists only one such surjection, and its kernel corresponds to the completion of $\Delta_{s_i}^{\bbJ}$ by Lemma~\ref{lem:exact-seq-Bs}.
\end{proof}

Once Lemma~\ref{lem:image-braid-group-bimodule-L} is proved, one also obtains that $\mathscr{L}_{\lambda,\lambda}^{-1}(\sfQ_{w \bullet \lambda,\lambda})$, which is the inverse of $\mathscr{L}_{\lambda,\lambda}^{-1}(\sfQ_{\lambda,w \bullet \lambda})$ (see~\S\ref{ss:Ug-D-bimodules}), is isomorphic to the image of $\Delta_{w^{-1}}^{\bbJ}$ under~\eqref{eqn:completion-bbJ}. (In fact, one can then check that for any $\mu \in W_\ext \bullet \lambda$, $\mathscr{L}_{\lambda,\lambda}^{-1}(\sfQ_{\mu,w \bullet \mu})$ is isomorphic to the image of $\Delta_{w}^{\bbJ}$.)

We can finally complete the proof of Theorem~\ref{thm:action}.

\begin{proof}[Proof of Theorem~\ref{thm:action}]
Consider the functor $\Psi_\bS^\lambda$ from~\eqref{eqn:functor-DBS-RepUS}. We claim that for any $s \in \Saff$ we have
\begin{equation}
\label{eqn:Psi-B}
\Psi_\bS^\lambda(B_s) \cong \sfP_\bS^{\lambda,\mu_s} \hatotimes_{\cU_\bS \fg} \sfP_\bS^{\mu_s,\lambda}.
\end{equation}
In fact, if $s \in W$ this is the content of Proposition~\ref{prop:image-Bs-finite}. Otherwise, as already seen in the course of the proof of Lemma~\ref{lem:essential-images}, there exist $x \in \Wext$ and $t \in \Saff \cap W$ such that $s=xtx^{-1}$. Then by Lemma~\ref{lem:conjugation-abe} the image of $B_s$ in $\mathsf{C}_\ext$ is
\[
B_s^{\Bim} \cong \Delta_x \otimes_R B_t^{\Bim} \otimes_R \Delta_{x^{-1}};
\]
using Lemma~\ref{lem:image-braid-group-bimodule-L} (together with the remark following it) and the known description of $\Psi_\bS^\lambda(B_t)$,
we deduce that
\[
 \Psi_\bS^\lambda(B_s) \cong \sfQ_{x^{-1} \bullet \lambda,\lambda} \hatotimes_{\cU_\bS \fg} \sfP_\bS^{\lambda,\mu_t} \hatotimes_{\cU_\bS \fg} \sfP_\bS^{\mu_t,\lambda} \hatotimes_{\cU_\bS \fg} \sfQ_{\lambda, x^{-1} \bullet \lambda}.
\]
% and Lemma~\ref{lem:image-braid-group-bimodule-L}, the image of $B_s$ is then
% \[
%  \sfQ_{\lambda,x \bullet \lambda} \hatotimes_{\cU_\bS\fg} \left( \sfP_\bS^{\lambda,\mu_t} \hatotimes_{\cU_\bS \fg} \sfP_\bS^{\mu_t,\lambda} \right) \hatotimes_{\cU_\bS\fg} \sfQ_{x \bullet \lambda,\lambda}.
% \]
% which by 
In view of Proposition~\ref{prop:conjugation-wall-crossing}, this implies~\eqref{eqn:Psi-B}.

Since each object of $\DBS$ is isomorphic to a shift of a product of objects $B_s$, and since both of the involved functors are monoidal, our claim implies that $\Psi_\bS^\lambda$ takes values in the essential image of the fully faithful functor
\[
 \HCBim_\diag^{\hat{\lambda},\hat{\lambda}} \to \HCBim_\bS^{\hat{\lambda},\hat{\lambda}}.
\]
(see Proposition~\ref{prop:rest-S-fully-faithful}). It follows that $\Psi_\bS^\lambda$ factors in a canonical way through a monoidal functor
\[
 \Psi^\lambda : \DBS \to \HCBim^{\hat{\lambda},\hat{\lambda}}
\]
sending $B_s$ to $\sfP^{\lambda,\mu_s} \hatotimes_{\Ug} \sfP^{\mu_s,\lambda}$ for any $s \in \Saff$, which finishes the proof.
%Since each object $\sfP^{\lambda,\mu_s} \hatotimes_{\Ug} \sfP^{\mu_s,\lambda}$ belongs to $\HCBim^{\hat{\lambda},\hat{\lambda}}_\diag$, and since this subcategory is monoidal, our functor factors through a functor $\DBS \to \HCBim^{\hat{\lambda},\hat{\lambda}}_\diag$, which finishes the proof.
\end{proof}

%---------------------------------------
\subsection{Images of generating morphisms}
\label{ss:image-morphisms}
%---------------------------------------

The category $\DBS$ is defined in~\cite{ew} in terms of generators and relations. In Theorem~\ref{thm:action} we have explained what are the images of the generating objects under $\Psi^\lambda$ (at least, up to isomorphism); by monoidality this determines the image of any object in $\DBS$ (again, up to isomorphism).
We finish the paper with a discussion of what can be said about the image under $\Psi^\lambda$ of the generating morphisms of $\DBS$.

\begin{rmk}
As explained in~\cite[Remark~5.1.2(3)]{rw}, although the original version of Conjecture~\ref{conj:intro} contained information about these images, this information is not needed for the applications considered in~\cite[Part~I]{rw}, and in particular the character formula for tilting modules in $\Rep_{[\lambda]}(G)$.
\end{rmk}

Recall that these generating morphisms fall into four families:
\begin{itemize}
 \item the polynomials (morphisms from $B_\varnothing$ to a shift of $B_\varnothing$, determined by homogeneous elements in $R=\scO(\ft^{*(1)})$);
 \item the ``dot'' morphisms for $s \in \Saff$
\[
       \begin{tikzpicture}[thick,scale=0.07,baseline]
      \draw (0,-5) to (0,0);
      \node at (0,0) {$\bullet$};
      \node at (0,-6.7) {\tiny $s$};
    \end{tikzpicture}
    \qquad \text{and} \qquad
      \begin{tikzpicture}[thick,baseline,xscale=0.07,yscale=-0.07]
      \draw (0,-5) to (0,0);
      \node at (0,0) {$\bullet$};
      \node at (0,-6.7) {\tiny $s$};
    \end{tikzpicture}
\]
 (morphisms from $B_s$ to a shift of $B_\varnothing$ and from $B_\varnothing$ to a shift of $B_s$);
 \item the ``trivalent'' morphisms for $s \in \Saff$
 \[
        \begin{tikzpicture}[thick,baseline,scale=0.07]
      \draw (-4,5) to (0,0) to (4,5);
      \draw (0,-5) to (0,0);
      \node at (0,-6.7) {\tiny $s$};
      \node at (-4,6.7) {\tiny $s$};
      \node at (4,6.7) {\tiny $s$};      
    \end{tikzpicture}
    \qquad \text{and} \qquad
        \begin{tikzpicture}[thick,baseline,scale=-0.07]
      \draw (-4,5) to (0,0) to (4,5);
      \draw (0,-5) to (0,0);
      \node at (0,-6.7) {\tiny $s$};
      \node at (-4,6.7) {\tiny $s$};
      \node at (4,6.7) {\tiny $s$};    
    \end{tikzpicture}
\]
 (morphisms from $B_s$ to a shift of $B_{ss}$ and from $B_{ss}$ to a shift of $B_s$);
 \item the ``$2m_{s,t}$-valent'' morphisms, for pairs $(s,t)$ of distinct elements of $\Saff$ generating a finite subgroup of $\Waff$.
\end{itemize}
%Let us say right away that we cannot say anything meaningful about images of the most subtle morphisms, namely those in the fourth family.
The information we can give only concerns the first two families of morphisms.

The image of polynomials is easy to describe: we have $\Psi^\lambda(B_\varnothing)=\cU^{\hat{\lambda}}$. If $\ZHC^{\hat{\lambda}}$ is as in Remark~\ref{rmk:completion-tensor-product-ZHC}, any element in $\ZHC^{\hat{\lambda}}$ determines an endomorphism of $\cU^{\hat{\lambda}}$. Now the natural morphisms
\[
 \ft^*/(W,\bullet) \leftarrow \ft^* \rightarrow \ft^{*(1)}
\]
are \'etale at $\ola$; they therefore determine an isomorphism between $\ZHC^{\hat{\lambda}}$ and the completion $\scO(\ft^{*(1)})^{\hat{0}}$ of $\scO(\ft^{*(1)})$ with respect to the ideal of $0$. The image under $\Psi^\lambda$ of a homogenous polynomial in $R$ is the endomorphism of $\cU^{\hat{\lambda}}$ determined by the corresponding element in $\ZHC^{\hat{\lambda}}$.

To describe the image of the other morphisms, we must be more specific about the isomorphism $\Psi^\lambda(B_s) \cong \sfP^{\lambda,\mu_s} \hatotimes_{\Ug} \sfP^{\mu_s,\lambda}$. First, assume that $s \in W$. In this case, after choosing an isomorphism $\Sim(\nu_s)^* \cong \Sim(-w_0(\nu_s))$ we obtain a canonical such isomorphism, see Remark~\ref{rmk:canonical-isom-s}. Using this isomorphism, Remark~\ref{rmk:canonicity-isom} shows that the image of the upper dot morphism
\[
       \begin{tikzpicture}[thick,scale=0.07,baseline]
      \draw (0,-5) to (0,0);
      \node at (0,0) {$\bullet$};
      \node at (0,-6.7) {\tiny $s$};
    \end{tikzpicture}
\]
is the morphism $\varphi_s : \sfP^{\lambda,\mu_s} \hatotimes_{\Ug} \sfP^{\mu_s,\lambda} \to \cU^{\hat{\lambda}}$ determined by the adjunction
\[
(\sfP^{\lambda,\mu_s} \hatotimes_{\Ug}(-), \ \sfP^{\mu_s,\lambda} \hatotimes_{\Ug}(-))
\]
defined by our choice of isomorphism $\Sim(\nu_s)^* \cong \Sim(-w_0(\nu_s))$. (Here we use an obvious variant of Lemma~\ref{lem:convolution-adjoint} for $\Ug$ in place of $\cU_\bS\fg$.)

Proposition~\ref{prop:rest-S-fully-faithful}, Corollary~\ref{cor:equiv-ModUg-ModZ} and Proposition~\ref{prop:localization-wall-crossing} (together with the various \'etale maps considered above) show that the $\scO(\ft^{*(1)})^{\hat{0}}$-modules
\[
 \Hom_{\HCBim^{\hat{\lambda},\hat{\lambda}}}(\cU^{\hat{\lambda}},\cU^{\hat{\lambda}}) \quad \text{and} \quad \Hom_{\HCBim^{\hat{\lambda},\hat{\lambda}}}(\cU^{\hat{\lambda}},\sfP^{\lambda,\mu_s} \hatotimes_{\Ug} \sfP^{\mu_s,\lambda})
\]
are both free of rank $1$; from this one can check that there exists a unique morphism
\[
 \psi_s : \cU^{\hat{\lambda}} \to \sfP^{\lambda,\mu_s} \hatotimes_{\Ug} \sfP^{\mu_s,\lambda}
\]
whose composition with $\varphi_s$ is the differential of the coroot of $(\bG,\bT)$ associated with $s$ (seen as an endomorphism of  $\cU^{\hat{\lambda}}$), and that this morphism is a generator of $\Hom_{\HCBim^{\hat{\lambda},\hat{\lambda}}}(\cU^{\hat{\lambda}},\sfP^{\lambda,\mu_s} \hatotimes_{\Ug} \sfP^{\mu_s,\lambda})$. In view of the ``barbell relation'' in $\DBS$, the image of the lower dot morphism
\[
      \begin{tikzpicture}[thick,baseline,xscale=0.07,yscale=-0.07]
      \draw (0,-5) to (0,0);
      \node at (0,0) {$\bullet$};
      \node at (0,-6.7) {\tiny $s$};
    \end{tikzpicture}
\]
is $\psi_s$. 
%This morphism is the coevaluation morphism for an adjunction
%\[
% (\sfP^{\mu_s,\lambda} \hatotimes_{\Ug}(-), \ \sfP^{\lambda,\mu_s} \hatotimes_{\Ug}(-)).
%\]
%In fact, by Lemma~\ref{lem:convolution-adjoint} there exists such an adjunction, and after choosing one the corresponding morphism $\Ug^{\hat{\lambda}} \to \sfP^{\lambda,\mu_s} \hatotimes_{\Ug} \sfP^{\mu_s,\lambda}$ and our morphism above are both generators of $\Hom_{\HCBim^{\hat{\lambda},\hat{\lambda}}}(\Ug^{\hat{\lambda}},\sfP^{\lambda,\mu_s} \hatotimes_{\Ug} \sfP^{\mu_s,\lambda})$; twisting our adjunction by an automorphism of the identity functor, we can therefore assume that it induces our morphism $\psi_s$.

In case $s \notin \Waff$, we do not have any canonical choice of isomorphism $\Psi^\lambda(B_s) \cong \sfP^{\lambda,\mu_s} \hatotimes_{\Ug} \sfP^{\mu_s,\lambda}$. What we can say is that there exists a choice of such an isomorphism (not unique) such that the images of the dot morphisms are described by the same rules as above.

%%%%%%%%%%%%%%%%%%%%%%%%%%%%
%%%%%%%%%%%%%%%%%%%%%%%%%%%%

\appendix

\section{Index of notation}
\label{sec:index}

Below is a list of the main notation used in the paper, listed by section of appearance. (We sometimes omit notation used only in one specific subsection.)

%--------------------------------------------------------------------------
\subsection{Section~\ref{sec:Hecke-univ-centralizer}}
%--------------------------------------------------------------------------

$\bk$: algebraically closed field of characteristic $p$, \S\ref{ss:Hecke-cat}.

$\bG$: connected reductive algebraic group over $\bk$, \S\ref{ss:Hecke-cat}.

$\bB$, $\bT$: Borel subgroup and maximal torus in $\bG$, \S\ref{ss:Hecke-cat}.

$\bg$, $\bb$, $\bt$: Lie algebras of $\bG$, $\bB$, $\bT$, \S\ref{ss:Hecke-cat}.

$\bX$, $\bX^\vee$: weight and coweight lattices of $\bT$, \S\ref{ss:Hecke-cat}.

$\Phi$, $\Phi^\vee$, $\Phi^+$, $\Phi^{\mathrm{s}}$: roots, coroots, positive roots, simple roots of $(\bG,\bB,\bT)$, \S\ref{ss:Hecke-cat}.

$\kappa$: choice of isomorphism $\bg \simto \bg^*$, \S\ref{ss:Hecke-cat}.

$\WW$: Weyl group of $(\bG,\bT)$, \S\ref{ss:Hecke-cat}.

$\Waff$, $\Wext$: affine and extended affine Weyl groups of $(\bG,\bT)$, \S\ref{ss:Hecke-cat}.

$\Saff$, $\Wext$: simple reflections in $\Waff$, \S\ref{ss:Hecke-cat}.

$\DBS$: diagrammatic Hecke category attached to $\bG$, \S\ref{ss:Hecke-cat}.

$B_{\uw}$: object of $\DBS$ attached to $\uw$, \S\ref{ss:Hecke-cat}.

$R=\scO(\bt^*)$, \S\ref{ss:Hecke-cat}.

$Q$: fraction field of $R$, \S\ref{ss:abes-incarnation}.

$\mathsf{C}$, $\mathsf{C}'$, $\mathsf{C}_\ext$, $\mathsf{C}'_\ext$: Abe's categories, \S\ref{ss:abes-incarnation}.

$B_s^{\Bim}$: bimodule in $\mathsf{C}$ attached to $s$, \S\ref{ss:abes-incarnation}.

$\Delta_x$: ``standard" object in $\mathsf{C}_\ext$ attached to $x$, \S\ref{ss:abes-incarnation}.

$\bU$: unipotent radical in $\bB$,  \S\ref{ss:Kostant-section-Abe}.

$\bn$: Lie algebra of $\bU$,  \S\ref{ss:Kostant-section-Abe}.

$\bg_\reg$, $\bg^*_\reg$: regular parts in $\bg$ and $\bg^*$, \S\ref{ss:Kostant-section-Abe}.

$\bbJ_\reg$, $\bbJ^*_\reg$: universal centralizers over $\bg_\reg$ and $\bg^*_\reg$, \S\ref{ss:Kostant-section-Abe}.

$\tbg$: Grothendieck resolution attached to $\bG$, \S\ref{ss:Kostant-section-Abe}.

$\tbg_\reg$: regular part in $\tbg$, \S\ref{ss:Kostant-section-Abe}.

$\pi$: natural morphism $\tbg \to \bg^*$, \S\ref{ss:Kostant-section-Abe}.

$\vartheta$: natural morphism $\tbg \to \bt^*$, \S\ref{ss:Kostant-section-Abe}.

$\bg_\rs$, $\bg^*_\rs$: regular semisimple parts in $\bg$ and $\bg^*$, \S\ref{ss:Kostant-section-Abe}.

$\bbJ_\rs$, $\bbJ^*_\rs$: restrictions of $\bbJ_\reg$, $\bbJ^*_\reg$ to $\bg_\rs$, $\bg^*_\rs$, \S\ref{ss:Kostant-section-Abe}.

$\bS$, $\bS^*$: Kostant sections in $\bg$ and $\bg^*$, \S\ref{ss:Kostant-section-Abe}.

$\bbJ^*_\bS$: restriction of $\bbJ^*_\reg$ to $\bS^*$, \S\ref{ss:Kostant-section-Abe}.

%$\bbJ^*_\adj$: regular centralizer over $\bt^*/\WW$, \S\ref{ss:Kostant-section-Abe}.

$\Rep^{\Gm}(\bt^* \times_{\bt^*/\WW} \bbJ^*_{\bS} \times_{\bt^*/\WW} \bt^*)$: category of $\Gm$-equivariant representations of $\bt^* \times_{\bt^*/\WW} \bbJ^*_{\bS} \times_{\bt^*/\WW} \bt^*$ on coherent sheaves on $\bt^* \times_{\bt^*/\WW} \bt^*$, \S\ref{ss:Rep-Abe}.

$\star$: monoidal product on $\Rep^{\Gm}(\bt^* \times_{\bt^*/\WW} \bbJ^*_{\bS} \times_{\bt^*/\WW} \bt^*)$, \S\ref{ss:Rep-Abe}.

$\Rep^{\Gm}_{\mathrm{fl}}(\bt^* \times_{\bt^*/\WW} \bbJ^*_{\bS} \times_{\bt^*/\WW} \bt^*)$: subcategory of $\Rep^{\Gm}(\bt^* \times_{\bt^*/\WW} \bbJ^*_{\bS} \times_{\bt^*/\WW} \bt^*)$ consisting of modules flat w.~r.~t.~the second projection $\bt^* \times_{\bt^*/\WW} \bt^* \to \bt^*$, \S\ref{ss:Rep-Abe}.

$\Delta^\bbJ_w$: object in $\Rep^{\Gm}(\bt^* \times_{\bt^*/\WW} \bbJ^*_{\bS} \times_{\bt^*/\WW} \bt^*)$ corresponding to $\Delta_w$, \S\ref{ss:Rep-Abe}.

%--------------------------------------------------------------------------
\subsection{Section~\ref{sec:HCBim}}
%--------------------------------------------------------------------------

$G$: connected reductive algebraic group such that $\bG=G^{(1)}$, \S\ref{ss:weights}.

$\Fr$: Frobenius morphism of $G$, \S\ref{ss:weights}.

$B$, $T$, $U$: subgroups of $G$ corresponding to $\bB$, $\bT$, $\bU$, \S\ref{ss:weights}.

$\fg$, $\fb$, $\ft$, $\fn$: Lie algebras of $G$, $B$, $T$, $U$, \S\ref{ss:weights}.

$W$: Weyl group of $(G,T)$, \S\ref{ss:weights}.

$\bX$, $\bX^\vee$: weight and coweight lattices of $T$, \S\ref{ss:weights}.

$\fR$, $\fR^\vee$, $\fR^+$, $\fRs$: roots, coroots, positive roots, simple roots of $(G,B,T)$, \S\ref{ss:weights}.

$w_0$: longest element in $W$, \S\ref{ss:weights}.

$\rho$: halfsum of the positive roots, \S\ref{ss:weights}.

$\bullet$: dot action of $W_\ext$ on $\bbX$ and $W$ on $\ft^*$, \S\ref{ss:weights}.

$\ola$: element of $\ft^*$ associated with $\lambda \in \bbX$, \S\ref{ss:weights}.

$\widetilde{\lambda}$: image of $\ola$ in $\ft^*/(W,\bullet)$, \S\ref{ss:weights}.

$\ft^*_{\mathbb{F}_p}$: ``integral'' part of $\ft^*$, \S\ref{ss:weights}.

$W_I$: subgroup of $W$ associated with $I \subset \fRs$, \S\ref{ss:weights}.

$\Ug$: universal enveloping algebra of $\fg$, \S\ref{ss:center}.

$\ZHC$, $\ZFr$: Harish-Chandra and Frobenius centers of $\Ug$, \S\ref{ss:center}.

$\AS$: Artin--Schreier morphism, \S\ref{ss:center}.

$\mathfrak{C}$: spectrum of $Z(\Ug)$, \S\ref{ss:center}.

$\mathfrak{m}_\eta$: ideal in $\ZFr$ associated with $\eta \in \fg^{*(1)}$, \S\ref{ss:central-reductions}.

$\mathfrak{m}^\xi$: ideal in $\ZHC$ associated with $\xi \in \ft^{*}/(W,\bullet)$, \S\ref{ss:central-reductions}.

$\cU_\eta \fg$, $\cU^\xi \fg$, $\cU_\eta^\xi \fg$: central reductions of $\Ug$, \S\ref{ss:central-reductions}.

$\cN^*$: nilpotent cone in $\fg^*$, \S\ref{ss:central-reductions}.

$\HCBim$: category of Harish-Chandra bimodules for $G$, \S\ref{ss:HCBim}.

$\cZ = Z(\Ug) \otimes_{\ZFr} Z(\Ug)$, \S\ref{ss:HCBim}.

$\Mod^G_\fin(\Ug \otimes_{\ZFr} \Ug^\op)$: category of $G$-equivariant f.g.~$\Ug \otimes_{\ZFr} \Ug^\op$-modules, \S\ref{ss:HCBim}.

$\mathfrak{D} = \ft^*/(W,\bullet) \times_{\ft^{*(1)}/W} \ft^*/(W,\bullet)$, \S\ref{ss:HCBim-completed}.

$\cI^{\lambda,\mu}$: ideal in $\scO(\fD)$ associated with $\lambda,\mu \in \bbX$,  \S\ref{ss:HCBim-completed}.

$\mathfrak{D}^{\hat{\lambda},\hat{\mu}}$: spectrum of the completion of $\scO(\fD)$ w.r.t.~$\cI^{\lambda,\mu}$,  \S\ref{ss:HCBim-completed}.

$\cU^{\hat{\lambda},\hat{\mu}} = (\Ug \otimes_{\ZFr} \Ug^\op) \otimes_{\scO(\fD)} \scO(\mathfrak{D}^{\hat{\lambda},\hat{\mu}})$, \S\ref{ss:HCBim-completed}.

$\Mod^G_\fin(\cU^{\hat{\lambda},\hat{\mu}})$: category of $G$-equivariant f.g.~$\cU^{\hat{\lambda},\hat{\mu}}$-modules, \S\ref{ss:HCBim-completed}.

$\HCBim^{\hat{\lambda},\hat{\mu}}$: subcategory of $\Mod^G_\fin(\cU^{\hat{\lambda},\hat{\mu}})$ of Harish-Chandra bimodules, \S\ref{ss:HCBim-completed}.

$\HCBim^{\hat{\lambda},\hat{\mu}}_\diag$: subcategory of $\HCBim^{\hat{\lambda},\hat{\mu}}$ of diagonally induced bimodules, \S\ref{ss:HCBim-completed}.

$\mathsf{C}^{\lambda,\mu}(-)=\scO(\mathfrak{D}^{\hat{\lambda},\hat{\mu}}) \otimes_{\scO(\fD)} (-)$, \S\ref{ss:HCBim-completed}.

$\cU^{\hat{\lambda}} = \mathsf{C}^{\lambda,\lambda}(\bk \otimes \Ug)$, \S\ref{ss:HCBim-completed}.

$h$: Coxeter number of $G$, \S\ref{ss:HCBim-completed}.

$\sfP^{\lambda,\mu}$: translation bimodule attached to $\lambda,\mu \in \bbX$, \S\ref{ss:HCBim-completed}.

$\Lambda$: set of representatives for the $(\Wext,\bullet)$-orbits on $\bbX$, \S\ref{ss:comparison}.

$\cI$: ideal of $\scO(\ft^{*(1)}/W)$ corresponding to the image of $0$, \S\ref{ss:comparison}.

$\fD^\wedge$: spectrum of the completion of $\scO(\fD)$ w.r.t.~$\cI \cdot \scO(\fD)$, \S\ref{ss:comparison}.

$\scO(\ft^{*(1)}/W)^\wedge$: completion of $\scO(\ft^{*(1)}/W)$ w.r.t.~$\cI$, \S\ref{ss:comparison}.

$\ZHC^\wedge$: completion of $\ZHC$ w.r.t.~$\cI \cdot \ZHC$, \S\ref{ss:comparison}.

$\ZHC^{\hat{\lambda}}$: completion of $\ZHC$ w.r.t.~$\mathfrak{m}^\lambda$, \S\ref{ss:comparison}.

$\cU^\wedge=(\Ug \otimes_{\ZFr} \Ug^\op) \otimes_{\scO(\fD)} \scO(\mathfrak{D}^\wedge)$, \S\ref{ss:comparison}.

$\Mod^G_\fin(\cU^\wedge)$: category of $G$-equivariant f.g.~$\cU^\wedge$-modules, \S\ref{ss:comparison}.

$\HCBim^\wedge$: subcategory of $\Mod^G_\fin(\cU^\wedge)$ of Harish-Chandra bimodules, \S\ref{ss:comparison}.

$\HCBim^\wedge_\diag$: subcategory of $\HCBim^\wedge$ of diagonally induced bimodules, \S\ref{ss:comparison}.

$\mathsf{C}^\wedge(-)=\scO(\mathfrak{D}^\wedge) \otimes_{\scO(\fD)} (-)$, \S\ref{ss:comparison}.

$(-) \hatotimes_\Ug (-)$: monoidal product for the categories $\Mod^G_\fin(\cU^{\hat{\lambda},\hat{\mu}})$, \S\ref{ss:monoidal-bimodules}.

$\bbI^*_\bS$: group scheme over $\bS^*$, \S\ref{ss:centralizer-Kostant}.

$\cU_{\bS}\fg := \Ug \otimes_{\ZFr} \scO(\bS^{*})$, \S\ref{ss:centralizer-Kostant}.

$\fC_\bS := \bS^{*} \times_{\ft^{*(1)} / W} \ft^*/(W,\bullet)$, \S\ref{ss:centralizer-Kostant}.

$\cZ_\bS = \cZ \otimes_{\ZFr} \scO(\bS^*)$, \S\ref{ss:HCBim-S}.

$\Mod^{\bbI}_{\mathrm{fg}}(\cU_\bS \fg \otimes_{\scO(\bS^{*})} \cU_\bS \fg^{\mathrm{op}})$: category of $\bbI^*_\bS$-equivariant f.g.~$\cU_\bS \fg \otimes_{\scO(\bS^{*})} \cU_\bS \fg^{\mathrm{op}}$-modules, \S\ref{ss:HCBim-S}.

$\HCBim_\bS$: subcategory of $\Mod^{\bbI}_{\mathrm{fg}}(\cU_\bS \fg \otimes_{\scO(\bS^{*})} \cU_\bS \fg^{\mathrm{op}})$ of Harish-Chandra bimodules, \S\ref{ss:HCBim-S}.

$\cI^{\lambda,\mu}_\bS = \cI^{\lambda,\mu} \cdot \cZ_\bS$, \S\ref{ss:HCBim-S}.

$\cZ_\bS^{\hat{\lambda},\hat{\mu}}$: completion of $\cZ_\bS$ w.r.t.~$\cI^{\lambda,\mu}_\bS$, \S\ref{ss:HCBim-S}.

$\cU^{\hat{\lambda},\hat{\mu}}_\bS = \cZ_\bS^{\hat{\lambda},\hat{\mu}} \otimes_{\cZ_\bS} ( \cU_\bS \fg \otimes_{\scO(\bS^{*})} \cU_\bS \fg^\op )$, \S\ref{ss:HCBim-S}.

$\bbI_\bS^{\hat{\lambda},\hat{\mu}} = \Spec(\cZ_\bS^{\hat{\lambda},\hat{\mu}}) \times_{\bS^*} \bbI^*_\bS$, \S\ref{ss:HCBim-S}.

$\Mod^{\bbI}_{\fin}(\cU_\bS^{\hat{\lambda},\hat{\mu}})$: category of $\bbI^{\hat{\lambda},\hat{\mu}}_\bS$-equivariant f.g.~$\cU_\bS^{\hat{\lambda},\hat{\mu}}$-modules, \S\ref{ss:HCBim-S}.

$\HCBim_\bS^{\hat{\lambda},\hat{\mu}}$: subcategory of $\Mod^{\bbI}_{\fin}(\cU_\bS^{\hat{\lambda},\hat{\mu}})$ of Harish-Chandra bimodules, \S\ref{ss:HCBim-S}.

$\sfP^{\lambda,\mu}_\bS = \scO(\bS^{*}) \otimes_{\scO(\fg^{*(1)})} \sfP^{\lambda,\mu}$, \S\ref{ss:HCBim-S}.

$(-) \hatotimes_{\cU_\bS \fg} (-)$: monoidal product for the categories $\Mod^{\bbI}_{\fin}(\cU_\bS^{\hat{\lambda},\hat{\mu}})$, \S\ref{ss:HCBim-S}.

$\cZ_\bS^\wedge$: completion of $\cZ_\bS$ w.r.t.~$\cI \cdot \cZ_\bS$, \S\ref{ss:HCBim-S}.

$\cU^\wedge_\bS = \cZ_\bS^\wedge \otimes_{\cZ_\bS} ( \cU_\bS \fg \otimes_{\scO(\bS^{*})} \cU_\bS \fg^\op )$, \S\ref{ss:HCBim-S}.

$\bbI_\bS^\wedge = \Spec(\cZ_\bS^\wedge) \times_{\bS^*} \bbI^*_\bS$, \S\ref{ss:HCBim-S}.

$\Mod^{\bbI}_{\fin}(\cU_\bS^\wedge)$: category of $\bbI^\wedge_\bS$-equivariant f.g.~$\cU_\bS^\wedge$-modules, \S\ref{ss:HCBim-S}.

$\sfC_\bS^\wedge (-) = \scO(\bS^*) \otimes_{\ZFr} \sfC^\wedge(-)$, \S\ref{ss:HCBim-S}.

$\cU_\bS^{\hat{\lambda}} = \scO(\bS^*) \otimes_{\scO(\fg^{*(1)})} \cU^{\hat{\lambda}}$, \S\ref{ss:HCBim-S}.

%--------------------------------------------------------------------------
\subsection{Section~\ref{sec:localization}}
%--------------------------------------------------------------------------

$\Ver_{\eta,B'}(\xi)$: baby Verma module, \S\ref{ss:Azumaya-Ug}.

$\Mod^{\bbI}_{\fin}(\cZ_\bS^\wedge)$: category of representations of $\bbI^\wedge_\bS$ on finitely generated $\cZ_\bS^\wedge$-modules, \S\ref{ss:splitting-bundles}.

$(-) \hatstar_{\bS} (-)$: monoidal product on $\Mod^{\bbI}_{\fin}(\cZ_\bS^\wedge)$, \S\ref{ss:splitting-bundles}.

$\bbJ_\bS^\wedge = \Spec(\cZ_\bS^\wedge) \times_{\bS^*} \bbJ^*_\bS$, \S\ref{ss:splitting-bundles}.

$\Mod^{\bbJ}_{\fin}(\cZ_\bS^\wedge)$: category of representations of $\bbJ^\wedge_\bS$ on finitely generated $\cZ_\bS^\wedge$-modules, \S\ref{ss:splitting-bundles}.

$\Mod^{\bbI}_{\fin}(\cZ_\bS^{\hat{\lambda},\hat{\mu}})$: category of representations of $\bbI^{\hat{\lambda},\hat{\mu}}_\bS$ on finitely generated $\cZ_\bS^{\hat{\lambda},\hat{\mu}}$-modules, \S\ref{ss:splitting-bundles}.

$\bbJ_\bS^{\hat{\lambda},\hat{\mu}} = \Spec(\cZ_\bS^{\hat{\lambda},\hat{\mu}}) \times_{\bS^*} \bbJ^*_\bS$, \S\ref{ss:splitting-bundles}.

$\Mod^{\bbJ}_{\fin}(\cZ_\bS^{\hat{\lambda},\hat{\mu}})$: category of representations of $\bbJ^{\hat{\lambda},\hat{\mu}}_\bS$ on finitely generated $\cZ_\bS^{\hat{\lambda},\hat{\mu}}$-modules, \S\ref{ss:splitting-bundles}.

$\sfM^{\lambda,\mu} = \sfP^{\lambda,-\rho} \hatotimes_{\Ug} \sfP^{-\rho,\mu}$, \S\ref{ss:splitting-bundles}.

$\sfM^{\lambda,\mu}_\bS = \scO(\bS^{*}) \otimes_{\scO(\fg^{*(1)})} \sfM^{\lambda,\mu}$, \S\ref{ss:splitting-bundles}.

$\tfD = \ft^* \times_{\ft^{*(1)}} \ft^*$, \S\ref{ss:study-fibers}.

$\tfD(\lambda)$, $\fD(\lambda)$: irreducible components associated with $\lambda \in \bbX$, \S\ref{ss:study-fibers}.

$\widetilde{\Lambda}$: set of representatives for $\bbX/p\bbX$, \S\ref{ss:study-fibers}.

$\ft^*_\circ$, open subset in $\ft^*$, \S\ref{ss:study-fibers}.

$\widetilde{\bS}^*$: preimage of $\bS^*$ in $\tbg$, \S\ref{ss:study-fibers}.

$\mathscr{L}_{\lambda,\mu}$: localization equivalence, \S\ref{ss:localization-HC}.

%--------------------------------------------------------------------------
\subsection{Section~\ref{sec:Ug-D}}
%--------------------------------------------------------------------------

$\cB=G/B$, \S\ref{ss:Ug-D-reg}.

$\omega$: natural morphism $G/U \to \cB$, \S\ref{ss:Ug-D-reg}.

$\tD$: universal twisted differential operators on $\cB$, \S\ref{ss:Ug-D-reg}.

$\tD_\bS := \tD_{|\widetilde{\bS}^{*} \times_{\ft^{*(1)}} \ft^*}$, \S\ref{ss:Ug-D-reg}.

$\widetilde{\cU}_\bS\fg := \cU_\bS\fg \otimes_{\ZHC} \scO(\ft^*)$, \S\ref{ss:Ug-D-reg}.

$\scO_{\cB}(\lambda)$: line bundle on $\cB$ attached to $\lambda \in \bbX$, \S\ref{ss:Ug-D-bimodules}.

$\widetilde{\cU}_\bS^{\hat{\lambda},\hat{\mu}}$:
completion of
$\widetilde{\cU}_\bS\fg \otimes_{\scO(\bS^{*})} (\widetilde{\cU}_\bS\fg)^\op$
at the ideal corresponding to $(\ola,\omu) \in \ft^* \times_{\ft^{*(1)}/W} \ft^*$, \S\ref{ss:Ug-D-bimodules}.

$\Mod^{\bbI}_{\fin}(\widetilde{\cU}_\bS^{\hat{\lambda},\hat{\mu}})$: category of equivariant f.g.~$\widetilde{\cU}_\bS^{\hat{\lambda},\hat{\mu}}$-modules, \S\ref{ss:Ug-D-bimodules}.

$(-) \hatotimes_{\widetilde{\cU}_\bS \fg} (-)$: monoidal product for the categories $\Mod^{\bbI}_{\fin}(\widetilde{\cU}_\bS^{\hat{\lambda},\hat{\mu}})$, \S\ref{ss:Ug-D-bimodules}.

$\sfQ_{\lambda,\mu}$: completion of
 $\Gamma \bigl( \widetilde{\bS}^{*} \times_{\ft^{*(1)}} \ft^*, (\scO_\cB(\lambda-\mu) \otimes_{\scO_\cB} \tD)_{|\widetilde{\bS}^{*} \times_{\ft^{*(1)}} \ft^*} \bigr)$, \S\ref{ss:Ug-D-bimodules}.
 
$\langle \eta \rangle$: twist functor on $\Mod^{\bbI}_{\fin}(\widetilde{\cU}_\bS^{\hat{\lambda},\hat{\mu}})$, \S\ref{ss:Ug-D-bimodules}.

$\fE=\ft^* \times_{\ft^{*(1)}/W} \ft^*/(W,\bullet)$, \S\ref{ss:relation}.

$\fE':=\ft^*/(W,\bullet) \times_{\ft^{*(1)}/W} \ft^*$, \S\ref{ss:relation}.

$\scO(\fE)^{\hat{\lambda},\hat{\mu}}$: completion of $\scO(\fE)$ at the ideal corresponding to $(\ola,\widetilde{\mu})$, \S\ref{ss:relation}.

$P_I$, $\bP_I$: standard parabolic subgroups in $G$ and $\bG$ associated with $I \subset \fRs$, \S\ref{ss:Ug-D-sing}.

$L_I$, $\bL_I$, $U_I$, $\bU_I$: Levi factor and unipotent radical of $P_I$ and $\bP_I$, \S\ref{ss:Ug-D-sing}.

$\cP_I=G/P_I$, \S\ref{ss:Ug-D-sing}.

$\omega_I$: natural morphism $G/U_I \to \cP_I$, \S\ref{ss:Ug-D-sing}.

$\tD_I$: twisted differential operators on $\cP_I$, \S\ref{ss:Ug-D-sing}.

$\tbg_I$: parabolic Grothendieck resolution for $\bG$, \S\ref{ss:Ug-D-sing}.

$\widetilde{\bS}^*_I$: preimage of $\bS^*$ in $\tbg_I$, \S\ref{ss:Ug-D-sing}.

$\tD_{I,\bS} = (\tD_I)_{|\widetilde{\bS}_I^{*} \times_{\ft^{*(1)}/W_I} \ft^*/(W_I,\bullet)}$, \S\ref{ss:Ug-D-sing}.

$\sfQ^I_{\lambda,\mu}$: completion of {\small
$\Gamma \bigl( \widetilde{\bS}_I^{*} \times_{\ft^{*(1)}/W_I} \ft^*/(W_I,\bullet), (\scO_{\cP_I}(\lambda-\mu) \otimes \tD_I)_{|\widetilde{\bS}_I^{*} \times_{\ft^{*(1)}/W_I} \ft^*/(W_I,\bullet)} \bigr)$}, \S\ref{ss:Ug-D-sing}.

%--------------------------------------------------------------------------
\subsection{Section~\ref{sec:Hecke-action}}
%--------------------------------------------------------------------------

$\Rep_{\mathbf{c}}(G)$: subcategory of $\Rep(G)$ associated with a $\Waff$-orbit $\mathbf{c} \subset \bbX$, \S\ref{ss:categories}.

$[\lambda]$: $\Waff$-orbit of $\lambda \in \bbX$, \S\ref{ss:categories}.

$\Rep_{\langle \lambda \rangle}(G)$: sum of the categories $\Rep_{\mathbf{c}}(G)$ with $\mathbf{c} \subset \Wext \bullet \lambda$, \S\ref{ss:categories}.

$\Mod^G_\fin(\Ug)$: category of $G$-equivariant f.g.~$\Ug$-modules, \S\ref{ss:categories}.

$\Mod^{G,\xi}_\fin(\Ug)$: full subcategory of $\Mod^G_\fin(\Ug)$ of modules annihilated by a power of $\mathfrak{m}^\xi$, \S\ref{ss:categories}.

$\Mod^{G,\wedge}_\fin(\Ug)$: full subcategory of $\Mod^G_\fin(\Ug)$ of modules annihilated by a power of $\cI$, \S\ref{ss:categories}.

$(-) \hatotimes_{\Ug} (-)$: bifunctor defining the action of $\Mod^G_{\fin}(\cU^{\wedge})$ on $\Mod_{\fin}^{G,\wedge}(\Ug)$, \S\ref{ss:action-bimodules}.

$T_\lambda^\mu$: translation functor (for $G$-modules) associated with $\lambda,\mu \in \bbX$, \S\ref{ss:translation}.

$\Theta_s$: wall-crossing functor associated with $s \in \Saff$, \S\ref{ss:main-result}.

$\Psi^\lambda$: functor from $\DBS$ to $\HCBim^{\hat{\lambda},\hat{\lambda}}$, \S\ref{ss:main-result}.

$\Psi^\lambda_\bS$: functor from $\DBS$ to $\HCBim_\bS^{\hat{\lambda},\hat{\lambda}}$, \S\ref{ss:proof-thm-action}.

%%%%%%%%%%%%%%%%%%%%%%%%%%%%%
%%%%%%%%%%%%%%%%%%%%%%%%%%%%%


\begin{thebibliography}{AMRW2}

\bibitem[Ab1]{abe}
N.~Abe, \emph{On Soergel bimodules}, preprint~\href{https://arxiv.org/abs/1901.02336}{arXiv:1901.02336}.

\bibitem[Ab2]{abe2}
N.~Abe, \emph{A Hecke action on $G_1T$-modules}, preprint~\href{https://arxiv.org/abs/1904.11350}{arXiv:1904.11350}.

\bibitem[Ab3]{abe3}
N.~Abe, \emph{A homomorphism between Bott--Samelson bimodules}, preprint \href{https://arxiv.org/pdf/2012.09414.pdf}{arXiv:2012.09414}.

\bibitem[AMRW1]{amrw}
P.~Achar, S.~Makisumi, S.~Riche, and G.~Williamson, \emph{Free-monodromic mixed tilting sheaves on flag varieties}, preprint~\href{https://arxiv.org/abs/1703.05843}{arXiv:1703.05843}.

\bibitem[AMRW2]{amrw2}
P.~Achar, S.~Makisumi, S.~Riche, and G.~Williamson, \emph{Koszul duality for Kac--Moody groups and characters of tilting modules}, 
% %preprint \texttt{arXiv:1706.00183}.
J. Amer. Math. Soc. \textbf{32} (2019), 261--310.

\bibitem[AM]{amcd}
M.~Atiyah and I.~Macdonald, \emph{Introduction to commutative algebra}, Addison-Wesley Publishing Co., 
%Reading, Mass.-London-Don Mills, Ont. 
1969.

\bibitem[Be]{bez}
R.~Bezrukavnikov, \emph{On two geometric realizations of an affine Hecke algebra}, Publ. Math. IHES~\textbf{123} (2016), 1--67.

\bibitem[BMR1]{bmr}
R.~Bezrukavnikov, I.~Mirkovi\'c, and D.~Rumynin, \emph{Localization of modules for a semisimple Lie algebra in prime characteristic},
with an appendix by R.~Bezrukavnikov and S.~Riche,
Ann. of Math. (2) \textbf{167} (2008), 
%no. 3, 
945--991. 

\bibitem[BMR2]{bmr2}
R.~Bezrukavnikov, I.~Mirkovi{\'c}, and D.~Rumynin, \emph{Singular localization and intertwining functors for reductive Lie algebras in prime characteristic},
Nagoya Math. J. \textbf{184} (2006), 1--55.

\bibitem[BM]{bm}
R.~Bezrukavnikov and I.~Mirkovi{\'c}, \emph{Representations of semisimple Lie algebras in prime characteristic and the noncommutative Springer resolution},
with an appendix by E.~Sommers,
Ann. of Math. (2) \textbf{178} (2013), 
%no. 3, 
835--919.

\bibitem[BR]{br}
R.~Bezrukavnikov and S.~Riche, \emph{Affine braid group actions on Springer resolutions}, Ann. Sci. {\'E}c. Norm. Sup{\'e}r. \textbf{45} (2012), 535--599.

\bibitem[BRR]{brr}
R.~Bezrukavnikov, S.~Riche, and L. Rider, \emph{Modular affine Hecke category and regular unipotent centralizer, I}, preprint~\href{https://arxiv.org/abs/2005.05583}{arXiv:2005.05583}.

\bibitem[BC]{bc}
A.~Bouthier and K.~Cesnavicius, \emph{Torsors on loop groups and the Hitchin fibration}, preprint~\href{https://arxiv.org/abs/1908.07480}{arXiv:1908.07480}, to appear in Ann. Sci. {\'E}c. Norm. Sup{\'e}r.

\bibitem[BK]{bk}
M.~Brion and S.~Kumar, \emph{Frobenius splitting methods in geometry and representation theory},
Progress in Mathematics 231, Birkh\"auser Boston, 
%Inc., Boston, MA, 
2005.

\bibitem[BG]{brown-goodearl}
K.~A.~Brown and K.~R.~Goodearl, \emph{Homological aspects of Noetherian PI Hopf algebras and irreducible modules of maximal dimension}, J. Algebra \textbf{198} (1997), 
%no. 1, 
240--265.

\bibitem[BGor]{brown-gordon}
K.~A.~Brown and I.~Gordon, \emph{The ramification of centres: Lie algebras in positive characteristic and quantised enveloping algebras},
Math. Z. \textbf{238} (2001), 
%no. 4, 
733--779. 

\bibitem[CG]{cg}
N.~Chriss and V.~Ginzburg, \emph{Representation theory and complex geometry}, Birkh\"auser Boston, 
%Inc., Boston, MA, 
1997.

\bibitem[Ci]{ciappara}
J.~Ciappara, \emph{Hecke category actions via Smith--Treumann theory}, pre\-print~\href{https://arxiv.org/abs/2103.07091}{arXiv:2103.07091}.

% \bibitem[De]{demazure}
% M.~Demazure, \emph{Invariants sym\'etriques entiers des groupes de Weyl et torsion}, Invent. Math. \textbf{21} (1973), 287--301.

\bibitem[Do]{dodd} 
C.~Dodd, \emph{Equivariant coherent sheaves, Soergel bimodules, and categorification
of affine Hecke algebras}, preprint~\href{https://arxiv.org/abs/1108.4028}{arXiv:1108.4028}.

% \bibitem[Ei]{eisenbud}
% D.~Eisenbud, \emph{Commutative algebra. With a view toward algebraic geometry}, Graduate Texts in Mathematics 150, Springer-Verlag, 
% %New York, 
% 1995.

\bibitem[El]{elias}
B.~Elias, \emph{Gaitsgory's central sheaves via the diagrammatic Hecke category}, preprint~\href{https://arxiv.org/abs/1811.06188}{arXiv:1811.06188}.

\bibitem[EL]{el}
B.~Elias and I.~Losev, \emph{Modular representation theory in type A via Soergel bimodules}, preprint \href{https://arxiv.org/abs/1701.00560}{arXiv:1701.00560}.

\bibitem[EW1]{ew}
B.~Elias and G.~Williamson, \emph{Soergel calculus}, Represent. Theory~\textbf{20} (2016), 295--374.

\bibitem[EW2]{ew2}
B.~Elias and G.~Williamson, \emph{Localized calculus for the Hecke category}, preprint~\href{https://arxiv.org/abs/2011.05432}{arXiv:2011.05432}.

%\bibitem[EGA1]{ega1}
%A.~Grothendieck, \emph{\'El\'ements de g\'eom\'etrie alg\'ebrique (r\'edig\'es avec la collaboration de J.~Dieudonn\'e) : I. Le langage des sch\'emas},
%Inst. Hautes \'Etudes Sci. Publ. Math. \textbf{4} (1960).

\bibitem[H\"a]{haerterich}
M. H\"arterich, \emph{Kazhdan--Lusztig-Basen, unzerlegbare Bimoduln und die Topologie der Fahnenmannigfaltigkeit einer Kac--Moody-Gruppe}, PhD thesis (1999), available at~\url{https://freidok.uni-freiburg.de/data/18}.

\bibitem[J1]{jantzen-Lie}
J.~C.~Jantzen, \emph{Representations of Lie algebras in prime characteristic},
notes by Iain Gordon, in \emph{Representation theories and algebraic geometry (Montreal, PQ, 1997)}, NATO Adv. Sci. Inst. Ser. C Math. Phys. Sci. 514, 185--235, Kluwer Acad. Publ.,
%Dordrecht, 
1998.

% \bibitem[J2]{jantzen-sub}
% J.~C.~Jantzen, \emph{Subregular nilpotent representations of Lie algebras in prime characteristic}, Represent. Theory \textbf{3} (1999), 153--222.

\bibitem[J2]{jantzen}
J.~C.~Jantzen, \emph{Representations of algebraic groups, Second edition}, Mathematical surveys and monographs 107, Amer. Math. Soc., 2003.

\bibitem[J3]{jantzen-nilp}
J.~C.~Jantzen, \emph{Nilpotent orbits in representation theory}, in \emph{Lie theory}, 1--211,
Progr. Math. 228, Birkh\"auser Boston, 
%Boston, MA, 
2004.

\bibitem[JeW]{jw}
L.~T.~Jensen and G.~Williamson, \emph{The $p$-canonical basis for Hecke algebras}, in \emph{Categorification and higher representation theory}, 333--361, Contemp. Math. 683, Amer. Math. Soc., 
%%Providence, RI, 
2017.

\bibitem[JMW]{jmw}
D.~Juteau, C.~Mautner, and G.~Williamson, \emph{Parity sheaves},
J. Amer. Math. Soc. \textbf{27} (2014), 1169--1212.

%\bibitem[KL1]{kl}
%D.~Kazhdan and G.~Lusztig, \emph{Schubert varieties and Poincar\'e duality}, in \emph{Geometry of the Laplace operator (Proc. Sympos. Pure Math., Univ. Hawaii, Honolulu, Hawaii, 1979)}, 185--203,
%Proc. Sympos. Pure Math., XXXVI, Amer. Math. Soc., 
%%Providence, R.I., 
%1980. 

\bibitem[KL]{kl2}
D.~Kazhdan and G.~Lusztig,
\emph{Proof of the Deligne--Langlands conjecture for Hecke algebras},
Invent. Math. \textbf{87} (1987), 
%no. 1, 
153--215. 

\bibitem[KO]{ko}
M.-A.~Knus and M.~Ojanguren, \emph{Th\'eorie de la descente et alg\`ebres d'Azumaya},
Lecture Notes in Mathematics 389, Springer-Verlag, 
%Berlin-New York, 
1974.

\bibitem[Lu]{lusztig}
G.~Lusztig, \emph{Some problems in the representation theory of finite Chevalley groups}, in \emph{The Santa Cruz Conference on Finite Groups (Univ. California, Santa Cruz, Calif., 1979)}, 313--317,
Proc. Sympos. Pure Math. 37, Amer. Math. Soc., 
%Providence, R.I., 
1980. 

\bibitem[MT]{mackaay-thiel}
M.~Mackaay and A.-L.~Thiel, \emph{Categorifications of the extended affine Hecke algebra and the affine $q$-Schur algebra $\widehat{\mathbf{S}}(n,r)$ for $3 \leq r < n$}, Quantum Topol. \textbf{8} (2017), 
%no. 1, 
113--203.

\bibitem[MR]{mr}
C.~Mautner and S.~Riche, {\em Exotic tilting sheaves, parity sheaves on affine
  Grassmannians, and the Mirkovi{\'c}--Vilonen conjecture}, J. Eur. Math. Soc.~\textbf{20} (2018), 2259--2332.

% \bibitem[Mi]{milne}
% J.~S.~Milne, \emph{\'Etale cohomology}, Princeton Mathematical Series 33, Princeton University Press, %Princeton, N.J., 
% 1980.

\bibitem[MR]{mirkovic-rumynin}
I.~Mirkovi\'c and D.~Rumynin, \emph{Centers of reduced enveloping algebras}, Math. Z. \textbf{231} (1999), 
%no. 1, 
123--132. 

\bibitem[Pe]{peskine}
C.~Peskine, \emph{An algebraic introduction to complex projective geometry. I.
Commutative algebra}, Cambridge Studies in Advanced Mathematics 47, Cambridge University Press, 
%Cambridge, 
1996.

\bibitem[PS]{ps}
A.~Premet and S.~Skryabin, \emph{Representations of restricted Lie algebras and families of associative $\mathscr{L}$-algebras}, J. Reine Angew. Math. \textbf{507} (1999), 189--218.

% \bibitem[R1]{riche-geom}
% S.~Riche, \emph{Geometric braid group action on derived categories of coherent sheaves}, 
% with an appendix joint  with R.~Bezrukavnikov, 
% Represent. Theory \textbf{12} (2008), 131--169.

\bibitem[R2]{riche}
S.~Riche, \emph{Koszul duality and modular representations of semi-simple Lie algebras}, Duke Math. J.~\textbf{154} (2010), 31--134.

\bibitem[R3]{riche-kostant}
S.~Riche, \emph{Kostant section, universal centralizer, and a modular derived Satake equivalence}, Math. Z.~\textbf{286} (2017), 223--261.

\bibitem[RW1]{rw}
S.~Riche and G.~Williamson, \emph{Tilting modules and the $p$-canonical basis}, 
Ast\'erisque \textbf{397} (2018).

\bibitem[RW2]{rw-smith}
S.~Riche and G.~Williamson, \emph{Smith--Treumann theory and the linkage principle}, 
preprint~\href{https://arxiv.org/abs/2003.08522}{arXiv:2003.08522}.

\bibitem[Ro]{rouquier}
R.~Rouquier, \emph{$2$-Kac--Moody algebras}, preprint~\href{https://arxiv.org/abs/0812.5023}{arXiv:0812.5023}.

\bibitem[SGA1]{sga1}
\emph{Rev\^etements \'etales et groupe fondamental (SGA 1)}.
S\'eminaire de g\'eom\'etrie alg\'ebrique du Bois Marie 1960--61. Directed by A. Grothendieck. With two papers by M. Raynaud. Updated and annotated reprint of the 1971 original.
%[Lecture Notes in Math., 224, Springer, Berlin; MR0354651]. 
Documents Math\'ematiques 
%(Paris) [Mathematical Documents (Paris)], 
3, Soci\'et\'e Math\'ematique de France, 
%Paris, 
2003.

%\bibitem[Sl]{slodowy}
%P.~Slodowy, \emph{Simple singularities and simple algebraic groups},
%Lecture Notes in Mathematics 815, Springer, 
%%Berlin, 
%1980.

\bibitem[S1]{soergel-comb}
W.~Soergel, \emph{The combinatorics of Harish-Chandra bimodules},
J. Reine Angew. Math. \textbf{429} (1992), 49--74. 

\bibitem[S2]{soergel-philo}
W.~Soergel, \emph{Langlands’ philosophy and Koszul duality}, in \emph{Algebra---representation
theory (Constanta, 2000)}, 379--414, NATO Sci. Ser. II Math. Phys. Chem., vol. 28, Kluwer
Acad. Publ., 
%Dordrecht, 
2001.

\bibitem[S3]{soergel}
W.~Soergel, \emph{Kazhdan--Lusztig-Polynome und unzerlegbare Bimoduln \"uber Polynomringen}, 
J. Inst. Math. Jussieu \textbf{6} (2007), 
%no. 3, 
501--525. 

\bibitem[SP]{stacks-project}
The {Stacks project authors}, \emph{The Stacks project}, \url{https://stacks.math.columbia.edu}, 2020.

%\bibitem[Ts]{tsimerman}
%J.~Tsimerman, \emph{Introduction to Etale Cohomology}, notes available at~\url{http://www.math.toronto.edu/~jacobt/Etale.html}.

\bibitem[Ve]{verma}
D.-N.~Verma, \emph{The r\^ole of affine Weyl groups in the representation theory of algebraic Chevalley groups and their Lie algebras}, in \emph{Lie groups and their representations (Proc. Summer School, Bolyai J\'anos Math. Soc., Budapest, 1971)}, 653--705, Halsted, 
%New York, 
1975.

\bibitem[Wi]{williamson}
G.~Williamson, \emph{Parity sheaves and the Hecke category}, in \emph{Proceedings of the International Congress of Mathematicians--Rio de Janeiro 2018. Vol. I. Plenary lectures}, 979--1015, World Sci. Publ., 
%Hackensack, NJ, 
2018. 

\end{thebibliography}
\end{document}